\documentclass[11pt]{article}
\usepackage{bbm}
 \usepackage{amssymb}
\usepackage{amssymb, amsthm, amsmath, amscd}
\usepackage[usenames,dvipsnames]{color}

\newtheorem{thm}{Theorem}[section]
\newtheorem{lem}[thm]{Lemma}
\newtheorem{prop}[thm]{Proposition}
\newtheorem{cor}[thm]{Corollary}
\theoremstyle{definition}
\newtheorem{NN}[thm]{}
\theoremstyle{definition}
\newtheorem{df}[thm]{Definition}
\theoremstyle{definition}
\newtheorem{rem}[thm]{Remark}

\theoremstyle{definition}

\newcommand{\red}{\textcolor{red}}
\newcommand{\blue}{\color{blue}}

\renewcommand{\phi}{\varphi}

\newcommand{\aff}{\rm aff}

\newcommand{\N}{\mathbb{N}}
\newcommand{\Z}{\mathbb{Z}}
\newcommand{\Q}{\mathbb{Q}}
\newcommand{\R}{\mathbb{R}}
\newcommand{\C}{\mathbb{C}}

\numberwithin{equation}{section}

\newcommand{\Aff}{\operatorname{Aff}}
\newcommand{\Inf}{\operatorname{Inf}}

\newcommand{\id}{\operatorname{id}}

\newcommand{\cpc}{completely positive contractive linear map}
\newcommand{\morp}{contractive completely positive linear map}

\newcommand{\hm}{homomorphism}
\newcommand{\dt}{\delta}
\newcommand{\ep}{\varepsilon}
\newcommand{\et}{\eta}

\newcommand{\td}{\tilde}

\newcommand{\lr}{\longrightarrow}
\newcommand{\ld}{\lambda}
\newcommand{\cd}{\cdots}

\newcommand{\qq}{{\quad \quad}}
\newcommand{\sbs}{{\subset}}
\newcommand{\sps}{{\supset}}
\newcommand{\sbseq}{{\subseteq}}

\newcommand{\tht}{\theta}

\newcommand{\bb}{\mathfrak{b}}
\newcommand{\cc}{\mathfrak{c}}
\newcommand{\dd}{\mathfrak{d}}
\newcommand{\p}{\mathfrak{p}}
\newcommand{\q}{\mathfrak{q}}

\newcommand{\LD}{\Lambda}

\newcommand{\gm}{\gamma}
\newcommand{\sm}{\sigma}

\newcommand{\DT}{\Delta}

\newcommand{\la}{\langle}
\newcommand{\ra}{\rangle}
\newcommand{\andeqn}{\,\,\,{\rm and}\,\,\,}
\newcommand{\rforal}{\,\,\,{\rm for\,\,\,all}\,\,\,}
\newcommand{\CA}{$C^*$-algebra}
\newcommand{\SCA}{$C^*$-subalgebra}

\newcommand{\af}{{\alpha}}
\newcommand{\bt}{{\beta}}

\newcommand{\diag}{{\rm diag}}

\newcommand{\wilog}{without loss of generality}
\newcommand{\Wlog}{Without loss of generality}

\newcommand{\D}{\mathbb D}
\newcommand{\beq}{\begin{eqnarray}}
\newcommand{\eneq}{\end{eqnarray}}
\newcommand{\tforal}{\,\,\,\text{for\,\,\,all}\,\,\,}
\newcommand{\tand}{\,\,\,\text{and}\,\,\,}
\newcommand{\zo}{{\cal Z}_0}

\newcommand{\cm}{{\cal M}}
\newcommand{\LAff}{{\rm LAff}}

\newcommand{\ds}{\displaystyle}
\usepackage{amsfonts}
\usepackage{mathrsfs}
\usepackage{textcomp}
\usepackage[all]{xy}

\title{On classification  of  non-unital  amenable  simple  C*-algebras III,
the range and the reduction}
\author{Guihua Gong and  Huaxin Lin
 }
\date{
}

\begin{document}

\maketitle

\begin{abstract}
Following Elliott's earlier work, we  show that the Elliott invariant of any finite separable simple \CA\, with finite nuclear dimension
can always be described as a scaled simple ordered group pairing together with a countable abelian group
(which unifies the unital  and nonunital, as well as
stably projectionless cases).
We also show that, {{for any  given such invariant set,}} there is a finite separable simple \CA\,
whose Elliott invariant is the given set, a refinement of the  range theorem of Elliott in the stable case.
 In the stably projectionless  case,  modified
 model
\CA s are constructed {{in such a way}} that they
are of generalized tracial rank one and  have other technical features.

We also show that every stably projectionless
separable simple amenable \CA\, in the UCT class has rationally generalized tracial
rank one.

\end{abstract}

\section{Introduction}
This paper is a part of the Elliott program of classification of simple separable
\CA s with finite nuclear dimension (or, equivalently, simple separable amenable ${\cal Z}$-stable \CA s
by \cite{CETWW} and \cite{CE}).  In fact it is the first part of the research results which gives a unified classification of separable {{finite}} simple \CA s
{of  finite nuclear dimension
which satisfy the Universal Coefficient Theorem (UCT).

Briefly, a full classification theorem for a class of \CA s, say ${\cal A},$ consists
of three parts.  The first part is a description of the Elliott invariant for the \CA s in ${\cal A}$ {{(see \cite{Ellicm}).}}
The second part is the  range (or model) theorem, i.e., for any given Elliott invariant set described in the first part,
there is a model \CA\, (with certain desired properties)  in ${\cal A}$ such that its Elliott invariant set is the given one {{(\cite{point-line}).}}
The third part is the isomorphism theorem {{which asserts that}} any two \CA s in the class ${\cal A}$ are isomorphic if and only
if they have the same Elliott invariant.

Let ${\cal A}$ be the class of separable simple  amenable ${\cal Z}$-stable \CA s
in the UCT class.
This paper contains the first two parts  of the classification of \CA s in  {{${\cal A},$}}
together with a reduction theorem.
X. Jiang and H. Su constructed (\cite{JS})
 a unital infinite dimensional separable simple amenable
\CA\, ${\cal Z}$
which has
exactly the same Elliott invariant as the complex field $\C.$ It is {{known}}  (see
{{\cite[Corollary 4.12]{EGLN}}})
that the Jiang-Su algebra ${\cal Z}$ is the unique  unital infinite dimensional separable simple \CA\,  with finite nuclear dimension in the UCT class to have this property.
Let $A$ be any unital separable simple \CA\, with weakly unperforated  $K_0(A).$
Then $A\otimes {\cal Z}$ and $A$ {{have}} exactly the same Elliott invariant {{(see  Theorem 1 part (b) of \cite{GJS})}}. The invariant
set may be described as
a six-tuple $(K_0(A), K_0(A)_+, [1_A], T(A), \rho_A, K_1(A)),$ where $T(A)$ is the tracial state space
of $A$ and $\rho_A: K_0(A)\to \Aff(T(A))$ (the space of all real affine continuous functions on $T(A)$)
is a pairing.
Therefore, currently we study the  classification for simple separable ${\cal Z}$-stable \CA s.
The classification results {{for}} unital separable simple amenable ${\cal Z}$-stable \CA s in
the UCT {{class}} can be found in
 {{\cite{Pclass}, \cite{KP},}}
\cite{GLN}, \cite{GLN2}, \cite{EGLN} and \cite{TWW}.
For separable simple {{ \CA s}} $A$ with  $K_0(A)_+\not=\{0\},$
the classification can  be easily reduced to the unital case.

This paper mainly studies the case that $K_0(A)_+=\{0\}$  {{and $A$ is finite,}} i.e., the stably projectionless  case.
Note that a separable simple \CA\, $A$ {{with}} $K_0(A)_+=\{0\}$ {{could}}  still have interesting
$K_0(A).$ Moreover, one could also have a non-trivial pairing $\rho_A.$

Let $A$ be a non-unital separable simple \CA\, and  {{${\tilde T}(A)$ the}} set of
densely defined lower semi-continuous traces. Let $\td A$ be the unitization of $A$ and {{$\pi: \td A\to \C$ the}} quotient map.
Suppose that  $A$ is {{an}} algebraically simple \CA\,
and $p\in M_n(\td A)$ is a non-zero projection and $m$ is the rank of $\pi(p).$
One may define $\rho_A([p]-[1_m])(\tau)=\tau(p)-m\|\tau\|$ (for $\tau\in {\tilde T}(A)$)
which gives a pairing $\rho_A: K_0(A)\to \Aff({\tilde T}(A))$ (the set of all real  continuous affine
functions on ${\tilde T}(A)$).
On the other hand,  if  $A$ is
stable, then a naive extension
of the above pairing   may not make sense
as $\tau(p)=\infty$ (when $A$ is stably projectionless)
and $\|\tau\|=\infty.$  If one chooses a hereditary \SCA\, $B$ of $A$ which is algebraically simple,
then one may define a pairing $\rho_B: K_0(B)\to \Aff({\tilde T}(B)).$
However, one needs
to define a pairing which is independent of $B.$ A little more careful pairing is deployed.
As the first part of the classification, {{following Elliott's earlier work (see \cite{Ellicm}),}} {{we state that, combining  several previous results, including \cite{eglnkk0} and
\cite{GJS},}}  for any separable {{amenable}} simple ${\cal Z}$-stable \CA\,
($A\cong A\otimes {\cal Z}$),
 its Elliott invariant
may always be described by a scaled simple ordered group pairing together with a countable abelian group (see
Definition \ref{Dparing} and Theorem \ref{TRangeM}).   {{For the second part of the classification,
we modify  the range theorem first proved by George A. Elliott (in \cite{point-line})
in the stable case.}}

{{To be more specific,}}  in the stably projectionless case,
modifying Elliott's original construction,
we show that any Elliott invariant (including the scale) mentioned above can be realized by a model simple \CA\, which is of generalized tracial
rank one. {{Moreover,  for technical purposes, we show that these model simple \CA s have other technical properties
(see Theorem \ref{Tmodel2}, {{Remark \ref{Rmodle2}}} and Theorem \ref{TclassM1-K}).}} We also fill in some subtle details   in the Elliott construction  similar to the one  mentioned in the first line
of page 88 of \cite{point-line}.

{{Furthermore, {{following Corollary A7 of \cite{eglnkk0},}} we found that,  {{in the stably projectionless case,}} a pairing mentioned above actually gives a previously unexpected stronger feature {{of  weak unperforation}} (see Corollary \ref{Ct=0}), a feature {{that}} plays
an important technical
role in the later part of \cite{GLIII} (see, for example, Theorem \ref{TWtraceD}).}}

{{As mentioned above,}} the isomorphism theorem in \cite{GLIII} is first established for those separable simple
\CA s of finite nuclear dimension which have  rationally generalized tracial rank one.  In this paper,
with {{the}}  model theorem mentioned above, we also show that every separable simple stably projectionless
\CA\, with finite nuclear dimension in the UCT class actually has rationally generalized tracial rank one (Theorem \ref {Treduction}).

The paper is organized as follows.
Section 2 serves as preliminaries.
Section 3 discusses
the existence of ${\cal W}$-traces.  In Section 4, we first
{{construct}} a class of simple \CA s which are inductive limits
of 1-dimensional  non-commutative CW complexes with arbitrary simple pairings.
Then, together with the construction in \cite{GLII},  we construct simple \CA s  with arbitrary simple pairings
and arbitrary $K_1$-groups. These are simple \CA s
which are locally approximated
by {{sub-homogeneous}} \CA s whose spectra have {{dimension}} no more than 3.
 In Section 5,
we discuss the range {{of the invariant sets}} of stably finite separable simple ${\cal Z}$-stable \CA s,
and
show that the models constructed in Section 5  exhaust all possible {{values of}} Elliott invariant for those  separable simple \CA s.
In Section 6, we show that, in the UCT class,
all separable {{finite}} simple \CA s {{with}} 
 finite nuclear dimension  have rationally generalized tracial rank at most one (Theorem \ref{Treduction}).

{{{\bf{{{Acknowledgements}}}}:    This research
began when both authors stayed
in the Research Center for Operator Algebras in East China Normal University
in the summer of 2016 and 2017.
    Both authors acknowledge the support by the Center
 which is in part supported  by NNSF of China (11531003)  and Shanghai Science and Technology
 Commission (13dz2260400),
 and  Shanghai Key Laboratory of PMMP. {{The first named author was also supported by NNSF of China (11920101001).}}
The second named author was also supported by NSF grants (DMS 1665183 and DMS 1954600).}}
The main part of this research was first reported in East Coast Operator Algebra Symposium, October, 2017.
Both authors would like to express their sincere gratitude {{to George A. Elliott}} for many  conversations
as well as for  his moral support.
Much of the work of this paper is based on \cite{eglnp} and \cite{eglnkk0}. We would also like to take the opportunity
to thank  G.A. Elliott and Z. Niu {{for the collaboration}}.

\section{Preliminaries}

\begin{df}\label{Dher}
Let $A$ be a \CA.  Denote by $A^{\bf 1}$  {{the}} unit ball of $A.$
Let  $a\in A_+.$  Denote by  ${\rm Her}(a)$ the hereditary \SCA\, $\overline{aAa}.$
If $a, b\in A_+,$ we write $a\lesssim b$ ($a$ is Cuntz smaller than $b$),
if there exists a sequence of $x_n\in A$ such that $a=\lim_{n\to\infty} x_n^*x_n$
and $x_nx_n^*\in {\rm Her}(b).$  {{If both $a\lesssim b$ and $b\lesssim a$, then we say $a$ is Cuntz equivalent to $b$.}} The Cuntz equivalence class represented by $a$ will
be denoted by $\la a\ra.$
A projection $p\in M_n(A)$ defines {{an element}} $[p]\in K_0(A)_+$.  {{We}} will also use $[p]$ to denote the Cuntz equivalence class represented by $p.$

\end{df}

\begin{df}\label{DTtilde}
Let $A$ be a \CA.  Denote by $T(A)$ the {{tracial state space}} of $A$ {{(which could be {{the}} empty set).}}
Let $\Aff(T(A))$ be the space of all real valued affine continuous functions on $T(A)$.
Let ${\tilde{T}}(A)$ be the cone of densely defined,
positive lower semi-continuous traces on $A$ equipped with the topology
of point-wise convergence on elements of the Pedersen ideal  ${\rm Ped}(A)$ of $A.$
{{
So $\td T(A)$ may be viewed as { {the}} cone {{in the}}  dual space of {{the}}
vector space ${\rm Ped}(A).$}}

Let $B$ be another \CA\, with $\td T(B)\not=\{0\}$
and let $\phi: A\to B$ be a \hm.  {{Since ${\rm Ped}(A)$ is the minimal dense ideal of $A,$
 $\phi({\rm Ped}(A))\subset {\rm Ped}(B).$}}
{{\it{In what follows we will also write  $\phi$ for $\phi\otimes \id_{M_k}: M_k(A)\to M_k(B)$
whenever it is convenient.}}}

We will write $\phi_T: \td T(B)\to \td T(A)$ for the induced continuous affine map.
Denote by $\td T^{b}(A)$ the subset of $\td T(A)$ which are bounded on $A.$
Of course $T(A)\subset \td T^b(A).$ Set $T_0(A):=\{t\in \td T(A): \|\tau\|\le 1\}.$ It is a compact convex subset of $\td T(A).$

Let $r\ge 1$ be an integer and $\tau\in {\tilde T}(A).$
We will continue to write $\tau$  on $A\otimes M_r$ for $\tau\otimes {\rm Tr},$ where ${\rm Tr}$ is the standard { {(unnormalized)}}
trace on $M_r.$
Let  $S$ be a convex subset (of a convex topological cone).
{{We assume  that {{a convex}} cone contains {{$0$, but a convex set $S$ may or may not contain $0$.}}  Denote by
$\Aff(S)$ the set of affine continuous functions on $S$ with the property that, if  $0\in S,$ then
$f(0)=0$ for all $f\in \Aff(S).$}}
Define (see \cite{Rl})
\beq
\Aff_+(S)&=&\{f: \Aff(S):  
 f(\tau)>0\,\,{\rm for}\,\,\tau\not=0\}\cup \{0\},\\
{\rm LAff}_{f,+}(S)&=&\{f:S\to [0,\infty): {{\exists\, \{f_n\},}}\,  f_n\nearrow f,\,\,
 f_n\in \Aff_+(S)\},\\
{\rm LAff}_+(S)&=&\{f:S\to [0,\infty]: {{\exists\, \{f_n\},}}\, f_n\nearrow f,\,\,
 f_n\in \Aff_+(S)\}\andeqn\\
 {{{\rm LAff}^{\sim}}}(S) &=&\{f_1-f_2: f_1\in {\rm LAff}_+(S)\andeqn f_2\in
 {\rm Aff}_+(S)\}.
 \eneq
 For the {{great}} part of this paper, {{$S={\tilde T}(A)$,}}  $S=T(A),$ or $S\subset T_0(A)$
in the above definition will be used.
Moreover, {{for $S\subset T_0(A),$ ${\rm LAff}_{b,+}(S)$}} is the subset of those bounded functions
 in ${\rm LAff}_{f,+}(S).$
Recall {{that}} $0\in {\tilde T}(A)$ and if $f\in \LAff(\td T(A)),$ then $f(0)=0.$


\end{df}

%
%

\begin{df}\label{DElliott23}
A convex topological cone  $T$ is a subset of a topological vector space such that for any $\af, \bt\in \R_+$ and $x, y\in T$, $\af x+\bt y\in T$, where $\R_+$ is the set of nonnegative real numbers.  A subset $\DT\subset T$ is called a base of $T$, if $\DT$ is convex and for any $x\in T\setminus \{0\}$, there is a unique pair ${{(\af_x, \tau_x)}}\in (\R_+\setminus \{0\})\times \DT$  
such that $x={{\af_x\tau_x}}$. {{In this article, {\it all convex topological cones are those with a metrizable Choquet {{simplex}} $\DT$ as its base.} Note also {{that}} the function from  $T\setminus \{0\}$ to $(\R_+\setminus \{0\})\times \DT$  sending $x$ to ${{(\af_x, \tau_x)}}$ is continuous.}}

A simple ordered group pairing is a triple $(G, T,  \rho),$ where
$G$ is a countable abelian group, $T$ is  {{a convex  topological cone}} with a Choquet simplex as its base
and
$\rho: G\to \Aff(T)$ {{is a \hm.}}
{\em In what follows,
for a pair of functions $f$ and $g$ on $T,$  we write $f>g$ if $f(\tau)>g(\tau)$ for all
$\tau\in T\setminus \{0\}.$}
Define $G_+=\{g\in G: \rho(g)>0\}\cup \{0\}.$ If $G_+\not=\{0\},$
then $(G, G_+)$ is an ordered group. It has the property that if $ng>0$ for some integer $n>0,$
then $g>0.$  In other words, $(G, G_+)$ is weakly unperforated. %
{{Moreover,}} if $G_+\not=\{0\},$
$(G, G_+)$ is a simple ordered group, i.e.,  every element $g\in G_+\setminus \{0\}$
is  an order unit {{(recall that the compact set $\Delta$ is a base for $T$).}}
{{In general, we  allow the case $G_+=\{0\}$ (but {{$\rho$
may not be $0$}}).}}

A scaled simple ordered group pairing is a quintuple $(G, \Sigma(G), T, s, \rho)$ such that
$(G, T, \rho)$ is a simple ordered group pairing,
where  $s\in \LAff_+(T)\setminus \{0\}$  {{and}}
\beq
\Sigma(G):=\{g\in G_+: \rho(g)<s\}\,\,{\rm or}\,\,\,
\Sigma(G):=\{g\in G_+: \rho(g)<s\}\cup\{u\},
\eneq
{{where}}  {{$u\in G_+$ and}} $\rho(u)=s.$
%
%
We allow $\Sigma(G)=\{0\}.$   Note also that $s(\tau)$ could be infinite for some $\tau\in T.$
It is called a unital scaled simple ordered group pairing, if $\Sigma(G)=\{g\in G_+: \rho(g)<s\}\cup \{u\}$
with $\rho(u)=s,$ in which case, $u$ is called the unit of $\Sigma(G).$
Note that, in this case,
$u$ is the maximum element of $\Sigma(G),$
and, one may write
$(G, u, T, \rho)$ for $(G, \Sigma(G), T, s,\rho).$
If $\Sigma(G)$ has no  unit {{(which includes the case that $\Sigma(G)$ has a maximum element $x$ but $\rho(x)<s$),}} then $\Sigma(G)$ is determined by $s.$
One may write $(G, T, s, \rho)$ for $(G, \Sigma(G), T, s, \rho)$
(see Theorem \ref{TRangeM} below).  (Note that $\Sigma(G)=\{0\}$ corresponds to {{the}} projectionless  case
and $G_+=\{0\}$ to {{the}} stably projectionless case).

Let $(G_i, \Sigma(G_i), T_i, s_i, \rho_i)$ be {{scaled}} simple ordered group pairings, {{$i=1,2.$}}
A map
$$\Gamma_0{{=(\kappa_0,\kappa_T)}}: (G_1, \Sigma(G_1), T_1, s_1, \rho_1)\to (G_2, \Sigma(G_2), T_2, s_2, \rho_2)$$
is said to be a \hm, if there is a group \hm\,
$\kappa_0: G_1\to G_2$ and a continuous cone map $\kappa_T: T_2\to T_1$
{{(preserving $0$)}}
such that
\beq
{{\rho_2}}(\kappa_0(g))(t)={{\rho_1(g)}}(\kappa_T(t))\rforal g\in G_1\andeqn t\in T_2,\andeqn\\
\kappa_0(\Sigma(G_1))\subset \Sigma(G_2{{),}} \andeqn   s_1(\kappa_T(t))\le s_2(t) \rforal t\in T_2.
\eneq
We say a \hm\, $\Gamma_0$ is an isomorphism
if $\kappa_0$ is an isomorphism, $\kappa_0(\Sigma(G_1))={{\Sigma(G_2)}},$ $\kappa_T$ is a cone
homeomorphism, and $s_1(\kappa_T(t))=s_2(t)$ for all $t\in T_2.$
\end{df}

\begin{df}\label{Dfep}
For any $\ep>0,$ define $f_\ep\in {{C([0,\infty))_+}}$ by
$f_\ep(t)=0$ if $t\in [0, \ep/2],$ $f_\ep(t)=1$ if $t\in [\ep, \infty)$ and
$f_\ep(t)$ is linear in $(\ep/2, \ep).$

Let $A$ be a \CA\, and $\tau$ be a quasitrace.
For each $a\in A_+$ define $d_\tau(a)=\lim_{\ep\to 0} {{\tau(f_\ep(a))}}.$
Note that $f_\ep(a)\in {\rm Ped}(A)$ for all $a\in A_+.$

{{Let $S$ be a  convex subset of $\tilde{T}(A)$ and $a\in M_n(A)_+.$
The function $\hat{a}(s)=s(a)$ (for $s\in S$) is an  affine function {{from $S$ to $[0,\infty]$}}.
Define $\widehat{\la a\ra }(s)=d_s(a)=\lim_{\ep\to 0}s(f_\ep(a))$ (for $s\in S$)
which is a lower semicontinuous function.
If $a\in {\rm Ped}(A)_+,$ then {{the map $\tau \mapsto \hat{a}(\tau)$}}
is in $\Aff_+(S)$
and
$\widehat{\la a\ra}\in \LAff_+(S)$ (see \ref{DTtilde}), in general. {{Note
that $\hat{a}$ is different from $\widehat{\la a\ra}.$}}
In most cases, $S$ is ${\tilde T}(A),$ $T_0(A),$ or $T(A).$}}
Note {{also}} that, there is a  canonical
 map from ${\rm Cu}(A)$ to {{${\rm LAff}_+(\tilde{T}(A))$}} {{sending}} $\la a \ra$ to $\widehat{\la a \ra}$.

\end{df}

\begin{NN}\label{range4.1}
{{If $A$ is a unital \CA\, and $T(A)\not=\emptyset,$ then
there is a {{canonical}} 
\hm\, $\rho_A: K_0(A)\to \Aff(T(A)).$}}

{{Now consider the case that $A$ is not unital.
Let $\pi_\C^A: \td A\to \C$ be the quotient map.
Suppose that $T(A)\not=\emptyset.$
Let $\tau_\C:=\tau_\C^A: \td A\to \C$ be
the tracial state which factors through $\pi_\C^A.$
Then
\beq
T(\td A) = \big\{t{{\tau_{_\C}^A}}+(1-t)\tau:~ t\in [0,1],~ \tau\in {{T(A)}}\big\}.
\eneq
The map ${{T(A)}}\hookrightarrow {{T(\td A)}}$ induces a map {{${{\Aff(T(\tilde{A}))}}\to \Aff(T(A))$}}. Then
 the  map {{$\rho_{\td A}: K_0({{\td A}})\to \Aff(T(\td A))$}} induces a \hm\,  $\rho':~ K_0(A)\to \Aff (T(A))$ by
\beq\label{April-10-2021}
{{\rho':~K_0(A)\to  K_0({{\td A}})\stackrel{\rho_{\td A}}{\longrightarrow}
\Aff(T(\td A))   \to \Aff(T(A)).}}
\eneq
However, in the case that $A\not={\rm Ped}(A),$ we will not use $\rho'$  in general, as it is possible
that $T(A)=\emptyset$ but ${\td T}(A)$ is rich (consider the case $A\cong A\otimes {\cal K}$).}}
\end{NN}

\begin{df}\label{Dparing}
Let $A$ be a \CA\, with $\td T(A)\not=\{0\}.$
If $\tau\in \td T(A)$ is bounded on $A,$ then $\tau$ can be extended naturally
to {{a trace}} on $\td A.$
Recall that $\td T^b(A)$ is the set of bounded traces on $A.$
Denote by $\rho_A^b: K_0(A)\to \Aff(\td T^b(A))$ the \hm\, defined by
$\rho_A^b([p]-[q])=\tau(p)-\tau(q)$ for all $\tau\in \td T^b(A)$
and for projections $p, q\in M_n(\td A)$ (for some integer $n\ge 1$)
{{with}} $\pi_\C^A(p)=\pi_\C^A(q).$ {{Note that}} $p-q\in  M_n(A).$
Therefore $\rho_A^b([p]-[q])$ is continuous on $\td T^b(A).$
In the case that $\td T^b(A)=\td T(A),$ for example, $A={\rm Ped}(A),$
we write $\rho_A:=\rho_A^b.$

Let $A$ be a $\sigma$-unital \CA\, with a strictly positive element $0\le e\le 1.$
Put  $e_n:=f_{1/2^n}(e).$ Then $\{e_n\}$ forms an approximate identity
for $A.$ Note  {{that}} $e_n\in {\rm Ped}(A)$ for all $n.$
Set $A_n={\rm Her}(e_n):=\overline{e_nAe_n}.$
Denote by $\iota_n:A_n\to A_{n+1}$   and $j_n: A_n\to A$ the embeddings. {{They extend}} {{to}}
$\iota_n^\sim: \td A_n\to \td A_{n+1}$ and $j_n^\sim :\td A_n\to \td A$ {{unitally.}}
Note that $e_n\in {\rm Ped}(A_{n+1}).$ Thus  $\iota_n$  and $j_n$ induce
continuous cone maps ${\iota_n}_T^b: \td T^b(A_{n+1})\to
\td T^b(A_n)$ and ${j_n}_T: \td T(A)\to \td T^b(A_n)$
(defined by ${\iota_n}_T^b(\tau)(a)=\tau(\iota_n(a))$
for $\tau\in {{\td T^b}}(A_{n+1}),$ and ${j_n}_T(\tau)(a)=\tau(j_n(a))$
for all $\tau\in \td T(A)$ and all $a\in A_n$), respectively.
Denote by $\iota_n^\sharp: \Aff(\td T^b(A_n))
\to {{\Aff(\td T^b(A_{n+1}))}}$  and $j_n^\sharp: \Aff(\td T^b(A_n))\to \Aff(\td T(A))$ the induced continuous linear maps.
Recall that $\cup_{n=1}^\infty A_n$ is dense in ${\rm Ped}(A).$
A direct computation shows that one may obtain  the following  {{inverse}} limit of {{convex  topological cones}} (with continuous cone maps):
\beq\label{Drho-trace}
\td T^b(A_1)\stackrel{{\iota_1}_T^b}
{\longleftarrow} \td T^b(A_2)\stackrel{{\iota_2}_{T}^b}{\longleftarrow} \td T^b(A_3)
\cd\longleftarrow\cd\longleftarrow \td T(A).
\eneq

{{To justify \eqref{Drho-trace}, set $T'=\lim_{\leftarrow} \td T^b(A_n)\subset \prod_{n=1}^\infty \td T^b(A_n).$
Define $\Gamma: \td T(A)\to T'$ by $\Gamma(\tau)=\{\tau_n\},$ where $\tau_n=\tau|_{A_n}.$
Let $I_n$ be the (closed two-sided) ideal of $A$ generated by $A_n.$
Let  {{$\sim_T$}} be the equivalence relation {{$\sim$}} defined in section 2 of \cite{CP}. Recall that
 $\{x\in (I_n)_+: \exists y\in {A_n}_+\, {\text{such\,\,that}}\, x\sim_T y\}$ is  a positive part of an ideal
 {{$J_n$}}
 containing $A_n$ {{(see the Remark after Proposition 4.7 of \cite{CP}{{).}}}}
 Therefore,
if $t\in \td T(A)$ and $t|_{A_n}=\tau|_{A_n},$ then $t|_{J_n}=\tau|_{J_n}.$
Put $J=\cup_{n=1}^\infty J_n.$
It follows that $t|_{{J}}=\tau|_{{J}}.$ Since $J$ is a dense ideal,
it contains ${\rm Ped}(A).$ Hence $t$ and $\tau$ are the same element in $\td T(A).$
This implies that the map $\Gamma$ is injective. It is also clear that $\Gamma$ is a continuous cone
map.  To see that $\Gamma$ is surjective, let $\{\tau_n\}\in T'.$  Recall that $\tau_n\in \td T^b(A_n)$
and ${\tau_{n+1}}|_{A_n}=\tau_n.$  Let $\td \tau_n$ be a trace in $\td T(A)$ which extends $\tau_n$
(see Lemma 4.6 of \cite{CP}). Then ${\td \tau_{n+1}}|_{J_n}={\td \tau_n}|_{J_n}$ (as argued  above).
Define $\td \tau$ on $J$ by ${{\td \tau}}|_{J_n}={\td \tau_n}|_{J_n}.$  Since $\td \tau$
is finite on $\cup_{n=1}^\infty A_n$ which is in ${\rm Ped}(A)$ and  {{dense}} in $A,$ $\td \tau$
is a lower semicontinuous densely defined trace.
Then  one  may view $\td \tau\in \td T(A),$ and note also
$\Gamma({{\td \tau}})=\{\tau_n\}.$ This shows that $\Gamma$ is surjective.}}

{{To see $\Gamma$ is open, consider $O_{a, >\bt}=\{\tau\in \td T(A): \tau(a)>\bt\},$ where $a\in {\rm Ped}(A)_+\setminus \{0\}$
and $\bt\in \R.$ Then
$O_{a,{{>}}\bt}=\cup_k\{\tau: \tau(e_kae_k)>\bt\}.$ Thus
$\Gamma(O_{a, >\bt})=\cup_k\{\{\tau_n\}\in T': \tau_k(e_kae_k)>\bt\},$ which is open in
the product topology.  {{Now consider $O_{a,<\bt}=\{\tau\in  \td T(A): \tau(a)<\bt\}.$
Since $a\in {\rm Ped}(A)\subset J,$ we may assume that $a\in J_N$ for some $N\ge 1.$
There is {{an element}} $b\in (A_{{N}})_+$ such that $b\sim_T a.$  Note that $\{\tau\in \td T^b(A_{{N}}): \tau(b)<\bt\}$ is open.
It follows that $O:=\{\{\tau_k\}\in  \prod_{k=1}^\infty \td T^b(A_{{k}}): \tau_N(b)<\bt\}$ is open
in $\prod_{k=1}^\infty \td T^b(A_n).$  Since $\Gamma(O_{a,<\bt})=O\cap T',$ it is also open in $T'.$}}
This implies that $\Gamma$ is open. }}

Note that the continuous cone map ${j_n}_T$ is  the  same as the cone map ${\iota_{\infty, n}}_T: \td T(A)\to \td T^b(A_n)$
given by the {{inverse  limit.}}  One also obtains the
induced inductive limit:
\beq
\Aff(\td T^b(A_1))\stackrel{\iota_1^\sharp}
{\longrightarrow} \Aff (\td T^b(A_2))\stackrel{{\iota_2}^\sharp}{\longrightarrow} \Aff (\td T^b(A_3))
\cd\longrightarrow\cd\longrightarrow \Aff (\td T(A)).
\eneq
Hence one also has the following commutative diagram:
\begin{displaymath}
    \xymatrix{
        K_0(A_1) \ar
        [d]_{\rho_{A_1}}\ar[r]^{\iota_{1*o}} & K_0(A_2) \ar[r]^{\iota_{2*o}} \ar
        [d]_{\rho_{A_2}}& K_0{{(A_3)}} \ar[r]
        \ar
        [d]_{\rho_{A_3}}& \cd K_0({{A}} )\\
        \Aff(\td T^b({{A}}_1))
        \ar[r]^{\iota^{\sharp}_{1,2}}
        &
         \Aff(\td T^b({{A}}_2)) \ar[r]^{\iota_2^{\sharp}}
         &
         \Aff(\td T^b({{A}}_3))
         \ar[r]
         & \cd
         \Aff(\td T({{A}})).}
\end{displaymath}
Thus one obtains a \hm\, $\rho: K_0(A)\to \Aff(\td T(A)).$
If $A$ is assumed to be simple, then each $A_n$ is simple
and $e_n$ is full in $A_n.$ Therefore, since $e_n\in {\rm Ped}(A)$ {{and}}  $A_n={\rm Ped}(A_n)$ (see 2.1 of \cite{T-0-Z}).
It follows {{that,}}  when $A$ is simple,
$\td T^b(A_n)=\td T(A_n)$ for all $n.$  {{But we do not assume that $A$ is simple in general.}}

Note {{that,}} if $\td T(A)=\td T^b(A),$ for any $n\ge 1,$ one also has the following commutative diagram:
{\small{\begin{displaymath}
    \xymatrix{
        K_0(A_n) \ar
        [d]_{\rho_{A_n}}\ar[r]^{\iota_{n*0}} & K_0(A_{n+1}) \ar[r]^{j_{n+1,*0}} \ar
        [d]_{{{\rho_{A_{n+1}}}}}& K_0{{(A)}}\ar[d]_{\rho_A}\\
        \Aff(\td T^b({{A}}_n))
        \ar[r]^{{\iota_n}^{\sharp}}
        &
         \Aff(\td T^b({{A}}_{n+1})) \ar[r]^{{j_{n+1}}^{\sharp}}
         &
         \Aff(\td T({{A}})) .}
\end{displaymath}}}
It follows that $\rho=\rho_A$ in the case that $\td T^b(A)=\td T(A).$

Let us briefly point out that the definition {{above}} does not depend on the choice of $e.$
Suppose that $0\le e'\le 1$ is another strictly positive element.
We similarly define $e_n'.$   Put $A_n'={\rm Her}(e_n').$
Note that, for any $m,$ there is $k(m)\ge n$ such that $\|e-e_{k(m)}'ee_{k(m)}'\|<1/2^m.$
By applying a result of R\o rdam (see, for example,  Lemma 3.3 of \cite{eglnp} and its proof),
one has, for each $n\ge 1,$ {{an integer $k(n) > k(n-1)\ge 1$}} {{and}} a partial isometry  $w_n\in A^{**}$ such that $w_nw_n^*e_n=e_nw_nw_n^*,$
 $w_n^*cw_n\in A_{k(n)}'$ for $c\in A_n$ and
 $\|w_ne_n-e_n\|< {{1/2^n}}.$
 {{Define $\phi_n: A_n\to {{A_{k(n)}'}}$ by $\phi_n(c)=w_{n+1}^* cw_{n+1}$ for $c\in A_n.$
 For each $m,$ let $E_{n,m}=e_n\otimes 1_{M_m}$ and $W_{n,m}=w_n\otimes 1_{M_m}.$
 Then $\|W_{n,m}E_{n,m}-E_{n,m}\|<1/2^n.$
 It follows {{that}} $\|(\phi_n\otimes \id_{M_m})(c)-c\|<{{(1/2^n)}}\|c\|$ for all $c\in M_m(A_n)$
 and $m\in \N.$}}
Moreover, for any $\tau\in \td T(A),$
$\tau(\phi_n(c))=\tau(c)$ for all $c\in A_n.$
Symmetrically, one has {{monomorphisms}}
$\psi_k: A_k'\to A_{N(k)}$ such that $\|{{(\psi_k\otimes \id_{M_m})}}(a)-a\|<(1/2^{{k}})\|a\|$
{{and $\tau(\psi_{{k}}(a))=\tau(a)$}}
{{for all $a\in M_m(A_k'),$}} $\tau\in \td T(A)$  {{and $m\in \N.$}}
Thus, by passing to  subsequences, one obtains the following  commutative diagram:
{\tiny{
\begin{displaymath}
    \xymatrix{
    K_0(A_1)\ar[rrd]^{\hspace{0.1in}\rho_{A_1}^b} \ar[r]^{\iota_{1*0}}\ar[ddd]^{\phi_{1*0}}\,\,\,& ~~~K_0(A_2)
    \ar[rrd]^{\rho^b_{A_{2}}} \ar[r]^{\iota_{2*0}}&\cdots&\lr\cdots
    & K_0(A){{\ar[ddd]^{\id_{K_0(A)}}}}\ar[rrd]^{\rho}&&\\
       && \Aff(\td T^b(A_1)) \ar[ddd]^{\phi_1^{\sharp}}\ar[r]^{\hspace{-0.25in}\iota_1^{\sharp}} & \Aff(\td T^b(A_{2})) \ar[r]^{~~~~~\iota_2^{\sharp}}~~~~
       & 
     \cdots  &\lr\cdots&  \Aff(\td T({{A}} ))\ar[ddd]^{\id_{\Aff(\td T(A))}}\\
       &&&&&&\\
     K_0(A_1')\ar[rrd]^{\rho_{A_1'}^b}\ar[r]^{\iota'_{1*0}} & ~~~K_0(A_2')\ar[uuu]_{\psi_{2*0}}\ar[rrd]^{\rho_{A_2'}^b}\ar[r]^{\iota'_{2*0}}&\cdots&\hspace{0.1in}\lr\cdots&K_0(A)\ar[rrd]^{\rho'}\ar[uuu] &&\\
       && \Aff(\td T^b({{A}}_1')) \ar[r]^{\hspace{-0.1in}{\hspace{-0.2in}\iota_1'}^{\sharp}} &
         \Aff(\td T^b({{A}}_{2}')) \ar[uuu]^{\psi^\sharp}\ar[r]^{~~~~~{\iota_2'}^{\sharp}}~~~~
         &
         \cdots&\lr\cdots&
         \Aff(\td T({{A}}))\ar[uuu],}
\end{displaymath}}}
%
%
This implies that $\rho'=\rho,$ where $\rho'$ is induced by choosing $e'$ instead of $e.$

Throughout,
we will define $\rho:=\rho_A.$
%
%
Moreover, this definition of pairing
 is consistent with {{the}} conventional definition of $\rho_A$ in the case
that {{$A$ is unital, or the case
$A
 ={\rm Ped}(A)$ (see \ref{range4.1}).}}

We {\it will write} $\pi_{\aff}^{\rho, A}: \Aff(\td T(A))\to \Aff(\td T(A))/\overline{\rho_A(K_0(A))}$
for the quotient map.   This may be simplified to $\pi_{\aff}^\rho$ if $A$ is clear.
 When $T(A)\not=\emptyset,$ we will use the same notation for
the quotient map $\Aff(T(A))\to \Aff(T(A))/\overline{\rho_A(K_0(A))}.$
In this case, we also write  $\rho_A^\sim:  K_0(\td A)\to \Aff(T(A))$ for the map
defined by $\rho_A^\sim([p])(\tau)=\tau(p)$ for projections $p\in M_l({{\td A}})$ (for all integer $l$) 
and for $\tau\in T(A).$

Suppose that $\Phi: A\to B$ is a \hm. Then $\Phi({\rm Ped}(A))\subset {\rm Ped}(B).$
Let $0\le e_A\le 1$ and $0\le e_B\le 1$ be strictly positive elements of $A$ and $B,$  {{respectively.}}
Let $e_n^A=f_{1/2^n}(e_A)$ and $e_n^B=f_{1/2^n}(e_B)$ {{be as}} defined above.
Define $A_n={\rm Her}(e_n^A)$ and $B_n={\rm Her}(e_n^B).$
Then $$\lim_{n\to\infty}\|\Phi(e_A)-e_n^B\Phi(e_A)e_n^B\|=0.$$ By passing to a subsequence,
as above (applying Lemma 3.3 of \cite{eglnp} repeatedly),
one has a sequence of \hm s {{$\Psi_n: A_n\to B_{k(n)}$}} 
 such that
{{$\|\iota_{B_{k(n)}}\circ\Psi_n(a)-\Phi\circ \iota_{A_n}(a)\|<(1/n)\|a\|$}} 
 and {{$\tau(\iota_{B_{k(n)}}\circ\Psi_n(a))=\tau(\Phi\circ \iota_{A_n}(a))$}} 
for all $a\in  {{M_m(A_n)}}$ {{(for every $m\in \N$)}} and
$\tau\in \td T(B)$ {{(recall that we {{write}} $H$ for $H\otimes \id_{M_k}$).}}  Drawing a similar diagram as above, one obtains the following commutative diagram:
{\small{\beq\label{2020-8-8-d1}
  \xymatrix{
K_0(A)~~~~\ar[rr]^{\rho_{A}} \ar[d]^{\Phi_{*0}}
 &&
~~~\Aff(\td T(A)) \ar@{->}[d]^{\Phi^\sharp}
\\
 K_0(B)~~~ \ar[rr]^{\rho_{B}}
 && ~~~\Aff(\td T(B)).}
\eneq}}
{{At least in the simple case, the construction above was pointed out by Elliott (see part (iv) of the subsection 7  of \cite{Ellicm} as well as Proposition \ref{Prhowd} below for more detail).}}
The above also works when we do not assume
that quasitraces are traces (but quasitraces would be used).
A more general definition will be left to avoid longer discussion.
\end{df}

\begin{df}\label{DElliott}
We now describe the Elliott invariant  for separable simple \CA s {{(see {{\cite{Ellicm} and}} \cite{point-line}).}}
%
%
%
Let us consider the case  ${\tilde T}\not=\{0\}.$
In this case the Elliott invariant is the six-tuple:
$$
{\rm{Ell}}(A):=((K_0(A), \Sigma(K_0(A)), {\tilde T}(A), \Sigma_A, \rho_A), K_1(A)),
$$
where $\Sigma(K_0(A))=\{x\in K_0(A): x=[p]\,\, {{{\rm for\,\, some\,\, projection}\,\, p}}\in A\},$
and $\Sigma_A$ is a function in $\LAff_+(\td T(A))$ defined by
\beq
\Sigma_A(\tau)=\sup\{\tau(a): a\in {{{\rm Ped}(A)_+,}}\,\,\|a\|\le 1\}
\eneq
(see \ref{TRangeM}). Let $e_A\in A$ be a strictly positive element.
Then $\Sigma_A(\tau)=\lim_{\ep\to 0} \tau(f_\ep(e_A))$ for all $\tau\in \td T(A),$ which
is independent of the choice of $e_A.$

Let $B$ be another separable  \CA.  A map $\Gamma: {\rm Ell}(A)\to {\rm Ell}(B)$
is a \hm\, if $\Gamma$ gives  group {{\hm s}} $\kappa_i: K_i(A)\to K_i(B)$ ($i=0,1$) {{and}}
a continuous {{cone map}}
$\gamma: \td {{T(B)\to T(A)}}$
 such
that $\rho_B(\kappa_0(x))(\tau)=\rho_A(x)(\gamma(\tau))$ for all $x\in K_0(A)$ and $\tau\in \td T(A),$
$\kappa_0(\Sigma(K_0(A)))\subset \Sigma(K_0(B))$ and $\Sigma_A(\gamma(\tau))\le \Sigma_B(\tau)$
for all $\tau\in \td T(B).$

We say {{that}} $\Gamma$ is an isomorphism,
if $\Gamma$ is a \hm,\,   $\kappa_i$ is a group
isomorphism ($i=0,1$),
 $\gamma$ is a cone homeomorphism, $\kappa_0(\Sigma(K_0(A))=\Sigma(K_0(B)),$
and  $\Sigma_A(\gamma(\tau))=\Sigma_B(\tau)$ for all $\tau\in \td T(B).$

In the case that $\rho_A(K_0(A))\cap {\rm LAff}_+(\td T(A))=\{0\},$ we often consider
the (special) reduced case that $T(A)$ is compact which gives a base for $\td T(A).$
{{Then, we}}
may write ${\rm Ell}(A)=(K_0(A), T(A), \rho_A, K_1(A)).$ Note {{that,}}  in {{this situation,}}
$\Sigma(K_0(A))=\{0\},$
$\td T(A)$ is determined by $T(A)$ and $\Sigma_A(\tau)=1$ for all $\tau\in T(A).$

\end{df}

\begin{df}{{(\cite{Rlz})}}\label{Dastablerk1}
Let $A$ be a \CA. We say $A$ has almost stable rank one,
if {{$A$}}
has the following property:
the set of  invertible elements of  {{the unitization}} $\widetilde B$ of
every {{nonzero}} hereditary \SCA\, $B$ of {{$A$}} is dense in $B.$
{{$A$ is said to stably {{have}} almost stable rank one,
if $M_n(A)$ has almost stable rank one for all integer $n\ge 1.$}}
\end{df}

\begin{df}\label{Dlambdas}
Let $A$ be a \CA\, with $T(A)\not=\emptyset.$
Suppose that $A$ has a strictly positive element $e_A\in {\rm Ped}(A)_+$ with $\|e_A\|=1.$
Then $0\not\in \overline{T(A)}^w,$ the closure of $T(A)$ in ${\tilde T}(A)$
(see Theorem 4.7  of \cite{eglnp}).
Define
\beq\nonumber
&&\lambda_s(A)=\inf\{d_\tau(e_A): \tau\in  {{\overline{T(A)}^w}}\}
=\lim_{n\to\infty}(\inf\{\tau(f_{1/n}(e_A)): \tau\in  {{T(A)}}\})>0.
\eneq
 Let $A$ be a \CA\, with $T(A)\not=\{0\}.$
There is an affine  map
$r_{\aff}: A_{s.a.}\to \Aff(T_0(A))$
defined by
$$
r_{\aff}(a)(\tau)=\hat{a}(\tau)=\tau(a)\tforal \tau\in T_0(A)
$$
and for all $a\in A_{s.a.}.$ Denote by $A_{s.a.}^q$ the space  $r_{\aff}(A_{s.a.})$ and
$A_+^q=r_{\aff}(A_+).$

\end{df}

\begin{df}\label{Dstrongaue}
Let $A$ and $B$ be two  \CA s.  A sequence of
linear maps $L_n: A\to B$ is said be approximately multiplicative
if
$$
\lim_{n\to\infty}\|L_n(a)L_n(b)-L_n(ab)\|=0\rforal a, b\in  A.
$$
Let
$\phi, \psi: A\to B$ be \hm s. We say $\phi$ and $\psi$ are asymptotically unitarily
equivalent if there is a continuous path of unitaries $\{u(t): t\in [1, \infty)\}$ in
$B$ (if $B$ is not unital, $u(t)\in \td B$) such
that
$$
\lim_{t\to\infty} u^*(t)\phi(a) u(t)=\psi(a)\rforal a\in A.
$$
We say $\phi$ and $\psi$ are
strongly  asymptotically unitarily equivalent if $u(1)\in  U_0(B)$ (or in $U_0(\td B)$).
\end{df}

\begin{df}\label{THfull}
Let $A$ and $B$ be \CA s, and let $T: A_+\setminus\{0\}\to \N\times \R_+\setminus \{0\}$  {{be}}
defined by $a\mapsto (N(a), M(a)),$ where $N(a)\in \N$ and $M(a)\in \R_+\setminus\{0\}.$
Let ${\cal H}\subset A_+\setminus \{0\}.$ A map $L: A\to B$ is said to be
$T$-${{\cal H}}$-full, if, for  any $a\in {\cal H}$ and any $b\in  B_+$ with $\|b\|\le 1,$  any $\ep>0,$ there are $x_1, x_2,...,x_N\in B$
with $N\le N(a)$ and $\|x_i\|\le M(a)$  such that
\beq
\|\sum_{j=1}^N {{x_j^*L(a)x_j}}-b\|\le \ep.
\eneq
$L$ is said to {{be}}  exactly $T$-${{\cal H}}$-full, if $\ep=0$ in the above formula.

\end{df}

 \begin{df}
 \label{DfC1}
{\rm
Let $A$ and $B$ be \CA s and $\phi_0, \phi_1: A\to B$ be \hm s.
By  {{the}}  mapping torus  $M_{\phi_0, \phi_1},$ we mean the following
\CA:
\beq\label{dmapping}
M_{\phi_0, \phi_1}=\{(f,a)\in C([0,1], B)\oplus A: f(0)=\phi_0(a)\andeqn f(1)=\phi_1(a)\}.
\eneq

One has the
short exact sequence
\begin{equation*}\label{Mtoruses}
0\to SB\stackrel{\imath}{\to}M_{{{\phi, \psi}}} \stackrel{\pi_e}{\to} A\to 0,
\end{equation*}
where $\imath: SB\to M_{{{\phi, \psi}}}$ is the embedding and $\pi_e$ is the
quotient map from $M_{{{\phi, \psi}}}$ to $A.$ { {Denote by  $\pi_t: M_{{{\phi, \psi}}}\to B$  the point evaluation at $t\in [0,1].$}}

Let $F_1$ and $F_2$ be two finite dimensional \CA s.
Suppose that there are  (not necessary unital)  \hm s
$\phi_0, \phi_1: F_1\to F_2.$
Denote the mapping torus $M_{\phi_1, \phi_2}$ by
$$
A=A(F_1, F_2,\phi_0, \phi_1)
=\{(f,g)\in  C([0,1], F_2)\oplus F_1: f(0)=\phi_0(g)\andeqn f(1)=\phi_1(g)\}.
$$

Denote by ${\cal C}$ the class of all  \CA s of the form $A=A(F_1, F_2, \phi_0, \phi_1).$
These \CA s {{are called  Elliott-Thomsen building blocks as well as
one  dimensional non-commutative CW complexes  (see  \cite{ET-PL} and  \cite{point-line}).}}

Recall that ${{\cal C}_0}$ is the class of all $A\in {\cal C}$
with $K_0(A)_+=\{0\}$ such that $K_1(A)=0$   and
$\lambda_s(A)>0,$ and ${\cal C}_0^{(0)}$ the class of all $A\in {\cal C}_0$ such that $K_0(A)=0.$  Denote by {{${\cal C}'$,}} ${\cal C}_0'$  and ${\cal C}_0^{0'}$ the class of all full hereditary \SCA s of \CA s in {{$\cal C$,}} ${\cal C}_0$ and
${\cal C}_0^{{(0)}},$ respectively.
}
\end{df}


 \begin{df}\label{DD0}{{(cf. 8.1 and 8.2 of \cite{eglnp})}}
{{Recall}} the definition of class ${\cal D}$ and ${\cal D}_0.$

 Let $A$ be a non-unital simple \CA\,  with a strictly positive element $a\in A$
with $\|a\|=1.$   Suppose that there exists
$1> \mathfrak{f}_a>0,$ for any $\ep>0,$  any
finite subset ${\cal F}\subset A$ and any $b\in A_+\setminus \{0\},$  there are ${\cal F}$-$\ep$-multiplicative \cpc s $\phi: A\to A$ and  $\psi: A\to D$  for some
\SCA\, $D\subset A$ with $D\in {\cal C}_0'$ (or ${\cal C}_0^{0'}$) {{such that}} $D\perp \phi(A),$ and
\beq\label{DNtr1div-1+++}
&&\|x-(\phi(x)+\psi(x))\|<\ep\rforal x\in {\cal F}\cup \{a\},\\\label{DNtrdiv-2}
&&c\lesssim b,\\\label{DNtrdiv-4}
&&t(f_{1/4}(\psi(a)))\ge \mathfrak{f}_a\rforal t\in T(D),
\eneq
where $c$ is a strictly positive element of $\overline{\phi(A)A\phi(A)}.$
  Then
 we say $A\in {\cal D}$ (or ${\cal D}_0$).

 {{ Note {{that,}}  by  Remark 8.11 of \cite{eglnp},
 $D$ can {\it always} be chosen to be in ${\cal C}_0$ (or ${\cal C}_0^{{(0)}}$).}}

{{When $A\in {\cal D}$ and is separable, {{then}}  $A=\mathrm{Ped}(A)$ (see 11.3  of \cite{eglnp}).
Let $a\in A_+$ with $\|a\|=1$ be a strict positive element.
Put
\beq\label{1225dd}
d=\inf\{\tau(f_{1/4}(a)): \tau\in T(A)\} {{>0}}.
\eneq
Then, for any $0<\eta<d,$ $\mathfrak{f}_a$ can be chosen to  be
$d-\eta$ (see  Remark 9.2 of \cite{eglnp}).
One may also assume that $f_{1/4}(\psi(a))$ is
{{full}} in $D.$
Furthermore,  there exists a map:  $T: A_+\setminus \{0\}\to \N\times \R_+\setminus \{0\}$
which is independent of
${\cal F}$ and $\ep$ such that, for any  finite subset ${\cal H}\subset A_+\setminus \{0\},$ we can further  require that $\psi$ is exactly $T$-${\cal H}$-full (see 8.3
and 9.2
of \cite{eglnp}).
}}
For any $n\ge 1,$ one can choose a strictly positive element $b\in A$ with $\|b\|=1$ such that
$f_{1/4}(b)\ge f_{1/n}(a).$ Therefore, if $A$
 has continuous scale, $d$ can be chosen to be {{$1$
if}} the strictly positive element is chosen accordingly.

In \cite{eglnp}, it is proved that  if $A$ a separable simple \CA\, in ${\cal D},$
then $A$ {{is stably projectionless,}}
has stable rank one  and ${\rm Cu}(A)=\LAff_+(\td T(A)),$  and,  every 2-quasitrace
on $A$ is a trace {{(see 9.3 and 11.11 of \cite{eglnp}).}}

Let $A$ be a {{nonzero}} separable stably projectionless simple \CA.  Recall that  $A$ has generalized tracial rank
one, {{written}} $gTR(A)=1,$ if there exists $e\in {\rm Ped}(A)_+$ with $\|e\|=1$ such that
$\overline{eAe}\in {\cal D}$ (see  11.6 of \cite{eglnp}).  It should be noted that, in  the definition of $D$ above,
if we assume that $A$ is unital, and replace ${\cal C}_0$ by ${\cal C},$ then $gTR(A)\le 1$ (see 9.1,
9.2 and 9.3 of \cite{GLN}). But {{the}} condition \eqref{DNtrdiv-4} and constant $\mathfrak{f}_a$ are not needed.
In  the case $K_0(A)_+\not=\{0\}$ but $A$ is not unital,  we may define $gTR(A)\le 1,$ if
for some nonzero projection {{$e\in M_k(A),$ $gTR(eM_k(A)e)\le 1$}} (see \cite{GLN}).
%
%
%
 \end{df}

 \begin{df}\label{DD1}
 Let $A\in {\cal D}$ {{be}} as defined {{in}} {{Definition}} \ref{DD0}.
 If, in  addition,
 for any integer $n,$ {{we can choose $D$ and ${{\psi: A\to D}}$ to satisfy the following condition:}}
 $D=M_n(D_1)$ for some $D_1\in {\cal C}_0$ such that
\vspace{-0.14in} \beq\label{DD1-1}
 \psi(x)=\diag(\overbrace{\psi_1(x), \psi_1(x),...,\psi_1(x)}^n)\rforal x\in {\cal F},
 \eneq
where $\psi_1: A\to D_1$ is an ${\cal F}$-$\ep$-multiplicative \cpc, then we say $A\in {\cal D}^d.$

Note that here, as in  8.3 and 9.2 of \cite{eglnp}, the map  $T$ mentioned in \ref {DD0}  is also assumed
{{to exist}} and
$\mathfrak{f}_a$ can be also
chosen as $d-\eta$ for any $\eta>0$ with $d$ as in \eqref{1225dd} for a certain strictly positive element $a.$

\end{df}

\begin{rem}\label{RDd}
  It follows  from
  {{10.4 and 10.7 of \cite{eglnp}}} that, if $A\in {\cal D}_0,$  then $A\in {\cal D}^d.$ Moreover,
  $D_1$ can be chosen in {{${\cal C}_0^{(0)},$}}
  and if $A\in {\cal D},$ then $D_1$ can be chosen in ${\cal C}_0.$ 
  If $A$ is a separable simple \CA\, in ${\cal D}$ and $A$ {{has an approximate divisible property defined in}}
  10.1 of \cite{eglnp}),
  then $A\in {\cal D}^d.$
\end{rem}

\begin{df}\label{DWtrace}
Throughout the paper,
${\cal W}$ is the separable simple \CA\, which is an inductive limit
of \CA s in ${\cal C}_0^{{(0)}}$ with a unique tracial state, which is first constructed in \cite{RzW}.
It is proved in \cite{eglnkk0} that ${\cal W}$ is the unique separable simple \CA\, with finite
nuclear dimension which is $KK$-contractible and with a unique tracial state.
Denote by $\tau_W$ the unique tracial state of ${\cal W}.$

Let $A$ be a \CA\, and let $\tau$ be a nonzero trace of $A.$
We say {{that}} $\tau$ is a ${\cal W}$-trace, if there exists a sequence of approximately multiplicative \cpc s
$\phi_n: A\to {\cal W}$ such that
\beq
\lim_{n\to\infty}\tau_W\circ \phi_n(a)=\tau(a)\rforal a\in A.
\eneq
{{Throughout, $Q$ will be the UHF-algebra with $K_0(Q)=\Q$ and $[1_Q]=1$ and with the unique tracial state ${\rm tr}.$
Recall that $T_{\rm{qd}}(A)$ is the set of those $\tau\in T(A)$ such that
there exists a sequence of approximately multiplicative \cpc s
$\phi_n: A\to Q$ such that
\beq
\lim_{n\to\infty}{\rm tr}\circ \phi_n(a)=\tau(a)\rforal a\in A.
\eneq
}}
\end{df}

\begin{df}[9.3 of \cite{GLII}]\label{DpropertyW}
%
Let $A$ be a separable \CA.  We say  {{that}} $A$ has property (W), if there is a map $T: A_+^{\bf 1}\setminus \{0\}\to \N\times \R_+\setminus \{0\}$ and
a sequence of approximately multiplicative \cpc s $\phi_n: A\to {{\cal W}}$ such
that, for any finite subset ${\mathcal H}\subset A_+^{\bf 1}\setminus \{0\},$
there exists an integer $n_0\ge 1$ such that
$\phi_n$ is exactly $T$-${\mathcal H}$-full (see 5.5   and 5.7  of \cite{eglnp}) for all $n\ge n_0.$

\end{df}

\section{W-traces}

\begin{thm}\label{Tfeature}
{{Let $\Delta$ be a compact  convex set
and let $G$ be
a countable abelian subgroup of $\Aff(\Delta).$
Suppose
that
$G\cap \Aff_+({{\DT}})=\{0\}.$
Then there exists $t\in \Delta$ such that $g(t)=0$ for all $g\in G.$}}

\end{thm}


\begin{proof} Let us assume that $0\notin \DT$, otherwise we can choose $t=0$.
Let $S_+=\{f\in  \Aff(\Delta): f(x)>0\rforal x\in \DT\}={{\Aff_+(\DT)}}\setminus\{0\}.$ It is an open convex {{subset}} of $\Aff(\Delta).$
Let $G_1$ be the convex hull of $G.$
Note that, if $g_1,g_2\in G$ and $r\in \Q$  with $0<r<1,$ then
$rg_1+(1-r)g_2\not\in S_+.$ To see this, we note that there is an integer $m\ge 1$ such
that $mr$ and  $m(1-r)$ are both integers.
In other words, $m(rg_1+(1-r)g_2)\in G.$ Therefore $m(rg_1+(1-r)g_2)\not\in S_+.$
Hence $rg_1+(1-r)g_2\not\in S_+.$ Since $S_+$ is open, this implies that $G_1\cap S_+=\emptyset.$

By the Hahn-Banach separating theorem, there is a real continuous linear functional  $f$ on $\Aff(\Delta)$
and $r_0\in \R$
such that
\beq\label{Tfeature-5}
f(s)<r_0\le f(g)\rforal s\in S_+\andeqn g\in G_1.
\eneq
If $f(g)\ge r_0,$ then $f(-g)\le -r_0.$
Note that, if  $g\in G,$ {{then $mg\in G$ for any $m\in \Z$. Hence $mf(g)\geq r_0$ for all $m\in \Z$. It follows that $f(g)=0$ for all $g\in G$ and $r_0\leq 0$.}}

By  \eqref{Tfeature-5},
\beq
-f(s)>{{-r_0\geq 0}}  \rforal s\in  S_+.
\eneq
Let $f'=-f.$ Since ${{\Aff_+(\Delta)}}=S_+\cup\{0\}\subset \overline{S_+},$
\beq
f'(s)\ge 0\rforal s\in
{{\Aff(\Delta)_+=\{s\in \Aff(\Delta): s(\tau)\ge 0\}.}}
\eneq
In other words, $f'$ is {a positive linear functional} on $\Aff(\Delta).$
Let $f_1=f'/\|f'\|.$ Then $f_1(1_\DT)=1.$
Consider $\Aff(\Delta)\subset C(\Delta).$
By the Hahn-Banach extension theorem
there is a linear functional ${\bar f}_1$ on $C(\Delta)$
such that $({\bar f}_1)|_{\Aff(\Delta)}=f_1,$ and $\|\bar{f}_1\|=\|f_1\|=1$.
 Since $\|\bar{f}_1\|={\bar f}_1(1_\DT)=1,$
${\bar f}_1$ is a positive functional (see Proposition 3.1.4 of \cite{Pbook}),  and therefore it is a  state of $C(\DT).$
%
%
%
%
%
Let $S(C(\Delta))$ be the state space. Then it is compact and convex.
By {{the}} Krein-Milman Theorem, ${\bar f}_1$ is the limit of $\{\mu_n\},$
where $\mu_n=\sum_{i=1}^{m(n)} \af_{n,i} \rho_{n,i},$
$0\le \af_{n,i}\le 1$ are positive numbers with $\sum_{i=1}^{m(n)}\af_{n,i}=1,$ and
${{\rho_{n,i}}}$ are pure states of $C(\Delta).$ Note that,
for each $i$ and $n,$ there is $t_{i,n}\in \Delta$ such that
${{\rho_{n,i}}}(a)=a(t_{i,n})$ for all $a\in C(\Delta).$ Since $\Delta$ is convex,
$\tau_n=\sum_{i=1}^{m(n)}\af_{i,n}t_{i,n}\in \Delta.$ Then
$a(\tau_n)=a(\mu_n)$ {{for all $a\in \Aff(\Delta),$  {{by the following computation:}}
$$
a(\mu_n)=\sum_{i=1}^{m(n)}\af_{n,i}a(\rho_{n,i})=\sum_{i=1}^{m(n)}\af_{n,i}a(t_{i,n})=a(\sum_{i=1}^{m(n)}\af_{n,i}t_{i,n})
=a(\tau_n) \rforal a\in \Aff(\Delta)).
$$}}
 Since $\Delta$ is compact,
we conclude that there is $\tau\in \Delta$ such that
\beq
{\bar f}_1(a)=a(\tau)\rforal a\in \Aff(\Delta).
\eneq


We have just shown that
\beq
g(\tau)=0\rforal g\in G.
\eneq


\end{proof}


\begin{cor}[compare with Corollary A.7 of \cite{eglnkk0}]\label{Ct=0}
Let $A\in {\cal D}$ be a separable simple \CA\, with continuous scale.
Then, there exists $t_o\in T(A)$ such that
$\rho_A(x)(t_o)=0$ for all $x\in K_0(A).$
\end{cor}

\begin{proof}
By 11.5 and 11.8 of \cite{eglnp}, $A$ has stable rank one and ${\rm Cu}(A)={\rm LAff}_+(T(A)).$
It follows from Corollary A.7 of \cite{eglnkk0} {{that}} $\rho_A(K_0(A))\cap \Aff_+(T(A))=\{0\}.$
By Theorem \ref{Tfeature}, there is $t_o\in T(A)$ such that {{$\rho_A(x)(t_o)=0$}} for all $x\in K_0(A).$
\end{proof}

\begin{thm}\label{TWtraceD}
Let $A\in {\cal D}$ be a separable simple \CA\, with continuous scale. Then
$A$
has at least one ${\cal W}$-tracial state.
Moreover, $A$ has property (W), i.e., there is a map $T: A_+\setminus \{0\}\to \N\times \R_+\setminus \{0\}$ and
a sequence of approximately multiplicative \cpc s $\phi_n: A\to {\cal W}$ such that, for any
finite subset ${\cal H}\subset  A_+\setminus \{0\},$ there exists $n_0\ge 1$ such that
$\phi_n$ is exactly $T$-${\cal H}$-full {{(see \ref{THfull})}} for all $n\ge n_0,$ and  there exists a
$\tau\in T(A)$ such that
$$
\tau(a)=\lim_{n\to\infty} t_W\circ \phi_n(a) {{\tforal}} a\in A.
$$

\end{thm}

\begin{proof}
{{Fix a strictly positive element $e\in A_+.$}}
It follows from {{9.2 of \cite{eglnp}}} that, for any $1/2>\eta>0,$ we may choose
$\mathfrak{f}_e>1-\eta$  in  {{\ref{DD0}.}} 
By  {{\ref{DD0},}} 
we {{obtain}} two sequences of  mutually orthogonal \SCA s $B_n, C_n\subset A,$
$B_n=\overline{a_nAa_n}$ for some
positive elements $a_n\in A$ with $\|a_n\|=1,$ $C_n\in {\cal C}_0,$ and two sequences
of \cpc s $\phi_{n,0}: A\to B_n$ and $\phi_{n,1}: A\to C_n$ such  that $\phi_{n,0}(A)\perp C_n,$
\beq\label{TWtraceD-5}
\lim_{n\to\infty}\|x-(\phi_{n,0}(x)+\phi_{n,1}(x))\|=0\rforal x\in A,\\\label{TWtraceD-6}
\lim_{n\to\infty}\|\phi_{n,i}(xy)-\phi_{n,i}(x)\phi_{n,i}(y)\|=0\rforal x, y\in A,\,\,\, i=0,1,\\\label{TWtraceD-7}
\lim_{n\to\infty}\sup\{d_\tau(a_n): \tau\in T(A)\}=0\andeqn\\\label{TWtraceD-8}
\tau(f_{1/2}(\phi_{n,1}(e)))>1-\eta\rforal \tau\in T(C_n).
\eneq

Let $b_n=f_{1/2}(\phi_{n,1}(e))\in C_n.$
Then $0\le {{b_n}}\le 1.$  The inequality \eqref{TWtraceD-8} implies
that
\beq\label{310}
\inf\{\tau(b_n): \tau\in T(C_n)\}>1-\eta.
\eneq
Hence $\lambda_s(C_n)>1-\eta$ {{(see \ref{Dlambdas}).}}

{{As reminded in the middle of \ref{DD0},   $A$ is stably projectionless, and has stable rank one,
${\rm Ped}(A)=A$ (see 11.3 of \cite{eglnp}),
$QT(A)=T(A)$ and ${\rm Cu}(A)=\LAff_+(\td T(A)).$  By
Theorem 7.3 of \cite{eglnp} (see also Theorem 6.2.3 of \cite{Rl} and 6.11 of \cite{RS}),
${\rm Cu}^\sim(A)=K_0(A)\sqcup \LAff_+^\sim (\td T(A)).$
By \ref{Ct=0}, there {{is}}  $t_D\in T(A)$ such that   $\rho_A(x)(t_D)=0$
for all $x\in K_0(A).$
Recall that (see Theorem 6.2.3 of \cite{Rl}) ${\rm Cu}^\sim({\cal W}) =\{0\}\sqcup (\R\cup\{\infty\}).$
Define $\gamma: {\rm Cu}^\sim(A)\to {\rm Cu}^\sim({\cal W})$ {{by}}
$\gamma|_{K_0(A)}=0$ and $\gamma(f)(t_W)=f(t_D)$ for $f\in \LAff_+^\sim(\td T(A)),$
where $t_W$ is the unique tracial state
of ${\cal W}.$   It is ready to see that $\gamma$ is  a  morphism
in ${\bf Cu}.$}}    {{Note that $\gamma(\la c\ra)\not=0$ for any $c\in A_+\setminus \{0\}$ as $A$ is simple.
Let $\iota_n: C_n\to A$ be the embedding.
By Theorem 1.0.1 of \cite{Rl}, there exists, for each $n,$
a \hm\,
$\psi_n: C_n\to {\cal W}$
such that ${\rm Cu}^\sim(\psi_n)=\gamma\circ {\rm Cu}^\sim(\iota_n).$
In particular, by \eqref{310}, $d_{t_W}(\psi_n(e_{C_n}))>1-\eta,$ where $e_{C_n}$ is a strictly positive element of $C_n.$
 Since $\iota_n$ is injective,
${\rm Cu}^\sim(\psi_n)(\la c\ra)\not=0$ for any $c\in {C_n}_+\setminus \{0\}.$ It follows that $\psi_n$ is injective.}}
%
{{Define $\Phi_n=\psi_n\circ \phi_{n,1}: A\to {\cal W}.$}}
Then $\Phi_n$ is a  {{sequence of}} \cpc s
satisfying the following:
\beq
\lim_{m\to\infty}\|\Phi_m(xy)-\Phi_m(x)\Phi_m(y)\|=0\rforal x, y\in A,\\
\lim_{m\to\infty}\|\Phi_m(x)\|=\|x\|\rforal x\in A\andeqn\\\label{TWtraceD-10}
\lim_{m\to\infty}\tau_W(f_{1/2}(\Phi_m(e)))\ge (1-\eta).
\eneq
Note that this holds for each $1/2>\eta>0.$
{{By choosing}} $\eta_n\to 0,$
{{we may}} further assume that,
there exists an increasing sequence $\{e_n\}$ of positive elements
with $0\le e_n\le 1$ such that, in the above,
{{we  have}}
\beq
\lim_{n\to\infty} \tau_W(\Phi_m(e_m))=1.
\eneq
Then, by passing to a subsequence, we may assume that there exists
$\tau\in T(A)$
such that
\beq
\lim_{n\to\infty} \tau_W(\Phi_m(a))=\tau(a)\rforal a\in A.
\eneq

The last part of the statement follows from the first part {{and Lemma 5.7}} of \cite{eglnp}.

\end{proof}
%
%
%
%
%
%
%
%


\begin{thm}\label{Tkk}
Let $C$ be a separable amenable \CA\, with the property (W),
let $A$ be a separable simple \CA\, with continuous scale which satisfies the UCT,
and has the property (W),
and let $\kappa\in KL(C, A).$
Then there exists a sequence of \cpc s $\{\phi_n\}$ from $C$ to $A\otimes M_{m(n)}$
(for some integers $m(n)$) such that
\beq
\lim_{n\to\infty}\|\phi_n(a)\phi_n(b)-\phi_n(ab)\|=0 {{\tforal}} a, b\in C,\,{{\tand}}
{[}\{\phi_n\}{]}=\kappa.
\eneq
\end{thm}

\begin{proof}
Note that $C$ satisfies the  conditions   in 9.3 of \cite{GLII}. Therefore the theorem follows from
the combination of Theorem 10.8 and 12.5 of \cite{GLII}, i.e.,  {{one first applies 10.8 of \cite{GLII} to obtain
maps from  $C$ to $A\otimes \zo\otimes M_{k(n)}$ and, then  applies 12.5 of \cite{GLII} to obtain
maps from $A\otimes \zo\otimes M_{k(n)}$ to $A\otimes M_{l(n)}$ (for some integers $k(n)$ and $l(n)$).}}


\end{proof}


\section{Range and Models}

This section is a refinement of Elliott's construction of model simple \CA s.  The main
results are stated as Theorem \ref{Tmodel2} and Theorem \ref{TclassM1-K}. These refinements are needed for our purposes. 
Some subtle details in the Elliott's construction are also dealt with.

\begin{NN}\label{range4.4} Let $\DT$ be any compact metrizable Choquet simplex {{and}}  let $G$ be any countable abelian group. Let $\rho: G\to  \Aff (\DT)$ be any homomorphism satisfying the following condition
$\\$\hspace{0.2in}(*)\hspace{0.3in} for any $g\in G$, there is a $\tau\in \DT$ such that
$\rho(g)(\tau)\leq 0.$\\
In other words, $\rho(G)\cap \Aff_+(\DT)\setminus \{0\}=\emptyset$
 (recall
that $\Aff_+(\DT)$
denotes the set  of all continuous affine functions $f: \DT \to \R$ such that $f(\tau)>0$ for any $\tau\in \DT$
and  {{the}} {{zero}} function). Note that we include the case that $\rho(G)=\{0\}$.

In  the first part of this section, we will assume that
$G$ is torsion free,
 and construct a stably projectionless simple $C^*$-algebra $A$ with continuous scale such that $K_0(A)=G$, $K_1(A)={{\{0\}}}$, ${{T(A)}}=\DT$ and  $\rho_{_A}: K_0(A)\to \Aff (T(A))$ is the map $\rho: G\to \Aff(\DT)$, when one identifies $K_0(A)$ with $G$, and ${{T(A)}}$ with  $\DT$.  {{The}} $C^*$-algebra $A$ is an inductive limit of
{{ \CA s $A_n\in {\cal C}_0$}} {{(Elliott-Thomsen building blocks)}} of the form
$$
{{A_n}}=A(F_n, E_n, \bt_{n,0}, \bt_{n,1}):=\Big\{ (f, a)\in C([0,1], E_n)\oplus F_n~ | ~ \bt_{n,0}(a)=f(0),  \bt_{n,1}(a)=f(1)  \Big\},
$$
where $F_n, E_n$ are finite dimensional $C^*$-algebras, and $\bt_{n,0}, \bt_{n,1}: F_n \to E_n$ are
(not necessarily unital)  {{homomorphisms.}}

Note that if $\bt_{n,0}\oplus\bt_{n,1}: F_n \to E_n\oplus E_n$ is injective, then the element $a$ is completely determined by $f$, so we can simply write $(f,a)$ as $f$. In our construction, we will always {{be}} in this situation. 

 Since the limit algebra $A$ to be constructed is {{stably}} projectionless, in each {{step,}} the algebra $A_n$
 {{will also be}}   stably projectionless.  {{The construction presented here}}  is {{a refinement of Elliott's construction in
 \cite{point-line} and is}}
 similar to \S 13 of  \cite{GLN} which is for the unital case.

\end{NN}

\begin{NN}\label{4.4.1}
  Let us keep the notation in 13.1 and 13.2 of \cite{GLN}.  In particular, $x^{\sim k}$ means
  ${{\{\overbrace{x,x,\cd,x}^{k}\}}}$ (see 13.1 of \cite{GLN}).  Let $A$ be an Elliott-Thomsen building block. Denoted by $Sp(A)$ {{the}} set of the equivalence classes of all irreducible representations of $A$, and $RF(A)$ the set of finite
  dimensional representations of $A$.  As in 13.1 and 13.2 of \cite{GLN}, each element of $RF(A)$ can be regarded as a subset of $Sp(A)$ with multiplicities.  For any homomorphism $\phi: A\to M_k(\C)$, let
  $Sp(\phi)=\{ x\in Sp(A); ~{\rm ker}(\phi)\,\sps\, {\rm ker}(x)\}$. 


  Suppose that $\phi$ is (unitarily equivalent to) a direct sum
  of $k_1$ copies of $x_1,$ $k_2$ copies of $x_2,$ ..., and $k_i$ copies of $x_i,$ where
  $x_1, x_2,...,x_i\in Sp(A),$ then we will write $SP(\phi):=\left\{ x_1^{\sim k_1},x_2^{\sim k_2},\cd, x_i^{\sim k_i} \right\}.$
  Note that if $\phi, \psi: A\to M_k$ are two \hm s then
  $\phi$ and $\psi$ are unitarily equivalent if and only if $SP(\phi)=SP(\psi).$

  %
  %
  %
%
%
Let $C$ be a  vector space  and $x=(x_1,x_2,...,x_n),$ where $x_i\in C,$ $i=1,2,...,n.$
For each integer $k\ge 1,$ consider $k$-tuple $S=(i_1, i_2, ...,i_k),$ where $i_s\in \{1,2,...,n\}.$
We write
$$
{\rm diag}_{j\in S}(x_j):=\diag(x_{i_1}, x_{i_2},...,x_{i_k})
$$
{{for a}} diagonal element in  $M_k(C).$
 In particular, we  use $\diag_{1\leq j\leq n}(x_j)$ to denote $\diag(x_1,x_2,\cd,x_n).$

We will adopt the following convention:
  $$\diag\big(\diag(a,b), \diag(c,d,e)\big)=\diag(a,b,c,d,e),$$
$$\diag\big(\diag_{1\leq i\leq 3}(a_i), \diag_{1\leq i\leq 2}(b_i)\big)=\diag(a_1,a_2,a_3,b_1,b_2),$$
$$\diag\big(\diag(a_1^{\sim 2}, a_2^{\sim 3},a_3), \diag(b_1, b_2^{\sim3})\big)=\diag(a_1,a_1,a_2,a_2,a_2,a_3,b_1,b_2,b_2,b_2)$$
%
%
%
%

As in 13.1 and 13.2 of \cite{GLN}, for any two sub-homogeneous algebras  $A$ and $B,$ and a  homomorphism $\phi: A\to B$, if $\tht\in Sp(B)$ {{is}} represented by $\tht: B\to M_k(\C)$, then we use $\phi|_{\tht}$ to denote
$\tht\circ \phi:~A \to M_k(\C)$.

\end{NN}

\begin{NN}{\label{4.6}}
Let us fix some notation for this section.  Let $\displaystyle F_n=\bigoplus_{i=1}^{p_n}M_{[n,i]}(\C),
~E_n=\bigoplus_{i=1}^{l_n}M_{\{n,i\}}(\C)$, where $p_n, l_n, [n,i], \{n,i\}$ are positive integers which will be constructed later.  Let $\bt_{n,0}, \bt_{n,1}: F_n\to E_n$ be two (not necessarily unital) homomorphisms.
{{Put}} $$A_n=A(F_n,E_n, \bt_{n,0},\bt_{n,1})
=\big\{(f,a)\in C([0,1], E_n)\oplus F_n; ~
\bt_{n,0}(a)=f(0), \bt_{n,1}(a)=f(1)\big\}.$$
%
Write $(f,a)\in A_n$ as $(f_1,f_2,\cd,f_{l_n}; a_1,a_2,\cd, a_{p_n}),$ {{where}} $f_i\in C\big([0,1], M_{\{n,i\}}(\C)\big)$
and ${{a_j}}\in M_{{{[n,j]}}}(\C)$.
Let $\eta_{n,i}(t), \tht_{n,j}\in Sp(A_n),~ 0<t<1,~  i=1,2,\cd,l_n,~j=1,2,\cd,p_n,$ be defined as:
$$\eta_{n,i}(t)(f,a)=f_i(t)\in M_{\{n,i\}}(\C)\sbs E_n;~~
{{\tht_{n,j}}}(f,a)=a_j\in M_{[n,j]}(\C)\sbs F_n.$$
{{We}} also use the notation $\eta_{n,i}(0)$ and $\eta_{n,i}(1)$ with
$$\eta_{n,i}(0)(f,a)=f_i(0)\andeqn \eta_{n,i}(1)(f,a)=f_i(1).$$
But $\eta_{n,i}(0), \eta_{n,i}(1)\in RF(A_{{n}})$ (rather than in $Sp(A_{{n}})$).
 {{Sometimes}} we use $(f,a)(\eta_{n,i}(t))$ and $(f,a)(\tht_{n,j})$ to denote $\eta_{n,i}(t)(f,a)$ and $\tht_{n,j}(f,a)$,
 {{respectively,}}
 or simply use $f(\eta_{n,i}(t)), f(\tht_{n,j})$ without $a$, as in this paper, $a$ is completely determined by $f$.  
 Moreover, we may write
\beq\label{SpAn}
Sp(A_n)=(\sqcup_{j=1}^{l_n} \{\eta_{n,j}(t): t\in (0,1)\})\sqcup \{\tht_{n,1},\tht_{n,2},...,\tht_{n, p_n}\}.
\eneq

Let $\dt>0.$  Let $S=(\sqcup_{j=1}^{l_n} S_j)\cup\{\tht_{n,1},\tht_{n,2},...,\tht_{n, p_n}\},$
and let $S_j=\{\eta_{n,j}(t): t\in T_j\subset (0,1)\},$ where $T_j$ is a $\dt$-dense subset of $(0,1).$
Then we say {{that}} $S$ is $\dt$-dense in $Sp(A_n).$ Let ${\cal F}\subset A_n\setminus \{0\}$ be
a finite subset.   Then, for all sufficiently small $\dt,$  if $S$ is $\dt$-dense in $A_n,$
then, for each $f\in {\cal F},$ there exists $s\in S,$ $f(s)\not=0.$

Write $(\bt_{n,0})_*,(\bt_{n,1})_*: K_0(F_n)=\Z^{p_n}\to K_0(E_n)=\Z^{l_n}$ as $\bb^n_0=(b^n_{0,ij}), \bb^n_1=(b^n_{1,ij})$, with $b^n_{0,ij}\in \N, b^n_{1,ij}\in \N$. (If there is no  confusion, we
may
omit $n$ from $\bb^n_0,b^n_{1,ij} $ etc. in the notation {{above}}.)
Then
\beq\label{eta=tht}
\eta_{n,i}(0)=\left\{\tht_{n,1}^{\sim b_{0,i1}},
\tht_{n,2}^{\sim b_{0,i2}},\cd,\tht_{n,p_n}^{\sim b_{0,i{p_n}}}\right\} {{\andeqn}} \eta_{n,i}(1)=\left\{\tht_{n,1}^{\sim b_{1,i1}}
\tht_{n,2}^{\sim b_{1,i2}},\cd,\tht_{n,p_n}^{\sim b_{1,i{p_n}}}
\right\}.
\eneq
{{By {{Proposition}} 3.6 of \cite{GLN} (note that the unital condition  is not used in the proof), if $\af_0, \af_1: F_n\to E_n$ satisfy
$$(\af_0)_*=(\bt_{n,0})_*, (\af_1)_*=(\bt_{n,1})_*: K_0(F_n)\to K_0(E_n),$$
then $A(F_n,E_n, \af_0,\af_1)=A(F_n,E_n, \bt_{n,0},\bt_{n,1})=A_n$. }}

{{Let us introduce the following notation. For $1\leq j\leq l_n$, let $\pi_{E_n^j}$ be the canonical projection from $E_n=\oplus_{k=1}^{l_n}M_{\{n,k\}}(\C)$ to $E_n^j=M_{\{n,j\}}(\C)$ and let $\bt_{n,0}^j=\pi_{E_n^j}\circ \bt_{n,0}$ and $\bt_{n,1}^j=\pi_{E_n^j}\circ \bt_{n,1}$. }}

{{If $(\bt_{n,0}^j)_*([1_{F_n}])\leq (\bt_{n,1}^j)_*([1_{F_n}])$, then there is a unitary $u_j\in E_n^j$ such that
$\bt_{n,0}^j(1_{F_n})\leq {\rm{Ad}}\, u_j \circ\bt_{n,1}^j(1_{F_n}),$ and if $(\bt_{n,0}^j)_*([1_{F_n}])\geq (\bt_{n,1}^j)_*([1_{F_n}])$, then there is a unitary $u_j\in E_n^j$ such that $\bt_{n,0}^j(1_{F_n})\geq {\rm Ad}\, u_j\circ\bt_{n,1}^j(1_{F_n}).$
Therefore, for convenience, replacing  $\bt_{n,1}^j$ by ${\rm Ad}\, u_j\circ \bt_{n,1}^j,$ if necessary, we may always assume:}}

(**) {{If}} $(\bt_{n,0}^j)_*([1_{F_n}])\leq (\bt_{n,1}^j)_*([1_{F_n}])$, {{then $\bt_{n,0}^j(1_{F_n})\leq \bt_{n,1}^j(1_{F_n});$
and if $(\bt_{n,0}^j)_*([1_{F_n}])\geq (\bt_{n,1}^j)_*([1_{F_n}])$, then $\bt_{n,0}^j(1_{F_n})\geq \bt_{n,1}^j(1_{F_n})$.}}
\vspace{0.1in}

\noindent
{{If (**) holds,}} $\max(\bt_{n,0}^j(1_{F_n}), \bt_{n,1}^j(1_{F_n}))$ makes sense. Let $P_{n,j}=1_{E_n^j}-\max(\bt_{n,0}^j(1_{F_n}), \bt_{n,1}^j(1_{F_n}))$.

\end{NN}

\begin{NN}\label{range4.6.1}


Let $A=A(F,E,\bt_0, \bt_1)$ be
with $F=\oplus_{i=1}^p M_{R_i}(\C), E=\oplus_{i=1}^l M_{r_i}(\C)$ and $(\bt_0)_*, (\bt_1)_*: K_0(F)=\Z^p\to K_0(E)=\Z^l$ 
represented {{by}} matrices
$(b_{0,ij}),~(b_{1, ij})$.
 From 3.4 of \cite{GLII} (see also \ref{Dlambdas}),  {{we have}}
$$
\ld_s(A)=\min_i\left\{\frac{\sum_{j=1}^{p}b_{0,ij}R_j}{r_i},
\frac{\sum_{j=1}^{p}b_{1,ij}R_j}{r_i}
\right\}~.$$

\end{NN}

\begin{NN}\label{range4.6.2} For $C^*$-algebra $A_n=A(F_n,E_n, \bt_{n,0},\bt_{n,1})$ as in \ref{4.6}, we will fix a strictly positive element $e_n^A\in A_n$ defined by $e_n^A=(f_1,f_2, \cdots f_{l_n},a_1,a_2,\cdots a_{p_n})$ with $a_i=1_{F_n^i}\in F_n^i$ and
$$f_j(t)=(1-t)\bt_{n,0}^j(1_{F_n})+t\bt_{n,1}^j(1_{F_n})+t(1-t)P_{n,j}.$$
From the definition of $P_{n,j}$, we know that for $0<t<1$, $${\rm rank} (f_j(t))={\rm rank}(\max(\bt_{n,0}^j(1_{F_n}), \bt_{n,1}^j(1_{F_n})))+{\rm rank}(P_{n,j})={\rm rank} (E_n^j)=\{n,j\}.$$
Hence $e_n^A$ is strictly positive.
%
{{Let
$
a^A_n:=(f, g)\in A_n
$
be defined by
\beq
g=1_{F_n}\andeqn f(t)=(1-t)\bt_{n, 0}(g)+t\bt_{n,1}(g).
\eneq}}
Then $a^A_n\leq e_n^A$. Note that $a^A_n$
{{may not be}} a strictly positive element.

\end{NN}

\begin{NN}{\label{4.8}} (Order unit and $M$ large)
 Suppose that $(G, \DT, \rho)$ is as in \ref{range4.4} with the condition (*). Let us assume that $G$ is torsion free.  Choose a countable dense subgroup $G^1\subset \Aff(\DT)$ with $1_{\DT}\in G^1$.   Put
$H=G\oplus G^1$ and define   ${\tilde \rho}: H\to \Aff(\DT)$ by
${\tilde \rho}((g,f))(\tau)=\rho(g)(\tau)+f(\tau)$ for all $(g,f)\in G\oplus G^1$ and
$\tau\in \DT.$ Define {{$H_+\ni 0$ such that}}
$H_+\setminus\{0\}$ {{is}} the set of  elements
$(g,f)\in G\oplus G^1$ with ${\tilde \rho}((g,f))(\tau)>0$ for all $\tau\in \DT.$
Then $(H,H_+,1_\DT)$ is a simple ordered group with Riesz interpolation property (see 13.9 of \cite{GLN}). Following 13.9-13.12 of \cite{GLN}, write $H$ as an inductive limit of
\beq\label{H'}
H_1'\stackrel{\gm'_{12}}{\lr}H_2'\stackrel{\gm'_{23}}{\lr}\cd\lr H
\eneq
of {{direct sum of finite}} copies of ordered group $(\Z,\Z_+)$  with the following property
\beq\label{june-9-2021}
\gm'_{n,\infty} (x)\in H_+\backslash \{0\}~~~\mbox{for any}~~x\in (H'_n)_+\backslash \{0\}.
\eneq
  Let $H'_n=\left(\Z^{p'_n},\Z^{p'_n}_+\right)$ {{be}} with a $p'_n$-tuple $\tilde{u}'_n$ of positive integers as the order unit.    Furthermore, $\gm'_{n,n+1}(\tilde{u}'_n)=\tilde{u}'_{n+1}$  and $\gm'_{n,\infty}(\tilde{u}'_n)=1_{\DT}$.

We will modify the order unit $\tilde{u}_n$ to a smaller one for future use.  Since $(H, H_+)$ is a simple Riesz group, we {{may}}  assume that all  maps $\gm'_{n,n+1}$ have multiplicity at least $3,$ i.e., $a_{ij}\ge 3$ for all $i,j,$
if $\gamma_{n,n+1}'$ is represented by {{the}} matrix $\big(a_{ij}\big)_{p'_{n+1}\times p'_n},$ {{then the condition}} that for all $\gamma_{n,n+1}'$,  $a_{ij}>0$ for all $i, j$ implies that  (\ref{june-9-2021}) holds. Consequently, the order unit $\tilde{u}_n=\big(a_1,a_2,\cd, a_{p'_n}\big)\in \Z^{p'_n}$ satisfies $a_i\geq 3$ for all $i,$ provided $n\geq 2$. We {{assume}}
{{that}} this is true for all $n$, since if it is not true for $n=1$, then we simply drop the first term from the limit procedure.

Suppose $\tilde{u}'_n=\big(a_1,a_2,\cd, a_{p'_n}\big)\in \Z^{p'_n}$.
Choose ${u'_n}=\big(a_1-2,a_2,\cd, a_{p'_n}\big)\in \Z^{p'_n}$.  
The assumption that $\gm'_{n, n+1}$ has multiplicities at least $3$ implies
that $\gm'_{n, n+1}(u'_n)<u'_{n+1}$.
Now set $H'_n=\left(\Z^{p'_n},\Z^{p'_n}_+, u'_n \right)$ with the new order unit $u'_n$. Note that $\gm'_{n,\infty}(u'_n)<1_\DT$, and for any element $x\in H_+$ with $x<1_\DT$, there is an integer $n$ such that $\gm'_{n,\infty}(u'_n)>x$.

The inductive limit
$\lim \big( H'_n, u_n', \gm'_{n,m}\big)$ of the scaled ordered groups has the following property. For any $n$ and $M>0$, there is  $N$ such that if $m>N$, then $\gm'_{n,m}$ is $M$-large in the following sense: suppose {{that}} the matrix $\big(a_{ij}\big)_{p'_{m}\times p'_n}$ represents $\gm'_{n,m}$ and $u_n'=(x_1,x_2,\cd, x_{p'_n}), u'_m=(y_1,y_2,\cd,y_{p'_m})$, then
\beq\label{M-large}
a_{ij}\geq M,~ y_i-\sum_{k=1}^{p'_n}a_{ik}x_k\geq M ~~\mbox{for }~~1\leq i\leq p'_m,~1\leq j\leq p'_n~.
\eneq

We will also use the inductive limit system $\lim \big( H'_n, {\tilde u}'_n, \gm'_{n,m}\big)$ which preserves the units. Compare two inductive limit systems, writing $\lim \big( (H'_n, (H'_n)_+, u'_n), \gm'_{n,m}\big)=(H,H_+,\Sigma H)$
and \\
$\lim \big( (H'_n, (H'_n)_+, {\tilde u}'_n), \gm'_{n,m}\big)=(H,H_+,\Sigma_1 H)$, {{we have}}
\beq
\Sigma H=\{h\in H_+, h<1_\DT\}~~\mbox{and} ~~\Sigma_1 H=\{h\in H_+, h\leq1_\DT\}.
\eneq
Furthermore $\gm'_{n,\infty}(u'_n)<1_\DT$ and $\gm'_{n,\infty}({\tilde u}'_n)=1_\DT$.

For any fixed positive integer $M,$ there is a positive integer $N$ such that if $n\geq N$, and {{if}} we  write  $u'_n=(b_1, b_2, \cdots, b_{p'_n})$
and $\tilde{u}'_n=(b_1+2,b_2, \cdots, b_{p'_n})$,  then $b_1\geq 2M$ and
\beq\label{u-tilde-u}
\tilde{\rho}(\gm'_{n,\infty}(u'_n))\geq \frac{b_1}{b_1+2}\tilde{\rho}(\gm'_{n,\infty}(\tilde{u}'_n))
=\frac{b_1}{b_1+2}\cdot1_\DT
\geq (1-1/M)\cdot 1_\DT.
\eneq

Let $G_n'=(\gm'_{n,\infty})^{-1}(\gm'_{n,\infty}(H'_n)\cap G)$.  Then $G'_n\sbs H_n'$. Since $G\cap H_+=\{0\}$, by (\ref{june-9-2021}), we have
\beq\label{june-9-2021-1}
G'_n\cap (H'_n)_+=\{0\}.
\eneq
Furthermore, $G$ is the inductive limit of
\beq
G'_1~~\stackrel{\gm'_{1,2}|_{G'_1}}
{\lr}~~G_2'\stackrel{\gm'_{2,3}|_{G'_2}}
{\lr}\cd
{\lr}~~G'_n\lr\cd \lr G~.
\eneq
{{Recall {{that}} a subgroup $G\subset H$ is {{said to be}}  relatively divisible,
if  {{for any}} $g\in G,$  $m\in \N\setminus \{0\},$ and $h\in H$
{{with}}
$g=mh,$
{{there}} is $g'\in G$ such that $g=mg'$.}}  As in 13.12 of \cite{GLN}, $G_n$ is a relatively divisible subgroup of $H_n$, and $H_n/G_n$ is a  {{torsion-free}}
finitely generated abelian group. 
Write $H_n/G_n=\Z^{l_n'}$.  We have the following {{commutative}} diagram (see 13.12 of \cite{GLN}).
\begin{displaymath}
    \xymatrix{G'_{1_{\,}} \ar[r]^{{{\gm'_{{12}}|_{{G_{1}}}}}} \ar@{^{(}->}[d] & G'_{2_{\,}} \ar[r]\ar@{^{(}->}[d]&\cd \ar[r]&G_{\,} \ar@{^{(}->}[d]\\
         H'_1 \ar[r]^{\gm'_{_{12}}} \ar[d] & H'_2 \ar[r]\ar[d]&\cd \ar[r]&H \ar[d]\\
         H'_1/G'_1 \ar[r]^{\td\gm'_{_{12}}}  & H'_2/G'_2 \ar[r]&\cd \ar[r]&H/G.}
\end{displaymath}
In the process of the construction {{above,}} we will use subsequences $(\{k(n)\})$ and  obtain
the following diagram:
\begin{displaymath}\label{Gkn}
    \xymatrix{G'_{{k(1)}_{\,}} \ar[r]^{{{\gm'_{{k(1),k(2)}}|_{{G_{k(1)}}}}}} \ar@{^{(}->}[d] & G'_{k(2)_{\,}} \ar[r]\ar@{^{(}->}[d]&\cd \ar[r]&G_{\,} \ar@{^{(}->}[d]\\
         H'_{k(1)} \ar[r]^{\gm'_{_{k(1),k(2)}}} \ar[d] & H'_{k(2)} \ar[r]\ar[d]&\cd \ar[r]&H \ar[d]\\
         H'_{k(1)}/G'_{k(1)} \ar[r]^{\td \gm'_{_{k(1),k(2)}}} & H'_{k(2)}/G'_{k(2)} \ar[r]&\cd \ar[r]&H/G.}
\end{displaymath}
Rewrite the sequence  {{above}} as
\beq\label{Hn}
   \xymatrix{G_{1_{\,}} \ar[r]^{{{\gm_{{12}}|_{{G_{1}}}}}} \ar@{^{(}->}[d] & G_{2_{\,}} \ar[r]\ar@{^{(}->}[d]&\cd \ar[r]&G_{\,} \ar@{^{(}->}[d]\\
         H_1 \ar[r]^{\gm_{_{12}}} \ar[d] & H_2 \ar[r]\ar[d]&\cd \ar[r]&H \ar[d]\\
         H_1/G_1 \ar[r]^{\td\gm_{_{12}}}  & H_2/G_2 \ar[r]&\cd \ar[r]&H/G.}
\eneq
Write $H_n=\Z^{p_n}$ ($p_n=p'_{k(n)}$ as $H_n=H'_{k(n)}$) and $H_n/G_n=\Z^{l_n}$  ($l_n=l'_{k(n)}$), and write $\gm_{n,n+1}=\gm'_{k(n),k(n+1)}: \Z^{p_n}\to \Z^{p_{n+1}}$ as $\cc^{n,n+1}=\left(c_{ij}^{n,n+1}\right)_{1\leq i\leq p_{n+1}, 1\leq j\leq p_n}$, and write $\td{\gm}_{n,n+1}: \Z^{l_n}\to \Z^{l_{n+1}}$ as $\dd^{n,n+1}=\left( d_{ij}^{n,n+1} \right)_{1\leq i\leq l_{n+1}, 1\leq j\leq l_n}$.
{{In case of no}}
confusion, we will write $\cc^{n,n+1}$ as $\cc=(c_{ij})$ and $\dd^{n,n+1}$ as $\dd=(d_{ij})$ (omitting $n, n+1$). Since $G_n\varsubsetneqq H_n$, we know that $l_n\geq 1$.

Write 
\begin{displaymath}
\xymatrix{
H_n=\Z^{p_n}~~~~~~\ar[r]^{\gm_{_{n,n+1}}} \ar@{->}[d]^{\pi_n
} &
~~~~~~H_{n+1}=\Z^{p_{n+1}}  \ar@{->}[d]^{\pi_{n+1}}
\\
 H_n/G_n=\Z^{l_n}~~~~~~ \ar[r]^{\td{\gm}_{_{n,n+1}}}& ~~~~~~H_{n+1}/G_{n+1}=\Z^{l_{n+1}},
  }
\end{displaymath}
{{where $\pi_n$ is the (surjective) quotient map.}}
{{We have}}
\beq\label{cd-commut}
\pi_{n+1}\cdot \cc^{n,n+1} = \dd^{n,n+1} \cdot \pi_n.
\eneq
\end{NN}

\begin{NN}{\label{4.9}} (Construction of $A_1$)
Choose $k(1)\geq 1$ such that $ u'_{k(1)}:=([1,1], [1,2],\cd, [1,p'_{k(1)}])\in \N^{p'_{k(1)}}$ satisfies $[1,1]\geq 2\cdot 8=16$. Let $H_1:=H'_{k(1)}$, $p_1:=p'_{k(1)}$, $u_1:=u'_{k(1)}:=([1,1], [1,2],\cd, [1,p_1])\in H_1$ and $\tilde{u}_1:= ([1,1]+2, [1,2],\cd, [1,p_1])\in H_1$. Then, by (\ref{u-tilde-u}),
 $$\tilde{\rho}(\gm'_{k(1),\infty}(u_1))\geq (1-1/8)\cdot 1_\DT$$
 (that is, $\tilde{\rho}(\gm_{1,\infty}(u_1))\geq(1-1/8)\cdot 1_\DT$, after the diagram  (\ref{Hn}) is obtained). Recall {{that}} we also have $\gm'_{k(1),\infty}(\tilde{u}_1)=1_\DT$.  
Let $G_1=G'_{k(1)}$ and $\pi_1: H_1=\Z^{p_1}\to H_1/G_1=\Z^{l_1}$ be the quotient map.

For  a homomorphism $\Z^k\to \Z^l$ represented by a matrix $B=\big(b_{ij}\big)_{1\leq i\leq l, 1\leq j\leq k},$  {{we define}}
$|||B|||{{:=}}\max_{i,j}|b_{ij}|\cdot k\cdot l$.

Let ${{{\cal M}}}_1=2^4|||{{\pi_1}}|||.$
The map $\pi_1$ can be written as the difference $\bb_1^1-\bb_0^1$ of two maps
$$\bb_0^1, \bb_1^1: \Z^{p_1}\to \Z^{l_1},$$
corresponding to two $l_1\times p_1$ matrices
\beq
\bb_0^1=\big(b^1_{0,ij}\big)_{1\leq i\leq l_1, 1\leq j\leq p_1} {{\andeqn}}
\bb_1^1=\big(b^1_{1,ij}\big)_{1\leq i\leq l_1, 1\leq j\leq p_1}
\eneq
such that
\beq\label{bb1}
b^1_{0,ij}\geq \cm_1 \qq \mbox{and} \qq b^1_{1,ij}\geq \cm_1 ~.
\eneq
(Later on we will define $\cm_{n+1}$  when we construct
$A_{n+1}$.) Namely, we construct $\bb_0^1$ and $\bb_1^1$ as below.  Assume $\pi_1: \Z^{p_1}\to \Z^{l_1}$ is given by
the matrix
$\big(p^1_{ij}\big)_{1\leq i\leq l_1, 1\leq j\leq p_1}$. Choose
$\bb_0^1=\big(b^1_{0,ij}\big)$ with $b^1_{0,ij}=|p_{ij}|+\cm_1$ and
$b^1_{1,ij}=b^1_{0,ij}+p_{ij}$.  Then $\bb_0^1$ and $\bb_1^1$ satisfy (\ref{bb1}) and $\bb_1^1- \bb_0^1={{\pi_1}}: \Z^{p_1}\to \Z^{l_1}$.

Recall that the unit $u_1'\in H'_1=H_1$ can be written as
$
([1,1], [1,2],\cd, [1,p_1])\in \N^{p_1}.
$
%
Since $b^1_{0,ij}\geq \cm_1,~ b^1_{1,ij}\geq \cm_1$, we have $$|b^1_{1,ij}-b^1_{0,ij}|\leq \frac{1}{p_1l_1}  |||{{\pi_1}}|||\leq \frac{1}{16}\cdot\frac{1}{2}\cm_1\leq\frac{1}{32}\min(b^1_{0,ij},b^1_{1,ij}).$$
Consequently, $\max(b^1_{0,ij},b^1_{1,ij})\leq (1+\frac{1}{32})\min (b^1_{0,ij},b^1_{1,ij})$. Hence
\beq\nonumber(1+\frac{1}{8})\sum_{j=1}^{p_1}\max(b^1_{0,ij},b^1_{1,ij})[1,j]\leq (1+\frac{1}{8})(1+\frac{1}{32})\sum_{j=1}^{p_1}\min(b^1_{0,ij},b^1_{1,ij})[1,j]\\ \label{(1,i)+}
\leq \left((1+\frac{1}{4})\sum_{j=1}^{p_1}\min(b^1_{0,ij},b^1_{1,ij})[1,j]\right)-
1.
\eneq
By  \eqref{(1,i)+}, {{we}}  may choose an integer $\{1,i\}$  such that
\beq\label{(1,i)}
(1+\frac{1}{8})\sum_{j=1}^{p_1}\max(b^1_{0,ij},b^1_{1,ij})[1,j]\leq \{1,i\}\leq (1+\frac{1}{4})\sum_{j=1}^{p_1}\min(b^1_{0,ij},b^1_{1,ij})[1,j].\eneq

Let $F_1=\bigoplus_{i=1}^{p_1}M_{[1,i]}(\C),$ $E_1=\bigoplus_{i=1}^{{l_1}}M_{\{1,i\}}(\C)$, and $\bt_{1,0}, \bt_{1,1}:~ F_1\to E_1$ be defined by
$$\bt_{1,0}(a_1\oplus a_2\oplus\cd \oplus a_{p_1})=\oplus_{i=1}^{l_1} {\rm{diag}}\left(a_1^{\sim b^1_{0,i1}},
a_2^{\sim b^1_{0,i2}},\cd, a_{p_1}^{\sim b^1_{0,ip_1}}
\right)\andeqn$$
$$\bt_{1,1}(a_1\oplus a_2\oplus\cd \oplus a_{p_1})=\oplus_{i=1}^{l_1} {\rm diag}\left(a_1^{\sim b^1_{1,i1}},
a_2^{\sim b^1_{1,i2}},\cd, a_{p_1}^{\sim b^1_{1,ip_1}}
\right).$$
 Evidently, $\bt_{1,0}\oplus\bt_{1,1}: F_1 \to E_1\oplus E_1$ is injective.
Then
$$(\bt_{0,1})_*=\bb_0^1~~ \mbox{and } ~~ (\bt_{1,1})_*=\bb_1^1:~~~
K_0(F_1)=\Z^{p_1}\to K_0(E_1)=\Z^{l_1}.$$
Define
$$A_1=A\big(F_1,E_1,\bt_{1,0},\bt_{1,1}\big)
=\big\{ (f,a)\in C([0,1],E_1)\oplus F_1{{:}}\, \bt_{1,0}(a)=f(0),\bt_{1,1}(a)=f(1)\big\}.$$
Then $K_0(A_1)=G_1$ and $K_0(F_1)=H_1$, and the quotient map
 $\pi^A_1: A_1\to F_1$ defined by $\pi_1^A(f,a)=a$ induces $(\pi_1^A)_*: K_0(A_1)={{G_1}}\to K_0(F_1)=H_1$, which is the inclusion map. {{Since $\bb_1^1- \bb_0^1={{\pi_1}}$
 is surjective (onto $\Z^{l_1}$),   $K_1(A_1)=0,$  by Proposition 3.5 of \cite{GLN}.}} Furthermore by (\ref{june-9-2021-1}), $K_0(A_1)_+=G_1\cap (H_1)_+=\{0\}$ {{(see Proposition 3.5 of \cite{GLN})}} and thus $A_1\in {\cal C}_0.$

By (\ref{(1,i)}), we have $$\ld_s(A_1)=\min_i\left\{\frac{\sum_{j=1}^{p_1}b^1_{0,ij}\cdot[1,j]}{\{1,i\}},
\frac{\sum_{j=1}^{p_1}b^1_{1,ij}\cdot[1,j]}{\{1,i\}}\right\}\geq 1/(1+\frac{1}{4})\geq\frac{3}{4}.$$.
\end{NN}

\begin{NN}{\label{4.10}} (Inductive assumption for $A_{n}$)
Suppose that we  have constructed
$A_n=A(F_n,E_n, \bt_{n,0},\bt_{n,1}){{\in {\cal C}_0}}$ with
$$F_n=\oplus_{j=1}^{p_n} M_{[n,j]},~ E_n=\oplus_{i=1}^{l_n} M_{\{n,i\}}(\C),~\bt_{n,0}, \bt_{n,1}: F_n\to E_n$$ such that the following conditions hold.

(a)  $K_0(F_n)=H_n=H'_{k(n)},~ K_0(A_n)=G_n=G'_{k(n)}$  (for some $k(n)$) {{and}} the quotient map $\pi_n^A: A_n\to F_n$ induces the inclusion map
{{$(\pi_n^A)_*: K_0(A_n)=G_n\to K_0(F_n)=H_n,$
{{ $G_n\cap (H_n)_+=\{0\},$ and $K_1(A_n)=\{0\}.$}}}}

(b)  $(\bt_{n,0})_*=\bb^n_0=\big(b^n_{0,ij}\big)_{1\leq i\leq l_n, 1\leq j\leq p_n}$ and
$(\bt_{n,1})_*=\bb^n_1=\big(b^n_{1,ij}\big)_{1\leq i\leq l_n, 1\leq j\leq p_n}$ with $b^n_{0,ij}\geq \cm_n,~ b^n_{1,ij}\geq \cm_n$ for a pre-given positive number $\cm_n$. (Note that $(\bb^n_1)-(\bb^n_0)=\pi_n: K_0(F_n)=H_n=\Z^{p_n}\to K_0(E_n)=H_n/G_n=\Z^{l_n}$
{{which is surjective.}})

(c)
{{In $K_0(F_n),$}} ${{[1_{F_n}]}}=u_n=u'_{k(n)}:=([n,1], [n,2], \cd, [n,p_n])\in \N^{p_n}$  with
$[n,1]\geq 2\cdot 8^n$. As a consequence (see (\ref{u-tilde-u})), {{we have}}
$$\tilde{\rho}(\gm'_{k(n),\infty}(u_n))\geq (1-1/8^n)\cdot 1_\DT.$$
 (Note that $\tilde{u}_n=([n,1]+2, [n,2], \cd, [n,p_n])\in H_n$ satisfies $\gm'_{k(n),\infty}(\tilde{u}_n)=1_\DT$.)


(d) $\bt_{n,0}$ and $\bt_{n,1}$ satisfy the property (**) in the last paragraph of \ref{4.6}.

(e)  $(1+\frac{1}{8^n})\cdot \sum_{j=1}^{p_n}\max\left\{b^n_{0,ij},~b^n_{n,ij} \right\}[n,j]~
\leq \{n,i\}\leq (1+\frac{1}{4^n})\cdot \sum_{j=1}^{p_n}\min\left\{b^n_{0,ij},~b^n_{n,ij} \right\}[n,j]~.$
{{Consequently,}} $\ld_s(A_n)> 1-\frac{1}{4^n}$.

Note that the homomorphism $\pi_n^A: A_n \to F_n$ induces the following commutative diagram
\begin{displaymath}
\xymatrix{
K_0(A_n)~~~~\ar[rr]^{\rho_{A_n}} \ar[d]^{(\pi_n^A)_{*0}}
 &&
~~~~~~\Aff(T_0(A_n)) \ar@{->}[d]^{{\pi_{n}^A}^\sharp}
\\
 K_0(F_n)~~~ \ar[rr]^{\rho_{F_n}}&& ~~~\Aff(T_0(F_n)).}
\end{displaymath}
From (a),  when we identify $K_0(A_n)=G_n$ and $K_0(F_n)=H_n$, the map $(\pi_n^A)_{*0}$ is identified
with  the inclusion from $G_n$ to $H_n$, we have
\beq\label{4.10+1}
\xymatrix{
G_n~~~~\ar[rr]^{\rho_{A_n}} \ar[d]^{(\pi_n^A)_{*0}}
 &&
~~~\Aff(T_0(A_n)) \ar@{->}[d]^{{\pi^A_{n}}^\sharp}
\\
 H_n~~~ \ar[rr]^{\rho_{F_n}}&& ~~~\Aff(T_0(F_n)).}
 \eneq
%
From $b^n_{0,ij}\geq \cm_n>0$ in part (b), we know $\bt_{n,0}$ is injective, hence $\bt_{n,0}\oplus\bt_{n,1}: F_n \to E_n\oplus E_n$ is injective.

We will construct $A_{n+1}=A\big(F_{n+1},E_{n+1},\bt_{n+1,0},\bt_{n+1,1}\big)$ and a positive integer $\cm_{n+1}$, and two homomorphisms $\phi^{o}_{n,n+1}, \phi_{n,n+1}: A_n\to A_{n+1}$ as below.

\end{NN}

\begin{NN}{\label{4.11}} (The definition of $A_{n+1}$)
Let
\beq\label{defL}
{\cal L}_{n+1}=2^{4(n+1)}\cdot\max\big\{\{n,i\}:~ 1\leq i\leq l_n\big\}\cdot nl_n~.
\eneq
Recall that $H_n=H'_{k(n)}$ {{(in the inductive limit system in (\ref{H'}))}} and $H$ is a simple Riesz group. There is  $k(n+1)>k(n)$ such that the map $\gm'_{k(n),k(n+1)}$ is ${\cal L}_{n+1}$ large in the sense of (\ref{M-large}) in \ref{4.8}.  Write
$H_{n+1}=H'_{k(n+1)}=\Z^{p_{n+1}}$ and represent the map
$\gm_{n,n+1}=\gm'_{k(n),k(n+1)}:~ \Z^{p_n}\to \Z^{p_{n+1}}$
by matrix $\cc^{n,n+1}=\big(c_{ij}^{n,n+1}\big)_{p_{n+1}\times p_n}$.  We further require that the  unit\\
$u_{n+1}={{u'_{k(n+1)}:=}}([n+1,1], [n+1,2], \cd, [n+1,p_{n+1}])\in \N^{p_{n+1}}$ satisfies the condition $[n+1,1]\geq 2\cdot 8^{n+1}$ (see (c) in \ref{4.10}). Then we have
\beq\label{c-M-large}
c_{ij}^{n,n+1}\geq {\cal L}_{n+1}~~\mbox{and} ~~ [n+1,i]-\sum_{k=1}^{p_n} c_{ik}^{n,n+1}[n,k]\geq {\cal L}_{n+1}.
\eneq
It follows from $\gm'_{k(n),k(n+1)}(\tilde{u}'_{k(n)})=\tilde{u}'_{k(n+1)}$, $\tilde{u}_n=\tilde{u}'_{k(n)}=([n,1]+2,[n,2],{{\cdots, [n, p_{n}]}})$, \\ $\tilde{u}_{n+1}=\tilde{u}'_{k(n+1)}=([n+1,1]+2,[n+1,2],{{\cdots, [n+1, p_{n+1}]}})$ and (c) of \ref{4.10} that
\beq\label{n-F-(n+1)}
\sum_{j=1}^{p_n}c_{ij}^{n,n+1}[n,j] \geq (1-1/8^n)\cdot [n+1,i].
\eneq
Write $G_{n+1}=G'_{k(n+1)}\sbs H_{n+1}$ and write
$H_{n+1}/G_{n+1}=\Z^{l_{n+1}}$.  Suppose that the map ${{\gm_{n,n+1}}}:~ H_n\to H_{n+1}$ induces
$\widetilde{\gm_{n,n+1}}:~ H_n/G_n=\Z^{l_n} \to H_{n+1}/G_{n+1}=\Z^{l_{n+1}}$.  Write $\widetilde{\gm_{n,n+1}}=\dd^{n,n+1}=
\big(d_{ij}^{n,n+1}\big)_{l_{n+1}\times l_n}$, where
$d_{ij}^{n,n+1}\in \Z$.
%
%
%
 Set
\beq\label{cm}
\cm_{n+1}=2^{4(n+1)}l_n\cdot n \cdot(|||\dd^{n,n+1}|||+2)\cdot \max \big\{ \{n,k\}; 1\leq k\leq l_n\big\}\cdot |||\pi_{n+1}|||,
\eneq
where $\pi_{n+1}:~ H_{n+1}=\Z^{p_{n+1}}\to H_{n+1}/G_{n+1}=\Z^{l_{n+1}}$ is the quotient map and $|||\cdot|||$ is defined in the end of \ref{4.8}.

We will construct the algebra $A_{n+1}$ as below (see 4.9 for similar construction of $A_1$). Write the map
$\pi_{n+1}:~ H_{n+1}=\Z^{p_{n+1}}\to H_{n+1}/G_{n+1}=\Z^{l_{n+1}}$ as a difference
$\pi_{n+1}=\bb_1^{n+1}-\bb_0^{n+1}$  of two matrices
$\bb_0^{n+1}=\big(b^{n+1}_{0,ij}\big)_{1\leq i\leq l_{n+1}, 1\leq j\leq p_{n+1}}$ and $
\bb_1^{n+1}=\big(b^{n+1}_{1,ij}\big)_{1\leq i\leq l_{n+1}, 1\leq j\leq p_{n+1}}$ satisfying
\beq\label{>cm}
b^{n+1}_{0,ij}\geq \cm_{n+1} {{\andeqn}}\,\,  b^{n+1}_{1,ij}\geq \cm_{n+1} \hspace{0.6in} \mbox{ for all } {{i, j.}}
\eneq
Consequently, {{we obtain}}
\vspace{-0.1in}
\beq\label{b-b}
|b^{n+1}_{1,ij}-b^{n+1}_{0,ij}|\leq \frac{1}{2^{4(n+1)}} \min(b^{n+1}_{1,ij},b^{n+1}_{0,ij}).
\eneq
As
in  the calculation in  (\ref{(1,i)+}), we have
$$(1+\frac{1}{8^{n+1}})\sum_{j=1}^{p_{n+1}}\max(b^{n+1}_{0,ij},b^{n+1}_{1,ij})
[n+1,j]\leq \left((1+\frac{1}{4^{n+1}})\sum_{j=1}^{p_{n+1}}
\min(b^{n+1}_{0,ij},b^{n+1}_{1,ij})[n+1,j]\right)-1.$$

One may then choose an integer $\{n+1,i\}$ which satisfies
\beq\label{2020-8-7-n1}
(1+\frac{1}{8^{n+1}})\sum_{j=1}^{p_{n+1}}\max(b^{n+1}_{0,ij},b^{n+1}_{1,ij})
[n+1,j]\leq \{n+1,i\}\\ \label{(n+1,i)}\leq (1+\frac{1}{4^{n+1}})\sum_{j=1}^{p_{n+1}}
\min(b^{n+1}_{0,ij},b^{n+1}_{1,ij})[n+1,j].
\eneq
{{From}} (\ref{cd-commut}), we have
\beq\label{cd-commut1}
\left(\bb_1^{n+1}-\bb_0^{n+1}\right)
\cdot \cc^{n,n+1}=
\dd^{n,n+1}\cdot\left(\bb_1^{n}-\bb_0^{n}\right).
\eneq

Let {{$F_{n+1}=\bigoplus_{{j=1}}^{p_{n+1}}M_{[n+1,{{j}}]}(\C)$, $E_{n+1}=\bigoplus_{{i=1}}^{l_{n+1}}M_{\{n+1,{{j}}\}}(\C)$}}
and
$\bt_{n+1,0}, \bt_{n+1, 1}:~ F_{n+1}\to {{E_{n+1}}}$
 be defined by
$$\bt_{n+1,0}(a_1\oplus a_2\oplus\cd \oplus a_{p_{n+1}})=\bigoplus_{i=1}^{l_{n+1}} \diag\left(a_1^{\sim b^{n+1}_{0,i1}},
a_2^{\sim b^{n+1}_{0,i2}},\cd, a_{p_{n+1}}^{\sim b^{n+1}_{0,ip_{n+1}}}, 0^{\sim k_0}
\right)\andeqn$$
$$\bt_{n+1,1}(a_1\oplus a_2\oplus\cd \oplus a_{p_{n+1}})=\bigoplus_{i=1}^{l_{n+1}} \diag\left(a_1^{\sim b^{n+1}_{1,i1}},
a_2^{\sim b^{n+1}_{1,i2}},\cd, a_{p_{n+1}}^{\sim b^{n+1}_{1,ip_{n+1}}}, 0^{\sim k_1}
\right),$$
where $k_0=\{n+1, i\}-\sum_{j=1}^{p_{n+1}} b^{n+1}_{0,ij}[{n+1},j]$ and $k_1=\{n+1, i\}-\sum_{j=1}^{p_{n+1}} b^{n+1}_{1,ij}[{n+1},j]$.
Define $A_{n+1}=A(F_{n+1},E_{n+1},\bt_{n+1,0},\bt_{n+1,1})$. Moreover  $$(\bt_{n+1,0})_*=\bb^{n+1}_0,~(\bt_{n+1,1})_*=\bb^{n+1}_1:~
K_0(F_{n+1})=\Z^{p_{n+1}}\to K_0(E_{n+1})=\Z^{l_{n+1}}.$$
$K_0(A_{n+1})=G_{n+1}\sbseq K_0(F_{n+1})={{H_{n+1}.}}$ {{Since $\pi_{n+1}$ is surjective, by Proposition 3.5 of \cite{GLN}),
$K_1(A_{n+1})=0.$  Also
$K_0(A_{n+1})_+=G_{n+1}\cap (H_{n+1})_+=\{0\},$ which implies $A_{n+1}\in {\cal C}_0.$ Thus}}
condition (a), (b), (c), (d) and (e) (see (\ref{(n+1,i)})) in  \ref{4.10} for $n+1$
{{follow.}}  This ends the construction of
$A_{n+1}$.

{{With $A_{n+1}$ constructed above we will construct two homomorphisms $\phi^o_{n,n+1},\phi_{n,n+1}:~ A_n\to A_{n+1}$ {{in}} the next few sections, namely from \ref{4.12} to \ref{4.20}.}}

\end{NN}

\begin{NN}{\label{4.12}} (Definition of $\psi$ and $\psi^o$)
As notation introduced in 4.6, we have
$$Sp(A_{n})=\{\tht_{n,1},\cd, \tht_{n,p_n}\}\cup \{\eta_{n,1}(t),\cd, \eta_{n,l_n}(t), 0<t<1
\}.$$

To simplify the {{notation,}}  let us use $\cc=(c_{ij})$ to denote $\cc^{n,n+1}=(c^{n,n+1}_{ij})$ and $\dd=(d_{ij})$ to denote $\dd^{n,n+1}=(d^{n,n+1}_{ij})$.
Let $\psi_{n,n+1}:~ F_n \to F_{n+1}$ be  {{a}} (non-unital) homomorphism defined by
\beq\label{psi-n-n+1}
\psi_{n,n+1}(a_1\oplus a_2\oplus\cd \oplus a_{p_n})=\bigoplus_{i=1}^{p_{n+1}}\diag (a_1^{\sim c_{i1}},a_2^{\sim c_{i2}},\cd,a_{p_n}^{\sim c_{ip_n}}, 0^{\sim \sim }),\eneq
where $0^{\sim\sim}$ denotes some suitable many copies of $0$'s---to avoid introducing too many notation,
{{sometimes}} we will not indicate how many copies  {{it has}} (but it is {{usually}} easy to calculate, {{for example}} it is  $[n+1,i]-\sum_{k=1}^{p_n} c_{ik}[n,k]$ copies here).  Then $(\psi_{n,n+1})_*=\tilde{\gm}_{n,n+1}=\cc=(c_{ij}): K_0(F_n) \to K_0(F_{n+1})$.

Let $\psi^o : A_n\to F_{n+1}$ be defined by $\psi^o=\psi_{n,n+1}\circ \pi_n^A.$
Let
\beq\label{c'}
c'_{1j}=c_{1j}-(n-1)\sum_{i=1}^{l_n} b^n_{0,ij}~~ \mbox{for}~~ 1\leq j\leq p_n.
\eneq
By  \ref{4.10}, (\ref{defL}), and (\ref{c-M-large}), {{we estimate}}
\beq\nonumber
c_{1,j}> c'_{1j}{{=}}c_{1,j}-(n-1)\sum_{i=1}^{l_n} b^n_{0,ij} \ge c_{1,j}-(n-1)l_n\max\{ \{n,i\}:1\le i\le p_n\}\\ \label{c'-large}
>c_{1,j}-{{\cal L}_{n, n+1}\over{2^{4(n+1)}}}
>(1-\frac{1}{2^{4(n+1)}})c_{1j}.
\eneq
Define $\psi: A_n\to F_{n+1}$ by sending $f=(f_1,f_2\cd, f_{l_n},a_1,a_2,\cd, a_{p_n})$ to
$\psi(f):=(b_1, b_2,\cd, b_{p_{n+1}})$,
where
\beq\nonumber\label{psi=b1}
&&\hspace{-0.6in}b_1=\diag\big( f_1(\frac{1}{n}),f_1(\frac{2}{n}),\cd, f_1(\frac{n-1}{n}),~f_2(\frac{1}{n}),f_2(\frac{2}{n}),\cd f_2(\frac{n-1}{n}),\\
&&\hspace{0.1in}\cd,f_{l_n}(\frac{1}{n}),f_{l_n}(\frac{2}{n}),\cd, f_{l_n}(\frac{n-1}{n}),~a_1^{\sim c'_{11}}, a_2^{\sim c'_{12}},\cd, a_{p_n}^{\sim c'_{1p_n}},0^{\sim\sim}\big)
\eneq
(from \eqref{c-M-large} and (\ref{defL}), $\sum_{i=1}^{l_n}(n-1)\{n,i\}+\sum_{j=1}^{p_n}c'_{1j}{ {[n,j]}}<[n+1,1]$), and for $i\geq 2$,
\beq\label{psi=b2}
b_i=\diag(a_1^{\sim c_{i1}}, a_2^{\sim c_{i2}},\cd, a_{p_n}^{\sim c_{ip_n}},0^{\sim\sim}).
\eneq
One should  note that $\pi_{n+1,i}(\psi^o(f))=\pi_{n+1,i}(\psi(f))$ for $i\geq 2$, where ${ {\pi_{n+1,i}:  F_{n+1}}}\to F_{n+1}^i$ is the projection map.


For each $0\leq t\leq 1$, define $\psi_t: A_n\to F_{n+1}$ by
sending
$f=(f_1,f_2\cd, f_{l_n},a_1,a_2,\cd, a_{p_n})$ to $\psi_t(f)=(b'_1, b'_2,\cd, b'_{p_{n+1}})$ with $b'_i=b_i$ for $i\geq 2$ (see (\ref{psi=b2})) and
\beq\nonumber\label{psit}
b'_1=\diag\big(f_1(\frac{1}{n}(1-t)),\cd, f_1(\frac{n-1}{n}(1-t)),~f_2(\frac{1}{n}(1-t)),\cd, f_2(\frac{n-1}{n}(1-t)),\\
\cd,f_{l_n}(\frac{1}{n}(1-t)),\cd, f_{l_n}(\frac{n-1}{n}(1-t)),~a_1^{\sim c'_{11}}, a_2^{\sim c'_{12}},\cd, a_{p_n}^{\sim c'_{1p_n}}\big).~~~~~~~~~
\eneq
When $t=0$, {{we have}} $b'_1=b_1$ {{(see \eqref{psi=b1}),}} and therefore $\psi=\psi_0$. For $t=1$,
$$b'_1=\diag\big(f_1(0)^{\sim(n-1)}, f_2(0)^{\sim(n-1)},\cd, f_{l_n}(0)^{\sim(n-1)},~a_1^{\sim c'_{11}}, a_2^{\sim c'_{12}},\cd, a_{p_n}^{\sim c'_{1p_n}}\big).$$
By (\ref{eta=tht}) and the definition of $c'_{1,j}$, we have $SP(\psi_1)=SP(\psi^o)$.
Hence
\beq\label{KKpsi}
KK(\psi)=KK(\psi_0)=KK(\psi_1)=KK(\psi^o)\in KK(A_n, F_{n+1}).
\eneq

\end{NN}

\begin{NN}{\label{4.13}} (Definition of $\chi$ and $\mu$) Recall {{that}}
$\dd=(d_{ij})_{l_{n+1}\times l_n}$ is the matrix which represents $\tilde{\gm}_{n,n+1}: H_n/G_n=Z^{l_n} \to H_{n+1}/G_{n+1}=Z^{l_{n+1}}$. For $1\leq j\leq l_{n+1}, 1\leq k\leq l_n$, define
\beq\label{tilde-d}
\tilde{d}_{jk}=\begin{cases} |d_{jk}|&\,\,\, \mbox{if}~ ~~ d_{jk}\not=0\\
                             2&\,\,\, \mbox{if}~~~ d_{jk}=0.
                             \end{cases}
\eneq
For $1\leq j\leq l_{n+1}$,
let $L_j=\sum_{k=1}^{l_n}\tilde{d}_{jk}\{n,k\}$.
Define $\chi: A_n\to \bigoplus_{j=1}^{l_{n+1}} M_{L_j}(C[0,1])$ by sending
$f=(f_1, f_2,\cdots, f_{l_n}, a_1, a_2, \cdots, a_{p_n})$
to $\chi(f)=\bigoplus_{j=1}^{l_{n+1}}(F_{j,1}(t), F_{j,2}(t), \cdots F_{j,l_{n}}(t))$, where
\beq\label{chi}
F_{jk}(t)=\begin{cases} \{f_k(t)^{\sim |d_{jk}|}\}&\,\,\, \mbox{if}~ ~~ d_{jk}>0,\\
 \{f_k(1-t)^{\sim |d_{jk}|}\}&\,\,\, \mbox{if}~ ~~ d_{jk}<0,\\
                           \{f_k(t), f_k(1-t)\}&\,\,\, \mbox{if}~~~ d_{jk}=0.
                             \end{cases}
\eneq
For any $1\leq j\leq l_{n+1}$, let $\pi_j^E: \bigoplus_{k=1}^{l_{n+1}} M_{L_k}(C[0,1])\to  M_{L_j}(C[0,1])$ be the projection. Then{{, for any $0<t<1,$}}
\beq\label{SP-chi}
\{\eta_{n,i}(t), \eta_{n,i}(1-t): 1\leq i\leq l_n\}\subset Sp((\pi_j^E\circ\chi)|_{t})\cup Sp((\pi_j^E\circ\chi)|_{1-t}).
\eneq
This implies that $Sp(A_n)=\cup_{0\leq t\leq 1}Sp(\chi|_t)$. Hence $\chi$ is injective.

Define two subsets $J_0, J_1$ of the index set $\{1, 2, \cdots, l_{n+1}\}$ by $j\in J_0$ if and only if $b^{n+1}_{0,j1}>b^{n+1}_{1,j1}$; and $j\in J_1$ if and only if $b^{n+1}_{1,j1}>b^{n+1}_{0,j1}$. (Note that if $b^{n+1}_{0,j1}=b^{n+1}_{1,j1}$, then $j$ is  neither {{ in $J_0$ nor  in $J_1$.}})

Let $B_j=|b^{n+1}_{0,j1}-b^{n+1}_{1,j1}|$ and $K_j=(n-1)B_j\cdot \sum_{i=1}^{l_n} \{n,i\}$.

Define $\mu: A_n\to \bigoplus_{j\in J_0\cup J_1} M_{K_j}(C[0,1])$ by sending
$f=(f_1, f_2,\cdots, f_{l_n}, a_1, a_2, \cdots, a_{p_n})$
to $\mu(f)=\bigoplus_{j\in J_0\cup J_1} (G_{j,1}(t), G_{j,2}(t), \cdots G_{j,l_n}(t))$, where
\beq\label{mu}
G_{jk}(t)=\begin{cases} \diag\big(f_k(\frac{1}{n}(1-t))^{\sim B_j},f_k(\frac{2}{n}(1-t))^{\sim B_j},\cd,f_k(\frac{n-1}{n}(1-t))^{\sim B_j}\big)  &\,\,\, \mbox{if}~ ~~ j\in J_0,\\
 \diag\big(f_k(\frac{1}{n}t)^{\sim B_j},f_k(\frac{2}{n}t)^{\sim B_j},\cd,f_k(\frac{n-1}{n}t)^{\sim B_j}\big)  &\,\,\, \mbox{if}~ ~~ j\in J_1
                           \end{cases}{{.}}~~~~~
\eneq
Note that, for $j\in J_0,$
\beq\label{mu0}
\hspace{-0.16in}G_{jk}(0)=\diag\big(f_k(\frac{1}{n})^{\sim B_j},f_k(\frac{2}{n})^{\sim B_j},\cd,f_k(\frac{n-1}{n})^{\sim B_j}\big), ~G_{jk}(1)=\diag(f_k(0)^{\sim(n-1)B_j}),~~~~
\eneq
and, for $j\in J_1,$
\beq\label{mu1}
\hspace{-0.2in}G_{jk}(0)=\diag(f_k(0)^{\sim(n-1)B_j}), ~G_{jk}(1)=\diag\big(f_k(\frac{1}{n})^{\sim B_j},f_k(\frac{2}{n})^{\sim B_j},\cd,f_k(\frac{n-1}{n})^{\sim B_j}\big).~~~~
\eneq

\end{NN}

\begin{NN}\label{4.14} ({{Notation}}:  $\kappa^j_0(\tht_{n,i}),\kappa^j_1(\tht_{n,i}),{\bar\kappa}^j_0(\tht_{n,i}),
{\bar\kappa}^j_1(\tht_{n,i})$--the multiplicities of $\tht_{n,i}$ for certain homomorphisms)

Let ~$\xi^o_0,\xi^o_1,\xi_0, \xi_1: A_n \to E_{n+1}=\bigoplus_{i=1}^{l_{n+1}}M_{\{n+1,i\}}(\C)
$
be defined by
$\xi^o_0=\bt_{n+1,0}\circ \psi^o,$ $\xi^o_1=\bt_{n+1,1}\circ \psi^o,~ \xi_0=\bt_{n+1,0}\circ \psi$ and $ \xi_1=\bt_{n+1,1}\circ \psi$.

It will be convenient to introduce the following {{notation}}: for a homomorphism $\phi: A\to M_l(C[0,1])$, we use $\phi|_0, \phi|_1: A \to M_l(\C)$ to denote the map given by $\phi|_0(a):=\phi(a)(0)$ and $\phi|_1(a):=\phi(a)(1)$.

For fixed $1\leq i\leq l_{n+1}$, denote also {{by}} $\pi_i^E$  the projection from $E_{n+1}=\bigoplus_{j=1}^{l_{n+1}}
M_{\{n+1,j\}}(\C)$
to {{$i$-th}} summand
$E_{n+1}^i=M_{\{n+1,i\}}(\C)$,
from $\bigoplus_{j=1}^{l_{n+1}}C([0,1],M_{L_j}(\C))$ to  {{$i$-th}} summand $ C([0,1],M_{L_i}(\C))$, or from $\bigoplus_{j=1}^{l_{n+1}}C([0,1],M_{L_j+K_j}(\C))$ to {{$i$-th}} summand $ C([0,1],M_{L_i+K_i}(\C)).$

In the next two lemmas,  we will compare $\pi_j^E\circ\xi^o_0$  and
$\pi_j^E\circ\xi^o_1 $ with $(\pi_j^E\circ \chi)|_0$ and ${{(\pi_j^E\circ \chi)}}|_1$; and compare $\pi_j^E\circ\xi_0 $  and
$\pi_j^E\circ\xi_1 $ with $(\pi_j^E\circ (\chi\oplus \mu))|_0$ and $\pi_j^E\circ (\chi\oplus \mu))|_1$.

Let $f=(f_1, f_2, \cd,f_{l_n}, a_1, a_2,\cd, a_{p_n}).$ Up to conjugating unitaries, we may write
$$(\pi_j^E\circ \chi)|_0(f)=\diag(a_1^{{\sim{\kappa_0^j(\tht_{n,1})}}}, a_2^{{\sim{\kappa_0^j(\tht_{n,2})}}},\cd, a_{p_n}^{{\sim{\kappa_0^j(\tht_{n,p_n})}}},0^{\sim\sim}), $$
$$(\pi_j^E\circ \chi)|_1(f)=\diag(a_1^{{\sim{\kappa_1^j(\tht_{n,1})}}}, a_2^{{\sim{\kappa_1^j(\tht_{n,2})}}},\cd, a_{p_n}^{{\sim{\kappa_1^j(\tht_{n,p_n})}}},0^{\sim\sim}), $$
$$\pi_j^E\circ \xi^o_0(f)=\diag(a_1^{{\sim{\bar{\kappa}_0^j(\tht_{n,1})}}}, a_2^{{\sim{\bar{\kappa}_0^j(\tht_{n,2})}}},\cd, a_{p_n}^{{\sim{\bar{\kappa}_0^j(\tht_{n,p_n})}}},0^{\sim\sim}),~~ \mbox{and} $$
$$\pi_j^E\circ \xi^o_1(f)=\diag(a_1^{{\sim{\bar{\kappa}_1^j(\tht_{n,1})}}}, a_2^{{\sim{\bar{\kappa}_1^j(\tht_{n,2})}}},\cd, a_{p_n}^{{\sim{\bar{\kappa}_1^j(\tht_{n,p_n})}}},0^{\sim\sim}). $$
(Here {{again}} we use $0^{\sim\sim}$ to denote some $0$'s without specifying  how many copies {{there are.}})  The motivation of the above notation is that $\tht_{n,j}$ is the representation of $A_n$ by sending $f$ to $a_j$.

\end{NN}

\newcommand{\kp}{\kappa}
\newcommand{\ts}{\textstyle}

\begin{lem}\label{4.15} For any $1\leq i\leq p_n$, we have
\beq\label{415-f}
\bar{\kappa}_0^j(\tht_{n,i})-\kappa_0^j(\tht_{n,i})=
\bar{\kappa}_1^j(\tht_{n,i})-\kappa_1^j(\tht_{n,i})
\geq (1-\frac{1}{2^{4(n+1)}})\bar{\kappa}_0^j(\tht_{n,i}).
\eneq
\end{lem}

\begin{proof}
The equality will {{follow}} from \eqref{cd-commut1}, i.e.,
$\left(\bb_1^{n+1}-\bb_0^{n+1}\right)
\cdot \cc^{n,n+1}=
\dd^{n,n+1}\cdot\left(\bb_1^{n}-\bb_0^{n}\right).
$
To see this, we first note  {{that}}
$$(\pi_j^E\circ \chi)|_0(f)=\hspace{5in}$$
$$ \diag\Big(\diag_{\{k; d_{jk}>0\}}(f_k(0)^{\sim d_{jk}}),\diag_{\{k; d_{jk}<0\}}(f_k(1)^{\sim |d_{jk}|}), \diag_{\{k; d_{jk}=0\}}(f_k(0), f_k(1))\Big)\andeqn$$
$$(\pi_j^E\circ \chi)|_1(f)=\hspace{5in}$$
$$ \diag\Big(\diag_{\{k; d_{jk}>0\}}(f_k(1)^{\sim d_{jk}}),\diag_{\{k; d_{jk}<0\}}(f_k(0)^{\sim |d_{jk}|}), \diag_{\{k; d_{jk}=0\}}(f_k(1), f_k(0))\Big).$$

Note that $\kp_0^j(\tht_{n,i})$ (and $\kp_1^j(\tht_{n,i})$ resp.) is  the multiplicity of $\tht_{n,i}$ in $SP((\pi_j^E\circ \chi)|_0)$ (and $SP((\pi_j^E\circ \chi)|_1)$ {{resp.}}).  Also note that the homomorphism $(f_1,f_2,\cd, f_{l_n}, a_1,a_2,\cd, a_{p_n})\to f_k(0)$ defines $\et_{n,k}(0)\in SP(A_n)$ 
{{(see \ref{4.6})}}.
Hence
\beq\nonumber
SP((\pi_j^E\circ \chi)|_0)=\{{{\eta_{n,k}(0)^{\sim d_{jk}}:\,}} d_{jk}>0\} {{\cup}}
\{\eta_{n,k}(1)^{\sim |d_{jk}|}: d_{jk}<0\} {{\cup}} \{\eta_{n,k}(0),\eta_{n,k}(1): d_{jk}=0\}.
\eneq
 By (\ref{eta=tht}), we have
$$\kp_0^j(\tht_{n,i})=\sum_{\{k;~d_{jk}>0\}}b^n_{0,ki}d_{jk}
+\sum_{\{k;~d_{jk}<0\}}b^n_{1,ki}|d_{jk}|
+\sum_{\{k;~d_{jk}=0\}}(b^n_{0,ki}+b^n_{1,ki}).$$
Similarly, {{we have}}
$$\kp_1^j(\tht_{n,i})=\sum_{\{k;~d_{jk}>0\}}b^n_{1,ki}d_{jk}
+\sum_{\{k;~d_{jk}<0\}}b^n_{0,ki}|d_{jk}|
+\sum_{\{k;~d_{jk}=0\}}(b^n_{1,ki}+b^n_{0,ki}).$$
Hence $\kp_1^j(\tht_{n,i})-\kp_0^j(\tht_{n,i})=
\sum_{k=1}^{l_n}(b^n_{1,ki}-b^n_{0,ki})d_{jk}$.
On the other hand,
\beq\label{74-10}
\bar{\kp}_0^j(\tht_{n,i})=\sum_{k=1}^{p_{n+1}}b^{n+1}_{0,jk}c_{ki}
\qq \mbox{and} \qq
\bar{\kp}_1^j(\tht_{n,i})=\sum_{k=1}^{p_{n+1}}b^{n+1}_{1,jk}c_{ki}.
\eneq
Hence
$$\bar{\kp}_1^j(\tht_{n,i})-\bar{\kp}_0^j(\tht_{n,i})
=\sum_{k=1}^{p_{n+1}}(b^{n+1}_{1,jk}-b^{n+1}_{0,jk})c_{ki}.$$
That is, $\kp_1^j(\tht_{n,i})-\kp_0^j(\tht_{n,i})$
is the $ji$-th entry of the matrix $\dd^{n,n+1}(\bb^n_1-\bb^n_0)$; and
$\bar{\kp}_1^j(\tht_{n,i})-\bar{\kp}_0^j(\tht_{n,i})$
is the $ji$-th entry of the matrix $(\bb^{n+1}_1-\bb^{n+1}_0)\cc^{n,n+1}$.
By (\ref{cd-commut1}), we have
$$\kp_1^j(\tht_{n,i})-\kp_0^j(\tht_{n,i})
=\bar{\kp}_1^j(\tht_{n,i})-\bar{\kp}_0^j(\tht_{n,i})\,\,\,\,\,\text{as\,\, desired.}$$

Furthermore, it follows from $b^{n+1}_{0,jk}\geq {\cm}_{n+1}$ that
\beq
\kappa_0^j(\tht_{n,i})<\sum_{k=1}^{l_n}(|d_{jk}|+2)\{n,k\}\leq \frac{1}{2^{4(n+1)}}\cm_{n+1}\leq \frac{1}{2^{4(n+1)}} b^{n+1}_{0,j1}\leq
\frac{1}{2^{4(n+1)}}\bar{\kp}_0^j(\tht_{n,i}).
\eneq
Hence the inequality in \eqref{415-f} also follows.
\end{proof}

\begin{NN}\label{4.16} (Notation: $\sm_{k}^j(\eta_{n,i}),{\bar{\sm}}_{k}^j(\eta_{n,i}), \ld_{k}^j(\tht_{n,i}), {\bar\ld}_{k}^j(\tht_{n,i})$ for $k=0,1$---the multiplicities of $\eta_{n,i}(l/n)$ and $\tht_{n,i}$ for certain homomorphisms)

Let $f=(f_1,f_2,\cd,f_{l_n},a_1,a_2,\cd,a_{p_n})$. Then, for each fixed $j,$ one may write
$$
\ts
(\pi_j^E\circ (\chi\oplus\mu))|_0(f)=\diag\Big(f_1(\frac{1}{n})^{\sim \sm^j_0(\et_{n,1})},f_1(\frac{2}{n})^{\sim \sm^j_0(\et_{n,1})},\cd,f_1(\frac{n-1}{n})^{\sim \sm^j_0(\et_{n,1})}, \hspace{2in}$$
$$
f_2(\ts\frac{1}{n})^{\sim \sm^j_0(\et_{n,2})},\cd,f_2(\frac{n-1}{n})^{\sim \sm^j_0(\et_{n,2})};\cd,
f_{l_n}(\frac{1}{n})^{\sim \sm^j_0(\et_{n,l_n})},\cd,f_{l_n}(\frac{n-1}{n})^{\sim \sm^j_0(\et_{n,l_n})},$$
$$
\ts
a_1^{\sim \ld^j_0(\tht_{n,1})},a_2^{\sim \ld^j_0(\tht_{n,2})},\cd,
a_{p_n}^{\sim \ld^j_0(\tht_{n,p_n})}, 0^{\sim\sim}
\Big){{,}}
$$
$$
\ts
(\pi_j^E\circ (\chi\oplus\mu))|_1(f)=\diag\Big(f_1(\frac{1}{n})^{\sim \sm^j_1(\et_{n,1})},f_1(\frac{2}{n})^{\sim \sm^j_1(\et_{n,1})},\cd,f_1(\frac{n-1}{n})^{\sim \sm^j_1(\et_{n,1})},\hspace{2in}$$
$$
f_2(\ts\frac{1}{n})^{\sim \sm^j_1(\et_{n,2})},\cd,f_2(\frac{n-1}{n})^{\sim \sm^j_1(\et_{n,2})},\cd,
f_{l_n}(\frac{1}{n})^{\sim \sm^j_1(\et_{n,l_n})},\cd,f_{l_n}(\frac{n-1}{n})^{\sim \sm^j_1(\et_{n,l_n})},$$
$$
\ts
a_1^{\sim \ld^j_1(\tht_{n,1})},a_2^{\sim \ld^j_1(\tht_{n,2})},\cd,
a_{p_n}^{\sim \ld^j_1(\tht_{n,p_n})},0^{\sim\sim}
\Big){{,}}
$$
$$
\ts
(\pi_j^E\circ \xi_0)(f)=\diag\Big(f_1(\frac{1}{n})^{\sim \bar{\sm}^j_0(\et_{n,1})},f_1(\frac{2}{n})^{\sim \bar{\sm}^j_0(\et_{n,1})},\cd,f_1(\frac{n-1}{n})^{\sim \bar{\sm}^j_0(\et_{n,1})},\hspace{2in}$$
$$
f_2(\ts\frac{1}{n})^{\sim \bar{\sm}^j_0(\et_{n,2})},\cd, {{f_2(\frac{n-1}{n})^{\sim \bar{\sm}^j_0(\et_{n,2})},}}\cd,
f_{l_n}(\frac{1}{n})^{\sim \bar{\sm}^j_0(\et_{n,l_n})},\cd,f_{l_n}(\frac{n-1}{n})^{\sim \bar{\sm}^j_0(\et_{n,l_n})},$$
$$
\ts
a_1^{\sim \bar{\ld}^j_0(\tht_{n,1})},a_2^{\sim \bar{\ld}^j_0(\tht_{n,2})},\cd,
a_{p_n}^{\sim \bar{\ld}^j_0(\tht_{n,p_n})}, 0^{\sim\sim}
\Big){{,\,\,\text{and},}}
$$
$$
\ts
(\pi_j^E\circ \xi_1)(f)=\diag\Big(f_1(\frac{1}{n})^{\sim \bar{\sm}^j_1(\et_{n,1})},f_1(\frac{2}{n})^{\sim \bar{\sm}^j_1(\et_{n,1})},\cd,f_1(\frac{n-1}{n})^{\sim \bar{\sm}^j_1(\et_{n,1})}, \hspace{2in}$$
$$
f_2(\ts\frac{1}{n})^{\sim \bar{\sm}^j_1(\et_{n,2})},\cd,{{f_2(\frac{n-1}{n})^{\sim \bar{\sm}^j_1(\et_{n,2})},}}\cd,
f_{l_n}(\frac{1}{n})^{\sim \bar{\sm}^j_1(\et_{n,l_n})},\cd,f_{l_n}(\frac{n-1}{n})^{\sim \bar{\sm}^j_1(\et_{n,l_n})},$$
$$
\ts
a_1^{\sim \bar{\ld}^j_1(\tht_{n,1})},a_2^{\sim \bar{\ld}^j_1(\tht_{n,2})},\cd,
a_{p_n}^{\sim \bar{\ld}^j_1(\tht_{n,p_n})}, 0^{\sim\sim}
\Big) {{.}}
$$
In particular, the multiplicities of $f_i(\frac{1}{n}), f_i(\frac{2}{n}),\cd,f_i(\frac{n-1}{n})$, in each of
{{the four homomorphisms above,}} are the same --- this is why we use notation $\sm^j_0(\et_{n,i})$ instead of $\sm^j_0(\et_{n,i}(\frac{l}{n}))$.

\end{NN}

\begin{lem}\label{4.17} For all $1\leq i\leq l_n$,
\beq\label{417-1}
\bar{\sm}^j_0(\eta_{n,i})-\sm^j_0(\eta_{n,i})=\bar{\sm}^j_1(\eta_{n,i})
-\sm^j_1(\eta_{n,i})\leq \min(b^{n+1}_{0,j1}, b^{n+1}_{1,j1});
\eneq
and for all $1\leq i\leq p_n$,
$$\bar{\ld}^j_0(\tht_{n,i})-\ld^j_0(\tht_{n,i})=\bar{\ld}^j_1(\tht_{n,i})
-\ld^j_1(\tht_{n,i})\geq  (1-\frac{2}{2^{4(n+1)}})\bar{\kappa}_0^j(\tht_{n,i}).$$

\end{lem}
\begin{proof}
By (e) of \ref{4.10} (and $[n,i]> 2$),
 (\ref{defL}) and (\ref{c-M-large}) as well as \eqref{74-10}, for any $1\leq i\leq p_n$,
we obtain
\beq\label{+4.21}\nonumber
(n-1)\max(b^{n+1}_{0,j1},b^{n+1}_{0,j1})(\sum_{k=1}^{l_n}b^n_{0,ki})
\leq \max(b^{n+1}_{0,j1},b^{n+1}_{0,j1})(\frac{1}{2}(n-1)
\sum_{k=1}^{l_n}\{n,k\})\\
\leq \min(b^{n+1}_{0,j1},b^{n+1}_{0,j1})\cdot c_{11}\cdot\frac{1}{2^{4(n+1)}}\leq \frac{1}{2^{4(n+1)}} \min(\bar{\kappa}_0^j(\tht_{n,i}), \bar{\kappa}_1^j(\tht_{n,i})).
\eneq
From the definition of $\xi_0,\xi_1$ (see the first paragraph {{of}} \ref{4.14}) and
(\ref{psi=b1}) and  (\ref{psi=b2})),
we have
$$
\ts
(\pi_j^E\circ \xi_0)(f)=\diag\Big(f_1(\frac{1}{n})^{\sim b^{n+1}_{0,j1}},f_1(\frac{2}{n})^{\sim b^{n+1}_{0,j1}},\cd,f_1(\frac{n-1}{n})^{\sim b^{n+1}_{0,j1}},\hspace{2in}$$
$$
f_2(\ts\frac{1}{n})^{\sim b^{n+1}_{0,j1}},\cd,f_2(\frac{n-1}{n})^{\sim b^{n+1}_{0,j1}},\cd,
f_{l_n}(\frac{1}{n})^{\sim b^{n+1}_{0,j1}},\cd,f_{l_n}(\frac{n-1}{n})^{\sim b^{n+1}_{0,j1}},\hspace{1.2in}
$$
\beq\label{pi-j-chi}
\ts
\hspace{0.2in}a_1^{\sim \left(c'_{11}b^{n+1}_{0,j1}+
\sum_{k=2}^{p_{n+1}}c_{k1}b^{n+1}_{0,jk}\right)},a_2^{\sim \left(c'_{12}b^{n+1}_{0,j1}+
\sum_{k=2}^{p_{n+1}}c_{k2}b^{n+1}_{0,jk}\right)},\cd,
a_{p_n}^{\sim \left(c'_{1p_n}b^{n+1}_{0,j1}+
\sum_{k=2}^{p_{n+1}}c_{kp_n}b^{n+1}_{0,jk}\right)}
\Big){{.}}\hspace{.3in}
\eneq
Hence we always have $\bar{\sm}^j_0(\eta_{n,i})=b^{n+1}_{0,j1}$. Similarly, we have $\bar{\sm}^j_1(\eta_{n,i})=b^{n+1}_{1,j1}$.
It follows {{that}}
$\bar{\sm}^j_0(\eta_{n,i})-\sm^j_0(\eta_{n,i})\leq b^{n+1}_{0,j1}$ and $\bar{\sm}^j_1(\eta_{n,i})-\sm^j_1(\eta_{n,i})\leq b^{n+1}_{1,j1}$.
So, for \eqref{417-1}, it remains to show {{that}}
$\bar{\sm}^j_0(\eta_{n,i})-\sm^j_0(\eta_{n,i})=
\bar{\sm}^j_1(\eta_{n,i})-\sm^j_1(\eta_{n,i}).$

Combining (\ref{pi-j-chi}),  (\ref{c'}),  and \eqref{74-10}, we calculate (see  also \ref{4.16})
\beq\nonumber\label{bar-ld-0}
\bar{\ld}^j_0(\tht_{n,i})= c'_{1i}b^{n+1}_{0,j1}+\sum_{k=2}^{p_{n+1}}c_{ki}b^{n+1}_{0,jk}=
(c_{1i}-\sum_{k=1}^{l_n}b^n_{0,ki})b^{n+1}_{0,j1}+
\sum_{k=2}^{p_{n+1}}c_{ki}b^{n+1}_{0,jk}\\
=\sum_{k=1}^{p_{n+1}}c_{ki}b^{n+1}_{0,jk}- (n-1)b^{n+1}_{0,j1}\big(\sum_{k=1}^{l_n}b^n_{0,ki}\big)
=\bar{\kappa}_0^j(\tht_{n,i})-
(n-1)b^{n+1}_{0,j1}\big(\sum_{k=1}^{l_n}b^n_{0,ki}\big).
\eneq
Similarly,
\beq\label{bar-ld-1}
\bar{\ld}^j_1(\tht_{n,i})= c'_{1i}b^{n+1}_{1,j1}+\sum_{k=2}^{p_{n+1}}c_{ki}b^{n+1}_{1,jk}=
\bar{\kappa}_1^j(\tht_{n,i})-
(n-1)b^{n+1}_{1,j1}\big(\sum_{k=1}^{l_n}b^n_{0,ki}\big).
\eneq

We will divide the proof into three cases: $j\notin J_0\cup J_1$, $j\in J_0$ and $j\in J_1$.

Case 1: $j\notin J_0\cup J_1$. In this case,  $b^{n+1}_{0,j1}=b^{n+1}_{1,j1}$, $K_j=0$ and $\pi_j^E\circ(\chi\oplus \mu)=\pi_j^E\circ \chi$. Consequently
\beq\label{bar-Id-n1}
\sm^j_0(\eta_{n,i})=0=\sm^j_1(\eta_{n,i}),~~
\ld^j_0(\tht_{n,i})=\kappa_0^j(\tht_{n,i})~~\mbox{and}~~
\ld^j_1(\tht_{n,i})=\kappa_1^j(\tht_{n,i}).
\eneq
Hence $\bar{\sm}^j_0(\eta_{n,i})-\sm^j_0(\eta_{n,i})=b^{n+1}_{0,j1}=b^{n+1}_{1,j1}=
\bar{\sm}^j_1(\eta_{n,i})
-\sm^j_1(\eta_{n,i}).$

It follows from  \eqref{bar-Id-n1},
$b^{n+1}_{0,j1}=b^{n+1}_{1,j1}$, (\ref{bar-ld-0}), (\ref{bar-ld-1}), Lemma \ref{4.15} and (\ref{+4.21}) that
\beq\nonumber
&&\hspace{-0.6in}\bar{\ld}^j_1(\tht_{n,i})-\ld^j_1(\tht_{n,i})=\bar{\kappa}_0^j(\tht_{n,i})-
(n-1)b^{n+1}_{0,j1}\big(\sum_{k=1}^{l_n}b^n_{0,ki}\big)-\kappa_1^j(\tht_{n,i})\\
&&=\bar{\ld}^j_0(\tht_{n,i})-\ld^j_0(\tht_{n,i})\hspace{2in} \\
&&=
\bar{\kappa}_0^j(\tht_{n,i})-\kappa_0^j(\tht_{n,i})-
(n-1)b^{n+1}_{0,j1}\big(\sum_{k=1}^{l_n}b^n_{0,ki}\big)
\geq(1-\frac{2}{2^{4(n+1)}})\bar{\kappa}_0^j(\tht_{n,i}).
\eneq

Case 2:  $j\in J_0$. Let $i\in\{1,2,\cd,l_n\}$. By \eqref{mu0},
\beq\label{Case2}
\sm^j_1(\eta_{n,i})=0~~\mbox{and}~~\sm^j_0(\eta_{n,i})=B_j=
b^{n+1}_{0,j1}-b^{n+1}_{1,j1}.
\eneq
Recall that we have computed  above that $\bar{\sm}^j_0(\eta_{n,i})=b^{n+1}_{0,j1}$ and  $\bar{\sm}^j_1(\eta_{n,i})=b^{n+1}_{1,j1}.$
Thus (by \eqref{Case2})
$$\bar{\sm}^j_0(\eta_{n,i})-\sm^j_0(\eta_{n,i})=
b^{n+1}_{0,j1}-(b^{n+1}_{0,j1}-b^{n+1}_{1,j1})=b^{n+1}_{1,j1}
=\bar{\sm}^j_1(\eta_{n,i})-\sm^j_1(\eta_{n,i}).$$
Now let $i\in\{1,2,\cd, p_n\}$. We will calculate $\ld^j_0(\tht_{n,i}),\ld^j_0(\tht_{n,i}), \bar{\ld}^j_0(\tht_{n,i})$ and $\bar{\ld}^j_1(\tht_{n,i})$. 
From (\ref{mu0}) (and (\ref{eta=tht})), we have  that
$$SP((\pi_j^E\circ\mu)|_1)=\bigcup\{\eta_{n,k}(0)^{\sim (n-1)B_j}: 1\leq k\leq l_n\}=\{\tht_{n,i}^{{\sim{(n-1)B_j\sum_{k=1}^{l_n} b^n_{0,ki}}}}: 1\leq i\leq p_n\},$$
and that $SP((\pi_j^E\circ\mu)|_0)$ does not contain any $\tht_{n,i}$.

Recall that $\ld^j_1(\tht_{n,i})$ is the multiplicity of $\tht_{n,i}$ of the spectrum of $(\pi_j^E\circ(\chi\oplus \mu))|_1.$
Hence
\beq\label{ld-kappa}
\ld^j_1(\tht_{n,i})=\kappa_1^j(\tht_{n,i})+(n-1)B_j\sum_{k=1}^{l_n} b^n_{0,ki}~~\mbox{and}~~\ld^j_0(\tht_{n,i})=\kappa_0^j(\tht_{n,i}).
\eneq
By (\ref{bar-ld-0}), (\ref{bar-ld-1})  and (\ref{ld-kappa}) as well as \ref{4.15}, we have
\beq\nonumber
\bar{\ld}^j_0(\tht_{n,i})-\ld^j_0(\tht_{n,i})
=\bar{\kappa}_0^j(\tht_{n,i})-\kappa_0^j(\tht_{n,i})
-(n-1)b^{n+1}_{0,j1}\big(\sum_{k=1}^{l_n}b^n_{0,ki}\big)\hspace{0.8in}
 \\ \nonumber
=\bar{\kappa}_1^j(\tht_{n,i})-\kappa_1^j(\tht_{n,i})
-(n-1)b^{n+1}_{0,j1}\big(\sum_{k=1}^{l_n}b^n_{0,ki}\big)\hspace{1.8in}\\
\nonumber
={ {\bar{\kappa}_1^j}}(\tht_{n,i})-\kappa_1^j(\tht_{n,i})
-(n-1)(B_j+b^{n+1}_{1,j1})\big(\sum_{k=1}^{l_n}b^n_{0,ki}\big)
=\bar{\ld}^j_1(\tht_{n,i})-\ld^j_1(\tht_{n,i}).
\eneq

{{Also,}} by (\ref{+4.21}) and Lemma \ref{4.15},  {{we  estimate}}
$$\bar{\ld}^j_0(\tht_{n,i})-\ld^j_0(\tht_{n,i})
=\bar{\kappa}_0^j(\tht_{n,i})-\kappa_0^j(\tht_{n,i})-
(n-1)b^{n+1}_{0,j1}\big(\sum_{k=1}^{l_n}b^n_{0,ki}\big)
\geq (1-\frac{2}{2^{4(n+1)}})\bar{\kappa}_0^j(\tht_{n,i}).$$

Case 3: $j\in J_1$. This case is  proved  exactly the same as the case 2 but
replacing (\ref{mu0}) {{with}} (\ref{mu1}).

\end{proof}






\begin{NN}\label{4.18} (Definition of $\phi^o,\phi$)~~ 
Set
$$
\kappa^j(i)=\bar{\kappa}^j_0(\tht_{n,i})-\kappa^j_0(\tht_{n,i})=
\bar{\kappa}^j_1(\tht_{n,i})-\kappa^j_1(\tht_{n,i}),$$
$$\sm^j(i)=\bar{\sm}^j_0(\eta_{n,i})-\sm^j_0(\eta_{n,i})=
\bar{\sm}^j_1(\eta_{n,i})-\sm^j_1(\eta_{n,i}), \andeqn$$
$$\ld^j(i)=\bar{\ld}^j_0(\tht_{n,i})-\ld^j_0(\tht_{n,i})=
\bar{\ld}^j_1(\tht_{n,i})-\ld^j_1(\tht_{n,i}).
$$
Let $\phi^o: A_n\to C([0,1],\bigoplus_{j=1}^{l_{n+1}}M_{o(j)}(\C))$ be the homomorphism defined by sending\\ $f=(f_1,f_2,\cd,f_{l_n}, a_1,a_2,\cd, a_{p_n})$ to
\beq\label{phi^o}
\phi^o(f)=\bigoplus_{j=1}^{l_{n+1}}\diag(\pi_j^E\circ\chi (f), a_1^{\sim \kappa^j(1)}, a_2^{\sim \kappa^j(2)}, \cd a_{p_n}^{\sim \kappa^j(p_n)}), \eneq
where $o(j)=L_j+\sum_{i=1}^{p_n}\kappa^j(i)[n,i].$
Let $\phi: A_n\to C([0,1], \bigoplus_{j=1}^{l_{n+1}}M_{o'(j)}(\C))$ be the homomorphism
defined by sending $f=(f_1,f_2,\cd,f_{l_n}, a_1,a_2,\cd, a_{p_n})$ to
\beq\nonumber
\ts
\phi(f)=\bigoplus_{j=1}^{l_{n+1}}\diag\Big((\pi_j^E\circ (\chi\oplus\mu))(f),f_1(\frac{1}{n})^{\sim \sm^j(1)},f_1(\frac{2}{n})^{\sim \sm^j(1)},\cd,f_1(\frac{n-1}{n})^{\sim \sm^j(1)}, \hspace{2in}\\ \nonumber
f_2(\ts\frac{1}{n})^{\sim \sm^j(2)},\cd,f_2(\frac{n-1}{n})^{\sim \sm^j(2)};\cd,
f_{l_n}(\frac{1}{n})^{\sim \sm^j(l_n)},\cd,f_{l_n}(\frac{n-1}{n})^{\sim \sm^j(l_n)},\hspace{1.8in}
\eneq
\beq\label{phi}
\ts
\hspace{1in} a_1^{\sim \ld^j(1)},a_2^{\sim \ld^j(2)},\cd,
a_{p_n}^{\sim \ld^j(p_n)}
\Big),
\eneq
where $o'(j)=L_j+K_j+(n-1)\sum_{k=1}^{l_n}\sm^j(k)\{n,k\}+
\sum_{i=1}^{p_n}\ld^j(i)[n,i].$

With the following lemma,  both maps $\phi^o$ and $\phi$ can be regarded as  maps from $A_n$ to $C([0,1], E_{n+1})$ by adding
suitably many copies of $0$'s in the equations  (\ref{phi^o}) and (\ref{phi}).
\end{NN}


\begin{lem}\label{4.21} We have the inequalities  $o(j)\leq \{n+1, j\}$ and $o'(j)\leq \{n+1,j\}$
(see \ref{4.18}).  Furthermore, we have
\beq\label{ld-j-i}
\ld^j(i)\geq (1-\frac{2}{2^{4(n+1)}})\sum_{k=1}^{p_{n+1}}b^{n+1}_{0,jk}c_{ki}\,\,\, and\,\,
\sum_{i=1}^{p_n}\ld^j(i)[n,i]>(1-\frac{2}{4^{n+1}})\{n+1,j\}.
\eneq
\end{lem}
\begin{proof}
Recall (see \ref{4.18}) that $\kappa^j(i)~\leq ~ \bar{\kappa}^j_0(\tht_{n,i}), \ld^j(i)~\leq ~ \bar{\ld}^j_0(\tht_{n,i})~\leq~  \bar{\kappa}^j_0(\tht_{n,i})$. Also from Lemma \ref{4.17}, {{we have}}  $\sm^j(k)\leq \min(b^{n+1}_{0,j1}, b^{n+1}_{1,j1})$. By (\ref{defL}) and (\ref{c-M-large}), we have
\beq\label{sigm<}\nonumber
n\sum_{k=1}^{l_n}\sm^j(k)\{n,k\}\leq n\sum_{k=1}^{l_n}\min(b^{n+1}_{0,j1}, b^{n+1}_{1,j1})
\leq \frac{1}{2^{4(n+1)}}{\cal L}_{n+1}\min(b^{n+1}_{0,j1}, b^{n+1}_{1,j1})\\
\leq \frac{1}{2^{4(n+1)}}c_{11}\min(b^{n+1}_{0,j1}, b^{n+1}_{1,j1})
\leq \frac{1}{2^{4(n+1)}}[n+1,1]\min(b^{n+1}_{0,j1}, b^{n+1}_{1,j1}).
\eneq
From the definition of ${\cal M}_{n+1}$ (see (\ref{cm})),
we know that
\beq\label{L+K}
\sum_{j=1}^{l_{n+1}}(L_j+K_j)\leq \frac{1}{2^{4(n+1)}}\cm_{n+1}.
\eneq
Combining with \eqref{74-10},  \eqref{cm},  \eqref{>cm}, and \eqref{b-b},
(and recall $L_j=\sum_{k=1}^{l_n}\tilde{d}_{jk}\{n,k\}$),
 we have
\beq\nonumber
&&\hspace{-0.4in}  o(j)=L_j+\sum_{i=1}^{p_n}\kappa^j(i)[n,i]
\leq L_j+\sum_{i=1}^{p_n}\bar{\kappa}^j_0(\tht_{n,i})[n,i]\hspace{0.8in}\\ \nonumber
&&= L_j+\sum_{i=1}^{p_n}\sum_{k=1}^{p_{n+1}}b^{n+1}_{0,jk}c_{ki}[n,i]
=\sum_{k=1}^{l_n}\tilde{d}_{jk}\{n,k\}+
\sum_{k=1}^{p_{n+1}}[n+1,k]b^{n+1}_{0,jk}\\ \nonumber
&&\le {\cal M}_{n+1}/2^{4(n+1)}+\sum_{k=1}^{p_{n+1}}[n+1,k]b^{n+1}_{0,jk}
\leq(1+\frac{2}{2^{4(n+1)}})\sum_{k=1}^{p_{n+1}}\min(b^{n+1}_{0,jk}, b^{n+1}_{1,jk})[n+1,k].
\eneq
Combining  \eqref{L+K}, \eqref{cm}, and (\ref{sigm<}), {{we obtain}}
\beq\nonumber
o'(j)=L_j+K_j+(n-1)\sum_{k=1}^{l_n}\sm^j(k)\{n,k\}+
\sum_{i=1}^{p_n}\ld^j(i)[n,i]\hspace{1in} \\ \nonumber \leq(1+\frac{2}{2^{4(n+1)}})\sum_{k=1}^{p_{n+1}}\min(b^{n+1}_{0,jk}, b^{n+1}_{1,jk})[n+1,k]+(n-1)\sum_{k=1}^{l_n}\sm^j(k)\{n,k\}\\ \nonumber
\leq(1+\frac{3}{2^{4(n+1)}})\sum_{k=1}^{p_{n+1}}\min(b^{n+1}_{0,jk}, b^{n+1}_{1,jk})[n+1,k].
\eneq
In summary, we conclude from (\ref{(n+1,i)}) {{that}}
\beq\nonumber
\max(o(j),o'(j))\leq (1+\frac{3}{2^{4(n+1)}})\sum_{k=1}^{p_{n+1}}\min(b^{n+1}_{0,jk}, b^{n+1}_{1,jk})[n+1,k]\leq \{n+1,j\}.
\eneq
The first inequality of (\ref{ld-j-i}) follows from Lemma \ref{4.17} and (\ref{74-10}).  Using (\ref{n-F-(n+1)}) and (e) in \ref{4.10} (with $n+1$ in place of $n$), we calculate
\beq\nonumber
\sum_{i=1}^{p_n}\ld^j(i)[n,i]\geq (1-\frac{2}{2^{4(n+1)}})\sum_{k=1}^{p_{n+1}} \sum_{i=1}^{p_n}b^{n+1}_{0,jk}c_{ki}[n,i]\hspace{1.4in}\\ \nonumber
\geq (1-\frac{2}{2^{4(n+1)}})(1-\frac{1}{8^n})
\sum_{k=1}^{p_{n+1}}b^{n+1}_{0,jk}[n+1,k]\hspace{1.8in}\\ \nonumber
\geq (1-\frac{2}{2^{4(n+1)}})(1-\frac{1}{8^n})(1-\frac{1}{4^{n+1}})\{n+1,j\}>
(1-\frac{2}{4^{n+1}})\{n+1,j\}.
 \eneq
\end{proof}


From definition of $\phi^o$ and $\phi$, we have {{the following lemma.}}

\begin{lem}\label{4.19} $SP(\phi^o|_0)=SP(\xi^o_0)$, $SP(\phi^o|_1)=SP(\xi^o_1)$, $SP(\phi|_0)=SP(\xi_0)$ and\\  $SP(\phi|_1)={{SP(\xi_1).}}$
\end{lem}
\begin{proof} Note that $SP(\pi_j^E\circ\phi^o|_0)=\left\{SP((\pi_j^E\circ \chi)|_0), \tht_{n,1}^{\sim \kappa^j(1)}, \tht_{n,2}^{\sim \kappa^j(2)},\cdots,\tht_{n,p_n}^{\sim \kappa^j(p_n)}\right\}=SP(\pi_j^E\circ \xi^o_0)$, since $\kappa^j(i)=\bar{\kappa}^j_0(\tht_{n,i})-\kappa^j_0(\tht_{n,i})$. The {{proofs}} of other three parts are similar.

\end{proof}

\begin{NN}\label{4.20} (Definition of $\phi_{n,n+1}^o$ and $\phi_{n,n+1}$) By the lemma above, there are unitaries $V_0, V_1, U_0, U_1\in E_{n+1}$ such that
$\mbox{Ad} V_0\circ (\phi^o|_0)=\xi^o_0$, $\mbox{Ad} V_1\circ (\phi^o|_1)=\xi^o_1$, $\mbox{Ad} U_0\circ (\phi|_0)=\xi_0$ and $\mbox{Ad} U_1\circ (\phi|_1)=\xi_1$. Since the unitary group of $E_{n+1}$ is path connected, there are two
continuous paths of unitaries ${{V:=}}V(t),{{U:=}}U(t), 0\leq t\leq 1$  such that $V(0)=V_0$, $V(1)=V_1$, $U(0)=U_0$ and  $U(1)=U_1$.

Now we define two \hm s
$\phi^o_{n,n+1},\phi_{n,n+1}: A_n\to A_{n+1}$ by
\beq\label{Dphinn+1}
\phi^o_{n,n+1}(f)=(\mbox{Ad} {{V}} \circ \phi^o)(f)\oplus \psi^o(f)\in A_{n+1} \subset C([0,1], E_{n+1})\oplus F_{n+1}~~~\mbox{and}\\\label{Dphinn+2}
\phi_{n,n+1}(f)=(\mbox{Ad} {{U}} \circ \phi)(f)\oplus \psi(f)\in A_{n+1} \subset C([0,1], E_{n+1})\oplus F_{n+1}.
\eneq
Note that $(\mbox{Ad} {{V}}\circ \phi^o)(f)(0)=(\mbox{Ad} V_0\circ (\phi^o|_0))(f)=\xi^o_0(f)=\bt_{n+1,0}(\psi^o(f))$ and \\$(\mbox{Ad} {{V}} \circ \phi^o)(f)(1)=(\mbox{Ad} V_1\circ (\phi^o|_1))(f)=\xi^o_1(f)=\bt_{n+1,1}(\psi^o(f))$.
{{In other words,}} the {{element}}  $\phi^o_{n,n+1}(f)$ is in $A_{n+1}$ rather than $C([0,1], E_{n+1})\oplus F_{n+1}$.  Similarly, $\phi_{n,n+1}$
is also a  \hm\,  from $A_n$ to
 $A_{n+1}$.

 So the construction is completed.
 We obtain two inductive systems $A^o=\lim(A_n,\phi^o_{n,n+1})$ and $A=\lim(A_n, \phi_{n,n+1}).$ We will  summarize the properties of $\phi_{n,n+1}^o$ and $\phi_{n,n+1}$ in \ref{4.23n+}, \ref{4.24} and Lemma \ref{4.25} {{below,}} and use these properties to prove that $A$ is a simple $C^*$-algebra which satisfies the desired properties in
\ref{range4.4}.

Before we present other properties of $\phi_{n,n+1}^o$ and $\phi_{n,n+1}$, let us {{point}} out that both of them are injective, since $\chi$ is injective. Also let us rewrite {{the}} part of property (c) in \ref{4.10} as\\
(c') $\tilde{\rho}(\gm_{n,\infty}(u_n))\geq (1-1/8^n)\cdot 1_\DT$ (since  $\gm_{n,\infty}=\gm'_{k(n),\infty}$).

\end{NN}

\begin{NN}\label{4.22} ($K$-theory of $\phi_{n,n+1}$)
Let $I_n=C_0\big((0,1),E_n\big)\sbs A_n$ and $I_{n+1}=C_0\big((0,1),E_{n+1}\big)\sbs A_{n+1}$ be the  ideals. From the definition of $\phi^o_{n,n+1},$ we have the following {{commutative}} diagram
\begin{displaymath}
\xymatrix{
A_n~~~~~~\ar[r]^{\phi^o_{_{n,n+1}}} \ar@{->}[d]^{\pi_n^A
} &
~~~~~~A_{n+1}  \ar@{->}[d]^{\pi^A_{n+1}}
\\
 A_n/I_n=F_n~~~ \ar[r]^{\psi_{_{n,n+1}}}& ~~~A_{n+1}/I_{n+1}=F_{n+1}.
  }
\end{displaymath}
Note that $K_0(A_n)=G_n$, $K_0(A_{n+1})=G_{n+1}$, $K_0(F_n)=H_n$, $K_0(F_{n+1})=H_{n+1}$, and the maps
$$(\pi^A_n)_*:K_0(A_n)\to K_0(F_n)=H_n\,\,\, \mbox{and}\,\,\,
(\pi^A_{n+1})_*:K_0(A_{n+1})\to K_0(F_{n+1})=H_{n+1}
$$
are inclusions.  Also, recall that $(\psi_{n,n+1})_*=\gm_{n,n+1}: H_n\to H_{n+1}$.  Consequently,
$(\phi^o_{n,n+1})_*=\gm_{n,n+1}|_{_{G_n}}:~ K_0(A_n)\to K_0(A_{n+1})$.
The {{the commutative diagram}}
\begin{displaymath}
\xymatrix{
A_1\ar[r]^{\phi^o_{_{12}}} \ar@{->}[d] &
A_2 \ar[r]^{\phi^o_{_{23}}} \ar@{->}[d]& \cd &\ar[r] & A^o \ar@{->}[d]
\\
 F_1 \ar[r]^{\psi_{_{12}}}& F_{2}\ar[r]^{\psi_{_{23}}} & \cd &\ar[r] & F
  }
\end{displaymath}
induces the {{following commutative diagram}}
\begin{displaymath}
\xymatrix{
K_0(A_1)=G_1\ar[r]^{\gm_{12}|_{_{G_1}}} \ar@{->}[d] &
K_0(A_2)=G_2 \ar[r]^{~~~~~\gm_{23}|_{_{G_2}}} \ar@{->}[d]& \cd &\ar[r] & K_0(A^o)=G \ar@{->}[d]
\\
 K_0(F_1)=H_1 \ar[r]^{\gm_{_{12}}}& K_0(F_2)=H_2\ar[r]^{~~~~\gm_{_{23}}} & \cd &\ar[r] & K_0(F)=H\,.
  }
\end{displaymath}
Hence $K_0(A^o)=G$ as a subgroup of $K_0(F)=H.$

On the other hand, $\pi_{n+1}^A\circ \phi^o_{n,n+1}=\psi^o$ is homotopy equivalent to $\pi_{n+1}^A\circ \phi_{n,n+1}=\psi$ (see (\ref{KKpsi})). Thus we know that\\
$\big(\pi_{n+1}^A\circ \phi_{n,n+1}\big)_*
=\big(\pi_{n+1}^A\circ \phi^o_{n,n+1}\big)_*:~K_0(A_n)=G_n\to K_0(F_{n+1})=H_{n+1}.$
Consequently, $(\pi_{n+1}^A)_*\circ (\phi_{n,n+1})_* =(\pi_{n+1}^A)_*\circ \gm_{n,n+1}|_{_{G_n}}$.
Since $(\pi_{n+1}^A)_*$ is an inclusion, we have $(\phi_{n,n+1})_*=\gm_{n,n+1}|_{_{G_n}}$.  Therefore, the inductive limit
$$A_1\stackrel{\phi_{12}}{\lr}A_2\stackrel{\phi_{23}}{\lr}A_3\lr \cd ~~~~~\lr A$$
also induces the K-theory maps
$$K_0(A_1)=G_1 \stackrel{\gm_{12}|_{_{G_1}}}{\lr}K_0(A_2)\stackrel{\gm_{23}|_{_{G_2}}}{\lr} {{K_0(A_3)}}\lr \cd ~~~~~\lr K_0(A) ~~( \mbox{or} ~ K_0(A^o)).$$
Hence $K_0(A)=G=K_0(A^o).$  {{Since $A_n\in {\cal C}_0,$ $K_1(A)=\{0\}.$}}
\end{NN}

\begin{lem}\label{4.23n+} $A$ is a simple $C^*$-algebra.
\end{lem}
\begin{proof}
Note  {{that,}} by \ref{4.17} and \ref{4.18}, $\lambda^j(i)>0.$ From  \eqref{Dphinn+2} and \eqref{phi}, it follows that
\beq\label{tht-eta}\{\tht_{n,i}: 1\leq i\leq p_n\}\subset Sp(\phi_{n,n+1}|_{\eta_{n+1,j}(t)})~~\mbox{for any}~~1\leq j\leq l_{n+1},0<t<1.\eneq
Note that from 
(\ref{Dphinn+2}), for any $\tht\in Sp(F_{n+1})$, $SP((\phi_{n,n+1}|)_\tht)=SP(\psi|_\tht)$ (see \ref{4.4.1}). It follows from definition of $\psi$ in \ref{4.12} (see (\ref{psi=b1}) and (\ref{psi=b2})) {{that}} we have
\beq\label{Sp-tht}
&&\{\tht_{n,i},1\leq i\leq p_n\}\subset Sp(\phi_{n,n+1}|_{\tht_{n+1,j}})~\mbox{ for all}~~ 1\leq j\leq p_{n+1}
\andeqn\\
\label{Sp-tht-1}
&&\{\eta_{n,j}({k\over{n}}), \theta_{n,i}:~ 1\leq i\leq p_n,1\leq j\leq l_n, 1\leq k\leq n-1\}\subset {\rm Sp}({\phi_{n,n+1}}|_{\tht_{n+1,1}})
\eneq
{{(see also  (\ref{psi=b1})).}}

From (\ref{phi}) (which {{tells}} us $\chi$ is a part of $\phi$), \eqref{Dphinn+2} (which {{tells}} us,  for any $\eta_{n+1,j}(t):=t\in (0,1)_j \subset C([0,1], M_{\{n+1,j\}}(\C))),$
 that $SP((\phi_{n,n+1})_{\eta_{n+1,j}(t)})=SP(\phi|_{\eta_{n+1,j}(t)})$), and (\ref{SP-chi}),  we know that for all $1\leq j\leq l_{n+1}$,
\beq\label{Sp-(n,n+1)}
\{\eta_{n,i}(t), \eta_{n,i}(1-t): 1\leq i\leq l_n\}\subset Sp(\phi_{n,n+1}|_{\eta_{n+1,j}(t)})\cup Sp(\phi_{n,n+1}|_{\eta_{n+1,j}(1-t)})~.
\eneq

 For the composition of two homomorphisms $\phi:A\to B$ and $\psi: B\to C$ among three sub-homogeneous algebras, it is
{{well known}} and easy to see that, for any $x\in Sp(C)$, one has $Sp(\psi\circ \phi)|_x=\bigcup_{y\in Sp(\psi|_x)} Sp(\phi|_y)$.  We will repeatedly use this fact.

From  (\ref{Sp-tht}), we have
\beq\label{Sp-tht-m}\{\tht_{n,i},1\leq i\leq p_n\}\subset Sp(\phi_{n,m}|_{\tht_{m,j}})~\mbox{ for all}~~ 1\leq j\leq p_{m}.\eneq
 From (\ref{Sp-(n,n+1)}), we know that, for any $m\geq n+1$ and  $1\leq j\leq l_{m}$,
\beq\label{Sp-(n,m)}
\{\eta_{n,i}(t), \eta_{n,i}(1-t): 1\leq i\leq l_n\}\subset  {{Sp}}(\phi_{n,m}|_{\eta_{m,j}(t)})\cup Sp(\phi_{n,m}|_{\eta_{m,j}(1-t)})~.
\eneq

{{Fix}}  $m+2>m\geq n$. Let $Z:=\{\eta_{m,j}({k\over{m}}), \theta_{m,i}:~ 1\leq i\leq p_m,1\leq j\leq l_m, 1\leq k\leq m-1\}$. It follows from (\ref{Sp-(n,m)}) and (\ref{Sp-tht-m}) {{that}}
$$\{\eta_{n,j}({k\over{m}}), \theta_{n,i}:~ 1\leq i\leq p_n,1\leq j\leq l_n, 1\leq k\leq m-1\}\subset \bigcup_{z\in Z} Sp(\phi_{n,m}|_z).$$
(When $n=m$, we use the convention $\phi_{n,n}=\id$.) {{Applying}} (\ref{Sp-tht-1}) with $m$ in place of $n$, we know that $Z\subset Sp(\phi_{m,m+1}|_{\tht_{m+1,1}})$.
For any $x\in Sp(A_{m+2})$,  from (\ref{Sp-tht-1}) with $m+1$ in place of $n$, we have  $\tht_{m+1,1}\subset Sp(\phi_{m+1, m+2}|_x)$. Hence
\beq\label{2020-917-1}
Sp(\phi_{n,m+2}|_x)\sps \left\{\et_{n,i}{{\left({\ts\frac{k}{m}}\right):}}~
   ~_{{1\leq k\leq m-1};~~
   { 1\leq i\leq l_n}}
\right\}
\bigcup
\Big\{{{\tht_{n,i}:}}~
    1\leq i\leq p_n
\Big\}.
\eneq
The latter is
{{$\frac{1}{m}$-dense}} in $Sp(A_n)$ (see \eqref{SpAn} and lines below that).

It is standard to show {{that}} $A$ is simple (see \cite{DNNP}). To see this, let $a, b\in A_+^{\bf 1}{{\setminus\{0\}}}.$  It suffices to show that $b$ is in the (closed) ideal
generated by $a.$  Let $1/2>\ep>0.$ There {{are}} $n\ge 1$ and $a_0\in (A_n)_+^{\bf 1}$ such that $\|\phi_{n,\infty}(a_0)-a\|<\ep/4.$
It follows from Lemma 3.1 of \cite{eglnp} that there is $r_0\in A$  such that
\beq\label{2020-917-2}
0\not=(\phi_{n, \infty}(a_0)-\ep/4)_+=r_0^*ar_0.
\eneq
Put $a_1:=(a_0-\ep/4)_+\in A_n.$  Choose an integer $m'>n$ and $b_1\in {A_{m'}}_+$ such that
$\|\phi_{m', \infty}(b_1)-b\|<\ep/4.$   By \eqref{SpAn} and lines below that, we may assume
that,  for some { {$\dt>0$}} 
and for any { {$\dt$-dense}} 
subset  $S$ of $Sp(A_n),$
there is $s\in S$ such that $a_1(s)>0.$  Choose     $m>n,$ such that $1/m<\dt$. {{Then}}
by what has been proved above (see \eqref{2020-917-1}), {{we have
}}
\beq
\phi_{n, m+2}(a_1)(x)>0\rforal x\in Sp(A_{m+2}).
\eneq
By choosing a large $m, $ we may assume that $m>m'.$
It follows from Proposition 6.3 of \cite{eglnp} that $\phi_{n,m+2}(a_1)$ is full in $A_{m+2}.$
Therefore, there are $x_1, x_2,...,x_K\in A_{m+2}$ such that
\beq
\|\sum_{i=1}^K x_i^* \phi_{n, m+2}(a_1)x_i-\phi_{m', m+2}(b_1)\|<\ep/4.
\eneq
This implies that (see \eqref{2020-917-2})
\beq
\|\sum_{i=1}^K\phi_{m+2, \infty}(x_i)^* r_0^*ar_0\phi_{m+2, \infty}(x_i)-b\|<\ep.
\eneq
This shows that $b$ is in the closed ideal generated by $a,$ whence $A$ is simple.
%
\end{proof}


\begin{NN}\label{4.23}  Let $\partial_e(T(A))$ denote the {{set}} of {{extremal}} tracial states {{of $A.$}}  It is well known (see Lemma 2.2 of \cite{Thm-K-theory}) that there is {{a}} one to one correspondence between $Sp(A_n)$ and $\partial_e(T(A_n))$ given by sending the irreducible representation
$\tht: A_n \to M_l(\C)$ to the {{extremal}} trace $\tau_\tht$ defined by $\tau_\tht(a)={\rm tr} (\tht(a))$, where ${\rm tr}$ is the normalized  trace on $M_l(\C)$.
Using the calculation in 3.8 of \cite{GLN} (see \cite{Thm-K-theory} also), we know that {{$\Aff(T_0(A_n))$ (for the definition of $T_0(A_n),$ see
  \ref{DTtilde})}} consists of elements $(g_1,g_2,\cd, g_{l_n}, x_1,x_2,\cd, x_n)\in C([0,1], \R^{l_n})\oplus \R^{p_n}$ with the following condition
\beq\label{Aff-bd}
g_j(0)=\frac{1}{\{n,i\}}\sum_{i=1}^{p_n}b^n_{0,ji}x_i[n,i]~~
\mbox{and}~~
g_j(1)=\frac{1}{\{n,i\}}\sum_{i=1}^{p_n}b^n_{1,ji}x_i[n,i].
\eneq
Note that the  norm on $\Aff({ {T_0}}(A_n))$ is given by
$$\|(g_1,g_2,\cd, g_{l_n}, x_1,x_2,\cd, x_n)\|=\max\{\sup_{1\leq t\leq 1}|g_j(t)|, |x_i|; 1\leq j\le l_n,1\leq i\leq p_n\}.$$

Let $\psi_{n,n+1}^{T}: T_0(F_{n+1})\to T_0(F_n)$ and  $\psi_{n,n+1}^{\sharp}: \Aff(T_0(F_n))\to \Aff (T_0(F_{n+1}))$ be
the affine {{maps}} induced by $\psi_{n,n+1}: F_n\to F_{n+1}$, and, $\phi_{n,n+1}^{T}: T_0(A_{n+1})\to T_0(A_n)$ {{and}}  $\phi_{n,n+1}^{\sharp}: \Aff(T_0(A_n))\to \Aff (T_0(A_{n+1}))$ be induced by $\phi_{n,n+1}: A_n\to A_{n+1}$. Note that $F_n$ is unital. There is a unique element in $\Aff(T_0(F_n))$, denoted by $1_{T(F_n)}$ such that $1_{T(F_n)}(\tau) =1$ for all $\tau\in T(F_n)$. Even though $\Aff(T_0(F_n))$ and $\Aff(T_0(F_{n+1}))$ have  {{units $1_{T(F_n)}$ and $1_{T(F_{n+1})},$}} respectively, $\psi_{n,n+1}^{\sharp}$ does not preserve the units, since $\psi_{n,n+1}$ is not unital.
\end{NN}

\begin{lem}\label{4.23.1} (a)  $\psi_{n,n+1}^{\sharp} (1_{T(F_n)})\geq (1-\frac{1}{8^n})\cdot 1_{T(F_{n+1})}.$ (Equivalently, for any $\tau \in T(F_{n+1})$, $\|\psi_{n,n+1}^{T}(\tau)\|\geq (1-\frac{1}{8^n})$.)
 Consequently,
$\psi_{n,m}^{\sharp} (1_{T(F_n)})\geq \left(\prod_{i=n}^{m-1}(1-\frac{1}{8^i})\right)\cdot 1_{T(F_{m})}$ {{for}} any $m>n$.

(b) Suppose that  $f\in \Aff (T_0(A_n))$ with $\|f\|\leq 1$  satisfying  $\pi_n^{A\,\sharp}(f)\geq \af\cdot 1_{T(F_n)}\in \Aff (T_0(F_n))$ for {{some}} $\af\in (0,1]$, {{where $\pi_n^{A\,\sharp}: \Aff (T_0(A_n))\to \Aff(T_0(F_n))$ is induced by $\pi_n^A: A_n\to F_n$.}}  Then, for any $\tau \in T(A_{n+1})$,
$\phi_{n,n+1}^{\sharp}(f)(\tau)\geq (1-\frac{2}{4^{n+1}})\af$.
Consequently, for   $a^A_n, e_n^A\in (A_n)_+$ (see \ref{range4.6.1}), we have
\beq\nonumber
\phi_{n,n+1}^{\sharp}({{\widehat{e_n^A})}}(\tau)\geq \phi_{n,n+1}^{\sharp}({{\widehat{a_n^A}}})(\tau)\geq 1-\frac{2}{4^{n+1}} ~~\mbox{for any}~~ \tau\in T(A_{n+1}),~~~\mbox{and}\\
\label{phi-sharp}
\phi_{n,m}^{\sharp}({{\widehat{e_n^A}}})(\tau)\geq \phi_{n,m}^{\sharp}({{\widehat{a_n^A}}})(\tau)\geq \left(\prod_{i=n}^{m-2}(1-\frac{1}{8^i})\right) (1-\frac{2}{4^{m}})~~\mbox{for any}~~m>n+1,
\eneq
where ${{\widehat{e_n^A}}}$ and ${{\widehat{a_n^A}}}$ are  the elements in $\Aff(T_0(A_n))$ corresponding to $e_n^A$ and $a^A_n,$ respectively. (Also,   for any $\tau \in {{\tilde{T}}}(A_{n+1})$, $\|\phi_{n,n+1}^{T}(\tau)\|\geq (1-\frac{2}{4^{n+1}})$.) Furthermore, for any $\tau\in T(A),$
\beq\label{phi-sharp-infty}
\phi_{n,\infty}^{\sharp}({{\widehat{e_n^A}}})(\tau)\geq \phi_{n,\infty}^{\sharp}({{\widehat{a_n^A}}})(\tau)\geq \prod_{i=n}^{\infty}(1-\frac{1}{8^i})\geq (1-\frac{1}{4^n}).
\eneq
\end{lem}
\begin{proof}(a) We only need to show {{that}} $\psi_{n,n+1}^{\sharp} (1_{T(F_n)}) (\tau)=\psi_{n,n+1}^{T}(\tau) (1_{F_n}){{\geq 1-1/8^n}}$ for $\tau\in \partial_e(T(F_{n+1}))$. Let $\tau= \tau_{\tht_{n+1,i}}$ be defined by
$\tau (a_1,a_2,\cdots, a_{p_{n+1}})={\rm tr}(a_i)$, where ${\rm tr}$ is
the normalized trace of $F_{n+1}^i=M_{[n+1,i]}(\C)$. Let ${\rm tr}_j$ denote the normalized trace on $F_n^j=M_{[n,j]}(\C)$. Then, for $b=(b_1,b_2,\cdots, b_{p_n})$, we have {{(recall $c_{ij}=c^{n, n+1}_{ij}$)}}
$$\psi_{n,n+1}^{T}(\tau) (b)=\frac{1}{[n+1,i]}\sum_{j=1}^{p_n} {\rm tr}_j (b_j) {{c_{ij}}}[n,j].$$
In particular, if $b=1_{F_n}$, then by (\ref{n-F-(n+1)}) (note that $\hat{1}_{F_n}=1_{T(F_n)}$), {{we have}}
$$\psi_{n,n+1}^{T}(\tau) (1_{F_n})=\frac{1}{[n+1,i]}\sum_{j=1}^{p_n} {\rm tr}_j (b_j) {{c_{ij}}}[n,j]\geq (1-1/8^n).$$

(b) Keep the notation from the proof of part (a). Again we only need to calculate  $\phi_{n,n+1}^{\sharp}(f)(\tau)$ for  $\tau \in \partial_e(T(F_{n+1}))$. First suppose that $\tau= \tau_{\tht_{n+1,i}}$ defined by
$\tau (f_1, f_2,\cdots, f_{l_n}, a_1,a_2,\cdots, a_{p_{n+1}})={\rm tr}(a_i)$. Writing $f=(g_1,g_2,\cdots,g_{l_n}, b_1,b_2, \cdots, b_{p_n})\in \Aff (T_0(A_n))$, we have $b_i\geq\af \in \R$ for all $i$.    Then, by (\ref{Dphinn+2}),(\ref{psi=b1}), (\ref{psi=b2}) and (\ref{c'-large}),
{{we have}}
\beq\nonumber
  \big( \phi_{n,n+1}^{\sharp}(f)\big)(\tau_{\tht_{n+1,i}})=
  \big( \psi^{\sharp}(f)\big)(\tau_{\tht_{n+1,i}})\geq \frac{1}{[n+1,i]}\left({\rm tr}_1(b_1){{c'_{i1}}}[n,1]+\sum_{j=2}^{p_n} {\rm tr}_j (b_j) {{c_{ij}}}[n,j]\right)\hspace{1.5in}\\ \nonumber
  \geq {{\frac{1}{[n+1,i]}}}(1-\frac{1}{2^{4(n+1)}})\sum_{j=1}^{p_n} {\rm tr}_j (b_j) {{c_{ij}}}[n,j]=(1-\frac{1}{2^{4(n+1)}}){{\frac{1}{[n+1,i]}}}\sum_{j=1}^{p_n} \af {{c_{ij}}}[n,j]\geq (1-\frac{2}{8^n})\af.\hspace{1in}
\eneq
Now suppose  {{that}} $\tau= \tau_{\eta_{n+1,j}(t)}$ (for $0<t<1$) is defined by
$\tau (f_1, f_2,\cdots, f_{l_n}, a_1,a_2,\cdots, a_{p_{n+1}})={\rm tr}(f_j(t))$, where ${\rm tr} $ is normalized trace on $E_{n+1}^j=M_{\{n+1,j\}}(\C)$. By (\ref{Dphinn+2}), (\ref{phi}) and (\ref{ld-j-i}), {{we have}}
\beq\nonumber
  \big( \phi_{n,n+1}^{\sharp}(f)\big)(\tau_{\eta_{n+1,j}(t)})=
  \big( \phi^{\sharp}(f)\big)(\tau_{\eta_{n+1,j}(t)})\geq \frac{1}{\{n+1,j\}}\left(\sum_{i=1}^{p_n} {\rm tr}_i (b_i) \ld^j(i) [n,i]\right)\\ \nonumber
  =\frac{1}{\{n+1,j\}}\left(\sum_{i=1}^{p_n} \af \ld^j(i) [n,i]\right)\geq (1-\frac{2}{4^{n+1}})\af.\hspace{2in}
\eneq
Other parts of (b) {{follow}} from $\pi_n^{A}(e^A)=\pi_n^A(a^A)=1_{F_n}$ and {{$e^A\ge a^A$}}.

\end{proof}

\begin{NN}\label{4.23.2}
Note that {{the map}} $\pi_n^{A\,\sharp}: \Aff (T_0(A_n))\to \Aff(T_0(F_n))$ induced by $\pi_n^A: A_n\to F_n$ is given by $$\pi_n^{A\,\sharp}(g_1,g_2,\cd, g_{l_n}, x_1,x_2,\cd, x_n)=(x_1,x_2,\cd, x_{p_n}).$$
Define $\xi_n: \Aff(T_0(F_n))\to \Aff(T_0(A_n))$ by
\beq\nonumber
&&\hspace{-0.5in}\xi_n(x_1,x_2,\cd,x_{p_n})=(g_1,g_2,\cd, g_{l_n}, x_1,x_2,\cd, x_{p_n}),\\
\label{xi-n}
&&\hspace{-0.4in}{\text{where}}\,\,\,\hspace{0.2in}\left(
  \begin{array}{c}
    \{n,1\}g_1(t) \\
    \{n,2\}g_2(t) \\
    \vdots \\
    \{n,l_n\}g_{l_n}(t) \\
  \end{array}
\right)=\big(\bb_0^{n}+t(\bb_1^{n}-\bb_0^{n})\big)
\left(
  \begin{array}{c}
    ~[n,1]x_1 \\
    ~[n,2]x_2  \\
    \vdots \\
    ~[n,p_n]x_{p_n}  \\
  \end{array}
\right).
\eneq
Then $\pi_n^{A\,\sharp}\circ\xi_n=\id|_{\Aff(T(F_n))}.$ Let $\xi_{n,n+1}: \Aff (T (A_n))\to \Aff (T (A_{n+1}))$ be defined by $\xi_{n,n+1}= \xi_{n+1}\circ \psi_{n,n+1}^{\sharp}\circ \pi_n^{A\,\sharp}$.
Note that if $\psi^{\sharp}_{n,n+1}(x_1,x_2,\cd, x_{p_n})=(y_1,y_2,\cd,y_{l_n})$, then
\beq\label{AffT-psi}
\left(
  \begin{array}{c}
    ~[n+1,1]y_1 \\
    ~[n+1,2]y_2  \\
    \vdots \\
    ~[n+1,p_{n+1}]y_{p_{n+1}} \\
  \end{array}
\right)=\cc^{n,n+1}
\left(
  \begin{array}{c}
    ~[n,1]x_1 \\
    ~[n,2]x_2  \\
    \vdots \\
    ~[n,p_n]x_{p_n}  \\
  \end{array}
\right).
\eneq
If $\xi_{n,n+1}(g_1,g_2,\cd, g_{l_n}, x_1,x_2,\cd, x_{p_n})=
(h_1,h_2,\cd, h_{l_{n+1}},y_1,y_2,\cd,y_{p_{n+1}})$, then (\ref{AffT-psi}) holds
\beq\label{AffT-psi1}
\hspace{-0,2in}\andeqn\hspace{0.1in} \left(
  \begin{array}{c}
    \{n+1,1\}h_1(t) \\
    \{n+1,2\}h_2(t) \\
    \vdots \\
    \{n+1,l_{n+1}\}h_{l_{n+1}}(t) \\
  \end{array}
\right)=\big(\bb_0^{n+1}+t(\bb_1^{n+1}-\bb_0^{n+1})\big)
\left(
  \begin{array}{c}
    ~[n+1,1]y_1 \\
    ~[n+1,2]y_2  \\
    \vdots \\
    ~[n+1,p_{n+1}]y_{p_{n+1}}  \\
  \end{array}
\right).
\eneq

\end{NN}

\begin{lem}\label{4.24} The following
estimate
holds:
\beq\label{4.24+1}
\|\phi_{n,n+1}^\sharp-\xi_{n,n+1}\|<\frac{1}{2^{n+1}}.
\eneq
\end{lem}
\begin{proof} {{Let $g=(g_1,g_2,\cd,g_{l_n},x_1,x_2,\cd,x_{p_n})\in {{\Aff(T(A_n))}}$ with $\|g\|\leq 1$. Without lose of generality, we assume that}}  $0\leq g_j(t)\leq 1$ and $0\leq x_i\leq 1$ for all $i, j, t$.  Write
\beq\nonumber
&&{\phi^o_{n,n+1}}^\sharp (g)=(h'_1,h'_2,\cd,h'_{l_{n+1}},y'_1,
y'_2,\cd, y'_{p_{n+1}})\in \Aff (T(A_{n+1})),\\\nonumber
&&\phi_{n,n+1}^{\sharp} (g)=(h_1,h_2,\cd,h_{l_{n+1}},y_1,
y_2,\cd, y_{p_{n+1}})\in \Aff (T(A_{n+1})),{{\andeqn}}\\\nonumber
&&\xi_{n,n+1} (g)=(f_1,f_2,\cd,f_{l_{n+1}},z_1,
z_2,\cd, z_{p_{n+1}})\in \Aff (T(A_{n+1})).
\eneq
Recall that $\pi_{n+1}^A\circ \phi^o_{n,n+1}=\psi^o=\psi_{n,n+1}\circ \pi_n^A: A_n\to F_{n+1}$ and $\pi^A_{n+1}\circ \phi_{n,n+1}=\psi: A_n\to F_{n+1}$.

Note that $$\pi_{n+1}^{A\,\sharp}\circ \xi_{n,n+1}=\pi_{n+1}^{\sharp}\circ\xi_{n+1}\circ (\psi_{n,n+1}\circ \pi_n^A)^{\,\sharp}=(\psi_{n,n+1}\circ \pi_n^A)^{\sharp}=\pi_{n+1}^{\sharp}\circ {\phi^o_{n,n+1}}^\sharp.$$
Hence we have $z_i=y_i'$ for all $1\leq i\leq p_{n+1}$.

Using (\ref{psi-n-n+1}), (\ref{psi=b2}) and (\ref{psi=b1}), we calculate that
$y_i=y'_i=z_i$ for $i\geq 2$, \\
$y'_1=\frac{1}{[n+1,1]}\sum_{i=1}^{p_n}c_{1,i}x_i[n,i]$ and
$y_1=\frac{1}{[n+1,1]}\big(\sum_{i=1}^{l_n}
\sum_{k=1}^{n-1}g_i(\frac{k}{n})\{n,i\}+
\sum_{i=1}^{p_n}c'_{1,i}x_i[n,i]\big)$. Since $\|{\phi^o_{n,n+1}}^\sharp\|\le 1,$ $|y_1'|\le 1.$
\,%
By (\ref{defL}) and (\ref{c-M-large}), we have
\beq\nonumber
|\sum_{i=1}^{l_n}
\sum_{k=1}^{n-1}g_i(\frac{k}{n})\{n,i\}|
\leq (n-1)l_n\max{{\{\{n,i\}; 1\leq i\leq p_n\}}}\\
\leq \frac{1}{2^{4(n+1)}}{\cal L}_{n+1}\leq \frac{1}{2^{4(n+1)}} c_{11}< \frac{1}{2^{4(n+1)}} [n+1,1].
\eneq
Hence $\frac{1}{[n+1,1]}|\sum_{i=1}^{l_n}
\sum_{k=1}^{n-1}g_i(\frac{k}{n})\{n,i\}|\leq \frac{1}{2^{4(n+1)}}.$ Combining with (\ref{c'-large})
{{(recall $c_{ij}=c_{ij}^{n,n+1}$),}}
we obtain
\beq\nonumber
&&\hspace{-0.4in}|y_1-y'_1|\leq \frac{1}{[n+1,1]}\sum_{i=1}^{p_n}(c_{1i}-{c'_{1,i}})x_i[n,i]
+\frac{1}{2^{4(n+1)}}\\ \nonumber
&&\hspace{0.2in}\leq \frac{1}{2^{4(n+1)}}\big(\frac{1}{[n+1,1]}
\sum_{i=1}^{p_n}c_{1i}x_i[n,i]\big)+\frac{1}{2^{4(n+1)}}
=\frac{1}{2^{4(n+1)}} y'_1+\frac{1}{2^{4(n+1)}}\leq \frac{2}{2^{4(n+1).}}
\eneq
By (\ref{b-b}), (\ref{(n+1,i)}) and (\ref{AffT-psi1}) (note that $z_k\in [0,1]$), we know that
\beq\nonumber
&&\hspace{-0.9in}|f_i(t)-f_i(0)|\leq \frac{1}{\{n+1,i\}}\sum_{k=1}^{p_{n+1}}
|b^{n+1}_{1,ik}-b^{n+1}_{0,ik}|[n+1,k]z_k\\ \nonumber
&&\leq\frac{1}{\{n+1,i\}}\frac{1}{2^{4(n+1)}}\sum_{k=1}^{p_{n+1}}
\max\{b^{n+1}_{1,ik},b^{n+1}_{0,ik}\}[n+1,k]\leq\frac{1}{2^{4(n+1)}}\eneq for any $1\leq i\leq l_{n+1}$ and $0\leq t\leq 1$.

Note that, by \eqref{L+K} and  \eqref{>cm} as well as \eqref{2020-8-7-n1},
\beq\label{2020-8-7-n2}
L_j+K_j\leq \frac{1}{2^{4(n+1)}}\cm_{n+1}\leq \frac{1}{2^{4(n+1)}} \{n+1,j\}.
\eneq
 It is easy to see from the definition
of  $\phi_{n,n+1}$ (see \eqref {Dphinn+2}
   and \eqref{phi})
 that all the functions {{$h'_j(t)$}} and $h_j(t)$ are approximately constant {{within}} $\frac{1}{2^{4(n+1)}}$.
 To be more precise,
 we  may regard $\phi_{n,n+1}$ as a homomorphism from $A_n$ to $C([0,1], E_{n+1})$ which is unitarily equivalent to $\phi$ (see (\ref{Dphinn+2})). Hence $(h_1,h_2,\cdots, h_{l_{n+1}})=\phi^{\sharp}(g)\in \Aff (T(C([0,1], E_{n+1})))$. On the other hand, {{by (\ref{phi}),}}  {{$\phi=\oplus_{j=1}^{l_{n+1}} \phi_j$}} can be written as {{$(\chi\oplus \mu)\oplus \phi'=\oplus_{j=1}^{l_{n+1}} (\chi\oplus \mu)_j\oplus \phi'_j$}}, where {{$(\chi\oplus \mu)=\oplus_{j=1}^{l_{n+1}}(\chi\oplus \mu)_j: A_n\to \bigoplus_{j=1}^{{l_{n+1}}}M_{L_j+K_j}(C[0,1])$}} is defined in \ref{4.13} and  ${{\phi'=\bigoplus_{j=1}^{{l_{n+1}}}\phi'_j:}}
  A_n\to \bigoplus_{j=1}^{l_{n+1}}M_{\{n+1,j\}-(L_j+K_j)}(C[0,1])$
 sends $f=(f_1,f_2,\cdots f_{l_n}, a_1,a_2,\cdots, a_{p_n})$ to
\beq\nonumber
\ts
\phi'(f)=\bigoplus_{j=1}^{l_{n+1}}\diag\Big( f_1(\frac{1}{n})^{\sim \sm^j(1)},f_1(\frac{2}{n})^{\sim \sm^j(1)},\cd,f_1(\frac{n-1}{n})^{\sim \sm^j(1)}, \hspace{2in}\\ \nonumber
f_2(\ts\frac{1}{n})^{\sim \sm^j(2)},\cd,{{f_2(\frac{n-1}{n})^{\sim \sm^j(2)},}}\cd,
f_{l_n}(\frac{1}{n})^{\sim \sm^j(l_n)},\cd,f_{l_n}(\frac{n-1}{n})^{\sim \sm^j(l_n)},
\hspace{1.5in}\\
a_1^{\sim \ld^j(1)},a_2^{\sim \ld^j(2)},\cd
a_{p_n}^{\sim \ld^j(p_n)}, 0^{\sim\sim}
\Big).\hspace{0.6in}
\eneq   In particular, $(\phi')^{\sharp}(g)$ is constant (that is $(\phi'_i)^{\sharp}(g)(t)=(\phi'_i)^{\sharp}(g)(0)$ for any $i\in\{1,2,\cdots, l_{n+1}\}$ and $t\in [0,1]$). Consequently, for any $1\leq i\leq l_{n+1}$ and {{$0\leq t\leq 1,$  we obtain}}
\beq\nonumber
|h_i(t)-h_i(0)|{{=|\phi_i^{\sharp}(g)(t)-\phi_i^{\sharp}(g)(0)|}}\hspace{4in}\\ \nonumber
={{{\big{\vert}}}}{\frac{(K_j+L_j)\big((\chi\oplus \mu)_j^{\sharp}(g)(t)-(\chi\oplus \mu)_j^{\sharp}(g)(0)\big)+(\{n+1,j\}-K_j-L_j)\big((\phi'_i)^{\sharp}(g)(t)-
(\phi'_i)^{\sharp}(g)(0)\big)}{\{n+1,j\}}}{{{\big{\vert}}}}\\ \nonumber
\leq \frac{K_j+L_j}{\{n+1,j\}}\leq \frac{1}{2^{4(n+1)}}.\hspace{5in}
\eneq
 Note that  $y_i=y'_i=z_i$ for all $i\geq 2$, $y'_1=z_1$ and $|y'_1-y_1|\leq \frac{1}{2^{4(n+1)}}$.   By the formulae (\ref{Aff-bd}), we have $h'_i(0)=f_i(0)$ and
$|h_i(0)-f_i(0)|<\frac{1}{2^{4(n+1)}}$ (as $\bt_{n+1,0}$ is a \hm).
Consequently
$$|h_i(t)-f_i(t)|\leq \frac{2}{2^{4(n+1)}}<\frac{1}{2^{n+1}}.$$

\end{proof}

\begin{NN}\label{4.24.1}
%
%
%
%
{{For a separable \CA\, $A,$ one has a standard  metric {{on}} $T_0(A)$ {{(see \ref{DTtilde}),}} namely,
$d(t_1,t_2):=\sum_{n=1}^\infty(1/2^{n+1}) |t_1(a_n)-t_2(a_n)|$ for all $t_1, t_2\in T_0(A),$ where
$\{a_n\}$ is a fixed dense sequence of $A^{\bf 1}_{s.a.}.$ In the following proof, we will use
this metric.}}
\end{NN}

\begin{thm}\label{4.25} The $C^*$-algebra $A=\lim(A_n, \phi_{n,m})$ satisfies the following {{conditions}}:

(a) A is simple,

(b) $K_0(A)=G$, $K_1(A)=0$, $T(A)=\DT$ and  $\rho_A: K_0(A)\to \Aff(T(A))$ is the  map
$\rho$ from $G$ to $\Aff \DT$ by identifying $K_0(A)$ with $G$ and $T(A)$ with $\DT$.


Moreover, $A\in {\cal D}$ {{(see \ref{DD0})}} {{and has}} continuous scale.

\end{thm}
\begin{proof}
From \ref{4.22} and Lemma \ref{4.23n+}, $A$ is simple, and $K_0(A)=G$.
Consider  the {{projective limit}}
\beq\label{hat-T-A}
T_0(A_1)\stackrel{\phi^{T}_{1,2}}
{\longleftarrow}T_0(A_2)\stackrel{\phi^{T}_{1,2}}{\longleftarrow}T_0(A_3)
\cd\longleftarrow\cd\longleftarrow T_0(A),
\eneq
where $\phi^T_{i,i+1}: T_0(A_{i+1})\to T_0(A_i)$ is the affine continuous map
induced by $\phi_{i,i+1}.$
Suppose that $\tau\in \overline{T(A)}^w$ is written as  $\tau=\lim_{k\to \infty} \tau_k$, where $\tau_k\in T(A)$.
Then,
 for any fixed $n$, by (\ref{phi-sharp-infty}),
$$\tau(\phi_{n,\infty}(e_n^A))=\lim_{k\to \infty}\tau_k(\phi_{n,\infty}(e_n^A))=\lim_{k\to \infty}\big(\phi^{\sharp}_{n,\infty}({{\widehat{e^A_n}}})\big)(\tau_k)\geq (1-\frac{1}{4^n}).$$
Hence $\|\tau\|\geq (1-\frac{1}{4^n})$. Since $n$ is  arbitrary, $\|\tau\|=1$ and $\tau\in T(A)$. That is,  $T(A)$ is  compact.
Note that $F$ is a non-unital simple AF algebra with
$$(K_0(F), K_0(F)_+, \Sigma (F))=(H, H_+, \{x\in H_+;~~ \tilde{\rho}(x)(\tau)<1 ~\forall~ \tau \in \DT\}).$$
{{Therefore $T(F)=\DT$, with $\rho_F: K_0(F) \to \Aff (T(F))$ identified with $\tilde{\rho}: H\to \Aff (\DT)$.}}
%
%
%

On the other hand, by Lemma \ref{4.24}, we have the following approximately commuting diagram
\begin{displaymath}
    \xymatrix{
       \Aff(T_0(A_1)) \ar[d]_{{\pi^A_1}^{\sharp}}\ar[r]^{\phi^{\sharp}_{1,2}} & \Aff(T_0(A_2)) \ar[r]^{\phi^{\sharp}_{2,3}} \ar[d]_{{\pi^A_2}{\sharp}}& \Aff(T_0(A_3)) \ar[r] \ar[d]_{{\pi^A}_3^{\sharp}}& \cd \ar[r] & \Aff(T_0({{A}} ))\\
        \Aff(T_0({{F}}_1)) \ar[r]^{\psi^{\sharp}_{1,2}}\ar[ur]^{\xi_2\circ \psi^{\sharp}_{1,2}}
        &
         \Aff(T_0({{F}}_2)) \ar[r]^{\psi^{\sharp}_{2,3}}\ar[ur]^{\xi_2\circ \psi^{\sharp}_{2,3}}&
         \Aff(T_0({{F}}_3)) \ar[r]^{{~~~~~\psi^{\sharp}_{3,4}}}\ar[ur]^{{\xi_3\circ \psi^{\sharp}_{3,4}}}& \cd \ar[r]&
         \Aff(T_0({{F}})) }
\end{displaymath}
\noindent
{{of}} real Banach spaces {{(recall $\xi_{n, n+1}=\xi_{n+1}\circ {{\psi^{\sharp}_{n,n+1}}}\circ {\pi^A_n}^\sharp$
and $\pi_n^{A\,\sharp}\circ\xi_n={{\id}}|_{\Aff(T(F_n))}$).}}
Let $\Pi^\sharp: \Aff(T_0(A))\to \Aff(T_0(F))$ be the  continuous linear isomorphism
induced by the above approximately {{commutative}} diagram.
Note that we also have the following projective limit:
\beq\label{hat-T-F}
T_0(F_1)\stackrel{\psi^{T}_{1,2}}
{\longleftarrow}T_0(F_2)\stackrel{\psi^{T}_{1,2}}{\longleftarrow}T_0(F_3)
\cd\longleftarrow\cd\longleftarrow T_0(F).
\eneq
%
%

 Together with \eqref{hat-T-A}, {{by}} applying the above approximately commutative diagram,  we obtain the following approximately commutative digram:
\begin{displaymath}
   \xymatrix{
       T_0(A_1)
       &
        \ar[l]^{\phi^{T}_{1,2}} T(A_2)
       & T_0(A_3) \ar[l]^{\phi^T_{2,3} }
       & \ar[l]\cd T_0({{A}} )\\
        T_0({{F}}_1) \ar[u]_{{\pi_1^A}^T} &\ar[l]^{\psi^{T}_{1,2}}
         T_0({{F}}_2) \ar[u]_{{\pi_2^A}^T}&\ar[l]^{\psi^{T}_{2,3}}
        T_0({{F}}_3) \ar[u]_{{\pi_3^A}^T}& \cd\ar[l]
         T_0({{F}}) }
     \end{displaymath}
as compact convex sets with {{the}} metric mentioned in \ref{4.24.1}, which gives an affine continuous map $\Pi^T: T_0(F)\to T_0(A).$
 Combining {{the}} two approximately commutative diagrams above,  we have  $\Pi^\sharp(f)(t)=f(\Pi^T(t))$ for all $t\in T(F).$
Consider the functions $g_n:=\phi_{n, \infty}^\sharp(e_n^A)\in \Aff(T_0(A)).$
By \eqref{phi-sharp-infty},  on $T(A)\subset T_0(A),$ $g_n$ converges uniformly to the affine function
$g_A$ with $g_A(\tau)=1$ for all $\tau\in T(A)$ ($g_A(0)=0$).
Since $g_n(r\tau)=r\tau$ for all $\tau\in T(A)$ and $0\le r\le 1,$
 and $T(A)$ is {{compact,}} $g_n$ converges {{to}}  $g_A$ uniformly on $T_0(A).$
 Note that  $\Pi^\sharp(g_n)=\psi_{n,\infty}^\sharp(\pi_n^A(e_n^A)).$
 It follows  from \eqref{phi-sharp}  that $\Pi^\sharp(g_n)$ converges to 1 uniformly on $T(F).$ Then, by
 the first approximately commutative diagram {{above,}}
 $\Pi^\sharp(g_A)=1$ on $T(F).$
 Since $\Pi^\sharp(f)(t)=f(\Pi^T(t))$ for all $t\in T(F),$
  $\Pi^T$ maps $T(F)$ to the compact set $\{t\in T_0(A): g_A(t)=1\}.$
 The fact that $\Pi^\sharp$ is an isomorphism implies that $\Pi^T$ is an affine homeomorphism.
 Since $T(A)=\{t\in T_0(A): g_A(t)=1\},$ this implies that $\Pi^T$ maps $\DT=T(F)$ onto
 $T(A).$

Recall from (\ref{4.10+1}) that we have, for each $n,$  the following commutative diagram:
\begin{displaymath}
\xymatrix{
G_n~~~~\ar[rr]^{\rho_{A_n}} \ar[d]^{(\pi_n^A)_{*0}}
 &&
~~~\Aff(T_0(A_n)) \ar@{->}[d]^{{\pi_{n}^A}^\sharp}
\\
 H_n~~~ \ar[rr]^{\rho_{F_n}}&& ~~~\Aff(T_0(F_n)),}
\end{displaymath}
where $\rho_{A_n}: G_n=K_0(A_n)\to \Aff(T(A_n))$ is induced by
$\rho_{\td A_n}: K_0(\td A_n)\to \Aff(T(\td A_n))$ (see (\ref{April-10-2021}) in \ref{range4.1}).
Let  $\rho_A: K_0(A)\to \Aff(T(A))$ be the map  given by $\rho_{\td A}: K_0(\td A)\to \Aff(T(\td A))$ (see \ref{range4.1}).
Then  $\rho_A=\lim_{n\to\infty}\rho_{A_n}.$
%
We  obtain the following approximately commutative diagram:
{\small{
\begin{displaymath}
    \xymatrix{
    G_n\ar[rrd]^{\rho_{A_n}} \ar[r]^{\phi_{n,n+1,*0}}\ar[ddd]^{\pi^A_{n*0}}\,\,\,& ~~~G_{n+1}\ar[ddd]^{\pi^A_{n+1,*0}}\ar[rrd]^{\rho_{A_{n+1}}} \ar[r]^{\phi_{n+1,n+2,*0}}&\cdots&\lr\cdots
    & G{\ar[ddd]^{\iota}}\ar[rrd]^{\rho_{A}}&&\\
       && \Aff(T_0(A_n)) \ar[ddd]^{{\pi_n^A}^{\sharp}}\ar[r]^{\phi^{\sharp}_{n,n+1}} & \Aff(T_0(A_{n+1})) \ar[r]^{~~~~~\phi^{\sharp}_{n+1,n+2}}~~~~ \ar[ddd]^{{\pi^A_{n+1}}^{\sharp}}& 
     \cdots  &\lr\cdots&  \Aff(T_0({{A}} ))\ar@{-->}[ddd]^{\Pi^\sharp}\\
       &&&&&&\\
     H_n\ar[rrd]^{\rho_{F_n}}\ar[r]^{\psi_{n,n+1*0}} & ~~~H_{n+1}\ar[rrd]^{\rho_{F_{n+1}}} \ar[r]^{\psi_{n+1,n+2*0}}&\cdots&\lr\cdots&H\ar[rrd]^{\rho_{F}} &&\\
       && \Aff(T_0({{F}}_n)) \ar[r]^{\psi^{\sharp}_{n,n+1}} &
         \Aff(T_0({{F}}_{n+1})) \ar[r]^{~~~~~\psi^{\sharp}_{n+1,n+2}}~~~~
         &
         \cdots&\lr\cdots&
         \Aff(T_0({{F}})),}
\end{displaymath}}}
\noindent
\noindent
\hspace{-0.05in}where the top, bottom  and the back diagrams  are commutative,
and the front plane is approximately commutative.
%
Thus,  we obtain the following
commutative diagram:
\begin{displaymath}
\xymatrix{
~~G~~\ar[rr]^{\rho_{A}} \ar[d]^{\iota}
 &&
~~\Aff(T_0(A)) \ar@{->}[d]^{\Pi^\sharp}
\\
 ~~H~~\ar[rr]^{\rho_{F}}&& ~\Aff(T_0(F)),}
\end{displaymath}
where the map from $G$ to $H$ is given by \ref{4.22}. Since $\rho_F=\td \rho,$ we
obtain $\rho_A=\rho: G\to \Aff(\DT)$ as desired. Note that, by the end of \ref{Dparing}, this is consistent with the definition
of \ref{Dparing} as $A$ has continuous scale.

To see $A\in {\cal D},$
choose $n_0\ge 1$ such that  $\lambda_s(A_n)\ge 63/64$ for all $n\ge n_0$ (see (e) of \ref{4.10}).
Let $e\in A_{n_0}$ be a strictly positive element. By \eqref{phi-sharp},
we may assume that
\beq\label{2020-429-n2}
\tau(\phi_{n_0, n}(e))\ge 63/64\rforal \tau\in T(A_n)\andeqn t(\phi_{n_0, \infty}(e))\ge 63/64\rforal t\in T(A).
\eneq
In particular,
$\phi_{n_0,n}(e)$ is full in $A_n$ for all $n\ge 1.$
Choose $\dt>0$ such that
\beq\nonumber
\tau(\phi_{n_0, n}((e-\dt)_+)\ge 31/32\rforal \tau\in T(A_n)\andeqn t(\phi_{n_0, \infty}(e-\dt)_+)\ge 31/32\rforal t\in T(A).
\eneq
Choose $k\ge 1$ such that
\beq\label{2020-sec4-n1}
f_{1/4}(e^{1/k})\ge (e-\dt)_+.
\eneq
Let  
$B=\overline{\phi_{n_0, \infty}(e)A\phi_{n_0, \infty}(e)}.$  Then $B$ is a hereditary \SCA\, of $A.$ Let us first show  $B\in{\cal D}$.
Put $a=\phi_{n_0, \infty}(e^{1/k})$ and choose  $\mathfrak{f}_a=1/2.$
Let $B_n=\overline{\phi_{{n_0,n}}(e)A_n\phi_{n_0,n}(e)}.$
Note that $B_n\in {\cal C}_0'$ (recall \ref{DfC1}).

Now fix a finite subset ${\cal F}\subset B$ and $0<\ep<1/16.$
We may assume that ${\cal F}\subset \phi_{n, \infty}({{B_n}})$ for some
$n>n_0.$   Choose $0<\eta<\ep$ such that, if
$a_1, a_2\in B_+$ with $0\le a_1, a_2
\le 1$ and $\|a_1-a_2\|<\eta,$ then $\|f_{1/4}(a_1)-f_{1/4}(a_2)\|<\ep/8.$

Choose (see 2.3.13 of \cite{Linbook}, for example) a \cpc ~$\psi: B\to B_n\cong \phi_{n,\infty}(B_n)$
such that
\beq\label{2020-724-n10}
\|\psi(b)-b\|<\ep/2\rforal b\in {\cal F}\cup \{a\}.
\eneq
Then
\beq
\|f_{1/4}(\psi(a))-f_{1/4}(a)\|<\ep/8.
\eneq
It follows, for all $\tau\in T(B_n),$ {{by}} identifying $\phi_{n,\infty}(B_n)$ with $B_n,$ {{that}}
\beq\label{2020-724-n11}
\hspace{-0.2in}\tau(f_{1/4}(\psi(a)))\ge \tau(f_{1/4}(\phi_{n_0, n}(e^{1/k})))-\ep/8\ge \tau((\phi_{n_0,n}(e-\dt)_+))\ge 31/32-\ep/8> \mathfrak{f}_a.
\eneq
Define $\phi=0.$
By  \eqref{2020-724-n10}, {{\eqref{2020-724-n11}}} and \ref{DD0}, $B\in {\cal D}.$
By Corollary 11.3 of \cite{eglnp}, $B={\rm Ped}(B),$ and by 9.4 of \cite{eglnp},
$B$ has strict comparison for positive elements.  It follows from
\cite{Br} {{that}} $A\otimes {\cal K}\cong B\otimes {\cal K}.$ Therefore $A$ is isomorphic
to a hereditary \SCA\, of $B\otimes {\cal K}.$ It follows that $A$ has strict comparison
for positive elements.  Let $e_A$ be a strictly positive element.
Since $T(A)$ is compact, $d_\tau(e_A)=1$
for all $\tau\in T(A).$ It follows that ${{\widehat{\la e_A \ra}}}$ is continuous on ${\td T}(A).$
By Theorem 5.4 of \cite{eglnp}, $A$ has continuous scale. Therefore $A={\rm Ped}(A)$ (see
Theorem 3.3 of \cite{Lincs1}). Hence $a\in {\rm Ped}(A).$
It follows from Proposition 11.7 of \cite{eglnp} (see also 11.6 of \cite{eglnp}) that $A\in {\cal D}.$

\end{proof}


%
%
%

\begin{NN}\label{4.26}

 Now let $G$ and $K$ be any countable abelian groups, $\DT$  {{a compact}} Choquet simplex and $\rho: G\to \Aff \DT$  a  homomorphism with following condition: for any $g\in G$, there is a $\tau\in \DT$ such that
$\rho(g)(\tau)\leq 0,$
(see condition (*) in \ref{range4.4}).

We will construct a simple stably projectionless, stably finite $C^*$-algebra $A$ such that $K_0(A)\cong G$, $K_1(A)\cong K$, $T(A)\cong \DT$, and the map $\rho_A: K_0(A) \to \Aff(T(A))$ is the  map $\rho$ when one identifies $K_0(A)$ with $G$ and $T(A)$ with $\DT$.

Note that if $K=0$ and $G$ is torsion free,
then {{the}} algebra satisfying the condition is already constructed (see Theorem \ref{4.25} above).

Note
${\rm Tor}(G)\subset {\rm ker}\rho.$
Write $G_T=G/{\rm ker}\rho$
which may be viewed as a subgroup of $\Aff \DT$,
{{and write $G_f={\rm ker}\rho/{\rm Tor}(G),$ which is  {{a}} torsion free group.}}
%
\end{NN}

\begin{lem}\label{G-ind} {{$G$ can be written as {{an}} inductive limit  of finitely generated subgroups\\
($G_n=G_{n, T}\oplus G_{n,f}\oplus G_{n,tor}, \gm_{n,m})$ with ${\rm Tor}(G_n)=G_{n, tor}$ such that:\\
(a)~ According to the decomposition $G_n=G_{n, T}\oplus G_{n,f}\oplus G_{n,tor}$ and $G_{n+1}=G_{n+1, T}\oplus G_{n+1,f}\oplus G_{n+1,tor}$, the map $\gm_{n,n+1}$ may be written  as
$$\left(
  \begin{array}{ccc}
    \gm^{n,n+1}_{T,T} & 0&0 \\
    \gm^{n,n+1}_{T,f} &\gm^{n,n+1}_{f,f} &0\\
    \gm^{n,n+1}_{T,tor} &\gm^{n,n+1}_{f,tor} &\gm^{n,n+1}_{tor,tor}
  \end{array}
\right){{,}}$$
that is, the components of $\gm_{n,n+1}$ from $G_{n,tor}$ to $G_{n+1,T}\oplus G_{n+1,f}$ and from $G_{n, f}$ to $G_{n,T}$ are zero maps. In particular, $\gm_{n,n+1}(G_{n,tor})\subset G_{n+1,tor}$ and $\gm_{n,n+1}(G_{n,\Inf})\subset G_{n+1,\Inf}$, where $G_{n,\Inf}=G_{n,f}\oplus G_{n,tor}$.\\
(b)~ ${\rm ker}\rho=\lim(G_{n,\Inf}, \gm_{n,m}|_{G_{n,\Inf}})$ and ${\rm Tor}(G)=\lim(G_{n,tor}, \gm_{n,m}|_{G_{n,tor}})$.\\
(c)~ Let $$\tilde{\gm}_{n,n+1}= \gm^{n,n+1}_{T,T}: G_{n, T}= G_n/G_{n,\Inf}\to G_{n+1, T}= G_{n+1}/G_{n+1,\Inf}~~\mbox{ and} ~$$
$$\tilde{\tilde{\gm}}_{n,n+1}=\left(
  \begin{array}{cc}
    \gm^{n,n+1}_{T,T} & 0 \\
    \gm^{n,n+1}_{T,f} &\gm^{n,n+1}_{f,f} \\
    \end{array}\right): G_{n, T}\oplus G_{n,f}= G_n/G_{n,tor}\to G_{n+1, T}\oplus G_{n+1,f}= G_{n+1}/G_{n+1,tor}$$ be the quotient maps induced by $\gm_{n,n+1}$. Then $G_T=G/{\rm ker}\rho
    =\lim(G_{n,T}, \tilde{\gm}_{n,m})$ and $G/tor(G)=\lim(G_{n,T}\oplus G_{n,f}, {{\tilde{\tilde{\gm}}_{n,m}).}}$\\
  (d)~$\gm^{n,n+1}_{T,T},$ $\gm^{n,n+1}_{f,f}$ and $\gm^{n,n+1}_{tor,tor}$ are injective. Consequently, $\gm_{n,n+1}, \tilde{\gm}_{n,n+1}$ and $\tilde{\tilde{\gm}}_{n,n+1}$ are injective}}.


 \end{lem}
\begin{proof}
Let $G_f={\rm ker}\rho/tor(G)$. Write $G_T=\cup_{n=1}^{\infty} G_{n, T}$ with $G_{1,T}\subset G_{2,T}\subset\cdots G_{n,T}\subset\cdots {{\subset G_T}},$ $G_f=\cup_{n=1}^{\infty} G_{n, f}$ with $G_{1,f}\subset G_{2,f}\subset\cdots {{\subset G_{n,f}}}\subset\cdots G_f,$ and ${\rm Tor}(G)=\cup_{n=1}^{\infty} G_{n, tor}$ with $G_{1,tor}\subset G_{2,tor}\subset\cdots G_{n,tor}\subset\cdots
\subset {\rm Tor}(G),$ where each $G_{n,T},$ $G_{n,f}$ and $G_{n,tor}$ are finitely generated.
Denote by {{$\iota_{G_{n,f}}: G_{n,f}\to G_f$}} the embedding.

Since $G_f$ is torsion free,  the extension $0\to {\rm Tor}(G)
 {{\to}}  {\rm ker}\rho\to G_f\to 0$ is pure, i.e., every
{{finitely}} generated subgroup lifts.   Thus, for each $n,$ there is a \hm\, $\xi_n: G_{n,f}\to {\rm ker}\rho$ such that
the following diagram commutes:
$$
\xymatrix{
~& G_{n,f}\ar[d]^{{\iota_{G_{n,f}}}} \ar@{-->}[dl]_{\xi_n}  \\
{\rm ker}\rho\ar[r]^{\pi} & {{G_f}}
}
$$
%
($\pi\circ \xi_n=\iota_{G_{n,f}},$ $n=1,2,...$).
 Define $\gm'_{n,f}: G_{n,f} \to {\rm Tor}(G)$ by
$\gm'_{n,f}(h)=\xi_n(h)-\xi_{n+1}(h)\in {\rm Tor}(G)$. Since $G_{n,f}$ is finitely generated, there is an  {{integer }}$m>n$ such that $\gm'_{n,f}(G_{n,f})\subset G_{m,tor}$. By passing to a subsequence, we may  assume $\gm'_{n,f}(G_{n,f})\subset G_{n+1,tor}$. Define $\gm^{n,n+1}_{f,tor}=\gm'_{n,f}$ and \\ $\chi_{n,n+1}: G_{n,\Inf}=G_{n,f}\oplus G_{n,tor} \to G_{n+1,\Inf}=G_{n+1,f}\oplus G_{n+1,tor}$ by
$$\chi_{n,n+1}=\left(
  \begin{array}{cc}
    \gm^{n,n+1}_{f,f} &0\\
     \gm^{n,n+1}_{f,tor} &\gm^{n,n+1}_{tor,tor}  \\
  \end{array}
\right),$$
where $\gm^{n,n+1}_{f,f}: G_{n,f} \to G_{n+1,f}$ and $\gm^{n,n+1}_{tor,tor}: G_{n,tor} \to G_{n+1,tor}$ are the inclusion
{{maps.}}
Then ${\rm ker}\rho=\lim(G_{n,f}\oplus G_{n,tor}, \chi_{n,n+1})$ and $\chi_{n,n+1}$ are injective.
{{Repeating}} this procedure with $G_T$ in place of $G_f$ and ${\rm ker}\rho$ in place of ${\rm Tor}(G)$, and of course
{{passing}} to a subsequence again, we obtain the other parts of the map $\gm_{n,n+1}$ as desired.


 \end{proof}


\begin{NN}\label{DYT}

Let us recall Theorem 7.11 of \cite{GLII}. Let $G_0$, $G_1$ be any {{countable}} abelian groups and  $T$ be any compact
metrizable Choquet simplex. There is a simple ${\cal Z}$-stable $C^*$-algebra $B_T\in {\cal D}_0$ with continuous scale such that $K_0(B_T)=\mathrm{ker}(\rho_{B_T})=G_0$, $K_1(B_T)=G_1$ and $T{{(B_T)}}=T$. Moreover $B_T$  is locally approximated by
sub-homogenous \CA s with spectrum having
dimension no more than 3 (see 7.7 of \cite{GLII}).
More precisely, ${{B_T}}=\lim_{k\to\infty}(E_{n(k)}\oplus W_k, \Phi_{{{k,k+1}}}),$
where $E_n=M_{(n!)^2}(A(W, \af_n))$ and $W_k$  is in ${\cal C}_0$ with $K_0(W_k)=\{0\},$
and $\Phi_{k,k+1}$ is injective
(see the constructions in 7.2  of \cite{GLII}; also  {{the}}  notation in 11.3 of \cite{GLII}
 and the discussions in 7.3--7.10 of \cite{GLII}{{).}} Note that, by Proposition 7.7 of \cite{GLII},  $B_T$ is locally approximated by
sub-homogenous \CA s with spectrum having
dimension no more than 3.


{{If $\DT$ is a compact metrizable Choquet simplex, then $\Aff (\DT)$ can be regarded as {{a subset}} of ${\rm LAff}_+(\tilde{\DT })$ (here ${\td \Delta}$ is the cone generated by $\DT$ and $0$)  by regarding $f:\DT \to \R$ as $\tilde{f}: {\td \Delta}\to \R$ defined by $\tilde{f}(\ld\tau)=\ld f(\tau)$ for $\ld\in [0,\infty)$ and $\tau\in\DT$. In particular $1_{\DT}\in {\rm LAff}_+(\tilde{\DT })$. For $\af>0$,
{{denote by}}
$\af\DT$  {{the subset}}
 $\{\af\tau:~\tau\in \DT\}$
 of  ${\td \Delta}$. Note that when $\DT$ is compact, we  may identify $\Aff (\DT)$ with $\Aff ({\td \Delta})$ (recall $f(0)=0$ and see the end
of \ref{DTtilde})
---that is,
for $f\in \Aff (\DT)$, we assume $f$ is extended to $f\in  \Aff ({\td \Delta})$ defined by $f(\af\tau)=\af f(\tau)$ for any $\af\in \R_+$ and $\tau\in \DT$. }}

\end{NN}

%

%

\begin{thm}\label{Tmodel2}
Let $\DT$ be a metrizable Choquet simplex, {{$G_0$}}  a countable abelian group,
$\rho: G_0\to \Aff(\DT)$ a \hm\, such that $\rho(G_0)\cap \Aff_+(\DT)=\{0\},$
and {{$G_1$}}  a countable abelian group.
Then there is a simple \CA\, $A=\lim_{n\to\infty} (B_n\oplus C_n\oplus D_n, \phi_{n,m}),$
where $C_n$  and $D_n$ are  {{simple \CA s}} in Theorem  \ref{4.25}
and $B_n$ is in \ref{DYT} ({{as $B_T$}}--see
Theorem  7.11  (and 7.2) of \cite{GLII})
such that $\phi_n$ maps strictly positive elements to strictly positive {{elements,}}
\beq
((K_0(A), \Sigma(K_0(A)), T(A), \rho_A, \Sigma_A), K_1(A))=((G_0,  \{0\}, \DT, \rho, 1_\DT), G_1),
\eneq
{{$\phi_{n, \infty}(K_0(C_n))\cap {\rm ker}\rho_A=\{0\}$,}}
${\rm ker}\rho_{C_n}=\{0\},$ ${\rm ker}\rho_{D_n}=K_0(D_n),$ and
$K_0(B_n)$ is torsion.
Moreover, $A$ has continuous scale, is in ${\cal D}$ and
$\lim_{n\to\infty}\inf\{d_\tau(\phi_{n, \infty}({{x_n}})):\tau\in A\}=0,$
where ${{x_n}}\in B_n\oplus D_n$ is any strictly positive element.

Moreover, one may require that ${{\phi_n}_{*i}}|_{K_i(B_n)},$ ${{\phi_n}_{*i}}|_{K_i(C_n)}$
and ${{\phi_n}_{*i}}|_{K_i(D_n)}$ are all injective,  and $K_i(B_n), K_i(C_n)$ and $K_i(D_n)$ are finitely generated.

\end{thm}

\begin{proof}  {{For convenience, we will write $G_0=G$ and $G_1=K$.
Choose {{finitely}} generated subgroups $K_1\subset K_2\subset \cdots K_n\subset\cdots\subset K$ such that $K=\cup_{n=1}^{\infty} K_n$.
{{Write}} {{$G=\lim(G_{n,T}\oplus G_{n,f} \oplus G_{n,tor}, \gm_{n,m})$}} as in Lemma \ref{G-ind}. We adopt the notations $\gm_{n,n+1}$ {{and}} $\gm^{n,n+1}_{a, b}$, where $a,b= T,$ {{$f$ and  $tor$}}  from Lemma \ref{G-ind}. }}

It follows from Theorem 7.11 and Proposition 7.8 of \cite{GLII} that there
is a simple \CA\, ${{B_n}}\in {\cal B}_T$ such that
$((K_0(B_n), \Sigma(K_0(B_n)), T(B_n), \rho_{B_n}), K_1(B_n))=((G_{n,tor}, \{0\},  \DT_0, 0), K_n),$
{{where $\DT_0$ is a single point.}} {{Here we mean $\rho_{{B_n}}=0.$}}
Note that ${{B_n}}\in {\cal D}_0\subset {\cal D}$ {{is}} a ${\cal Z}$-stable simple \CA\, with continuous scale
({{see}} 7.7 of \cite{GLII}).

By Theorem  \ref{4.25},
there is also
a simple \CA\, $C_n'$  with continuous scale and with the form in \ref{4.25}
such that
$(K_0(C_n'), T(C_n'), \rho_{C_n'})=({{G_{n,T}, \DT, \xi_{T, n, \infty}}}),$
where ${{\xi_{T, n,\infty}}}$ is the  \hm\, {{of the form}} {{$G_{n, T}\stackrel{\tilde{\gm}_{n,\infty}}{\lr} G/{\rm ker}\rho=G_T\subset \Aff(\DT)$}} induced by
the inductive limit $G_T=\lim_{n\to\infty}(G_{n,T}, \tilde{\gm}_{n,m}).$
Note that, by 6.2.3 of \cite{Rl},
\beq
{\rm Cu}^\sim(C_n')=K_0(C_n')\sqcup{\rm LAff}_+(\td\DT),
\eneq
where $\td \DT$ is the cone generated by $\DT$ and $\{0\}.$
Note that ${{\tilde{\gm}_{n,n+1}}}$ and ${\rm id}_{\DT}$ induce
a morphism ${{\xi^{cu}_{T,n}}}: {\rm Cu}^\sim(C_n')\to {\rm Cu}^\sim (C_{n+1}').$
By Theorem 1.0.1 of \cite{Rl}, there is a \hm\, $\psi_n: C_n'\to C_{n+1}'$
which sends strictly positive elements to strictly positive elements and
${\rm Cu}^\sim(\psi_n)={{\xi^{cu}_{T,n}}}.$ In particular $\psi_{n*0}={{\xi^{cu}_{T,n}}}|_{K_0(C_n')}{{=\tilde{\gm}_{n,n+1}}}.$
We will continue to {{write}} $\psi_n$ for the extension $\psi_n\otimes \id_{M_3}: M_3(C_n')\to M_3(C_{n+1}').$
Note {{that,}} since ${{C_1'}}\in {\cal D},$ by Proposition 11.8  of \cite{eglnp}, one may
choose ${{c_1}}\in (C_1')_+$ 
such that $d_\tau({{c_1}})=1/2$ for all $\tau\in T({{C_1'}}).$
{{Let $C_1=\overline{c_1C_1'c_1}$.}}
{{Let $n\geq 2$.}} Choose
$c_{n,0}, c_{n,b}, c_{n,c}, c_{n,d} \in (C_n')_+$
{{such that}} 
$d_\tau(c_{n,0})=1/2^{n+1},$  and $d_\tau(c_{n,b})=d_\tau(c_{n,d})=1/2^{n+2}$
(and $d_\tau(c_{n,c})=1-1/2^n$)
for all $\tau\in T(C_n'),$ $n=2,3,....$
Put $c_n:=c_{n,b}\oplus c_{n,c}\oplus c_{n,d}\in M_3(C_n')$ and  $C_n=\overline{c_nM_3(C_n')c_n}.$   {{Note that $c_n\in M_3(C_n')$ satisfies that $\tau(c_n)=(1-1/2^n)+1/2^{n+2}+1/2^{n+2}=1-1/2^{n+1}$ and defines {{$\widehat{\la c_n\ra}=(1-1/2^{n+1})\cdot 1_\DT\in {\rm LAff}_+(\td\DT)\subset {\rm Cu}^\sim(C_n')$}}.  Similarly, $c_{n+1,c}$ also defines  {{$\widehat{\la c_{n+1,c}\ra}=(1-1/2^{n+1})\cdot 1_\DT\in {\rm LAff}_+(\td\DT)\subset {\rm Cu}^\sim(C_{n+1}')$}}.  Consequently, $\la  c_{n+1,c}\ra=\la \psi_n(c_{n})\ra$ in ${\rm Cu}(C_{n+1}')$.   }}
{{By}} Theorem 1.2 of \cite{Rlz},  {{one has}} ${\rm Her}(\psi_n(C_n))
\cong \overline{c_{n+1,c}C_{n+1}'c_{n+1,c}}.$ Therefore there is a \hm\,
 $\phi_{n,c,c}: C_n\to C_{n+1}$
such that ${\rm Cu}^\sim(\phi_{n,c,c})={\rm Cu}^\sim({\psi_n}|_{C_n})$ and $\la \phi_{n,c,c}(c_n)\ra =\la c_{n+1,c}\ra.$
{{Note the following fact which will be used later: when we identify both $T(C_n)$ and $T(C_{n+1})$ with $\DT$, the map $\phi_{n,c,c}^{\sharp}: \Aff (T(C_n))=\Aff (\DT)\to \Aff (T(C_{n+1}))=\Aff (\DT)$ is given by
\beq\label{august3-2020}
\phi_{n,c,c}^{\sharp}(f)=\frac{1-1/2^{n+1}}{1-1/2^{n+2}}f {{\rforal f\in \Aff(T(C_n)).}}
\eneq}}

Denote by $C_{n,\p}=\overline{c_{n,\p}C_nc_{n,\p}},$ where $\p=b,c,d.$

By \ref{4.25}, there is a simple \CA\,  $D_n$ {{of}}  the form in \ref{4.25} such that $K_0(D_n)={{G_{n,f}}}$  and ${\rm ker}\rho_{D_n}=K_0(D_n),$
and $T(D_n)=\DT_0,$ the single point.
As in the previous case, there is $\psi_{d,n}: D_n\to D_{n+1}$
such that 
{{${\psi_{d,n}}_{*0}=\gm^{n,n+1}_{f,f}: G_{n,f}\to G_{n+1,f}$ (see Lemma \ref{G-ind}),}} and $\psi_{n,d}$ sends strictly positive elements
to strictly positive elements (using again 1.0.1 of \cite{Rl}).
Choose  $d_{n,b}, d_{n,c}, d_{n,d}\in {{(D_n)_+}}$ such
that $d_\tau(d_{n,b})=d_\tau(d_{n,d})=1/2^{n+2},$ and $d_\tau(d_{n,c})=1-1/2^{n+1}$ for all $\tau\in T(D_n).$
Define $D_n'=\overline{(d_{n,b}\oplus d_{n,c}\oplus d_{b, d})M_3(D_n)(d_{n,b}\oplus d_{n,c}\oplus { {d_{n, d}}})}.$
Then, by \cite{Rlz}, $D_n'\cong D_n.$ \Wlog, {{we}} may assume
that $D_n=D_n'$ and $d_{n,b}, d_{n,c}$ and $d_{n,d}$ are mutually orthogonal in $D_n.$
Put $D_{n,\p}={\rm Her}(d_{n,\p}),$ $\p=b,c,d.$

Choose  $b_{n,b}, b_{n,c}, b_{n,d}\in {{(B_n)_+}}$ such that
$d_\tau(b_{n,b})=d_\tau(b_{n,d})={{1/2^{n+2}}}$ and $d_\tau(b_{n, c})=1-1/2^{n+1}$
 for all $\tau\in {{T(B_n)}}.$
Define ${{B'_n}}:=\overline{(b_{n,b}\oplus b_{n,c}\oplus b_{n,d})M_2(B)(b_{n,b}\oplus b_{n,c}\oplus b_{n,d}}).$
Since {{$B_n$}} is stably projectionless and ${\cal Z}$-stable (see \cite{GLII}), by Theorem 1.2 of \cite{Rlz}, ${{B'_n}}\cong B_n.$
We may assume that $b_{n,b}, b_{n,c}$ and $b_{n,d}$ are  mutually orthogonal in $B_n$ and ${{B'_n}}=B_n.$
Define
$B_{n,\p}=\overline{b_{n,\p}Bb_{n,\p}},$  where $\p=b,c,d.$

Denote by  $\iota_{n,\p}^b: B_{n,\p}\to B_n$ the embedding ($\p=b,c,d$).
Note that, for $\p=b,c,d,$
\beq\label{Tmodel2-5}
((K_0(B_{n,\p}), \td T(B_{n,\p}) ,0), K_1(B_{n,\p}))\cong {{((K_0(B_n),\td T(B_n), 0),  K_1(B_n))=((G_{n,tor},{{\tilde{\DT}_0}}, 0), K_n)}}.\hspace{0.3in}
\eneq
%
 By Theorem 12.8  of \cite{GLII}, there is a \hm\, $\phi_{n,b,b}': B_n\to B_{n+1,b}\subset B_{n+1}$
which sends {{strictly}} positive elements to strictly positive elements such that {{${\phi_{n,b,b}}_{*0}=\gm^{n,n+1}_{tor,tor}: K_0(B_n)=G_{n,tor}\to K_0(B_{n+1})=G_{n+1,tor}$ and ${\phi_{n,b,b}}_{*1}=\iota: K_1(B_n)=K_n\hookrightarrow K_1(B_{n+1})=K_{n+1}$.}}
Let $\phi_{n,b,b}:={{\iota_{n+1,b}^b\circ \phi_{n,b,b}'}}.$
Let $W_n$ be a simple \CA\, which is an inductive limit of \CA s in ${\cal C}_0$
such that $K_0(W_n)=K_1(W_n)=0$ and $T(W_n)=\DT_0,$ $n=1,2,....$
It follows from 12.8  of \cite{GLII} again that there is $h_{n,b,w}: B_n\to W_{n}$
which sends strictly positive elements to strictly positive elements and
${h_{n,b,w}}_T$ gives {{the}} identity on $\DT_0.$
Note that, for {{$\p=b,c, d,$}}
\beq\label{Tmondel2-6}
((K_0(C_{n,\p}),\Sigma(K_0(C_{n,\p})), \tilde{T}(C_{n,\p}), \rho_{C_{n,\p}}), K_1(C_{n,\p}))\cong ((K_0(C_n'), \{0\}, \,{{\tilde{\DT},}} \rho_{C_n'}), \{0\}{{).}}
\eneq
By Theorem 1.0.1 of \cite{Rl}, there is a \hm\, $h_{n,w,c}: W_n\to C_{n+1,b}$
which maps strictly positive elements to strictly positive elements.
Define $\phi_{n, b,c}: B_n\to C_{n+1,b}\subset C_{n+1}$
by
$\phi_{n,b,c}:=h_{n,w,c}\circ h_{n,b,w}.$
Similarly,  one obtains a \hm\, $\phi_{n, b,d}: B_n\to D_{n+1,b}\subset D_{n+1}$
which factors through $W_n$
{{and sends}} strictly positive elements to strictly positive elements.

Note that, by 11.5  of \cite{eglnp}, {{$B_{n+1}$}} has {{stable}} rank one, and by 6.2.3 of \cite{Rl} (see also 7.3
 of 11. 3 and 11.8 of  \cite{eglnp}),
\beq
{\rm Cu}^\sim(B_{n+1,\p})=K_0(B_{n+1})\sqcup {\rm LAff}^\sim(T(B_{n+1}){{)}}\,\,\,\,\,(\p=c,b,d).
\eneq
Let $\tau_{w,0}\in T(C_n)$ {{be}}
such that $\rho_{C_n}(x)(\tau_{w, 0})=0$ for all $x\in K_0(C_n)$
given by Theorem \ref{Tfeature}.
Define $\xi_{c,b}: {\rm Cu}^\sim(C_n)\to {\rm Cu}^\sim ({{B_{n+1,c}}})$
by 
{{$\xi_{c,b}|_{K_0(C_n)}:=\gm^{n,n+1}_{T,tor}:K_0(C_n)=G_{n,T}\to K_0(B_{n+1,c})=K_0(B_{n+1})=G_{n+1,tor}$}}
and $\xi_{c,b}|_{{\rm LAff}^\sim(T(C_n))}$
{{by}}
%
${ {\xi_{c,b}}}(f)(t)=f(\tau_{w,0})$
for all $f\in {\rm LAff}^\sim(T(C_n))$ {{and $t\in T(B_{n+1,c})$.}}
Note that $\rho_{B_n}=0.$ One then checks that $\xi_{c,b}$ is
a morphism in ${\bf Cu}.$
Since $B_{n+1,c}$ has stable rank one, by applying Theorem 1.0.1 of \cite{Rl} again, one obtains a \hm\,
$\phi_{n,c,b}': C_n\to B_{n+1,c}$  such that ${\rm Cu}(\phi_{n,c,b}')=\xi_{c,b}$
which sends strictly positive elements to strictly positive elements.
Define $\phi_{n,c,b}:=\iota_{n+1,c}^b\circ \phi_{n,c,b}',$  where $\iota_{n+1,c}^b: B_{n+1,c}\to B_{n+1}$ is the embedding.

Using $\tau_{w,0}$ and Theorem 1.0.1 of \cite{Rl} again, one obtains a \hm\, $\phi_{n,c,d}: C_n\to D_{n+1,c}\subset D_{n+1}$
such that 
${\phi_{n,c,d}}_{*0}={{\gm^{n,n+1}_{T,f}: G_{n, T}\to G_{n+1, f}}}$
 {{ which sends}} strictly positive {{elements}} of $C_n$ to strictly positive {{elements}} of $ D_{n+1,c}$.

Denote {{by}} $\iota_{n,\p}^d: D_{n,\p}\to D_n$ the embedding ($\p=b,c,d$).
As above, applying Theorem 1.0.1 of \cite{Rl}, as $B_{n+1}$ has stable rank one,
one obtains a \hm\, $\phi_{n,d,b}: D_n\to B_{n+1, d}\subset B_{n+1}$ such
that 
${\phi_{n,b,d}}_{*0}={{\gm^{n,n+1}_{f,tor}:K_0(D_n)=G_{n,f}\to K_0(B_{n+1, d})=G_{n+1,tor}}}.$  By factoring
{{through}} $W_n$ again,
one obtains a \hm\, $\phi_{n,d,c}: D_n\to C_{n+1,d}\subset  C_{n+1}$, {{which sends {{strictly}} positive {{elements}} of $D_n$ to strictly positive {{elements}} of $C_{n+1,d}$.}}
By applying Theorem 1.0.1 of \cite{Rl} again (recall $\rho_{D_n}=0$ for all $n$),
one also has a \hm\, $\phi_{n, d,d}: D_n\to D_{n+1,d}\subset D_{n+1}$
such that 
${\phi_{n, d,d}}_{*0}={{\gm^{n,n+1}_{f,f}:K_0(D_n)=G_{n,f}\to K_0(D_{n+1, d})=G_{n+1,f}}}$,
which sends strictly positive elements to
strictly positive elements.

Now define $\phi_{n, n+1}: B_n\oplus C_n\oplus D_n \to B_{n+1}\oplus C_{n+1}\oplus D_{n+1}$
by $\phi_{n,n+1}|_{B_n}=\phi_{n,b,b}\oplus \phi_{n,b,c}\oplus \phi_{n,b,d},$
$\phi_{n,n+1}|_{C_n}=\phi_{n,c,b}\oplus \phi_{n,c,c}\oplus \phi_{n,c,d},$
and $\phi_{n,n+1}|_{D_n}=\phi_{n, d,b}\oplus \phi_{n, d, c}\oplus \phi_{n, d,d}.$
Put $A_n=B_n\oplus C_n\oplus D_n.$ Then 
{{$K_0(A_n)=G_n=G_{T,n}\oplus G_{n,f}\oplus  G_{n,tor}$ and $K_1(A_n)=K_n$.}}
It is clear that $\phi_{n,n+1}$ sends strictly positive elements
to strictly positive elements as constructed above.
Moreover,
\beq
{{{\phi_{n,n+1}}_{*0}:=\gm_{n,n+1}:G_n\to G_{n+1} \andeqn
{\phi_{n,n+1}}_{*1}=\iota:K_n\hookrightarrow {{K_{n+1}.}}}}
\eneq
Denote by $\iota_n: C_n\to C_n'$ the embedding  and by ${\iota_n}_T:
T(C_n')\to T(C_n)$ the induced affine homeomorphism defined by
${\iota_n}_T(\tau)(c)=\frac{1}{1-1/2^{n+1}}\tau(\iota_n(c))$ for all $c\in C_n$
(recall $T(C_n')=\DT$).
Then,  for any {{$(f, g_f, g_t)\in G_{n,T}\oplus G_{n,f}\oplus G_{n, tor}=G_n,$ we have}}
\beq\label{Tmodel2-9}
\rho_{A_n}(f\oplus g_f\oplus  g_t)(\tau)&=&
\rho(\gm_{n,\infty}(f))
\big({\iota_{n}}_T^{-1}(\tau)\big)
\rforal \tau\in T(C_n),\\
\rho_{A_n}(f\oplus g_f\oplus g_t)(\tau)&=&0\rforal \tau\in T(B_n)\andeqn\\
\rho_{A_n}(f\oplus g_f\oplus g_t)(\tau)&=&0\rforal \tau\in T(D_n)
\eneq
{{(recall $\rho: G\to \Aff (\DT)=\Aff (\tilde{\DT})$)}}. Define $A:=\lim_{n\to\infty}(A_n, \phi_{n,n+1}).$ Then,
$K_0(A){{=G}}=G_0$ and $K_1(A){{=K}}=G_1.$   Since $B_n, \,C_n$  and $D_n$ are simple,
and all maps { {$\phi_{n,\p,\q}$ (both $\p$ and $\q$ are among $b,c,$ and $d$)}} 
are non-zero,
$\phi_{n,n+1}$ maps any nonzero element of $B_n\oplus C_n\oplus D_n$ to a full element in $B_{n+1}\oplus C_{n+1}\oplus D_{n+1}.$
It follows that $A$ is simple.
Note that, for any $b\in B_n\oplus D_n$ and any $\tau\in T(A_{n+1}),$
\beq\label{Tmodel2-n10}
|\tau(\phi_{n,n+1}(b))|<(1/2^n)\|b\|.
\eneq

{{Let}}  $q_{n,a,b}: A_n\to B_n,$   $q_{n,a,c}: A_n\to C_n$  and $q_{n,a,d}: A_n\to D_n$
be the projection maps.
Denote by  $j_{n,c,a}:
C_n\to A_n$ the \hm\, defined by $j_{n,c,a}(c)=0\oplus c\oplus 0$ for all $c\in C_n.$
Identify $T(C_n)=\DT$ as above. {{Let}} $\lambda_n: \Aff(T(C_n))\to \Aff({{(T(C_{n+1}))}}$
{{be defined as $\id: \Aff(\DT)\to \Aff(\DT)$, {{where}} we identify both $T(C_n)$ and $T(C_{n+1})$ with $\DT$. {{Then, by}} (\ref{august3-2020})}},
\beq\label{Tmodel2-n11}
\|\lambda_n(f)-\phi_{n,c,c}^\sharp(f)\|
{{=\|(1-\frac{1-1/2^{n+1}}{1-1/2^{n+2}})f\|}}<(1/2^n)\|f\|.
\eneq

For each $f\in \Aff(T(A_n)),$ we may
write $f=f_b\oplus f_c\oplus f_d,$ where  $f_b=q_{n,a, b}(f),$ {{$f_c=q_{n,a,c}(f)$}} and
{{$f_d=q_{n,a,d}(f).$}}
By \eqref{Tmodel2-n10} and \eqref{Tmodel2-n11},  {{we have}}
\beq
&&\hspace{-0.3in}\|{{j_{n+1,c,a}^\sharp}}\circ \lambda_n\circ q_{n,a,c}^\sharp(f)-\phi_{n,n+1}^\sharp(f)\|\\\label{2020-808-n1}
&&\le \|\lambda_n(f_c)-\phi_{n, c,c}^\sharp (f_c)\|
+\|{{\phi_{n, n+1}^{\sharp}}}(f_b\oplus 0\oplus f_d)\|<(2/2^n)\|f\|.
\eneq
Let $\tilde{\ld}_n=\lambda_n\circ q_{n,a,c}^\sharp$. Then we have
\beq\label{August 10,2020}
\|j_{n+1,c,a}^\sharp\circ \tilde{\ld}_n(f)-\phi_{n,n+1}^\sharp(f)\|<(2/2^n)\|f\|~~\mbox{and}~~
\tilde{\ld}_n\circ j_{n,c,a}^\sharp=\ld_n.
\eneq
%
Recall that $T(A_n)$ and $T(C_n)=\DT$ are all compact as constructed above.
Moreover,  by \eqref{August 10,2020}, 
 we obtain the following approximately commutative diagram:
\beq\label{August11-2020-di}
    \xymatrix{
        \Aff(T(A_1)) \ar[r]^{\phi_{1,2}^{{\sharp}}} \ar[dr]^{\tilde{\ld}_1}& \Aff(T(A_2)) \ar[r]^{\phi_{2,3}^{\sharp}}
        \ar[dr]^{\tilde{\ld}_2}& \Aff(T(A_3)) \ar[r] & \cd \Aff(T(A))\\
        \Aff(T({{C}}_1)) \ar[r]^{\lambda_1}\ar[u]_{j^\sharp_{1,c,a}}&
         \Aff(T({{C}}_2)) \ar[r]^{\lambda_2}\ar[u]_{j^\sharp_{2,c,a}}&
        \Aff(T({{C}}_3)) \ar[u]_{j^\sharp_{3,c,a}} \ar[r]& \cd
         \Aff(\DT). }
         \eneq
Note that, since $\phi_{n,n+1}$  maps strictly positive elements
to strictly positive elements, $\phi_{n,n+1}^\sharp: \Aff(T(A_n))\to \Aff(T(A_{n+1}))$
and $\lambda_n: \Aff(T(C_n))\to
\Aff(T(C_{n+1}))$  {{are identities when we identify all $T(C_n)$ with $\DT$}}.
Therefore   there is an affine homeomorphism $\LD: \DT\to T(A)$ which induces the diagram above.
In other words, $T(A)=\DT.$
Moreover, by 
\eqref{Tmodel2-9}, {{identifying $K_0(A_n)$ with $G_n$,}}
we obtain the following {{commutative}} diagram:
\begin{displaymath}
    \xymatrix{
        \Aff(T(A_1)) \ar[r]^{\phi_{1,2}^{{\sharp}}} & \Aff(T(A_2)) \ar[r]^{\phi_{2,3}^{\sharp}}
        & \Aff(T(A_3)) \ar[r] & \cd \Aff(T(A))\\
        G_1 \ar[r]^{{\phi}_{1,2*0}}\ar[u]_{\rho_{A_1}} &
         G_2\ar[r]^{{\phi}_{2,3*0}}\ar[u]_{\rho_{A_2}}&
         G_3 \ar[r] \ar[u]_{\rho_{A_3}}& \cd
         G\ar[u]_{\rho_A}. }
\end{displaymath}

\noindent
{{Combining with (\ref{August11-2020-di}), we obtain}}
\beq
((K_0(A), T(A), \rho_A), K_1(A))=((G_0,\DT,\rho), G_1).
\eneq
Since ${{\gm^{n,n+1}_{T,T}}}$ is injective, ${\phi_{n, n+1}}_{*0}({{C_n}})\cap {\rm ker}\rho_A=\{0\}.$
By \eqref{Tmodel2-n10},\\ {{$\lim_{n\to\infty}\inf\{d_\tau(\phi_{n, \infty}({{x_n}})):\tau\in A\}=0,$}}
where ${{x_n}}\in B_n\oplus D_n$ is any strictly positive element.
{{It follows from Lemma \ref{G-ind} that ${{\phi_n}_{*i}}|_{K_i(B_n)},$ ${{\phi_n}_{*i}}|_{K_i(C_n)}$
and ${{\phi_n}_{*i}}|_{K_i(D_n)}$ are all injective, and $K_i(B_n), K_i(C_n)$ and $K_i(D_n)$ are finitely generated. }}
It remains to show that $A\in {\cal D}.$  However, this follows from the fact that $B_n,\,C_n$  and $D_n$ are in ${\cal D}$
{{and have continuous scales}}
(and $A$ is simple). {{By Theorem 9.4 of \cite{GLII}, $A$ has strict comparison.
Since $T(A)=\Delta$ is compact, by Theorem 5.3 of \cite{GLII}, $A$ has continuous scale.}}
\end{proof}
%
%
%
%
%
%

\begin{rem}\label{Rmodle2}
With  the last part of  \ref{G-ind} in mind and with  some obvious modification, one  may also have {{following forms of inductive limit:}}

{{(1)}} $A=\lim_{n\to\infty}(B_n\oplus C_n, \phi_{n,n+1}),$
where $B_n\in {\cal B}_T,$ $K_0(B_n)={\rm Tor}(K_0(A)),$ $K_1(B_n)=K_1(A),$
$T(B_n)={{\DT}},$ and  {{$C_n$ is a simple \CA\, as in \ref{4.25},}} $K_0(C_n)=G_{n, T}\oplus G_{n,f}$ and $T(C_n)=T(A).$

{{(2) $A=\lim_{n\to\infty}(B_n\oplus C_n, \phi_{n,n+1}),$
where $B_n\in {\cal B}_T,$ $K_0(B_n)={\rm Inf}(K_0(A)),$ $K_1(B_n)=K_1(A),$
$T(B_n)={{\DT}},$ and {{$C_n$ is a simple \CA\, as in \ref{4.25},}} $K_0(C_n)=G_{n, T}$ and $T(C_n)=T(A)$, ${\rm ker}\rho_{C_n}=0$. }}

   In particular, in this modified
construction, $B_n=B_1$ for all $n\ge 1.$  However, while $K_0(C_n)$ is finitely generated,
$K_i(B_n)$ is not.
{{Thus, one also have the following form}}

{{(3) $A=\lim_{n\to\infty}(B_n\oplus C_n\oplus D_n)$ as in \ref{Tmodel2}  except that $B_n=B_1$ for all
$n\ge 1$ (thus $K_i(B_n)$ may not be finitely generated).}}

\end{rem}

\begin{df}\label{DM1}
Let ${\cal M}_1$ denote the class of stably projectionless simple \CA s  with continuous scale
 constructed in Theorem \ref{Tmodel2}, {{or  in Remark \ref {Rmodle2}.}}

{{By, \ref{4.25}, there is a  simple \CA\, $A$ which is {{an}} inductive limit of \CA s in ${\cal C}_0$ such that $A$ has a unique
tracial state,
$K_0(A)=\Z,$ ${\rm ker}\rho_A=K_0(A)=\Z,$ and $K_1(A)=\{0\}.$
By  Corollary {{15.7}} of \cite{GLII}, $A\cong \zo,$ the unique stably finite separable simple {{\CA\,}} with finite
nuclear dimension in the UCT class.}}
\end{df}


For the future usage, let us state the following corollary.

\begin{thm}\label{TclassM1-K}
{Let $\DT$ be a metrizable Choquet simplex, {{$G_0$}} a countable abelian group,
$\rho: G_0\to \Aff(\DT)$ {{a \hm\,}} such that $\rho(G_0)\cap \Aff_+(\DT)=\{0\},$
and {{$G_1$}} a countable abelian group.}
%
%
%
%
%
%
%
%
%
%
%
 There is  a simple \CA\, $A\in {\cal  M}_1$ with continuous scale  such that
 \beq\label{TM1-K-0}
((K_0(A), T(A), \rho_A), {{K_1(A)}})=((G_0, \DT, \rho), G_1).
\eneq
{{Moreover, if $A$ is as constructed in \ref{Tmodel2} or as (3) in Remark \ref{Rmodle2}, then
it}} also satisfies the following conditions:
for any finitely generated { {group}} 
$G_{0,T}\oplus G_{f,inf}\oplus G_{0,tor}\oplus G_r\subset \underline{K}(A),$
where  $G_r\cap K_0(A)=\{0\},$
$G_{0,T}\subset K_0(A)$ is a free subgroup with
$G_{0, T}\cap {\rm ker}\rho_A=\{0\},$
$G_{f,inf}\subset {\rm ker}\rho_A$ is a free subgroup,
$G_{0,tor}\subset {\rm Tor}(K_0(A)),$
any $\ep>0,$
{{any}} finite subset ${\cal F}\subset A,$  and
any $\sigma>0,$ there are mutually orthogonal \SCA s  ${{E_n}}$
{{as defined in 7.2 and 7.7  of \cite{GLII} (see \ref{DYT})}} with strictly positive element $a_{e,n}$
and
$C_n,\, D_n\in {\cal C}_0$ with strictly positive elements $a_{c, n}$  and $a_{d,n},$ respectively, satisfying the following:
\beq\label{TM1-K-10}
&&\hspace{-0.2in} {{G_{0,T}\oplus G_{0,tor}\oplus  G_{0,inf}\oplus G_r\subset [\iota_n](E_n\oplus C_n\oplus D_n),\,\,
K_0(E_n)={\rm Tor}(K_0(E_n)),}}
\\\label{TM1-K-11}
&&G_{0, T}\subset {\iota_n}_{*0}(K_0(C_n)),   G_{0,\inf}\subset {\iota_n}_{*0}(K_0(D_n)), G_{0, tor}\subset {\iota_n}_{*0}(E_n),\\
&&{\iota_{n}}_{*0}(K_0(C_n))\cap {\rm ker}\rho_A=\{0\},  {\iota_n}_{*0}(K_0(D_n))\cap {\rm Tor}(K_0(A))=\{0\},\andeqn\\
&&{\iota_n}_{*0}(K_0(D_n))\subset {\rm ker}\rho_A,
\eneq
where $\iota_n: E_n\oplus C_n\oplus D_n\to A$
 is the embedding,
\beq\label{TM1-K-3}
a\approx_{\ep} \phi_{e,n}(a)\oplus \phi_{c,n}(a)\oplus { {\phi_{d,n}(a)}}\tforal a\in {\cal F},
\eneq
where $\phi_{e,n}: A\to E_n$,  $\phi_{c,n}: A\to C_n$  and $\phi_{d,n}: A\to D_n$
are \cpc s which are ${\cal F}$-$\ep/2$-multiplicative,
\beq\label{TM1-K-4}
d_\tau(a_{e,n})+d_\tau(a_{d,n})<\sigma\andeqn d_\tau(a_{c,n})>1-\sigma \rforal \tau\in T(A),\andeqn
\eneq
\beq\label{TM1-K-5}
\lambda_s(C_n)>1-\sigma, \andeqn\lambda_s(D_n)>1-\sigma
\eneq
\end{thm}

\begin{proof}
Let $A=\lim(A_n=B_n\oplus C'_n\oplus D'_n, \phi_{n,m})$ {{be}}
as in Theorem \ref{Tmodel2} such that (\ref{TM1-K-0}) holds {{(or $A$ is as in (3) of Remark \ref{Rmodle2}).}}
Note that $\phi_{n,n+1}$ are injective, so we can regard $A_n$ as {{a subalgebra}} of $A$.  Choose $m$  with the strictly positive elements $b\in B_m$, $c\in C'_m$ and $d\in D'_m$ such that\\
(i) ${\cal F}\subset_{\ep/10} A_m=B_m\oplus C'_m\oplus D'_m$,\\
(ii) $G_{0,T}\oplus G_{0,tor}\oplus  G_{0,inf}\oplus G_r\subset [\iota](A_m)$, and  $G_{0, T}\subset {\iota}_{*0}(K_0(C'_m)),   G_{0,\inf}\subset {\iota}_{*0}(K_0(D'_m)), G_{0, tor}\subset {\iota}_{*0}(B_m)$ and \\
(iii)$d_{\tau}(b+d)<\sigma/2$ and $d_{\tau}(c)>1-\sigma/4$ for all $\tau\in T(A)$.

Note that {{$D'_m$ and $C_m'$}}
can be written, {{respectively,}}  as inductive limits of $D_{m,n}$ and $C_{m,n}$ with $D_{m,n}, C_{m,n}\in {\cal C}_0$. {{Also from \ref{DYT},   $B_m=\lim_{n\to\infty}(E_{k(n)}\oplus W_n, \Phi_{n,n+1}),$
where $E_k=M_{(k!)^2}(A(W, \af_k))$ and $W_n$  is in ${\cal C}_0$ with $K_0(W_k)=\{0\}$. Since $B_m$ has continuous scale
{{(see \ref{DYT}),}}  for $n$ large enough, we have $\ld_s(E_{k(n)})>1-\sigma$ and $\ld_s(W_n)>1-\sigma$ (see 5.3 of \cite{EGLN}). }} One can choose $n$ large enough such that \\
(iv) ${\cal F}\subset_{\ep/5} {{E_{k(n)}\oplus W_n}}\oplus C_{m,n}\oplus D_{m,n}$ \\
(v) $G_{0,T}\oplus G_{0,tor}\oplus  G_{0,inf}\oplus G_r\subset [\iota]({{E_{k(n)}\oplus W_n}}\oplus C_{m,n}\oplus D_{m,n})$, and  $G_{0, T}\subset {\iota}_{*0}(K_0(C_{m,n})), \\  G_{0,\inf}\subset {\iota}_{*0}(K_0(D_{m,n}))$\\
(vi) for the strictly positive element ${ {e^A_n}}\in C_{m,n}$, $d_\tau({ {e^A_n}})>1-\sigma/4$ for all $\tau \in T(\tilde{C}_m)$, and\\
(vii) $\ld_s(C_{m,n})>1-\sigma$, {{$\ld_s(D_{m,n})>1-\sigma$ and $\ld_s(W_n)>1-\sigma$.}}

{{Let}} $C_n=C_{m,n}$,  $D_n=D_{m,n}\oplus W_n$, $E_n=E_{k(n)}$ and  {{let}} $a_{c,n}={ {e^A_n}}\in C_n$, $a_{e,n}\in E_n$, $a_{d,n}\in D_n$ be  strictly positive elements. Then (\ref{TM1-K-4}) follows from (iii) and (vi). It is standard to construct \cpc s $\phi_{e,n}: A\to E_n$,  $\phi_{c,n}: A\to C_n$  and $\phi_{d,n}: A\to D_n$ to finish the proof.
\end{proof}


\begin{rem}\label{ReModelK}
In the statement  of Corollary  \ref{TclassM1-K}
we may replace $B_n$ by a simple \CA\,
of the form $B_T$ as {{in}} Theorem 7.11 of \cite{GLII}  with continuous scale, and $C_n$  and $D_n$ are   simple
\CA s  which are  constructed in \ref{4.25} with continuous scale and retain \eqref{TM1-K-10} and \eqref{TM1-K-11}.
Moreover, we may assume that $a_{e,n}+a_{c,n}+a_{d,n}$ is a strictly positive element.
Note that all  {{ $K_i(B_n)$, $K_i(C_n)$ and  $K_i(D_n)$ }}  are finitely generated ($i=0,1$).
\end{rem}


%
%
%

%

%
%
%

%
%
%

%
%
%
%
%
%




%
%
%


%
%
%
%




%
%
%

\section{Range of the Elliott invariant}

The following statement {{implies}} that, in the case that $A$ is simple,  {{the}} pairing does not depend {{on}}
where the unitization occurs.
\begin{prop}\label{Prhowd}
Let $A$ be a $\sigma$-unital simple \CA\, with $a\in {\rm Ped}(A)_+\setminus \{0\}.$
Suppose that $\iota: B:={\rm Her}(a)\to A$ is the embedding.
Then  $\iota_{*0}$ is an isomorphism and $\rho_{B}(x)(\iota_T(\tau))=\rho_A(\iota_{*0}(x))(\tau)$  for all $x\in {{K_0(B)}}$ and
for all $\tau\in \td T(A),$ where $\iota_T: \td T(A)\to \td T(B)$ is the map induced by $\iota.$
\end{prop}

\begin{proof}
Claim: in general, without assuming $A$ is simple,  if $B$ is a $\sigma$-unital full hereditary \SCA\, of $A,$ then
$\iota_{*i}: K_i(B)\to K_i(A)$ is an isomorphism ($i=0,1$).  This is a reconstruction of Brown's argument in \cite{Br}.

If $B$ is a full corner of $A,$ i.e., $B=pAp$ for some projection $p\in M(A),$  then the claim follows from Corollary 2.6 of \cite{Br}.
In general, let $C$ be the \SCA\, of $A\otimes M_2$
consisting of the sum $\sum a_{ij}\otimes e_{ij}$ such that
$a_{11}\in B,$ $a_{12}\in \overline{BA},$ $a_{21}\in \overline{AB}$ and
$a_{22}\in A,$ where {{$\{e_{ij}\}_{1\le i,j\le 2}$}} is a system of matrix units.
We identify $B$ with the full corner $B\otimes e_{11}$ and
$A$ with the full corner $A\otimes e_{22}$ of $C$ (see the proof of Theorem 2.8 of \cite{Br}).
Let $j_A: A\to C$ and $j_B: B\to C$ be
embeddings.   Then  $(j_A)_{*i}$ and $(j_B)_{*i}$ are isomorphisms ($i=0,1$).

On the other hand,  $j_B(B)$ and $j_A\circ \iota(B)$ may be identified with $B\otimes e_{11}$ and
$B\otimes e_{22}$ in $M_2(B)\subset C.$  Denote by  $j: M_2(B)\to C$ the embedding.
Then $j\circ j_B=j_B$ and ${j\circ j_A}|_{B\otimes e_{22}}={j_A}|_{B\otimes e_{22}}.$
 Since  $j_{*i}=(j\circ j_B)_{*i}=({j\circ j_A}|_{B\otimes e_{22}})_{*i}: K_i(B)\to K_i(C),$
one has  that $(j_A)_{*i}\circ \iota_{*i}=(j_A\circ \iota)_{*}=(j\circ j_B)_{*i}=(j_B)_{*i}$  is also
an isomorphism.  Since $(j_A)_{*i}$ is an isomorphism, so is $\iota_{*i}$ ($i=0,1$). This proves the claim.

The lemma  follows from  the claim and the commutative  diagram  (\ref{2020-8-8-d1}) immediately.
%
\end{proof}

Let us state the following result (see \ref{DElliott}) which also holds if we replace
the condition that $A$ has finite nuclear dimension by that  {{$A$}} is ${\cal Z}$-stable (i.e., $A\cong A\otimes {\cal Z}$) and all 2-quasitraces are traces.

\begin{thm}\label{TRangeM}
Let  $A$ be a separable finite simple
\CA\,
with finite nuclear dimension.
Then
$(K_0(A), \Sigma(K_0(A)), {\tilde T}(A), {{\widehat{\la e_A\ra}}}, \rho_A)$
is a scaled simple ordered group pairing (see \ref{DElliott23} and \ref{Dparing}).
\end{thm}

\begin{proof}
It follows from \cite{Winter-Z-stable-02} and \cite{T-0-Z} that $A$ is ${\cal Z}$-stable.
By Corollary 5.1 of \cite{Rrzstable},
$\td T(A)\not=\{0\}.$
Then, if $A$ is unital, $(K_0(A), K_0(A)_+, [1_A])$ is
a weakly unperforated simple ordered group {{(see \cite{GJS})}} with the scale determined by
the order unit $[1_A],$ and
$g\in K_0(A)_+\setminus \{0\}$ if and only if $\rho_A(g)(\tau)>0$ for all $\tau\in T(A).$
So the unital case follows.
Suppose that  $A$ is not unital and $K_0(A)_+\not=\{0\}.$  Let $x\in K_0(A)_+{{\setminus\{0\}}}.$
Then there is a  nonzero projection $p\in M_r(A)$ for some integer $r\ge 1$ such that $[p]=x.$
If follows that $A_1:=pAp$ is a unital simple \CA\,  with finite nuclear dimension (see 2.8 of \cite{WZ-ndim}).
Therefore  $(K_0(A_1), K_0(A_1)_+)$ is a weakly unperforated simple ordered group
such that $g\in K_0(A_1)_+\setminus \{0\}$ if and only if $\rho_{A_1}(g)(\tau)>0$
for all $\tau\in T(A_1).$  Note that $A\otimes {\cal K}\cong A_1\otimes {\cal K}.$
It follows that $(K_0(A), {\tilde T}, \rho)$ is a simple ordered group pairing.

Let $e_A\in A$ be a strictly positive element with $\|e_A\|=1$ and let $s(\tau)=d_\tau(e_A)$
for all $\tau\in {\tilde T}(A).$
{{If $p\in A$ is a projection, then, $f_{1/n}(e_A)p\approx_{1/2} p$ for some integer
$n\ge 1.$ It follows that there is a projection
$q\in {\rm Her}(f_{1/n}(e_A))$  such that $[q]=[p].$ It follows that $f_{1/2n}(e_A)q=q$
and $(1-q)f_{1/2n}(e_A)(1-q)\not=0.$
Since $A$ is simple, this implies that $\tau(p)=\tau(q)<d_\tau(e_A).$}}
Then
$$
\Sigma(K_0(A))=\{g\in G_+: g=[p]\,\, {\rm for\,\, some \,\, projection}\,\, p\in A\}=\{g\in G_+: \rho(g)< s\}.
$$
It follows that $(K_0(A), \Sigma(K_0(A)), {\tilde T}(A), s, \rho_A)$ is a scaled simple ordered group
pairing. {{(Even though $K_0(A)_+\not=\{0\}$, it is still possible {{that}} ${{\Sigma (K_0(A))}}=\{0\}$.)}}

Now assume that $K_0(A)_+=\{0\}.$ Therefore $A$ is stably projectionless.
It follows from 5.2 of \cite{eglnp},  for example, that one may choose $a\in A_+\setminus \{0\}$ such that $A_1:=\overline{aAa}$
has continuous scale.  {{It follows from (1) of 5.3 of \cite{eglnp} that $T(A_1)$ is compact.
By 5.2.2 of \cite{Pbook}, every tracial state $\tau\in T(A_1)$ extends to a lower semi-continuous trace on $A$
which is finite on ${\rm Ped}(A)$ as $A$ is simple.
Again, since $A$ is simple, the extension is unique.
It follows that $T(A_1)$ is a
base for the cone ${\tilde T}(A).$ Since ${ {\Aff({\tilde T}(A))}}$ is a lattice (see  Corollary 3.3 of \cite{PdMI}
and Theorem 3.1 of \cite{PdMIII}),
${\tilde T}(A)$ is a {{convex}} topological cone with
a Choquet simplex as its base.}}
 For any $x\in K_0(A_1),$ by A.7 of \cite{eglnkk0}, $\rho_{A_1}(x)(\tau)=0$ for some
$\tau\in T(A_1).$   Since $A\otimes {\cal K}\cong A_1\otimes {\cal K},$
this implies that
$(K_0(A), \{0\}, {\tilde T}(A), {{\widehat{\la e_A\ra}}}, \rho_A)$ is a  scaled simple ordered group pairing (see \ref{Prhowd}).
\end{proof}

{{The following range theorem was given by Elliott in \cite{point-line} in the stable case.}}

\begin{thm}\label{TrangeGG}
Let $(G_0, \Sigma(G_0), T, s, \rho)$ be a scaled simple ordered group pairing and
$G_1$ be a countable abelian group.
Then, there is {{a}} simple separable  amenable \CA\, $A$  which satisfies the UCT such that
\beq
{{((K_0(A), \Sigma(K_0(A)), {\tilde T}(A), {{\widehat{\la e_A\ra}}}, \rho_A), K_1(A))=((G_0, \Sigma(G_0), T, s, \rho), G_1).}}
\eneq
$A$  is unital if and only if
$\Sigma(G_0)$ has a {{unit}} $u.$ {{(This means that $u$ is the non-zero maximum in $\Sigma(G_0)$ and $\rho(u)=s$. See Definition \ref{DElliott23}.)}} If  $\rho(G_0)\cap \Aff_+(T)\not=\{0\},$
 then
$A$ can be chosen to have  rationally generalized tracial rank at most one and {{be}}
an inductive limit
of sub-homogeneous  \CA s of spectra with dimension no more than 3.
If ${{\rho(G_0)\cap \Aff_+(T)=\{0\}}},$ then $A$ is stably projectionless {{and}}
$A$ can be chosen to have  generalized tracial rank one and {{be}}
 locally approximated by sub-homogeneous \CA s with {{the spectra having}} dimension no more than {{3}}.

If $G_1=\{0\}$ and $G_0$ is torsion free,  then $A$ can be chosen to be an inductive limit of 1-dimensional NCCW complexes
{{in ${\cal C}_0.$}}

\end{thm}

\begin{proof}
Let us first consider the case that $\Sigma(G_0)$ has a {{unit}} $u.$
Let $\DT:=\{t\in T: \rho(u){{(t)}}=1\}.$  Then $\DT$ is a base for the  cone $T.$
Recall that $\rho(G)\subset \Aff(T).$  Therefore $\DT$ is a compact convex subset.
Moreover it is a base for $T.$ Since $T$ has a metrizable Choquet simplex as a base,
{ {$\Aff(T)=\Aff(\DT)$}} is a lattice. Therefore
$\DT$ is a Choquet simplex.
Let ${G_0}_+=\{g\in G: \rho(g)>0\}\cup \{0\}.$
It follows from Theorem {{13.50}} of \cite{GLN} that there is a unital simple \CA\, $A$
which  has rationally  generalized tracial rank at most one
{{and}} is an inductive limit of sub-homogeneous \CA s of spectra  with dimension no more than 3 such
that
$$
(K_0(A), K_0(A)_+, [1_A], K_1(A), T(A), \rho_A), K_1(A)) =((G_0, {G_0}_+, u, G_1, \DT, \rho), G_1).
$$
Put ${\tilde T}(A):=\{r\tau: r\in \R_+, \tau\in T(A)\}.$
Then ${\tilde T}(A)=T.$ Moreover
$\Sigma(K_0(A))=\Sigma(G_0).$ This proves the case that $\Sigma(G_0)$ has  a unit.

Consider the case $\rho(G_0)\cap \Aff_+(T)\not=\{0\}$ and $\Sigma(G_0)$ has no unit.
 Choose $v\in G_0$ such that $\rho(v)\in \Aff_+(T)\setminus \{0\}.$
Put $\Sigma_1(G)=\{g\in G_+: \rho(g)<v\}\cup\{v\}.$
Since $\rho(v)\in \Aff_+(T),$  as above, $\DT:=\{t\in T: {\rho(v)(t)}=1\}$ is a Choquet simplex.
Then, by what has been  shown,
there is a \CA\, $A_1$ which has  rationally tracial rank at most one such that
\beq
((K_0(A_1), \Sigma(K_0(A)), T(A_1), \rho_A([1_{A_1}]), \rho_{A_1}), K_1(A_1))
=((G_0, \Sigma_1(G_0), T_1, \rho'), G_1),
\eneq
where $T_1:=\{r \xi: r\in \R_+, \xi\in \DT\}=T$ and
$\rho': G\to \Aff(T_1)$
is defined to be {{the}} same as $\rho$, when we identifying $T_1$ with $T.$
Choose an element $e_A\in A_1\otimes {\cal K}$ such
that 
${{\widehat{\la e_A\ra}}}(\tau)=s(\tau)$ for all $\tau\in {\tilde T}(A_1)=T_1=T$
{{(see Theorem 5.5 of \cite{BPT}).}}
Define $A=\overline{e_A(A_1\otimes {\cal K})e_A}.$
One then checks {{that}}
\beq
((K_0(A), \Sigma(K_0(A)), {\tilde T}(A), {{\widehat{\la e_A\ra}}}, \rho_A), K_1(A))\cong ((G_0, \Sigma(G_0), T, s, \rho), {{G_1}}).
\eneq

Now we consider the case that ${{\rho(G)\cap \Aff_+(T)=\{0\}.}}$
Let $\DT$ be a base of $T$ which is a Choquet simplex.
Define $\rho': G\to \Aff(\DT)$ {{to be the same map as $\rho$ by restricting a function in $\Aff(T)$ to $\Aff(\DT)$.}}
By Theorem \ref{Tmodel2}, there is a simple \CA\, {{$A_1\in {\cal D}$}} with continuous scale which is  an inductive limit
of $B_n\oplus C_n\oplus D_n,$ where $B_n$ is locally approximated by
sub-homogenous \CA s with {{spectra}} having
dimension no more than 3 (see \ref{DYT}), 
and $C_n\oplus D_n$ is an inductive limit of \CA s in ${\cal C}_0$
such that
\beq
(K_0(A_1),T(A_1), \rho_{A_1}, K_1(A_1))=(G_0, G_1, \DT, \rho', G_1).
\eneq
Choose $e_A\in (A_1\otimes {\cal K})_+\setminus \{0\}$ such
that 
${{\widehat{\la e_A\ra}(t)=\lim_{n\to \infty}t(e_A^{1/n})=s(t)}}$  for all $t\in \DT.$
Define $A:=\overline{e_A(A_1\otimes {\cal K})e_A}.$  {{Then $A$ has generalized tracial rank one (see \ref{DD0}).}}
One then checks {{that}}
\beq
{{((K_0(A), \Sigma(K_0(A)), {\tilde T}(A), {{\widehat{\la e_A\ra}}}, \rho_A), K_1(A))=
((G_1, (G_0, \{0\}, T, s, \rho), G_1).}}
\eneq
The last part of the lemma follows immediately from
\ref{4.25}.

\end{proof}

%
%
%
%
%
%

%

\begin{cor}\label{smallmap}
Let {{$A_1$}} be a simple separable C*-algebra in {{$\mathcal D$ with continuous scale,  {{$U$}} be an infinite dimensional UHF algebra {{and}} $A=A_1\otimes U$.}}

 There exists an
  inductive limit algebra $B$ as constructed in Theorem {{\ref{Tmodel2}}} such that ${{A=A_1}}\otimes U$
  and $B$ have the same  Elliott invariant{.}
%
   Moreover, the C*-algebra $B$ {{has}} the following properties:

Let $G_0$ be a finitely generated subgroup of $K_0(B)$ with decomposition $G_0=G_{00}\oplus G_{01}$, where $G_{00}$  vanishes under all states of ${{K_0(B)}}$. Suppose $\mathcal P\subset \underline{K}(B)$ is a finite subset which generates a subgroup $G$ such that $G_0\subset G\cap K_0(B)$.

Then, for any $\epsilon>0$, any finite subset $\mathcal F\subset B$, any $1>r>0$, and any positive integer $K$, there is an
$\mathcal F$-$\epsilon$-multiplicative map $L:B\to B$ such that:
\begin{enumerate}
\item $[L]|_{\mathcal P}$ is well defined.
\item $[L]$ induces the identity maps on 
    {{$G_{00}$,}} $G\cap K_1(B)$,
      $G\cap K_0(B,\mathbb Z/k\mathbb Z)$ and $G\cap K_1(B, \mathbb Z/k\mathbb Z)$
      for $k=1, 2, ...$, and $i=0,1$.
\item $\|\rho_B\circ[L](g)\|\leq r\|\rho_B(g)\|$ for all $g\in
      G\cap K_0(B)$, where $\rho_B$ is the canonical positive homomorphism from $K_0(B)$ to
      $\Aff(T(B)).$
\item For any
element $g\in G_{01}$, we have $g-[L](g)=Kf$ for some $f\in K_0(B)$.
\item $d_\tau(e_0)<{{r}}\rforal \tau\in T(B),$
\end{enumerate}
where $e_0$ is a strictly positive element of $\overline{L(B)BL(B)}.$
\end{cor}

\begin{proof}

 Consider ${\rm Ell}(A_1).$
{{By}} Theorem \ref{Tmodel2},  there is {an} inductive
system $B_1=\varinjlim(T_i\oplus C_i, \psi_{i, i+1})$ (where $T_i:=B_i\oplus D_i$ in Theorem \ref{Tmodel2}) {{such}} that

(i) $\rho_{T_i}=0: K_0(T_i)\to {{\Aff (T(T_i))}}$ and $C_i\in {\mathcal C}_0$
with $K_1(C_i)=\{0\}$,

(ii) For the strictly positive element $e_{T_i}\in T_i$,  $\lim \tau(\phi_{i, \infty}(e_{T_i}))= 0$ uniformly on $\tau\in T(B_1)$,

(iii) $\ker(\rho_{B_1})=\bigcup_{i=1}^\infty (\psi_{i, \infty})_{*0}(K_0(T_i))$, and

(iv) ${\rm Ell}(B_1)={\rm Ell}(A_1)$.

Put $B=B_1\otimes U$. Then ${\rm Ell}(A)={\rm Ell}(B)$. Let $\mathcal P {{\subset}} \underline{K}(B)$ be a finite subset, and let $G$ be the subgroup generated by {{$\mathcal P$,}} which we may assume {{to contain}} $G_0$. Then there is a positive integer $M'$ such that $G \cap K_*(B, \mathbb Z/ k\mathbb Z)={\{0\}}$ if $k>M'$. Put $M=M'!$. Then $Mg=0$ for any $g\in G \cap K_*(B, \mathbb Z/ k\mathbb Z)$, {{$k=1, 2, ... .$}}

Let $\ep>0$, {{a finite subset  $\mathcal F\subset B$,}} and $0< r<1$ be given. Choose a finite subset $\mathcal G\subset B$ and  $0<\ep'<\ep$
such that $\mathcal F\subset \mathcal G$  and for any $\mathcal G$-$\epsilon'$-multiplicative map $L: B\to B$,
the map $[L]_{\mathcal P}$ is well defined, and $[L]$ is
a homomorphism on $G$.

Since $B=B_1\otimes U,$ we may write
$U=\varinjlim(M_{m(n)}, \imath_{n, n+1}),$
where $m(n)|m(n+1)$ and $\imath_{n, n+1}: M_{m(n)}\to M_{m(n+1)}$ is defined
by $a\mapsto a\otimes 1_{m(n+1)}.$  Choosing a sufficiently large $i_0$ and $n_0$, we may assume that
$[\psi_{i_0, \infty}](\underline{K}{{((T_{i_0}\oplus C_{i_0})\otimes M_{m(n_0)})}})\supset G.$
In particular, we may assume, by (i) and (iii) above, that $\rho_{T_i\otimes M_{m(n_0)}}=0$ {{and}}
$G\cap {\rm ker}\rho_{B_1\otimes M_{m(n_0)}}\subset (\psi_{i_0, \infty})_{*0}(K_0(T_i)\otimes M_{m(n_0)}).$ Let $G'\subset \underline{K}((T_{i_0}\oplus { {C_{i_0}}})\otimes M_{m(n_0)})$ be such that $[\psi_{i_0, \infty}]({ G}'){{\supset}}{ G}.$

One may assume that, for each $f\in\mathcal G$, {there exists $i> i_0, n_0$ such that}
\begin{equation}\label{diag-f}
f=(f_0\oplus f_1)\otimes 1_m
\in (T_i'\oplus C_i')\otimes M_m  
\end{equation}
for some $f_0\in T_i'$, ${{f_1}}\in C_i'$, and $m>2MK/r,$ where
$m= m(n)/m(i),$
$T_i'=\psi_{i, \infty}'(T_i\otimes M_{m(i)}),$  $C_i'=\psi_{i, \infty}'(C_i\otimes M_{m(i)}),$ and
where $\psi'_{i, \infty}=\psi_{i, \infty}\otimes \imath_{i,\infty}.$
Moreover, one may assume that $\tau(1_{T_i'})<r/2$ for all $\tau\in T(A_1)$. 

Choose a large $n$ {{such}} that $m={M_0}+l$ with ${M_0}$ divisible by $KM$ and $0\leq l<KM$. Then define {{a map}} $L: (T_i'\oplus C_i')\otimes M_m \to (T_i'\oplus C_i')\otimes M_m$ to be
$$
{L((f_{i, j} \oplus g_{i, j})_{m\times m})=(f_{i,j})_{m\times m}\oplus E_l(g_{i,j})_{m\times m}E_l,}
$$
where $E_l={\rm diag}(\underbrace{1_{(C_i')^{\sim}}, 1_{(C_i')^{\sim}},...,1_{(C_i')^{\sim}}}_l, \underbrace{0,0,...,0}_{M_0}).$
{{Note that $L$}}
is {{also}} a \morp\, from $(T_i'\oplus C_i')\otimes M_m$ to $B$, where we identify $B$ with $B\otimes M_m$. {{(Note that $E_l\notin C'_i\otimes M_m$ but $E_l(g_{ij})E_l\in C'_i\otimes M_m$).}}
We then
 extend $L$ to a completely positive linear map  from $B$ to $B$. 
 Also define $R: {(T_i'\oplus C_i')\otimes M_m \to T_i'\oplus C_i'}$ to be
 {{\begin{equation}
R
\left({{\left(
\begin{array}{cccc}
f_{1,1}\oplus g_{1,1} & f_{1,2}\oplus g_{1,2}&\cdots & f_{1,m}\oplus g_{1,m}\\
f_{2,1}\oplus g_{2,1} & f_{2,2}\oplus g_{2,2}&\cdots &f_{{2,m}}\oplus g_{2,m}\\
 & &{{\ddots}}  &\\
f_{m,1}\oplus g_{m,1} & f_{m,2}\oplus g_{m,2}&\cdots & f_{m,m}\oplus g_{m,m}
\end{array}
\right)}}\right)=g_{1,1},
\end{equation}
where $f_{j,k}\in T'_i$ and $g_{j,k}\in C'_i$,}}
~and extend it to a {\morp}\, $B\to B$,
where ${T_i'\oplus C_i'}$ is regarded as a corner of ${(T_i'\oplus C_i')}\otimes M_m\subset B$. Then $L$ and $R$ are $\mathcal G$-$\epsilon'$-multiplicative. Hence $[L]|_\mathcal{P}$ is well defined. Moreover, $$d_\tau(e_0)=d_\tau(L(e_B))<d_\tau(e_{T_i'})+\frac{l}{m}<\frac{r}{2}+\frac{MK}{2MK/r}=r\rforal \tau\in T(B),$$ where $e_B$ and $e_{T_i'}$ are strictly positive {{elements}} in $B$ and {{$T_i',$}} respectively.

Note that, for any $f$ in the form  \eqref{diag-f}, if $f$ is written in the form ${{(f_{jk}\oplus g_{jk})}}_{m\times m}$, then {{$g_{jj}=g_{11}$ and $g_{jk}=0$ for $j\not= k$}}. Hence one has
$$f=L(f)+{\overline{R}}(f),$$
{where $\overline{R}(f)$ may be written as}
$$
{\overline{R}(f)=\mathrm{diag}\{\underbrace{0,0,...,0}_l, \underbrace{(0\oplus g_{1,1}), ..., (0\oplus g_{1,1})}_{M_0}\}}.
$$
Hence for any $g\in G$,
$$g=[L](g)+ {M_0}[R](g).$$
Then, if $g\in\ (G_{0, 1})_+\subset (G_0)_+$, one {{has}}
$$g-[L](g)={M_0}[R](g)=K((\frac{M_0}{K})[R](g)).$$
{{Also}} if $g\in G\cap K_i(B, \mathbb Z/ k\mathbb Z)$ ($i=0,1$), one also has
$$g-[L](g)={M_0}[R](g).
$$
Since $Mg=0$  and {$M|M_0$}, one has $g-[L](g)=0$.



{Since $L$ is {{the}} identity on $\psi_{i, \infty}'(T_i\otimes M_{m(i)})$ and $i>i_0,$ by (iii),
{{$[L]$}} is the identity map on $G\cap \ker \rho_{B}.$  Since $K_1(S_i)=0$ for all $i,$
$L$ induces the identity map on $G\cap K_1(B).$ It follows that}
 $L$ is the desired map.
\end{proof}

%


%

%
%
%

%


\begin{lem}\label{Lstrictpositive}
Let $C=A(F_1, F_2, \bt_0, \bt_1)\in {\cal C}_0$ and let $N_0\ge 1$ be an integer.
There exists $\sigma>0$ satisfying the following condition:
For any  order preserving \hm\, $\kappa: K_0(\td C)\to \R$ such that,
for any $x\in K_0(\td C)_+\setminus \{0\}$ {{with}} $N_0\kappa(x)>1$ and $\kappa({ {[1_{\td C}]}})=1,$
there exists $t\in T(C)$ such
that
\beq
t(h)\ge {{\sigma \int_{s\in [0,1]} T(\lambda(h)(s))d\mu(s)}}{{\tforal h\in C_+\tand}}\\
{{\kappa(x)}}=\rho_{C}(x)(t) {{\tforal}} x\in K_0(\td C),
\eneq
where $\lambda: C\to C([0,1], F_2)$ is the natural embedding  and
$T(b)=\sum_{j=1}^k {\rm tr}_j (\psi_j(b))$ for all $b\in F_2,$
where $F_2=\bigoplus_{j=1}^k M_{r(j)},$ $\psi_j: F_2\to M_{r(j)}$ is the projection map,
${\rm tr}_j$ is the normalized trace on $M_{r(j)},$ and $\mu$ is
{{Lebesgue}}
measure.
\end{lem}
\begin{proof}
Let us write $F_1=\bigoplus_{i=1}^lM_{R(i)}.$  Denote by $q_i: F_1\to M_{R(i)}$ the projection map
and $\bar t_i$ the tracial state of $M_{R(i)},$
$i=1,2,...,l.$   Let $\pi_e: C\to F_1$   {{be}} the quotient map and
$T: K_0(F_2)\to \R$ be defined by $T(x)={{\sum_{j=1}^k \rho_{M_{r(j)}}\circ \psi_{j*0}(x)({\rm tr}_j)}}$
for all $x\in K_0(F_2).$

Define \hm s $\bt_0': \C\to F_2$
by $\bt_0'(1_\C):=1_{F_2}-\bt_0(1_{F_1})$ and $\bt_1'(1_\C):=1_{F_2}-\bt_1(1_{F_1})$
(at least one of them is not zero as $C$ is not unital).
We may write $\td C=A(F_1^\sim, F_2, \bt_0^\sim , \bt_1^\sim),$
where $F_1^\sim=F_1\oplus \C$ and  $\bt_i^\sim=\bt_i\oplus \bt_i',$ $i=0,1.$
Denote by  $q_{l+1}: F_1^\sim \to \C$  the projection map and
$\pi_e^\sim: \td C\to F_1^\sim$  {{the}} quotient map which extends $\pi_e.$
{{ Recall {{that}} $T([\bt_0^\sim (\pi_e^\sim (1_{\td C}))])=k$ and for any projection $p\in M_m(\tilde{C})$, $T(p)\leq mT([\bt_0^\sim (\pi_e^\sim (1_{\td C}))])=mk.$}}

Let $p_1, p_2,...,p_s\in M_m(\td C)$ be a set of minimal projections for some integer $m\ge 1$
such that they generate $K_0(\td C)_+$ (see 3.15 of \cite{GLN}).
There is $\sigma_{0}>0$ such that ${{\sigma_{0}}}m{{k}}<1/2.$ 
Choose $\sigma:=\sigma_0/2N_0.$ Since $N_0\kappa(x)>1$ for all $x\in K_0(\td C)_+\setminus \{0\},$
for any $p\in \{p_1,p_2,...,p_s\},$  {{we have}}
\beq
N_0\kappa([p])-\sigma_0T\circ (\bt_0\circ \pi_e)(p)>0.
\eneq
Define $\Gamma: K_0(\td C)\to \R$  by
\beq
\Gamma(x)=\kappa(x)-\sigma_0 T\circ  (\bt_0^\sim \circ \pi_e)_{*0}(x)\rforal x\in K_0(\td C).
\eneq
By 3.5  of \cite{GLN},  $\pi_{e*0}:K_0(\td C)\to K_0(F_1^\sim)$ is an order embedding.
It follows from Theorem 3.2 of \cite{GH} that there is an order preserving \hm\,
$\Gamma^\sim: {{K_0(F_1^\sim)}}\to
\R$ such that $\Gamma^\sim\circ \pi_{e*0}=\Gamma.$
Put $\af_j:=\Gamma^\sim([e_j])\ge 0,$
where $e_j:=q_j\circ \pi_e^\sim (1_{\td C}),$ $j=1,2,...,l+1.$
Define, for $({{f}},b)\in C^\sim={{\{}}C([0,1], F_2)\oplus F_1^\sim: {{f}}(0)=\bt_0^\sim(b)\andeqn {{f}}(1)=\bt_1^\sim(b)\},$
\beq
t(({{f}},b))=\sigma{{\sum_{j=1}^k\int_{s\in (0,1)} {\rm tr}_j(f)(s) dm(s)}}+\sum_{i=1}^{l+1}\af_i {\bar t}_i(q_i(\pi_e^\sim(b))).
\eneq
For any projection $p=(p, \pi_e^\sim (p))\in M_N({\tilde C})$ (for some $N\ge 1$), {{we have}}
\beq
\rho_{C}([p])(t)=t(p, \pi_e^\sim(p))=\sigma{{\sum_{j=1}^k\int_{s\in (0,1)} {\rm tr}_j(p)(s) dm(s)}}+\sum_{i=1}^{l+1}\af_i {\bar t}_i(q_i(\pi_e^\sim(p)))\\
=\sigma_0T\circ  (\bt_0^\sim \circ \pi_e)_{*0}([p])+\Gamma^\sim([\pi_e^\sim(p)])=\kappa([p]).
\eneq
Moreover, if $h=({{f}}, b)\in C_+,$  {{we have}}
\beq
t(h)\ge \sigma{{\sum_{j=1}^k\int_{s\in (0,1)} {\rm tr}_j(f)(s) dm(s)=\sigma \int_{s\in [0,1]}T(\lambda(h))d\mu(s)}}.
\eneq
\end{proof}

\begin{lem}\label{LHahn-B-1}
{{Let $C\in {\cal C}_0$ and let $G_1\subset \Aff(\DT)$ be a countable subgroup
such that $G_1\cap \Aff_+(\DT)=\{0\},$ where $\DT$ is a metrizable Choquet simplex.
Let $\eta: K_0({\tilde C})\to G_1+\Z1_{\DT}$ be an order preserving \hm,
where $1_\DT\in \Aff(\DT)$ is the constant function with value 1 (with
$\Aff_+(\DT)$ as the (strictly) positive cone of $\Aff(\DT)${{--see \ref{DTtilde}}}).
Then there is a morphism $\eta^\sim: {\rm Cu}^\sim({{\tilde C}})\to G_1\sqcup {\rm LAff}^\sim_+(\DT)$
in ${\bf Cu}$ such that $\eta^\sim|_{K_0(\td C)}=\eta.$}}
\end{lem}

\begin{proof}
Write $C=A(\bt_0, \bt_1, F_1, F_2),$
where $F_1=\bigoplus_{i=1}^{k}M_{R(i)},$ $F_2=\bigoplus_{j=1}^{l} M_{r(i)},$
and $\bt_i: F_1\to F_2$ are \hm s.
Recall that $\Aff(T(C))$ is identified as a subspace of
$C([0,1], \R^l)\oplus \R^k$ and {{${\rm Cu}^\sim(\td C)$}} is identified with
a subgroup of
\beq
K_0({{\tilde C}})\sqcup LSC([0,1], \Z^{{ l}})\oplus \Z^{{k+1}}
\eneq
(see  3.6 of \cite{GLII}).
Since $K_0(\td C)_+$ is finitely generated (see 3.15 of  \cite{GLN}) by  $g_1,g_2,...,g_s,$
there exists an integer $N_0\ge 1$  such that
\beq
N_0\eta(x)>1_\DT\rforal x\in {{K_0(\td C)_+\setminus\{0\}}}.
\eneq
Let $\sigma>0$ be given by Lemma \ref{Lstrictpositive} for $2N_0.$

Recall that the map $\rho: K_0({{\tilde C}})\to \Aff(T({{\tilde C}}))$ is injective (see 3.5 of \cite{GLN} for example).
By Lemma 5.1 of \cite{Tsang}, we may write $\Aff(\DT)=\lim_{n\to\infty}(\R^{a(n)}, \lambda_n),$
where $a(n)\ge 1$ are integers and $\lambda_n: \R^{a(n)}\to \R^{a(n+1)}$ is an order preserving  map
which also {{preserves}} the canonical order unit.
Let $G_2:=\eta(K_0(C)).$ Note that $G_2$ is a finitely generated subgroup of $G_1.$
Therefore there is a sequence of subgroups $G_{n,2}\subset \R^{a(n)}$
such that
$G_2=\lim_{n\to\infty} (G_{n,2}, \lambda_n)$ and ${\lambda_n}|_{{G_{n,2}}}$ is injective.
Let $G_{n,2}^\sim=G_{n,2}+\Z\cdot 1$ and $G_2^\sim:=G_2+\Z \cdot 1_{\DT}.$

Since $K_0({\tilde C})_+$ is finitely generated, one obtains  (for all large $n$) an
order preserving map $\eta_n: K_0({\tilde C})\to \R^{a(n)}$ such
that $\lambda_{n, \infty}\circ \eta_n=\eta,$ where
$\lambda_{n, \infty}$ is induced by the inductive limit system.
There exists $n_1\ge 1$ such that, for all $n\ge n_1,$
\beq
2N_0\eta_n(x)>1\rforal x\in K_0(\td C)_+\setminus \{0\}.
\eneq
Write $\R^{a(n)}=\bigoplus_{i=1}^{a(n)}\R_i.$
Define $q_i: \R^{a(n)}\to \R_i$ {{to}} be the projection.
Consider the order preserving map $q_i\circ \eta_n: K_0({\tilde C})\to \R$ which
{{preserves}} the order unit.
Since $\rho$ is injective, we may view $K_0({\tilde C})$ as {{an order}} subgroup of $C([0,1], \R^l)\oplus \R^{k+1}.$
By Lemma \ref{Lstrictpositive},
there is an order preserving  {{map}} $\gamma_{n,i}: C([0,1], \R^l)\oplus {{\R^{k+1}}}\to \R_i$
such that ${\gamma_{n,i}}|_{K_0({\tilde C})}=q_i\circ \eta_n$ and
\beq\label{LHB-nn1}
\gamma_{n,i}(\hat{h})\ge (\sigma/2)\int_{[0,1]}T(\hat{h})(t) d\mu,
\eneq
where $h\in {\tilde C}_+\setminus \{0\},$ {{which}} we identify  {{with}} the corresponding element in $C([0,1], F_2),$
and $\hat{h}$ is the associated affine function
(for all large $n\ge n_1$).

Define $\gamma_n: C([0,1], \R^l)\oplus \R^{{k+1}}\to
\R^{a(n)}$ by
$\gamma_n(g)={{(\gamma_{n,1}(g), \gamma_{n,2}(g),..., \gamma_{n,a(n)}(g))}}$ for all   {{$g$.}} 
Note {{that}} ${\gamma_n}|_{K_0(\td C)}=\eta_n$ for all $n.$
Define
$\gamma:C([0,1], \R^l)\oplus \R^{k+1} \to \Aff(\DT)$
by $\gamma(g):=\lambda_{n, \infty}\circ \gamma_n.$
 Then $\gamma$ is {{an}} order preserving map
which preserves the order unit. It is linear and therefore continuous (with {{the supremum}} norm).
The condition $\lambda_{n, \infty}\circ \eta_n=\eta$ implies
that $\gamma|_{K_0(\td C)}=\eta.$
Note {{also that}} the condition \eqref{LHB-nn1} implies  that, for each $f\in (C([0,1], \R^{{l}})\oplus \R^{{k+1}})_+\setminus \{0\},$
\beq
\gamma(f)(t)>0\rforal t\in \DT.
\eneq
Then $\gamma^\sim$ extends to an order preserving affine map
from $LSC([0,1], \Z^{{l}})\oplus \Z^{{k+1}}$ to ${\rm LAff}^\sim_+(\DT).$
Define $\eta^\sim:{{ {\rm Cu}^\sim(\td C)}} \to G_1\sqcup {\rm LAff}_+^\sim (\DT)$ by
$\eta^\sim|_{K_0(\td C)}=\eta$ and {{$\eta^\sim|_{LSC([0,1], \Z^{{ l}})\oplus \Z^{{k+1}}}=
\gamma^\sim|_{LSC([0,1], \Z^{{ l}})\oplus \Z^{{k+1}}}.$}}
It is then straightforward to verify that   {{$\eta^\sim$}}
 is a map in ${\bf Cu}.$

\end{proof}

\begin{thm}\label{TExtCtoD}
{{Let $A\in {\cal D}$  {{be}} with continuous scale and $C\in {\cal C}_0$.}}
Suppose that there is a strictly positive \hm\, $\af: K_0({\tilde C})\to K_0({\tilde A})$ such
that $\af([1_{\tilde C}])=[1_{\tilde A}]$ and $\af(K_0(C))\subset K_0(A).$
Then there exists a \hm\, $h: C\to A$ such that $h_{*0}=\af.$
%
\end{thm}

\begin{proof}
{{Since $A$ is {{a}} stably projectionless ${\cal Z}$-stable simple {{\CA\,}} with stable rank one {{(all {{\CA s}} in ${\cal D}$ {{have}} stable rank one, see 11.11 of \cite{eglnp})}},
by  6.2.3 of \cite{Rl} (see also  7.3  of \cite{eglnp} and 6.11 of  \cite{RS}),
\beq
{\rm Cu}^\sim(A)= K_0(A)\sqcup {\rm LAff}_+^\sim (T(A)).
\eneq
By Lemma \ref{LHahn-B-1}, $\af$ extends to a morphism  $\af^\sim: {\rm Cu}^\sim ({{\tilde C}})\to {\rm Cu}^\sim({{\tilde A}})$
in ${\bf Cu}.$   Let $e_C\in C$ be a strictly positive element with $\|e_C\|=1.$ Then
$\af^\sim(\la e_C\ra) \le [1_{\tilde C}].$ Let $e_A\in A$ be a strictly positive element. Since $A$ has
continuous scale, $d_\tau(e_A)=1$ for all $\tau\in T(A).$
Therefore, by 7.3 of \cite{eglnp} (see also 6.10 of \cite{RS}),
$\af^\sim(\la e_C\ra)\le \la e_A\ra$ in ${\rm Cu}^\sim(A).$
Since $A$ has stable rank one, by
Theorem 1.0.1 of \cite{Rl},  {{there}} is
a \hm\, $h: C\to A$ such that ${\rm Cu}^\sim(h)=\af^\sim{{|_{{\rm Cu}^{\sim}(C)}}}.$ In particular, $h_{*0}=\af.$}}
\end{proof}
%

\section{Reduction}

This section is a nonunital version of the corresponding results in {{\cite{EGLN}}}.
Most of the results are taken from \cite{EGLN} with some modification.

\begin{lem}[Lemma 3.1 of \cite{EGLN}]\label{0.5}
Let $A$ be a non-unital simple  separable amenable quasidiagonal C*-algebra satisfying the UCT. Assume that {{$A\cong A\otimes Q$.}}

Let a finite subset $\mathcal G$ of ${\tilde{A}}\otimes Q$ and $\ep_1, \ep_2>0$ be given. Let $p_1, p_2, ..., p_s\in  M_m(\tilde{A}\otimes Q)$
(for some integer $m\ge 1$)
be projections such that
$[1], [p_1], [p_2], ..., [p_s]\in K_0(\tilde{A}\otimes Q)$ are $\Q$-linearly independent. (Recall that $K_0({\tilde{A}}\otimes Q)\cong K_0(({\tilde A}\otimes Q)\otimes Q)\cong K_0(\tilde{A}\otimes Q)\otimes \Q$.) There {{are}}
a $\mathcal G$-$\ep_1$-multiplicative completely positive {{linear}} map $\sigma: \tilde{A}\otimes Q\to Q$ with $\sigma(1)$ a projection satisfying
$$\mathrm{tr}(\sigma(1))<\ep_2$$
(where $\mathrm{tr}$ denotes the unique {{tracial}} state on $Q$), and $\delta>0$,
such that, for any $r_1, r_2, ..., r_s\in\Q$ with
$$|{r_i}|<\dt,\  i=1, 2,..., s,$$
there is a $\mathcal G$-$\ep_1$-multiplicative completely positive  {{linear}} map $\mu: \tilde{A}\otimes Q\to Q $, with $\mu(1)=\sigma(1)$, such that
$$[\sigma(p_i)]-[\mu(p_i)]=r_i,\quad i=1,2, ...,s.$$

\end{lem}

\begin{proof}
Let us first consider the case $m=1.$
The proof  in this case is exactly the same as that of Lemma 3.1 of \cite{EGLN}. The only place in that proof
{{mentioning}} simplicity is the lines shortly after equation (3.1), where one claims
that the algebra is the closure of an increasing sequence of  unital amenable RFD \CA s.
We will replace this part as follows: Since $A$ is simple,  separable, amenable and quasidiagonal,
by Corollary 5.5 of \cite{BK2}, $A$ is a strong NF-algebra.  Since $Q$ is a strong NF-algebra,
by  {{Corollary}} 7.1.6 of \cite{BK1}, ${\tilde A}\otimes Q$   is a strong NF-algebra.  It follows from {{Corollary}}
6.16 of \cite{BK1} that, indeed, ${\tilde A}\otimes Q$ is the closure of an increasing sequence of  unital amenable RFD \CA s.
The rest of the proof  of this case remains   the same.

If $m>1,$ one {{notes}} that there are $p_i'\in {\tilde A}\otimes Q$ such that $[p_i']=(1/m)[p_i],$ $i=1,2,...,s.$
Then, $[1], [p_1'],...,[p_s']$ are $\Q$-linearly independent.  {{Replace}} $r_i$ by $r_i/m,$ $i=1,2,...,s.$
If $\sigma([p_i'])-\mu{{([p_i'])}}=r_i/m,$ $i=1,2,...,s,$ then
$\sigma{{([p_i])}}-\mu([p_i])=r_i,$ $i=1,2,...,s.$

\end{proof}

\begin{cor}[Lemma 3.1 of \cite{EGLN}]\label{0.5c}
Let $A$ be a non-unital simple  separable amenable quasidiagonal C*-algebra satisfying the UCT. Assume that $A\cong A\otimes Q.$

Let a finite subset $\mathcal G$ of ${\tilde{A}}$
and $\ep_1, \ep_2>0$ be given. Let $p_1, p_2, ..., p_s\in M_m(\tilde{A})$
(for some integer $m\ge 1$)
be projections such that
$[1], [p_1], [p_2], ..., [p_s]\in K_0(\tilde{A})$ are $\Q$-linearly independent.
There {{are}}  a $\mathcal G$-$\ep_1$-multiplicative completely positive {{linear}} map $\sigma: \tilde{A}\to Q$ with $\sigma(1)$ a projection satisfying
$$\mathrm{tr}(\sigma(1))<\ep_2$$
(where $\mathrm{tr}$ denotes the unique {{tracial}} state on $Q$), and $\delta>0$,
such that, for any $r_1, r_2, ..., r_s\in\Q$ with
$$|{r_i}|<\dt,\  i=1, 2,..., s,$$
there is a $\mathcal G$-$\ep_1$-multiplicative completely positive {{linear}} map $\mu: \tilde{A}\to Q $, with $\mu(1)=\sigma(1)$, such that
$$[\sigma(p_i)]-[\mu(p_i)]=r_i,\quad i=1,2, ...,s.$$

\end{cor}

\begin{proof}
Let $B=\tilde{A}\otimes Q.$  One has the following split short exact sequence
$$
0\to  A\otimes Q\to {\tilde A}\otimes Q\to Q\to 0.
$$
This gives the split short exact sequence
$$
0\to K_0(A\otimes Q)\to K_0({\tilde A}\otimes Q)\to \Q\to 0.
$$
Since $A\cong A\otimes Q,$ $K_0(A)=K_0(A\otimes Q)$ and $K_0(\tilde A)=K_0(A\otimes Q)\oplus \Z$ is a subgroup
of $K_0(\tilde{A}\otimes Q).$
 {{Apply}} Lemma \ref{0.5} to $B,$ {{and}} then choose $\sigma|_{\tilde{A}}$ and
$\mu|_{\tilde{A}}.$

\end{proof}

\begin{rem}\label{Rred-r1}
Let $A\cong A\otimes Q$ be a separable stably projectionless  simple \CA.
Suppose that $K_0(A)\not=\{0\}$.
Then there exists {{$x=[p]-k[1_{\tilde A}]\in K_0(A)\setminus \{0\}$}}, where $k\in \N$, and $p\in M_{n}({\tilde A})$. If $[p]=r[1_{\tilde A}]$ for some rational number $r\in \Q,$
then $x=(r-k)[1_{\tilde A}].$ Since $x\not=0,$ $r\not=k.$ But, then either $x=(r-k)[1_{\tilde A}]$ or $-x:=(k-r)[1_{\tilde A}]$ is a non-zero
positive element in $K_0(A).$ This contradicts the fact that $K_0(A)_+=\{0\}.$ In other words, $[p]$ and {$[1_{\tilde A}]$} are $\Q$-linearly independent.  {{Put}} $p_1:=p.$
Choose $r_1\not=0$ and let $[\sigma(p_1)]-[\mu(p_1)]=r_1.$ Then $[\mu](x)\not=[\sigma](x).$
In other words, at least one of the maps $\mu|_A$ and $\sigma|_A$ is not zero.

\end{rem}

\begin{lem}[Lemma 3.3 of \cite{EGLN}]\label{0.6}
Let $A$ be a non-unial  simple  separable  amenable quasidiagonal C*-algebra satisfying the UCT.
Assume that $A\cong A\otimes Q.$

Let ${\mathcal G}$ be a finite subset of $A$, let $\ep_1, \ep_2>0,$ and let $p_1, p_2,..., { {p_s}}\in  M_m({\tilde{A}})$
(for some integer $m$)
be projections such that $[1_A], [p_1], [p_2],...,[p_s]\in K_0(\tilde{A})$ are $\Q$-linearly independent.
There exists  $\dt>0$ satisfying the following condition.

Let $\psi_k: Q^l \to Q^r,$ $k=0,1,$ be unital homomorphisms, where  $l, r\in \{1, 2, ...\}.$
Set
$$\D=\{x\in \Q^l: (\psi_0)_{*0}(x)=(\psi_1)_{*0}(x)\}\subset \Q^l.$$
{There exists
a}  ${\mathcal G}$-$\ep_1$-multiplicative completely positive {{linear}} map $\Sigma: \tilde{A}\to Q^l$, such that $\Sigma(1_{\tilde A})$ is a projection, with the following properties:
$$\tau(\Sigma(1_{\tilde A}))<\ep_2,\quad \tau \in  T(Q^l),$$
$$ [\Sigma(1_{\tilde{A}})],\ [\Sigma(p_j)]\in  \D, \quad j=1,2,...,s,$$
 and, for any $r_1, r_2,...,r_s\in \Q^r$ satisfying
$$
|r_{i,j}|<\dt,
$$
where $r_i=(r_{i,1},r_{i,2},...,r_{i,r}),$ $i=1,2,...,s,$
there is a ${\mathcal G}$-$\ep_1$-multiplicative completely positive {{linear}} map $\mu: \tilde{A}\to Q^r$, with $\mu(1_{\tilde A})$ a projection, such that
$$ [\psi_0\circ \Sigma(p_i)]-[\mu(p_i)]
=r_i, \quad i=1,2,...,s,\tand
[\mu(1_{\tilde A})]=[\psi_0\circ \Sigma(1_{\tilde A})].
$$
\end{lem}

\begin{proof}
The proof is {{exactly}} the same as that {{of}} Lemma 3.3 of \cite{EGLN}  but using Lemma \ref{0.5} (and {{Corollary  \ref{0.5c}}}) above instead of
Lemma 3.1 in \cite{EGLN}.

\end{proof}

\begin{lem}[Lemma 3.4 of \cite{EGLN}]\label{0.7}

Let $A$ be a non-unital simple separable  amenable quasidiagonal C*-algebra satisfying the UCT.
Assume that $A\cong A\otimes Q.$

Let ${\mathcal G}\subset A$ be a finite subset, let $\ep_1, \ep_2>0$ and let $p_1, p_2,..., p_s \in M_m(\tilde{A})$
be
projections such that $[1_{\tilde A}], [p_1], [p_2],...,[p_s]\in K_0({\tilde A})$ are $\Q$-linearly independent.
Then there exists $\dt>0$ satisfying the following condition.

 Let $\psi_k: Q^l\to Q^r$, $k=0, 1$, be  unital homomorphisms, where $l, r\in\{1, 2, ...\}.$
Set $$\D=\{x\in \Q^l: (\psi_0)_{*0}(x)=(\psi_1)_{*0}(x)\}\subset \Q^l.$$
{There exists
a}
${\mathcal G}$-$\ep_1$-multiplicative completely positive {{linear}}  map $\Sigma: \tilde{A}\to Q^l$, such that $\Sigma(1_{\tilde A})$ is a projection, with the following properties:
$$\tau(\Sigma(1_{\tilde A}))<\ep_2,\quad \tau \in  T(Q^l),$$
$$[\Sigma(1_{\tilde A})],\  [\Sigma(p_i)]\in  \D,\quad  i=1,2,...,s,$$
and, for any $r_1, r_2,...,r_s\in \Q^l$ satisfying
$$
|r_{i,j}|<\dt,
$$
where $r_i=(r_{i,1},r_{i,2},...,r_{i,l}),$ $i=1,2,...,s,$
there is a ${\mathcal G}$-$\ep_1$-multiplicative completely positive {{linear}} map $\mu: {\tilde A}\to  Q^l$, with $\mu(1_{\tilde A})=\Sigma(1_{\tilde A})$, such that
\beq\nonumber
[\Sigma(p_i)]-[\mu(p_i)]=r_i,\quad {{i}}=1,2,...,s.
\eneq

\end{lem}

\begin{proof}
The proof is  exactly the same as that of Lemma 3.4 of \cite{EGLN} {{also}} using Lemma \ref{0.5} (and Corollary  \ref{0.5c})
instead of Lemma 3.1 of \cite{EGLN}.

\end{proof}




\begin{lem}\label{T06}
Let $A$ be a non-unital simple separable  amenable 
\CA\,
with $ T(A)= T_{\mathrm{qd}}(A)\not=\emptyset$
which satisfies the UCT. Assume that $A\otimes Q\cong A$ and $A$ {{have}} continuous scale.

For any $\sigma>0$, $\ep>0,$ and any finite subset
${\mathcal F}$ of $A$,
there exist a finite set of
${{{\mathcal P}\subset K_0(A)}}$
and $\dt>0$ with the following property.

Denote by $G\subseteq K_0(A)$ the subgroup
generated by $\mathcal P$.
Let $\kappa: G\to K_0(C)$ be a
 homomorphism which extends to
 a positive \hm\, $\kappa^\sim : G+\Z\cdot [1_{\td A}]\to {{K_0( C)}}$
 such that $\kappa^\sim([1_{\td A}])=[1_C],$
where  {{$C=C([0,1], Q)$}}, and let $\lambda:  T(C)\to  T(A)$ be a continuous affine map
such that
\begin{equation}\label{07-1-1}
|\tau(\kappa(x)) -{ {\rho_A}}(\lambda(\tau))(x)|<\dt,\quad  x\in {\mathcal P},\ \tau\in T(C).
\end{equation}
Then
there {{is an}} ${\mathcal F}$-$\ep$-multiplicative completely positive {{linear}} map  $L: A\to C$ such that
\begin{equation}\label{07-1}
| \tau\circ L(a)-\lambda(\tau)(a)|<\sigma,\quad a\in\mathcal F,\ \tau\in  T(C).
\end{equation}
\end{lem}

\begin{proof}
Let $\ep,$ $\sigma$ and ${\mathcal F}$ be given. We may assume that every element of $\mathcal F$ has norm at most one.

Fix a strictly positive element  $e\in A$ with $\|e\|=1.$
Choose
$$
0<d<\inf\{\tau(f_{1/2}(e)):\tau\in T(A)\}.
$$
This is possible since $T(A)$ is compact.

Let $\delta_1$ (in place of $\delta$), $\mathcal G$, and $\mathcal P_1$
(in place of ${\cal P}$) be as assured by Lemma 7.2 of \cite{eglnkk0}
for $\mathcal F$
and $\ep,$  as well as $d.$
We may assume
that ${\cal F}
\subset {\cal G}$ and $f_{1/2}(e), f_{1/4}(e)\in {\cal G}.$

We may assume
${\cal P}_1=\{x_1,x_2,...,{ {x_{s'}}}\},$ where
$x_i=[q_i]-[{\bar q}_i],$
where $q_i\in M_m({\tilde{A}})$ is a projection and ${\bar q}_i\in M_m(\C \cdot 1_{\td A})$ is a scalar matrix.
Choose ${\cal P}^\sim =\{1_{\tilde A}, p_1, p_2,...,p_s\}$ such that $x_i$ is in the subgroup  $G_0$ generated by
$\{[1_{\tilde A}], [p_1], [p_2],...,[p_s]\}$ (in $K_0({\tilde A})$).
 Deleting one or more of $p_1, p_2, ..., p_s$ (but not $1_{\tilde A}$), we may assume that the set
$\{[1_A], [p_1],...,[p_s]\}$ is $\Q$-linearly independent.
Define ${\cal P}=\{[p_i]-[{\bar p}_i]: 1\le i\le s\},$
where ${\bar p}_i\in M_m(\C \cdot 1_{\td A})$ is a scalar matrix such that
$\pi_{\C}^A(p_i)={\bar p}_i$ ($1\le i\le s$), where $\pi_{\C}^A: \td A\to \C$ is the quotient map.

Note {{that}} $p_i={\bar p}_i+y_i,$ where ${\bar p_i}\in M_m(\C\cdot 1_{\tilde A})$ is a scalar projection and $y_i\in A,$
$i=1,2,...,s.$   Let $c_i$ be the rank of ${\bar p}_i,$ $i=1,2,...,s.$
We may assume that $G_0\cap K_0(A)$ is generated by ${\cal P}$ (by enlarging ${\cal P}$ if necessary).
Then $G_0$ is the subgroup of $K_0(\tilde A)$ generated by $[1_{\td A}]$ and $G.$
Therefore we may assume {{that}}
{{$[{\bar p}_i]=c_i[1_{\td A}]$}}
for some integer {{$1\le c_i\le m,$ }}  $i=1,2,...,s.$
Choose 
{{$M=1+m$}}  and
choose $0<\sigma_1<\min\{\sigma, d/32\}$ such that ${1-\sigma_1\over{1+\sigma_1}}>{127\over{128}}.$

Let $0<\dt_2<1$ (in place of $\dt$) be as assured by Lemma \ref{0.5} for
$\ep_1=\dt_1,$ $\ep_2=\sigma_1/4,$ ${\mathcal G}$, and $\{p_1,p_2,...,p_s\}.$
Write $y_i=(y^{(i)}_{j,k})_{m\times m},$ where $y^{(i)}_{j,k}\in A.$
Choose ${\cal G}_1=\{y^{(i)}_{j,k}: j,k, i\}$ and ${\cal G}_2={\cal G}\cup {\cal G}_1.$

Put $\dt_3=\min\{\dt_1/M, {\dt_2/32M}, d/128 \}.$
We choose $0<\dt<\dt_3$ and a finite subset ${\cal G}_3\supset {\cal G}_2$
such that, any ${\cal G}_3$-$\dt$-multiplicative \morp\, $L: A\to B$
(any unital \CA\, $B$),  $[L(p_i)]$ is well defined and
\beq\label{August 21-2020}
\|[L(p_i)]-L(p_i)\|<\dt_3,\,\,\,i=1,2,...,s.
\eneq
Let us show that $\mathcal P$ and $\delta$ are as desired.

Let $\kappa$ and $\lambda$ be given satisfying \eqref{07-1-1}. We will write
$\kappa$ for $\kappa^\sim$ for convenience.
{{{{Then}} recall  $\kappa([1_{\td A}])=[1_C].$}}
Note  that
\beq\label{829number}
\lambda(\tau)({\bar p}_i)=c_i=\tau(\kappa({\bar p}_i)),\,\,\, i=1,2,....
\eneq
{{So, as $\kappa$ is positive, we may identify $\kappa([p_i])$ with a
projection in $M_m(C)$ as $M_m(C)$ has stable rank one.}}
{{Hence, by \eqref{07-1-1} and by \eqref{829number}, for all $\tau\in { {T(C)}},$}}
\beq\label{2020-829-nl1}
{{|\tau(\kappa([p_i])) -{ {\rho_A}}(\lambda(\tau))([p_i])|<\dt,\,\,\,i=1,2,...,s.}}
\eneq
Let $\lambda_*: \Aff( T(A))\to \Aff( T(C))$ be defined by
$\lambda_*(f)(\tau)=f(\lambda(\tau))$ for all $f\in \Aff( T(A))$ and $\tau\in  T(C).$
Identify $\partial_e{( T(C))}$ with $[0,1]$, and put $\eta=\min\{\dt, \sigma_1/12\}.$
 Choose a partition
$$
0=t_0<t_1<t_2<\cdots <t_{n-1}<t_n=1
$$
of the interval $[0, 1]$ such that 
\begin{equation}\label{T06-2}
|\lambda_*(\hat{g})(t_j) - \lambda_*(\hat{g})(t_{j-1})| < \eta/m^2,\quad g\in {\mathcal G_3},\  j=1,2,...,n.
\end{equation}

Since $ T(A)= T_{{\rm qd}}(A),$  there are unital ${\mathcal G}_3$-$\dt$-multiplicative completely positive {{linear}} maps
$\Psi_j:  A\to Q$, $j=0, 1, 2,  ..., n$, such that
\begin{equation}\label{T06-3}
|\mathrm{tr}\circ \Psi_j(g)-\lambda_*(\hat{g})(t_j)|<\eta/m^2,\quad g\in {\mathcal G}_3.
\end{equation}
Denote by $\Psi_j^\sim: {\tilde A}\to Q$ the unitization (which maps $1_{\tilde A}$ to $1_Q$).
Recall that $[p_i]=m_i[1_{\tilde A}]+x_i,$ $i=1,2,...,s.$
It then follows from {{\eqref{August 21-2020},}}
\eqref{T06-3}, and {{\eqref{2020-829-nl1},}}
that, for each $i=1,2,...,s,$ and each $j=1, 2, ..., n$,
\beq\label{T06-4}
 |{\rm tr}([\Psi_j^\sim(p_i)])-{\rm tr}([\Psi_0^\sim(p_i)])| &<& |{\rm tr}(\Psi_j^\sim(p_i))-{\rm tr}(\Psi_0^\sim(p_i))|+{{2\dt_3}}\\
  &<& {{2 \dt_3}}+
  2\eta +|\lambda_*(\hat{p_i})(t_j) - \lambda_*(\hat{p_i})(t_{0})| \nonumber\\
&<&{{2\dt_3}}+
2\eta+2\dt
+|{\rm tr}\circ \pi_{t_j}(\kappa([p_i]))-{\rm tr}\circ\pi_0(\kappa([p_i])| \nonumber \\
&=&{{2 \dt_3}}+2\eta+2\dt
\leq  {{6\dt_3}}
\leq \delta_2.
\eneq
(Here, as before, $\pi_t$ is the point evaluation at $t\in [0,1]$.)
We also have, by (\ref{T06-2}) and (\ref{T06-3}), that
\beq\label{T06-4+n}
|\mathrm{tr}(\Psi_j(g))-\mathrm{tr}(\Psi_{j+1}(g))|<3\eta,\quad  g\in {\cal G}_3,\ j=1, 2, ..., n.
\eneq
Recall that $\lambda(\tau)(\widehat{f_{1/2}(e)})\ge d$ for all $\tau\in T(C).$
So we  also have, by \eqref{T06-2},
\beq\label{T06-4+nn}
{\rm tr}(\Psi_j(f_{1/2}(e)))>{31\over{32}}d.
\eneq
Consider the differences
\begin{equation}\label{1228-1}
r_{i,j}:={\rm tr}([\Psi_j^\sim(p_i)])-{\rm tr}([\Psi_0^\sim(p_i)]),\quad i=1,2,...,s,\ j=1,2, ...,n.
\end{equation}
Recall that $[\Psi_j^\sim(p_i)]$ and $[\Psi_0^\sim(p_i)]$ are in $K_0(Q).$
By \eqref{T06-4}, $|r_{i, j}|<\delta_2$. Applying Lemma \ref{0.5} we obtain a projection $e\in Q$ with $\mathrm{tr}(e)<\sigma_1/4$ and ${\mathcal G}$-$\dt_1$-multiplicative  unital completely positive {{linear}} maps
$\psi_0^\sim, \psi_j^\sim: {\tilde A}\to eQe$, $j=1,2,...,n,$ such that
\beq\label{T06-5-NN}
[\psi_0^\sim(p_i)]-[\psi_j^\sim(p_i)]=r_{i,j},\quad i=1,2,...,s,\  j=1,2,...,n.
\eneq

Consider the direct sum maps
$$\Phi_j'^\sim:=\psi_j^\sim\oplus \Psi_j^\sim: {{\tilde{A}}}\to (1\oplus e)\mathrm{M}_2(Q)(1\oplus e),\quad j=0, 1,2,...,n.$$
Since $\delta\leq \delta_1$, these {maps} are $\mathcal G$-$\delta_1$-multiplicative.  By  (\ref{1228-1}) and (\ref{T06-5-NN}),
\beq\label{T06-7}
[\Phi_j'^\sim(p_i)]=[\Phi_0'^\sim(p_i)],\quad i=1,2,...,s,\  j=1,2,...,n.
\eneq
Define $s: \Q\to \Q$ by $s(x)=x/(1+\mathrm{tr}(e))$, $x\in \Q.$
Choose a (unital) isomorphism $$S: (1\oplus e)\mathrm{M}_2(Q)(1\oplus e)\to Q$$ such
that $S_{*0}=s.$

Consider the composed maps, still $\mathcal G$-$\delta_1$-multiplicative, and now unital,
$$\Phi_j:=S\circ \Phi_j': A\to Q, \quad j=0,1, 2, ..., n.$$
By (\ref{T06-7}),
$$[\Phi_j]|_{\mathcal P}=[\Phi_{j-1}]|_{\mathcal P},\quad j=1, 2, ..., n,$$
and by (\ref{T06-4+n}) and the fact that $\mathrm{tr}(e)<\sigma_1/4,$
\begin{equation}\label{sm-tr-ed-pts}
|\mathrm{tr}\circ\Phi_j(a)-\mathrm{tr}\circ\Phi_{j-1}(a)|<3\eta+\sigma_1/4\leq \sigma/2,\quad a\in {\mathcal F},\ j=1, 2, ...,n.
\end{equation}
Moreover,  by \eqref{T06-4+nn} and by the choice of $\sigma_1,$
\beq
{\rm tr} (\Psi_j(f_{1/2}(e))\ge d/2,\,\,\, j=1,2,...,n.
\eneq
{{Applying Lemma 7.2 of \cite{eglnkk0}}}
successively for $j=1, 2, ..., n$ (to the pairs $(\Phi_0, \Phi_1)$, $(\mathrm{Ad}\, u_1\circ\Phi_1, \mathrm{Ad}\, u_1\circ\Phi_2)$, ..., $(\mathrm{Ad}\, u_{n-1}\circ\cdots\circ\mathrm{Ad}\, u_1\circ\Phi_{n-1}, \mathrm{Ad}\, u_{n-1}\circ\cdots\circ\mathrm{Ad}\, u_1\circ\Phi_{n})$),  {{one obtains,}}
for each $j$, a unitary $u_j\in Q$ and a unital
${\mathcal F}$-$\ep$-multiplicative completely positive {{linear}} map $L_j: A\to {\rm C}([t_{j-1}, t_{j}], Q)$ such that
\begin{equation}\label{prop-1}
\pi_0\circ L_1=\Phi_0, \quad \pi_{t_1}\circ L_1={\rm Ad}\, u_1\circ \Phi_1,
\end{equation}
and
\begin{equation}\label{prop-2}
\pi_{t_{j-1}}\circ L_j=\pi_{t_{j-1}}\circ L_{j-1},\quad \pi_{t_{j}}\circ L_j={\rm Ad}\, u_j\circ\cdots\circ\mathrm{Ad}\, u_1\circ \Phi_j, \quad j=2, 3, ..., n.
\end{equation}
Furthermore, in view of \eqref{sm-tr-ed-pts}, we may choose the maps $L_j$ such that
\begin{equation}\label{prop-tr}
|{\rm tr}\circ \pi_t\circ L_j(a)-\lambda({\rm tr}\circ \pi_t)(a)|<\sigma,\quad  t\in [t_{j-1}, t_{j}],\ a\in\mathcal F,\ j=1, 2, ..., n.
\end{equation}
%
Define $L: A\to {\rm C}([0,1], Q)$ by
$\pi_t \circ L= \pi_t \circ L_j,\quad t\in [t_{j-1}, t_{j}],\ j=1,2,...,n.$
Since $L_j$, $j=1, 2, ..., n$, are $\mathcal F$-$\ep$-multiplicative (use \eqref{prop-1} and \eqref{prop-2}), we have that $L$ is a
$\mathcal F$-$\ep$-multiplicative completely positive {{linear}} map $A \to \mathrm{C}([0, 1], Q)$.
It follows from \eqref{prop-tr} that $L$ satisfies \eqref{07-1}, as desired.
\end{proof}

\begin{NN}\label{R14}
Let $A$ be a separable stably projectionless simple \CA\, with continuous scale.
{{Recall that $\tau_\C^A$ is
the tracial {{state}} of $\td A$
which vanishes
on $A$ and
$T(\tilde A)=\{ s\cdot  t_\C^A+(1-s)\cdot \tau: \tau\in T(A),\, \, 0\le s\le 1\}$ (see \ref{range4.1}).}}
For each projection $p\in M_m(\tilde A),$ one may write
$p={\bar p}+a,$ where ${\bar p}$ is a scalar matrix in $M_m(\C\cdot 1_{\tilde A})$ and $a\in M_m(A)_{s.a.}.$
Let ${\bar p}$ have rank $k(p).$ { {For}} each $\tau\in T(A),$  {{define}} $r_{\tilde A}(\tau)([p])=k(p)+\tau(a)$ and $r_{\tilde A}(\tau_\C^A)([p])=k(p).$ {{This gives a map $r_{\tilde A}: T({\tilde A}) \to {\rm Hom}(K_0({\tilde A}), \R)$. Let $r_{ A}: T({ A}) \to {\rm Hom}(K_0({ A}), \R)$ be defined by $r_{A}(\tau)=r_{\tilde A}(\tau)|_{K_0(A)}$ for any $\tau\in T(A)$ and $r^{\sim}_{ A}: T({ A}) \to {\rm Hom}(K_0({\tilde A}), \R)$ be defined by $r^{\sim}_{A}=r_{\tilde A}|_{T(A)}$.}}



Suppose that $C$ {{is}} another  separable stably projectionless simple
\CA\, with continuous scale and
suppose that there is an isomorphism
\beq
\Gamma: (K_0(A),  T(A), {{\rho_A}})\cong (K_0(C), T(C), \rho_C).
\eneq
Recall that {{this}} means $\Gamma|_{K_0(A)}$ is a group
isomorphism, $\Gamma|_{T(A)}$ is an affine homeomorphism,  and ${{\rho_A}}(\Gamma^{-1}(\tau))(x)=\rho_C(\tau)(\Gamma(x))$
for $x\in K_0(A)$ and $\tau\in T(C).$

Then $\Gamma$ extends to an order isomorphism
$$\Gamma^\sim:
(K_0(\tilde A), K_0({\tilde A})_+, [1_{\tilde A}],  T(\tilde A), r_{\tilde A})\to (K_0({\tilde C}), K_0({\tilde C})_+,
[1_{\tilde C}], T({\tilde C}), r_{{\tilde C}})$$
by defining $\Gamma^\sim([1_{\tilde A}])=[1_{\tilde C}]$ and $\Gamma^\sim(\tau_\C^A)=\tau_\C^C.$
To see this,  we note  that $\Gamma^\sim |_{K_0(\tilde A)}$ is an isomorphism and $\Gamma^\sim _{T(\tilde A)}$
is an affine homeomorphism.
{{If}} $y:=m\cdot [1_{\tilde A}]+x\ge 0$ for some positive integer  $m$ and $x\in K_0(A),$
then there is a projection $p\in M_K(A)$ for some integer $K\ge 1$ such that $[p]=y.$
Assume that $y\not=0.$ Then $p\not=0.$
Choose $a\in { {(pAp)}}_+^{\bf 1}\setminus \{0\}.$   Then $a\le p.$ It follows that $\tau(a)>0$ for all $\tau\in T(A).$
Therefore $\tau(p)>0$ for all $\tau\in T(A).$    This also means
that $m+r_A(\tau)(x)>0$ for all $\tau\in T(A).$

On the other hand, $\pi_\C^A(p)\not=0,$
where $\pi_\C^A: {\tilde A}\to \C$ is the quotient map. It follows that $\tau_\C^A(p)>0.$
This implies that $t(p)>0$ for all $t\in T(\tilde A).$  One checks that, for $\tau\in T(C),$
\beq
r_A^\sim ((\Gamma^\sim)^{-1}(\tau))(y)={ {r_A^\sim}}(\Gamma^{-1}(\tau)(m+x))
=m+r_A(\Gamma^{-1}(\tau))(x)> 0.
\eneq
Also
\beq
r_A^\sim((\Gamma^\sim)^{-1}{{(\tau_\C^C))}}(y)=r_A(\tau_\C^A)(y)=m>0.
\eneq
This implies that
\beq
r_C^\sim(t)(\Gamma^\sim(y))=r_A^\sim((\Gamma^\sim)^{-1})(y)>0.
\eneq
Therefore $\Gamma^\sim$ is an order isomorphism.
\end{NN}

\begin{lem}\label{Lcontinuous}
Let $A$ be a non-unital  but $\sigma$-unital  simple \CA\, with strict comparison for positive {{elements}}
which has almost stable rank one.
Suppose that ${\rm{QT}}(A)={\rm{T}}(A),$ $A={\mathrm{Ped}(A)}$ and
{{the canonical map}} $\imath: W(A)\to {\rm LAff}_{b,+}({\overline{{ T(A)}}^\mathrm{w}})$
is surjective.
Fix a strictly positive element $a\in A.$
Then $A$ has an approximate identity $\{e_n\}$ such that
$\overline{e_nAe_n}$ has continuous scale of each $n.$    Moreover,
$e_nae_n\sim e_n$ for all $n.$
\end{lem}

\begin{proof}
Let $a\in A_+$ with $\|a\|=1$ be a strictly positive element. We may assume that $a$ is not Cuntz equivalent to
a projection as $A$ is not unital.
For $\ep_1=1/2,$  by Lemma 7.2 of \cite{eglnp}, there are $0<\ep_2<1/4$ and $e_1\in A$ such that
$0\le f_{\ep_1}(a)\le e_1\le  f_{\ep_2}(a)$
and $\overline{e_1Ae_1}$ has continuous scale  (see also Lemma 5.3 of \cite{eglnp}, for example).
By {{an}} induction and repeatedly applying Lemma 7.2 of \cite{eglnp},
one obtains a sequence {{$\{e_n\}\subset  A_+^1$}} such that
\beq
0\le f_{\ep_n}(a)\le e_n\le f_{\ep_{n+1}}(a),
\eneq
{{$\overline{e_nAe_n}$ has continuous scale}}
and $0<\ep_n<1/2^n,$ $n=1,2,....$
Since $\{f_{\ep_n}(a)\}$ forms an approximate identity, one then verifies that $\{e_n\}$ also forms
an approximate identity for $A.$

To see the last part of the lemma, we note that
\beq
e_n\sim e_ne_ne_n\le e_nf_{\ep_{n+1}}(a)e_n
\lesssim e_nae_n
\le e_n.
\eneq
It follows that $e_n\sim e_nae_n$ for all $n.$
\end{proof}

The following is a non-unital version of 2.2 of {{\cite{Wccross}}}.

{{In the next statement, as in 5.3 of \cite{eglnkk0},  ${\cal S}$ is a fixed class of non-unital separable amenable $C^*$-algebras $C$ such that
$T(C)\not=\emptyset$ and $0 \notin \overline{ T(C)}^w$. A simple $C^*$-algebra A is said to be in  the class ${\cal R}$, if $A$ is separable, has
continuous scale  and $T(A)\not=\emptyset$.}}


\begin{thm}[cf. {{Theorem 2.2 of \cite{Wccross} and}} Theorem 5.4 of \cite{eglnkk0}]\label{fdim}

Let $A$ be a stably  projectionless separable simple \CA\,
with continuous scale and
with $\mathrm{dim}_{\mathrm{nuc}} A=m<\infty.$

Fix a
 positive element $e\in A_+$ with $0\le e\le 1$
such that $\tau(e),\,\tau(f_{1/2}(e))\ge r_0>0$ for all $\tau\in  T(A).$
Let $C=\overline{\bigcup_{n=1}^{\infty} C_n}$ be a non-unital simple \CA\, with continuous scale,
where $C_n\subseteq C_{n+1}$ and $C_n\in {\cal S},$ which also satisfies condition (1) in
Definition 5.3 of \cite{eglnkk0}.
Suppose that there is an affine homeomorphism $\Gamma:   T(C)\to  T(A)$
and suppose that there
{{is a  sequence}}
of \cpc s $\sigma_n: A\to C$ with {{${\rm im}(\sigma_n)\subset C_n$}} and
a sequence of  injective \hm s
$\rho_n: C_n\to A$
such that
\beq\label{TWv-1}
&&\lim_{n\to\infty}\|\sigma_n(ab)-\sigma_n(a)\sigma_n(b)\|=0,\quad a, b\in A,\\\label{TWv-2}
&&\lim_{n\to\infty}\sup\{|t\circ \sigma_n(a)-\Gamma(t)(a)|: t\in  T(C)\}=0,\quad a\in A,{ {\tand}}\\\label{TWv-2+}
&&\lim_{n\to\infty}\sup\{|\tau(\rho_n\circ \sigma_n(a))-\tau(a)|:\tau\in  T(A)\}=0,\quad a\in { {A.}}
\eneq

Then $A$ has the following property: For any finite set $\mathcal F\subseteq A$ and any $\ep>0$, there are a projection $p\in \mathrm{M}_{4(m+2)}(\widetilde{A})$, a {{\SCA}}
$S\subseteq p\mathrm{M}_{4(m+2)}(A)p$ with $S\in\mathcal {\cal S},$ and
an  ${\cal F}$-$\ep$-multiplicative \cpc\, $L: A\to S$ such that
\begin{enumerate}
\item $\|[p, 1_{4(m+2)} \otimes a]\| < \ep$\,\, {{for all}} $a\in\mathcal F$,
\item $p(1_{4(m+2)} \otimes a) p\in_\ep S$ \,\, {{for all}} $a\in\mathcal F$,
\item $\|L(a)-p(1_{4(m+2)} \otimes a) p\|<\ep$ \,\, {{for all}} $a\in \mathcal F,$
\item $p\sim e_{11}$ in $\mathrm{M}_{4(m+2)}(\widetilde{A})$,
\item $\tau(L(e)),\, \tau(f_{1/2}(L(e)))>7r_0/32(m+2)$ for all $\tau\in  T(\mathrm{M}_{4(m+2)}(A)),$
\item ${{(1_{4(m+2)}-p)\mathrm{M}_{4(m+2)}(A)(1_{4(m+2)}-p)}} \in \mathcal R$, and
\item $t(f_{1/4}(L(e)))\ge (3r_0/8)\lambda_s(C_1)\tforal t\in  T(S)$ \,\,\, {{(see \ref{Dlambdas} for $\lambda_s$).}}
\end{enumerate}
\end{thm}

\begin{proof}
{{This is a slight modification of Theorem 5.4 of \cite{eglnkk0} which {{is a variation of}}
that of Proposition 2.2
of \cite{Wccross}.
The proof is {{almost}}
the same as that of Theorem 5.4 of \cite{eglnkk0}
which itself is a repetition of the original proof of Proposition 2.2 of \cite{Wccross}.}}

Since $A$ has finite nuclear dimension, one has that $A\cong A\otimes \mathcal Z$ (\cite{Winter-Z-stable-02} for the unital case and \cite{T-0-Z} for the non-unital case). Therefore, $A$ has strict comparison for positive elements (Corollary 4.7 of {{\cite{Rrzstable}}}).

The proof is essentially the same as that of Theorem 2.2 of \cite{Wccross}. We give the proof in the present very much analogous situation for the convenience of the reader.
Let  $e\in A_+$ with
$\|e\|=1$, {{$\tau(e) > r_0$}}, and $\tau(f_{1/2}(e))>r_0$ for all $\tau\in  T(A).$

 Let $(e_n)$ be an
(increasing)
 approximate unit for  $A$. Since $A\in\mathcal R$, {{and}}
 $A$ is also assumed to be projectionless, one may assume that $\mathrm{sp}(e_n)=[0, 1]$. Since $\mathrm{dim}_\mathrm{nuc}(A)\leq m$, by Lemma {{5.2 of \cite{eglnkk0}}}, there is a system of $(m+1)$-decomposable completely positive approximations
$$
\xymatrix{
\widetilde{A} \ar[r]^-{\psi_j} & F_j^{(0)}\oplus F_j^{(1)}\oplus\cdots\oplus F_j^{(m)} \oplus \C \ar[r]^-{\phi_j} & \widetilde{A},\quad j=1, 2, ...
}
$$
such that 
\beq
&&\phi_j(F_j^{(l)})\subseteq A, \quad l=0, 1, ..., m,\andeqn\\
\label{defn-phi-j}
&&\phi_j|_\C(1_\C) = 1_{\widetilde{A}}-e_{j},
\eneq
where $e_j$ is an element of $(e_n)$.

Write
$$\phi_{j}^{(l)} = \phi_j|_{F_j^{(l)}}\quad\mathrm{and}\quad {\phi_{j}^{(m+1)}}= \phi_j|_\C,\quad l=0, 1, ..., m.$$
{{Set $F_j^{(m+1)}=\C$. Let $\psi_j^{(l)}=\pi_l\circ \psi_j$ for $j=0,1,2,\cdots, m+1$, where $\pi_l: \oplus_{k=0}^{m+1}F_j^{(k)} \to F_j^{(l)}$ is the projection.}} As in Lemma {{5.2 of \cite{eglnkk0}}}, one may assume that
\begin{equation}\label{refine}
\lim_{j\to\infty}\|\phi^{(l)}_j\psi^{(l)}_j(1_{\widetilde{A}})a - \phi_j^{(l)}\psi^{(l)}_j(a)\|=0,\quad l=0 ,1 , ..., m, \ a\in A.
\end{equation}
Note that $\phi_{j}^{(l)}: F_j^{(l)} \to A$ is of order zero,
and
the relation for
an
order zero map is weakly stable
(see  (${\mathcal P}$) and (${\mathcal P} 1$) of {{2.5 of \cite{KW}}}{{).}}
On the other hand, if $i$
is
large enough, then $ \sigma_i\circ\phi_{j}^{(l)}$ satisfies the relation for order zero
to within an arbitrarily small
tolerance, since $\sigma_i$ will be sufficiently multiplicative. It
follows that there are order zero maps $$\tilde{\phi}_{j, i}^{(l)}: F_j^{(l)} \to {{C_i}}$$
such that
$$\lim_{i\to\infty}\| \tilde{\phi}_{j, i}^{(l)}(c) - \sigma_i(\phi_{j}^{(l)}(c))\|=0,\quad c\in F_{j}^{(l)}.$$
We will identify ${{C_i}}$ with  $S_i=\rho_i({{C_i}})\subseteq A$, $\sigma_i: A \to {{C_i}}$ with $\rho_i\circ \sigma_i: A \to S_i\subseteq A$, and $\tilde{\phi}_{j, i}^{(l)}$ with $\rho_i\circ \tilde{\phi}_{j, i}^{(l)}$. There is
a positive linear map (automatically order zero)
$$\tilde{\phi}_{j, i}^{(m+1)}: \C\ni 1 \mapsto 1_{\widetilde{A}}-\sigma_i(e_j)\in \tilde{S_i}=\textrm{C*}(S_i, 1_{\widetilde{A}})\subseteq \widetilde{A},\quad i\in\mathbb N.$$
Note that
\begin{equation}\label{defn-phi-j-p}
 \tilde{\phi}_{j, i}^{(m+1)}(\lambda) = \sigma_i(\phi_{j}^{(m+1)}(\lambda)),\quad \lambda \in F_{j}^{(m+1)}=\C,
\end{equation}
where one still uses $\sigma_i$ to denote the induced map $\widetilde{A} \to \widetilde{S_i}$.

Note that for each $l=0, 1, ..., m$,
$$\lim_{i\to\infty}\|{f(\tilde{\phi}_{j, i}^{(l)})(c)} - \sigma_i ({f(\phi_{j}^{(l)})(c)})\| = 0,\quad c\in (F_j^{(l)})_+, \  f\in\mathrm{C}_0((0, 1])_+,$$
{{(see 4.2 of \cite{WZ-OR0} for the definition of $f(\psi)$ where $\psi$ is an order zero map)}}  and hence, from (\ref{TWv-2+}),
$$\lim_{i\to\infty}\sup_{\tau\in{ T(A)}}| \tau({f(\tilde{\phi}_{j, i}^{(l)})(c)} - {f(\phi_{j}^{(l)})(c)})| = 0,\quad c\in (F_j^{(l)})_+, \  f\in\mathrm{C}_0((0, 1])_+.$$
Also note that
$$\limsup_{i\to\infty} \|f(\tilde{\phi}_{j, i}^{(l)})(c)\| \leq \| f(\phi_j^{(l)})(c)\|, \quad c\in (F_j^{(l)})_+, \  f\in\mathrm{C}_0((0, 1])_+.$$

Applying {{Lemma 5.1 of \cite{eglnkk0}}}
to $(\tilde{\phi}_{j, i}^{(l)})_{i\in\mathbb N}$ and $\phi_j^{(l)}$ for each $l=0, 1, ..., m$, we obtain contractions $$s^{(l)}_{j, i} \in \mathrm{M}_4(A) \subseteq \mathrm{M}_4(\widetilde{A}),\quad i\in\mathbb N,$$
such that,   {{for all $c\in F_j^{(l)},$}}
\beq
&&\lim_{i\to\infty}\|s_{j, i}^{(l)}(1_4\otimes \phi_j^{(l)}(c)) - (e_{1, 1}\otimes\tilde{\phi}_{j, i}^{(l)}(c)) s_{j, i}^{(l)}\| = 0
\andeqn\\
&&\lim_{i\to\infty}\|(e_{1, 1} \otimes {\tilde \phi}_{j,i}^{(l)}(c))s_{j,i}^{(l)}(s_{j, i}^{(l)})^*- e_{1, 1} \otimes {\tilde \phi}_{j,i}^{(l)}(c)\| = 0.
\eneq

Note that $\mathrm{sp}(e_j)=[0,1].$  Put $C_0=C_0((0,1]).$ Define
\beq
\Delta_j(\hat{f})=\inf\{\tau(f(e_j)): \tau\in  T(A)\}\rforal f\in (C_0)_+\setminus\{0\}.
\eneq
Since $A$ is assumed to have continuous scale, $ T(A)$ is compact and $\Delta_j(\hat{f})>0$
for all $f\in {{(}}C_0)_+\setminus\{0\}.$
For $l=m+1$, since $\mathrm{sp}(e_j) = [0, 1]$, by considering $\Delta_j$
for each $j,$ since $i$ is chosen after $j$ is fixed,  by applying  {{A.16 of \cite{eglnkk0},}}
one obtains unitaries $$s^{(m+1)}_{j, i} \in \widetilde{A},\quad i\in\mathbb N,$$ such that $$\lim_{i\to\infty}\|s_{j, i}^{(m+1)}e_j-\sigma_i(e_j)s_{j, i}^{(m+1)}\|=0,$$ and hence
$$\lim_{i\to\infty}\|s_{j, i}^{(m+1)}(1_{\widetilde{A}}-e_j)-(1_{\widetilde{A}}-\sigma_i(e_j))s_{j, i}^{(m+1)}\|=0.$$
By \eqref{defn-phi-j} and \eqref{defn-phi-j-p}, one has
\begin{equation*}
\lim_{i\to\infty}\|s_{j, i}^{(m+1)}\phi_j^{(m+1)}(c) - \tilde{\phi}_{j, i}^{(m+1)}(c) s_{j, i}^{(m+1)}\| = 0,\quad c\in F_j^{(m+1)}=\C.
\end{equation*}

Considering the
element
$e_{1, 1}\otimes s_{j, i}^{(m+1)} \in \mathrm{M_4}\otimes \widetilde{A}$, and still denoting it by $s_{j, i}^{(m+1)}$,
{{one has}}
\begin{equation*}
\lim_{i\to\infty}\|s_{j, i}^{(m+1)}(1_4\otimes \phi_j^{(m+1)}(c)) - (e_{1, 1}\otimes\tilde{\phi}_{j, i}^{(m+1)}(c)) s_{j, i}^{(m+1)}\| = 0,\quad c\in {{F_j^{(m+1)}=\C}}
\end{equation*}
and $$(e_{1, 1} \otimes {\tilde \phi}_{j,i}^{(m+1)}(c))s_{j,i}^{(m+1)}(s_{j, i}^{(m+1)})^* = e_{1, 1} \otimes
{\tilde \phi}^{(m+1)}_{j,i}(c). $$
Therefore,
{{\beq\label{pre-s-1}
&&\hspace{-0.5in}\lim_{i\to\infty}\|s_{j, i}^{(l)}(1_4\otimes\phi_j^{(l)}(c)) - (e_{1, 1} \otimes \tilde{\phi}_{j, i}^{(l)}(c)) s_{j, i}^{(l)}\| = 0,\quad c\in F_j^{(l)},\ l=0, 1, ..., m+1,\andeqn\\
\label{pre-s-1-1}
&&\hspace{-0.5in}\lim_{i\to\infty}\|(e_{1, 1}\otimes {\tilde \phi}_{j,i}(c))s_{j,i}^{(l)}(s_{j, i}^{(l)})^*- {\tilde \phi}_{j,i}(c)\| = 0, \quad c\in F_j^{(l)},\ l=0, 1, ..., m+1.
\eneq}}
Let $\tilde{\sigma}_i: \widetilde{A} \to
{{\widetilde C_i}} $ and $\tilde{\rho}_i:
{{\widetilde C_i}} \to\widetilde{A}$ denote the unital maps induced by
$\sigma_i: {A} \to {{C_i}}$ and $\rho_i:  {{C_i}}\to {A}$, respectively.

Consider the contractions
$$s_j^{(l)} := (s_{j, i}^{(l)})_{i\in\mathbb N}\in (\mathrm{M}_4\otimes \widetilde{A})_\infty,\quad l=0, 1, ..., m+1, \quad j=1, 2, ....$$
By \eqref{pre-s-1} and \eqref{pre-s-1-1}, these satisfy
$$s_j^{(l)}(1_4\otimes { {\iota}}(\phi_j^{(l)}(c))) = (e_{1, 1}\otimes \bar\rho\bar\sigma({\phi}_j^{(l)}(c))) s_{j}^{(l)}\andeqn$$
$$(e_{1,1}\otimes \bar{\rho}\circ \bar{\sigma}(\phi_j^{(l)}(c)))s_j^{(l)}(s_j^{(l)})^* =
 (e_{1,1}\otimes \bar{\rho}\circ {\bar{\sigma}}(\phi_j^{(l)}(c))),$$
where
$$\bar{\sigma}: { {\widetilde{A}}}
\to  {{\prod_{n=1}^\infty}} {{\widetilde{C_n}}}/{{\bigoplus_{n=1}^\infty \widetilde{C_n}}}~~~~\mbox{and}~~~~ \bar\rho: {{\prod_n^{\infty} \widetilde{C_n}/\bigoplus_n^\infty \widetilde{C_n} }}\to (\widetilde{A})_\infty$$
are the homomorphisms induced by $\tilde{\sigma}_i$ and $\tilde{\rho}_i$,
and the map { {$\iota: \widetilde{A}\to ({\widetilde{A}})_{\infty}$ is the canonical embedding. Let}}
$$\bar\iota: (\widetilde{A})_\infty \to ((\widetilde{A})_\infty)_\infty$$ {{be}} the embedding induced by the canonical embedding {{$\iota$}} 
{{and let}} $$\bar\gamma: (\widetilde{A})_\infty \to ((\widetilde{A})_\infty)_\infty$$ denote the homomorphism induced by
the composed map
$$\bar{\rho}\bar{\sigma}: {{\widetilde{A}}} 
\to (\widetilde{A})_\infty,$$
For each $l=0, 1, ..., m+1$, let
\beq
\bar{\phi}^{(l)}: \prod_j F_j^{(l)}/\bigoplus_j F_j^{(l)} \to A_\infty\andeqn
\bar{\psi}^{(l)}: A \to \prod_j F_j^{(l)}/\bigoplus_j F_j^{(l)}
\eneq	
denote the maps induced by $\phi_j^{(l)}$ and $\psi_j^{(l)}$. 

Consider the  contraction
$$\bar{s}^{(l)} = (s_j^{(l)}) \in (\mathrm{M}_4\otimes \widetilde{A}_\infty)_\infty.$$
Then
$$
{\bar{s}^{(l)}}(1_4\otimes \bar\iota\bar{\phi}^{(l)}\bar\psi^{(l)}(a)) = (e_{1, 1} \otimes \bar\gamma\bar\phi^{(l)}\bar\psi^{(l)}(a)) {\bar{s}^{(l)}},\quad a\in \widetilde{A},\andeqn$$
$$
\hspace{-0.7in}(e_{1, 1} \otimes \bar\gamma\bar\phi^{(l)}\bar\psi^{(l)}(a)) {\bar{s}^{(l)}}({\bar{s}^{(l)}})^* = (e_{1, 1} \otimes \bar\gamma\bar\phi^{(l)}\bar\psi^{(l)}(a)).
$$

By \eqref{refine}, one has
$$\bar\phi^{(l)}\bar\psi^{(l)}(1_{\widetilde A})\iota(a) = \bar\phi^{(l)}\bar\psi^{(l)}(a),\quad a\in A.$$
{{Hence $[\bar\phi^{(l)}\bar\psi^{(l)}(1_{\widetilde A}),\, \iota({ {b}})]=0,$ and,
$(\bar\phi^{(l)}\bar\psi^{(l)}(1_{\widetilde A})\iota({ {b}}))^{1/2}=(\bar\phi^{(l)}\bar\psi^{(l)}(1_{\widetilde A}))^{1/2}\iota({ {b}}^{1/2})$ { {for any $b\in A_+.$}}
It follows}}
$$((\bar\phi^{(l)}\bar\psi^{(l)}(1_{\widetilde A}))^{\frac{1}{2}}\iota(a) \in \textrm{C*}(\bar\phi^{(l)}\bar\psi^{(l)}(A))~{ {\mbox{for}~ a\in A,}}$$
and hence
\begin{eqnarray}\label{twrist}
\hspace{-0.5in}\bar{s}^{(l)}(1_4 \otimes (\bar{\iota}\bar{\phi}^{(l)}\bar{\psi}^{(l)}(1_{\widetilde{A}}))^{\frac{1}{2}})(1_4\otimes \bar{\iota}\iota(a)) &= & \bar{s}^{(l)}(1_4\otimes \bar{\iota}\bar{\phi}^{(l)}\bar{\psi}^{(l)}(1_{\widetilde{A}})^{\frac{1}{2}}\iota(a)) \nonumber \\
& = & (e_{1, 1} \otimes \bar{\gamma}(\bar{\phi}^{(l)}\bar{\psi}^{(l)}(1_{\widetilde{A}}))^{\frac{1}{2}}\iota(a)) \bar{s}^{(l)} \nonumber \\
& = & (e_{1, 1} \otimes \bar{\gamma}\iota(a)(\bar{\phi}^{(l)}\bar{\psi}^{(l)}(1_{\widetilde{A}}))^{\frac{1}{2}}) \bar{s}^{(l)} \nonumber \\
& = & {{(e_{1, 1} \otimes \bar{\gamma}(\iota(a)))(e_{1, 1}\otimes \bar{\gamma}(\bar{\phi}^{(l)}\bar{\psi}^{(l)}(1_{\widetilde{A}}))^{\frac{1}{2}})\bar{s}^{(l)}.}}
\end{eqnarray}
Set
\vspace{-0.1in}\begin{eqnarray*}
\bar{v} & = & \sum_{l=0}^{m+1} e_{1, l} \otimes ((e_{1, 1} \otimes \bar{\gamma}\bar{\phi}^{l}\bar{\psi}^{(l)}(1_{\widetilde{A}}))^{\frac{1}{2}}\bar{s}^{(l)}) \\
& = & \sum_{l=0}^{m+1} e_{1, l} \otimes (\bar{s}^{(l)}(1_4\otimes \bar{\iota}\bar{\phi}^{l}\bar{\psi}^{(l)}(1_{\widetilde{A}}))^{\frac{1}{2}}) \in \mathrm{M}_{m+2}(\C) \otimes \mathrm{M}_4(\C)\otimes ({\widetilde{A}}_\infty)_\infty.
\end{eqnarray*}
Then
\begin{eqnarray*}
\bar{v}\bar{v}^* & = & \sum_{l=0}^{m+1} e_{1, 1} \otimes (e_{1, 1} \otimes \bar{\gamma}\bar{\phi}^{l}\bar{\psi}^{(l)}(1_{\widetilde{A}})) =  e_{1, 1} \otimes e_{1, 1} \otimes \bar{\gamma}(1_{\widetilde{A}}).
\end{eqnarray*}
Thus, $\bar{v}$ is {{a}} partial isometry.
Moreover, for any $a\in \widetilde{A}$,
\begin{eqnarray*}
\bar{v}(1_{(m+2)}\otimes 1_4 \otimes\bar{\iota}\iota(a)) & = & \sum_{l=0}^{m+1} e_{1, l} \otimes (\bar{s}^{(l)}(1_4\otimes \bar{\iota}\bar{\phi}^{l}\bar{\psi}^{(l)}(1_{\widetilde{A}})^{\frac{1}{2}})(1_4\otimes \bar{\iota}\iota(a)))\\
& = & \sum_{l=0}^{m+1} e_{1, l} \otimes (e_{1, 1} \otimes \bar{\gamma}(\iota(a)))(e_{1, 1} \otimes \bar{\gamma}(\bar{\phi}^{(l)}\bar{\psi}^{(l)}(1_{\widetilde{A}})^{\frac{1}{2}}\bar{s}^{(l)}) \quad\quad\textrm{(by \eqref{twrist})}\\
& = & (e_{1, 1}\otimes e_{1, 1} \otimes\bar{\gamma}(\iota(a))) \sum_{l=0}^{m+1} e_{1, l} \otimes e_{1, 1} \otimes \bar{\gamma}(\bar{\phi}^{(l)}\bar{\psi}^{(l)}(1_{\widetilde{A}})^{\frac{1}{2}}\bar{s}^{(l)}) \\
& = & (e_{1, 1}\otimes e_{1, 1} \otimes \bar{\gamma}(\iota(a)))\bar{v}.
\end{eqnarray*}
Hence, {{we obtain}}
\begin{equation*}
\bar{v}^*\bar{v} (1_{m+2}\otimes 1_4 \otimes \bar\iota\iota(a))  =  \bar{v}^*(e_{1, 1}\otimes e_{1, 1} \otimes \bar\gamma\iota(a)) \bar{v} =  (1_{m+2}\otimes 1_4\otimes\bar\iota\iota(a)) \bar{v}^*\bar{v},\quad a\in \widetilde{A}.
\end{equation*}
Then, for any finite set $\mathcal G \subseteq \widetilde{A}$ and any  $\dt>0,$ there are
{{$i \in \mathbb N$}} and ${{v_i}}\in \mathrm{M}_{m+2}(\C)\otimes \mathrm{M}_{4}(\C) \otimes \widetilde{A}$ such that
\beq\label{53-200110-10}
&&{{v_iv_i^*}}= e_{1, 1}\otimes e_{1, 1}\otimes \tilde{\rho}_i(1_{\tilde{S_i}})= e_{1, 1}\otimes e_{1, 1}\otimes 1_{\widetilde{A}},\\
&&\|[{{v_i^*v_i}}, 1_{m+2} \otimes1_4\otimes  a]\| < \dt \rforal
a\in \mathcal G,\\
&&\|{{v_i^*v_i}}(1_{m+2}\otimes 1_4\otimes  a) - {{v_i}}^*(e_{1, 1}\otimes e_{1, 1}\otimes\tilde{\rho}_i\tilde{\sigma}_i(a)){{v_i}}\|< \dt
\rforal a\in \mathcal G\andeqn\\\label{53-200110-11}
&& \tau(\rho_i\circ \sigma_i(e)),\, \tau(f_{1/2}(\rho_i\circ \sigma_i(e)))\ge 15r_0/16\rforal \tau\in  T(A).
\eneq

Define
${{\kappa_i}}: \widetilde{S}_i \to \mathrm{M}_{m+2}\otimes \mathrm{M}_{4} \otimes \widetilde{A}$
by
$${{\kappa_i}}(s)={v_i}^*(e_{1, 1}\otimes e_{1, 1} \otimes{{\rho_i(s)}}){ v_i}.$$
Note that
$${{\kappa_i}}(S_i) \subseteq \mathrm{M}_{m+2}\otimes \mathrm{M}_{4}\otimes A.$$

Then ${{\kappa_i}}$ is an embedding; and on setting $p_i=1_{{\kappa_i}(\tilde{S}_i)}=v_i^*v_i$, {{we have}}
\begin{enumerate}
\item[(i)] ${p_i} \sim e_{1, 1}\otimes e_{1, 1}\otimes 1_{\widetilde A}$,
\item[(ii)] $\|[{p_i}, 1_{m+2}\otimes 1_4 \otimes a]\| <\dt$, $a\in \mathcal G$,
\item[(iii)] ${p_i}(1_{m+2}\otimes 1_4 \otimes a){p_i} \in_{\dt} {{\kappa_i}}(\tilde{S}_i)$, $a\in \mathcal G$.
\end{enumerate}
%
%
Note that $A$ is ${\cal Z}$-stable (by \cite{Winter-Z-stable-02}) and hence has strict comparison
(by \cite{Rlz}).  Let ${{e'}} 
\in (1_{4(m+2)}-{p_i})\mathrm{M}_{4(m+2)}(A)(1_{4(m+2)}-{p_i})$
be a strictly positive element.
By (i), ${\mathrm d}_\tau({{e'}})=\tau(1_{4(m+2)}-p_i)=\tau(1_{4(m+2)}-e_{1, 1}\otimes e_{1, 1}\otimes 1_{\widetilde A})$
for all $\tau\in T(A)$, where $\tau$ is naturally extended to $\tilde{A}.$
Since $A$ and $M_{4(m+2)}(A)$ have continuous scale, $\tau\mapsto {\text{d}}_\tau$
is continuous on $T(A).$ Hence
$(1_{4(m+2)}-{p_i})\mathrm{M}_{4(m+2)}(A)(1_{4(m+2)}-{p_i})$ also
has continuous scale (see 5.4 of \cite{eglnp}) and
is still in the reduction class $\mathcal R$ (so condition (6) holds).
%

Define $L_i: A\to {{\kappa_i(S_i)}}$ by  $L_i(a)={v_i}^*(e_{1, 1}\otimes e_{1, 1} \otimes{\rho}_i(\sigma_i(a))){v_i}$
for all $a\in A.$  Then
\begin{enumerate}
\item[(iv)] $\|L_i(a)-{p_i}(1_{4(m+2)}\otimes a){p_i}\|<\dt\rforal a\in {\cal G}$ and
\item[(v)] $\tau(L_i(e)),\, \tau(f_{1/2}(L_i(e)))\ge \displaystyle{15r_0\over{64(m+2)}}\rforal \tau\in  T(\mathrm{M}_{4(m+2)}(A)).$
\end{enumerate}
{{Let $\tau_i\in  T(\kappa_i(S_i)).$
Then  $\tau_i\circ L_i$ is a positive linear functional. Let ${{\bar t}}$ be a weak *-limit of $\{\tau_i\circ L_i\}.$
Note that, for any $1/2>\ep>0,$ since $A$ has continuous scale, there is $e_A\in A$  with $\|e_A\|=1$
such that $\tau(e_A)>1-\ep/2$ for all $\tau\in T(A).$
By \eqref{TWv-2+} (see also \eqref{53-200110-10}),  we may assume that
$\tau_i\circ L_i(e_A)>1-\ep$ for all large $i.$ It follows that ${{\bar t}}(e_A)\ge 1-\ep.$ Hence $\|{{\bar t}}\|\ge 1-\ep$
for any $1/2>\ep>0.$ It follows that ${{\bar t}}$ is a state of $A.$  Then, by \eqref{TWv-1} and \eqref{TWv-2+},
${{\bar t}}$ is a tracial state of $A.$
Therefore,}} with sufficiently small $\dt$ and large ${\cal G}$ (and sufficiently large $i$), {{also by}}  \eqref{53-200110-11},
we may assume that
\beq\label{lambdas-2}
t(f_{1/4}(L_i(e)))\ge 7r_0/8\rforal t\in  T({{\kappa_i(S_i)}}).\quad 
\eneq

{{Recall {{that}} $\kappa_i(S_i)\cong C_i.$ {{Fix}} (a large $i$) above and set $S=\kappa_i(S_i)$ and $L=L_i.$}}
{{Recall also {{that}} $\lambda_s(C_1)\le 1.$
So \eqref{lambdas-2} also implies that
\beq\label{Claim518}
t(f_{1/4}(L(e)))\ge (3r_0/8)\lambda_s(C_1)\rforal t\in  T(S)).\quad 
\eneq
}}

%
%
The conclusion of the theorem follows from (i),(ii), (iii), (iv), (v),
and \eqref{Claim518}.
\end{proof}



\begin{thm}[c.f. Theorem 5.7 of \cite{eglnkk0}]\label{TTTAD}
Let $A$ be a stably projectionless  separable simple \CA\,  with continuous scale and
with $\mathrm{dim}_{\mathrm{nuc}} A=m<\infty$.

Suppose that every hereditary \SCA\, $B$ of $A$ with continuous scale   has the following properties:
Let $e_B\in B$ be a strictly positive element with $\|e_B\|=1$ and $\tau(e_B)>1-1/64$
for all $\tau\in  T(B).$
Let $C$ be a non-unital simple \CA\, which is an inductive limit { {$C=\overline{\cup_{n=1}^\infty C_n}$, where $C_n\subset C_{n+1}$ and
$C_n\in {\cal C}_0$,}}
with continuous scale such that
$T(C)\cong T(B).$
For each affine homeomorphism $\gamma: { T(B)\to  T(C)},$
there
exist sequences
of \cpc s $\sigma_n: B\to C$  {{with ${\rm im}(\sigma_n)\subset C_n$}} and injective
\hm s $\rho_n: C_n\to B$
such that
\beq\label{TTWv-1}
&&\lim_{n\to\infty}\|\sigma_n({{xy}})-\sigma_n({{x}})\sigma_n({{y}})\|=0\tforal x, y\in B,\\\label{TTWv-2}
&&{{\lim_{n\to\infty}\sup\{|t\circ \sigma_n({ {b}})-\gamma^{-1}(t)({ {b}})|: t\in T(C)\}=0\tforal { {b\in B}},\tand}}\\\label{TTWv-3}
&&\lim_{n\to\infty}\sup\{|\tau (\rho_n\circ \sigma_n(b))-{\tau(b)}|:\tau\in  T(B)\}=0{{\tforal b\in B.}}
\eneq
%
Suppose also that every hereditary \SCA\, $A$ is tracially approximately divisible.
Then $A\in {\cal D}.$
\end{thm}

\begin{proof}
The proof is exactly the same as that of Theorem 5.7 {of \cite{eglnkk0}.}
By \cite{T-0-Z} (see also \cite{Winter-Z-stable-02}),
$A'\otimes {\cal Z}\cong A'$ for every hereditary \SCA\, $A'$ of $A.$ It follows from \cite{Rlz}
that $A$ has almost stable rank one.
Let $B$ be a hereditary  \SCA\, with continuous scale. Then $B$ has finite nuclear dimension (see \cite{WZ-ndim}).
By \cite{T-0-Z} again, $B$ is ${\cal Z}$-stable.
It follows from 6.6 of \cite{ESR-Cuntz}
that the map from ${ {\mathrm{Cu}(B)}}$ to ${\mathrm{LAff}}_+({{\tilde T}}(B))$ is surjective.
Note that
{{the map}} from ${{W}(B)}$ to ${\mathrm{LAff}}_{b,+}( T(B))$ is {{also}} surjective.
We will apply {{Theorem
\ref{fdim} above}}
and Lemma 5.5 of \cite{eglnkk0}.

{{Since $B$ has continuous scale, we may choose a strictly positive element $e\in B$ with $\|e\|=1$
and $e'\in B_+$ with $\|e'\|\le 1$ such that }}
$f_{1/2}(e)e'=e'f_{1/2}(e)=e'$
and
$d_\tau(f_{1/2}(e'))> 1-1/128(m+2)$  for all $\tau\in  T(B).$ 
Let $1>\ep>0,$ ${\cal F}\subset B$ be a finite { {subset}}
and let $b\in B_+\setminus \{0\}.$
Choose $b_0\in B_+\setminus \{0\}$ {{such that}} $64(m+2)\la b_0\ra \le {\la b \ra}$ in
$\mathrm{Cu}(A).$
Since we assume that ${{B}}$ is tracially approximately divisible (see 10.1 of {{\cite{eglnp},}} or 5.6 of \cite{eglnkk0}),
{{there}} are
$e_0\in B_+$ and a   hereditary \SCA\,
$A_0$ of ${{B}}$ such that $e_0\perp M_{4(m+2)}(A_0),$ $e_0\lesssim b_0$ and
$$
{\mathrm{dist}}(x, B_{1,d})<\ep/64(m+2)\rforal x\in {\cal F}\cup \{e\},
$$
where $B_{1,d}\subseteq  B_s:= \overline{e_0Be_0}\oplus \mathrm{M}_{4(m+2)}(A_0)\subseteq  B$ and
\beq\label{TTTAD-1}
B_{1,d}=\{ x_0\oplus (\overbrace{x_1\oplus x_1\oplus \cdots \oplus x_1}^{4(m+2)}):x_0\in \overline{e_0Be_0},\, x_1\in A_0\}.
\eneq
\Wlog, we may further assume that ${\cal F}\cup \{e'\}\subseteq B_{1,d}.$
{{Let $P:
B_s\to
\mathrm{M}_{4(m+2)}(A_0)$  be a projection map and $P^{(1)}: \mathrm{M}_{4(m+2)}(A_0)
\to A_0=A_0\otimes e_{11}$ be defined by $P^{(1)}(a)=({ {1_{\tilde A_0}}}\otimes e_{11})a({ {1_{\tilde A_0}}}\otimes e_{11})$
for $a\in M_{4(m+2)}(A_0),$
 where $\{e_{ij}\}_{4(m+2)\times 4(m+2)}$ is a system of matrix unit.}}
{{Put ${\cal F}_0=\{P(x): x\in {\cal F}\}.$}}
Therefore,  we  may assume, \wilog,  that
$\|e_0x-xe_0\|<\ep/64(m+2),$
and there is $e_1\in M_{4(m+2)}(A_0)$ with $0\le e_1\le 1$ such that
$\|e_1x-xe_1\|<\ep/64(m+2)$  and
{{$\|e_1P(x)-P(x)\|<\ep/(64(m+2)$
{{for all}} $x\in {\cal F}\cup \{e, e',f_{1/2}(e), f_{1/4}(e), f_{1/2}(e')\}.$}}
Moreover,  since  the map from $W(A)_+$ to ${\mathrm{LAff}}_{b,+}( T(A))$
is surjective,  by \ref{Lcontinuous},
\wilog, we may assume
that $A_0$ has continuous scale.

\Wlog, we may further assume that ${\cal F}\cup \{e\}\subseteq B_{1,d}.$
Write
$$
x=x_0+\overbrace{x_1\oplus  x_1\oplus\cdots\oplus x_1}^{4(m+1)}.
$$
Let ${\cal F}_1=\{x_1: x\in \mathcal F\cup \{e\}\}
{{\subset A_0}}.$
Note that we may write $\overbrace{x_1\oplus  x_1\oplus\cdots\oplus x_1}^{4(m+2)}=x_1\otimes 1_{4(m+2)}.$
{{Note that}} ${ \mathrm{dim}}_{\mathrm{nuc}}A_0=m$ (see \cite{WZ-ndim}).
Also, $A_0$ is a non-unital separable simple \CA\,
which has continuous scale.
We may then apply {{Theorem
%
\ref{fdim}}}
to $A_0$ with {{${\cal S}={\cal C}_0.$}}
%
By  {{\ref{4.25},}}
$C=\overline{\bigcup_{n=1}^{\infty} C_n},$ where $C_n\in {\cal C}_0,$
$C_n\subset  C_{n+1},$ {{and satisfies the condition (1) in 5.3 of \cite{eglnkk0}.}}
Moreover, $\lambda_s(C_n)\ge 1/2$ for all $n.$
Put $r_0=(1-1/64(m+2)).$
Choose $\eta_0=7/32(m+2)$ and $\lambda=3/16.$
Thus, by applying  {{Theorem
\ref{fdim},}}
we have {{the following estimates,}} with $\phi_1(b)=(E-p)b(E-p)$ for all $b\in M_{4(m+2)}(A_0),$
where $E=1_{M_{4(m+2)}({\widetilde{A_0}})},$ and $p\in M_{4(m+2)}({\widetilde{A_0}})$
 is a projection given by
 {{Theorem
\ref{fdim},}}
and $L: A_0\to D_1\subset pM_{4(m+2)}(A_0)p$ is an ${\cal F}_1$-$\ep$-multiplicative {{\cpc,}}
\beq
&&\hspace{-0.4in}\|x\otimes 1_{4(m+2)}-(\phi_1(x\otimes 1_{4(m+2)})+L(x))\|<\ep/4\rforal x\in {\cal F}_1,\\
&&\hspace{-0.4in}{d_\tau}(\phi_1(P(e)))\le 1-1/4(m+2)\rforal \tau\in  T(M_{4(m+2)}(A_0)),\\
&&\hspace{-0.65in}\tau'(\phi_1(P(e))), \tau'(f_{1/2}(\phi_1(P(e)))\ge r_0-\ep/4\rforal \tau'\in { T}((1-p)M_{4(m+1)}(A_0)(1-p)),\\\label{751831+}
&&\hspace{-0.4in}D_1\in {\cal C}_0, D_1 {\subseteq} pM_{4(m+2)}(A_0)p,\\
&&\hspace{-0.4in}\tau{{(L(P^{(1)}(e)))}}\ge r_0\eta_0\rforal \tau\in   T( M_{4(m+2)}(A_0))\andeqn\\\label{571831}
&&\hspace{-0.4in}t{{(f_{1/4}(L({{P^{(1)}(e)}})))}}\ge r_0\lambda\rforal t\in { T}(D_1).
\eneq
Let  $B_1=(1-p)M_{4(m+1)}{{(A_0)}}(1-p)\oplus \overline{e_0Be_0}$ and
$\phi: {B} \to B_1$ be defined by $\phi(a)=\phi_1(e_1ae_1)+e_0ae_0$ for $a\in A.$
{{Define $L_1: B\to  D_1$ by $L_1(b)=L(P^{(1)}(e_1^{1/2}be_1^{1/2})).$ Then both $\phi$
and $L_1$ are}}  ${\cal F}$-$\ep$-multiplicative.
Put $\eta=\eta_0/2<{\eta_0\over{1+\ep/64(m+2)}}.$
Then, in addition to  \eqref{571831} and \eqref{751831+},
\beq\nonumber
&&\|x-(\phi(x)+L_1(x))\|<\ep\rforal x\in {\cal F},\\\nonumber
&&{d_\tau}(\phi(e))\le 1-\eta\rforal \tau\in  T(B),\\\nonumber
&&\tau'(\phi(e)), \tau'(f_{1/2}(\phi(e))\ge r-\ep\rforal \tau'\in { T}(B_1),\\\nonumber
&&\tau(L_1(e))\ge r_0\eta\rforal \tau\in   T(B).
\eneq
 Note {{that}} these hold
for every such $B.$
Thus,  the
hypotheses
of  Theorem 5.5  of \cite{eglnkk0}
are satisfied.  We then   apply Theorem 5.5 of \cite{eglnkk0}.
%
\end{proof}

%
%
%


\begin{thm}[cf.~Theorem 4.4 of \cite{EGLN}]\label{Treduction1}
Let $A$ be a separable stably projectionless simple \CA\, of finite nuclear dimension which satisfies the UCT.
Assume that  $A$ has continuous scale, $T(A)=T_{qd}(A)\not=\emptyset,$  ${\rm Cu}(A)={\rm LAff}_+(T(A))$.
Suppose that  there exists {{a}} simple \CA\, $C$ with continuous scale such that
$C=\lim_{n\to\infty}({{C_n, \iota_n}}),$ where each $C_n=C_n'\otimes Q$ and $C_n'\in {\cal C}_0,$ and
 such that
\beq
(K_0(A\otimes Q),  T(A\otimes Q), r_{A\otimes Q})
\cong
(K_0(C), T(C), r_C).
\eneq
Then, $A\otimes Q\in  {\cal D}.$
\end{thm}

\begin{proof}
Since $A$ is simple, the assumption
$ T(A)= T_{{\rm qd}}(A)\not={{\emptyset}}$
immediately implies that  $A$ is both stably finite and quasidiagonal.
We may assume that $A\otimes Q\cong A.$
{{Let $B\subset A$ be a hereditary \SCA\, with continuous scale.
We may write $B={\rm Her}(b)$ for some $b\in A_+.$
Since $A\cong A\otimes Q,$ there is $a_0\in A_+$ such that
$d_\tau(a_0\otimes 1_U)=d_\tau(b)$ for all $\tau\in T(A).$
Since we assume that ${\rm Cu}(A)=\LAff_+(T(A)),$ $\la a_0\otimes 1_U\ra=\la b\ra.$
{{By Theorem 1.2 of}} \cite{Rlz}, $B\cong {\rm Her}(a_0\otimes 1_U).$
However, ${\rm Her}(a_0\otimes 1_U)=\overline{a_0Aa_0}\otimes Q.$
It follows that $B\cong \overline{a_0Aa_0}\otimes U\cong B\otimes Q.$}}

By  {{Theorem \ref{4.25},}}
together with the assumption $A\cong A\otimes Q$, there is a simple C*-algebra $C=\lim_{n\to\infty}(C_n, \imath_n)$
with continuous scale,
where each $C_n$ is the tensor product of a C*-algebra in $\mathcal C_0$ with $Q$ 
and $\imath_n$ is injective, such that
$$
(K_0(A),  T(A), r_A)\cong (K_0(C),   T(C), r_C)
$$
{{is}} given by $\Gamma.$
It follows that there is an order isomorphism
$$
{{\Gamma^\sim:}}\,\, (K_0({\tilde A}), K_0(\tilde A)_+, [1_{\tilde A}],  T(\tilde A), r_{\tilde A})\cong (K_0(\tilde C), K_0(\tilde C)_+, [1_{\tilde C}],
 T(\tilde C), r_{\tilde C}).
$$
We will continue to {{write}} $\imath_n$ and $\imath_{n,\infty}$
for the { {inclusions}} of $\imath_n: {\tilde C}_n\to {\tilde C}_{n+1}$ and $\imath_{n,\infty}: {\tilde C}_n\to {\tilde C},$
$n=1,2,....$
Let us {{write}}
$\Gamma_{\Aff}$  and $\Gamma_{\Aff}^\sim$ for the corresponding maps from $\Aff( T(A))$ to $\Aff( T(C))$
and from $\Aff( T(\tilde A))$ to $\Aff( T(\tilde C){{).}}$
 Since $A$ and $C$ have continuous {{scale,}} $T(A)$ and $T(C)$ are compact.

Let a finite subset  ${\mathcal F}$ of $ A$ and $\ep>0$ be given.
Fix a strictly positive element  $e\in A$ with $\|e\|=1.$
Choose
\beq\label{14f12(e)-n0}
0<d<\inf\{\tau(f_{1/2}(e)):\tau\in T(A)\}.
\eneq
This is possible since $T(A)$ is compact.
Let $0<\sigma<{{\min\{d, \ep\}/2^{10}}}.$
Since $A$ has continuous scale, one may choose $e_1=f_{\ep'}(e)$ for some $1/4>\ep'>0$ such that
\beq\label{14app-0}
\tau(e_1)>1-\sigma/128\rforal \tau\in T(A).
\eneq
\Wlog, we may assume that $e, f_{1/2}(e),$ $f_{1/4}(e)$ and $e_1\in {\cal F}.$


Let the finite set $\mathcal P$ of $K_0(A),$
the finite subset $\mathcal G_1$ (in place of ${\cal G}$) of $A$, and
$\dt_0>0$ (in place of $\dt$) be as assured by Lemma  7.2 of \cite{eglnkk0}
for ${\mathcal F}$ and $\ep/4$.
We may also assume that, for any ${\cal G}_1$-$\dt_0$-multiplicative \morp\, $L$ from $A,$
\beq\label{14n20526-n1}
\|f_{1/2}(L(e))-L(f_{1/2}(e))\|<d/2^{10}\andeqn \|f_{\ep'}(L(e))-L(f_{\ep'}(e))\|<d/2^{10}.
\eneq
Choose a finite subset ${\cal P}_0$ of  projections in $M_m({\tilde A})$ (for some integer $m\ge 1$)
such that ${\cal P}\subset \{[p]-[q]: p, q\in {\cal P}_0\}.$
We may assume that $1_{\tilde A}\in {\mathcal P}.$
Write  $\mathcal P_0=\{1_{\tilde A}, p_1,p_2,...,p_s\}.$  Deleting some elements (but not $1_A$), we may assume that the set
$${ {{\cal P}_0=}}\{[1_{\tilde A}], [p_1], [p_2],...,[p_s]\}\subset K_0({\tilde A})$$ is $\Q$-linearly independent.

Choose $\dt_0'=\dt_0/4m^2.$
We may also assume, \wilog, that  $L^\sim$ is ${\cal G}_1\cup {\cal P}_0$-$\dt_0$-multiplicative,
if $L$ is a ${\cal G}_1$-$\dt_0'$-multiplicative \morp\, from $A$ to a \CA\, $B,$ and
$L^\sim: M_m({\tilde A})\to M_m({\tilde B})$ is the usual extension of the unitization of $L.$

Put ${\cal G}={\cal F}\cup {\cal G}_0$
and $\dt=\min\{\ep/8, \dt_0'/2, d/2^{10}\}.$
We may further  assume, \wilog, that
every element of $\mathcal G$ has norm at most one. 


Let $\dt_1>0$ (in place of $\delta$) be as assured by Lemma \ref{0.6} for ${\mathcal G}$, $\dt$ (in place of $\ep_1$), and
$\sigma/64$ (in place of $\ep_2$) and for ${\cal P}_0.$
We may assume that $\delta_1\leq  \dt.$


Let $\dt_3>0$ (in place of $\dt$) be as assured by Lemma \ref{0.7}  for ${\cal G},$ $\dt_1/8$ (in place of $\ep_1$)
and
$\min\{\dt_1/32,  \sigma/256\}$ (in place of $\ep_2$) (and for ${\cal P}_0$).

Let ${\cal P}_1$ (in place of ${\cal P}$) and
$\dt_2>0$ (in place of $\dt$)  be as assured by Lemma \ref{T06} for
$\dt_1/8$ (in place of $\ep$),
$\min\{\dt_1/32,\sigma/256\}$ (in place of $\sigma$), and ${\cal G}$
(in place of ${\cal F}$).
Replacing $\mathcal P$ and $\mathcal P_1$ by their union, we may assume that ${\cal P}={\cal P}_1.$
Note that we still use notation ${\cal P}_0$ for the related set of projections.


{{By
Lemma 2.8 and  Lemma 2.9 of \cite{EN-K0-Z},}}
there are  unital positive linear maps
$$
{{\gamma: \Aff( T(\tilde A))\to \Aff( T({\tilde C_{n_1}}))}}
$$ 
for some $n_1\ge 1$ such that
\begin{equation}\label{TT-1}
\|(\imath_{n_1, \infty})_{\Aff}\circ{{\gamma}}(\hat{f})-\Gamma^\sim_{\Aff}(\hat{f})\|<
{{\min\{\sigma/128, \delta_2, \delta_3/2\},}}
\quad f\in {\mathcal F}\cup {\mathcal P}_0,
\end{equation}
{{(recall {{that}} $\hat{f}$ is the element in $\Aff( T(\tilde A))$ corresponding to $f\in A_+\subset \tilde{A}_+$).}}

 We may assume, \wilog,
that there are projections  $p_1',p_2',...,p_s'\in M_m({\tilde{C_{n_1}}})$ such that
${{\Gamma^{\sim}}}([p_i])=\imath_{n_1,\infty}([p_i']),$ $i=1,2,...,s.$
To simplify notations, assume that
$n_1=1.$
Let ${\bar G_0}$ denote the subgroup of $K_0({\tilde A})$
generated by ${\mathcal P},$ and $G_0=K_0(A)\cap {\bar G}_0.$  Since ${\bar G}_0$ is free abelian,
there is a homomorphism $\Gamma': {\bar G}_0\to K_0(\tilde C_1)$
such that
\begin{equation*}
(\imath_{1, \infty})_{*0}\circ \Gamma'{{|_{G_0}}}=\Gamma|_{G_0}\andeqn (\imath_{1, \infty})_{*0}\circ \Gamma'=\Gamma^\sim |_{{\bar G}_0}.
\end{equation*}
We may assume (since $\mathcal P$ is a basis for ${\bar G}_0$) that
\begin{equation}\label{lft-proj}
\Gamma'([p_i])=[p_i'],\quad i=1,2,...,s.
\end{equation}
Since the pair $(\Gamma_{\Aff}^\sim , \Gamma^\sim |_{K_0(\tilde A)})$ is compatible, as a consequence of \eqref{TT-1} and \eqref{lft-proj} we have
\begin{equation}\label{TT-3}
\|\widehat{p_i'}- {{\gamma}}(\widehat{p_i})\|_{\infty}<\min\{\delta_2, \delta_3/2\},
\quad i=1,2,....,s.
\end{equation}
Write
$$C_1=(\psi_0, \psi_1, Q^r, Q^l)=\{(f,a)\in {\rm C}([0,1], Q^r)\oplus Q^l: f(0)=\psi_0(a)\andeqn f(1)=\psi_1(a)\},$$
where $\psi_0, \psi_1: Q^l\to Q^r$ are homomorphisms.
Note  {{that}} since we assume that $K_0(C_1)_+=\{0\},$ $C_1$ is {{stably projectionless}}.

Set $1^{\bar r}:=\id_{Q^r}$ and $1^{\bar l}:=\id_{Q_l}.$
Define $\psi_0^\sim: Q^l\oplus \C\to Q^r$ by
$\psi_0^\sim(a,{{c}})=\psi_0(a)+{{(1^{\bar r}-\psi_0(1^{\bar l}))}}c$
for all $(a,c)\in Q^l\oplus \C$ ($a\in Q^l$ and $c\in \C$), and
define $\psi_1^\sim: Q^l\oplus \C\to Q^r$ by
$\psi_1^\sim(a,c)=\psi_1(a)+{{(1^{\bar r}-\psi_1(1^{\bar r}))}}c$
for all $(a, c)\in Q^l\oplus \C.$ It is understood that if $\psi_0$ is unital, $\psi_0^\sim=\psi_0,$ and
if $\psi_1$ is unital, $\psi_1^\sim=\psi_1.$
We then identify
\beq\nonumber
{\widetilde C}_1=(\psi_0^\sim, \psi_1^\sim, Q^r, Q^l)=\{(f,b)\in {\rm C}([0,1], Q^r)\oplus (Q^l\oplus \C):
f(0)=\psi_0^\sim(b)\andeqn f(1)=\psi_1^\sim(b)\}.
\eneq
%
Denote by $$\pi_{\mathrm{e}}: C_1 \to Q^l,\ (f, a)\mapsto a,\andeqn
\pi_{\mathrm{e}}^\sim: {\widetilde C}_1\to Q^l\oplus \C,\, (f,b)\to b$$
the canonical quotient map, and by  $j: C_1\to {\rm C}([0,1], Q^r)$ the canonical {{maps}}
$$j((f,a))=f, \quad (f,a)\in {{C_1.}}$$ 

For convenience, in what follows,
we also consider ${\tilde C}_1\otimes Q.$ We will continue to
use $\psi_0^\sim$ for the extension $\pi_0^\sim(a, x):=\psi_0(a)+{{(1^{\bar r}-\psi_0(1^{\bar l}))}}x$
for all $(a,x)\in Q^l\oplus Q$ ($a\in Q^l$ and $x\in Q$), and
define $\psi_1^\sim: Q^{l+1}\to Q^r$ by
$\psi_1^\sim(a,{{x}})=\psi_1(a)+{{(1^{\bar r}-\psi_1(1^{\bar r}))}}x$
for all $(a, x)\in Q^{l+1}.$
We also identify
$$
{\widetilde{C_1}}\otimes Q=\{(f,b)\in {\rm C}([0,1], Q^r)\oplus Q^{l+1}:
f(0)=\psi_0^\sim(b)\andeqn f(1)=\psi_1^\sim(b)\}.
$$
We also continue to {{write}} $\pi_e^\sim$ for the extension $\pi_e^\sim: {\widetilde{C}_1}\otimes Q\to Q^{l+1}.$
{{Denote by 
$\gamma^*:  T({\tilde C}_1)\to  T(\tilde A)$  the continuous affine {{map}}
dual to $\gamma$}}.
Let $\pi_{\C}^{C_1}: \widetilde{C}_1\to \C$  be the quotient map and $\tau_\C^{C_1}$ be the trace of $\widetilde{C}_1$
which factors through $\C.$
We will also {{write}} $\pi_\C^{C_1}$ for the extension $\pi_\C^{C_1}: \widetilde{C}_1\otimes Q\to Q$ and
$\tau_\C^{C_1}$ for the tracial state of ${\widetilde{C}}_1\otimes Q$ vanishing on $C_1.$

{{It follows from (\ref{TT-1}) {{that}}
\beq\label{TT-1+Aug29}
|\gm^*(\tau_\C^{C_1})(a)|<\min\{\sigma/128, \dt_2, \dt_3/2\}~~\mbox{for all}~~a\in {{{\cal F}.}}
\eneq}}
Note that, for any $\tau\in  T({\widetilde{C_1}}),$
$\tau=s\tau_\C+(1-s)\tau',$ where $\tau'\in  T(C_1)$ and $0\le s\le 1.$


Denote by  $\theta_1,\theta_2,...,\theta_l$  the extreme tracial states of $C_1$ factoring through $\pi_{\mathrm{e}}: C_1 \to Q^l$.
Put $\theta_{l+1}=\tau_\C.$

For any finite subset ${\cal G}'\supset {\cal G},$  and $0<\eta<\min\{\dt_1/8, \dt_3/8\},$
by the assumption $ T(A)= T_{\mathrm {qd}}(A),$ there exists a unital
${\mathcal G}'$-$\eta$-multiplicative completely positive {{linear}} map $\Phi: {\tilde A}\to Q^{l+1}$
(in fact, we can define each component $\pi_j\circ \Phi: {\tilde A}\to Q $ of $\Phi$ separately)
such that
\begin{equation}\label{TT-4}
|\mathrm{tr}_j\circ \Phi(a)-{{\gamma^*}}(\theta_j)(a)|< \min\{13\delta_1/32, \delta_3/4, \sigma/32\},\quad a\in {\mathcal G}\cup {\cal P}_0,\ j=1,2,..., l,l+1.
\end{equation}
where $\mathrm{tr}_j$ is the tracial state supported on the $j$th direct summand of $Q^l,$
$j=1,2,...,l,$ and ${\rm tr}_{l+1}$ is the tracial state of $\C$
(recall also that we also {{write}} $\mathrm{tr}\circ \Phi$
for $\mathrm{tr}\otimes {\rm Tr}_m\circ(\Phi\otimes \id_m),$ where
${\rm Tr}_m$ is the non-normalized trace on $M_m$).
We also assume that $\Phi|_{A}$ maps $A$ to $Q^l$ (namely, $\pi_{l+1}\circ\Phi: {\tilde A}\to Q $ can be defined to be the homomorphism taking $A$ to $0\in Q$ and $1_{\tilde A}$ to $1\in Q$)
and $\Phi(1_{\tilde A})=(\overbrace{1,1,...,1}^{l+1}).$
We may also assume that
\beq\label{TT-4n}
\mathrm{tr}_j(f_{1/2}(\Phi(e)))\ge 63d/64,\,\,j=1,2,...,l.
\eneq
Moreover,
we may also assume that
\begin{equation}\label{pert-p}
|\mathrm{tr}_j([\Phi(p_i)]) - \mathrm{tr}_j(\Phi(p_i))| < \delta_3/4,\quad i=1, 2, ..., s,\ j=1, 2, ..., l+1.
\end{equation}
Set
\beq\nonumber
&&\D_0:=(\pi_{\mathrm{e}})_{*0}(K_0(C_1))=\ker((\psi_0)_{*0}-(\psi_1)_{*0})\subseteq \Q^{l}\andeqn\\
&&\D:=(\pi_{\mathrm{e}}^\sim)_{*0}(K_0({\tilde C_1}\otimes Q))=\ker((\psi_0^\sim)_{*0}-(\psi_1^\sim)_{*0})\subseteq \Q^{l+1}.
\eneq
It follows from 
\eqref{TT-4} that 
\begin{equation}\label{trace-infty}
|\tau(\Phi(a)) -  (\pi_{\mathrm{e}}^\sim)_{\Aff}(\gamma (\hat{a}))(\tau)| < \min\{13\delta_1/32, \delta_3/4, \sigma/32\},\quad a\in \mathcal G,\ \tau\in  T(Q^{l+1}),
\end{equation}
where $(\pi_{\mathrm{e}}^\sim)_{\Aff}: \Aff( T({\widetilde{C_1}}\otimes Q)) \to \Aff( T(Q^{l+1}))$ is the map induced by $\pi_{\mathrm{e}}^\sim$.
By \eqref{trace-infty} for $a\in\mathcal P_0\,$
together with \eqref{pert-p} and \eqref{TT-3},
\begin{equation*}
|\tau([\Phi(p_i)]) - \tau\circ (\pi_{\mathrm{e}})_{*0}\circ \Gamma'([p_i])|<\delta_3,\quad\tau\in  T(Q^{l+1}),\ i=1, 2, ...,s.
\end{equation*}
Therefore, applying Lemma \ref{0.7}, with $r_i=[\Phi(p_i)]-(\pi_{\mathrm{e}})_{*0}\circ \Gamma'([p_i]){{\in \Q^{l+1}}}${{(note that $|r_{ij}|<\dt_3$ for $j=1,2,\cdots,l+1$ and $i=1,2,\cdots, s$)}}, we obtain ${\mathcal G}$-$\dt_1/8$-multiplicative completely positive {{linear}} maps
$\Sigma_1, \mu_1 : {\tilde A}\to Q^{l+1}$, with ${{\Sigma_1}}({{1_{\tilde A}}})=\mu_1({{1_{\tilde A}}})$ a projection, such that
\begin{equation}\label{pert-1-1}
 \tau(\Sigma_1(1_{\tilde A}))<\min\{\dt_1/32,\sigma/256\},\quad \tau\in  T(Q^{l+1}),
\end{equation}
\begin{equation}\label{pert-1-2}
[\Sigma_1({\mathcal P}_0)]\subseteq \D,\quad\mathrm{and}
\end{equation}
\begin{equation} \label{pert-1-3}
{[}\Sigma_1(p_i){]}-[\mu_1(p_i)]{{=r_i}}=[\Phi(p_i)]-(\pi_{\mathrm{e}})_{*0}\circ \Gamma'([p_i]),\quad  i=1,2,...,s.
\end{equation}

Consider the (unital) direct sum map
\begin{equation}\label{defn-Phi-P}
\Phi':=\Phi\oplus\mu_1: {\tilde A}\to (1\oplus\Sigma_1(1_{\tilde A}))\mathrm{M}_2(Q^{l+1})(1\oplus\Sigma_1(1_{\tilde A})).
\end{equation}
Note that $\Phi'$, like $\mu_1$ and $\Phi$, is $\mathcal G$-$\delta_1/8$-multiplicative. 
It follows from \eqref{pert-1-3} that
{{\beq \label{TT-7-1}
&&\hspace{-0.4in}[\psi_0(\Phi'(p_i))] = (\psi_0^\sim)_{*0}([\mu_1(p_i)]+[\Phi(p_i)])=(\psi_0^\sim)_{*0}({[}\Sigma_1(p_i){]}+(\pi_{\mathrm{e}}^\sim)_{*0}\circ \Gamma'([p_i]))
\andeqn\\
\label{TT-7-2}
&&\hspace{-0.4in}{[}\psi_1(\Phi'(p_i)){]}= (\psi_1^\sim)_{*0}([\mu_1(p_i)]+[\Phi(p_i)])=(\psi_1^\sim)_{*0}({[}\Sigma_1(p_i){]}+(\pi_{\mathrm{e}}^\sim)_{*0}\circ \Gamma'([p_i]))
\eneq}}
 {{for $i=1, 2, ..., s.$}}
It follows from \eqref{TT-7-1} and \eqref{TT-7-2}, in view of \eqref{pert-1-2} and the fact ({{using}} \eqref{lft-proj}) that $(\pi_{\mathrm{e}})_{*0}\circ \Gamma'([p_i]))\in (\pi_{\mathrm{e}}^\sim )_{*0}(K_0({\tilde{C_1}}\otimes Q))=\D,$
that $[\Phi'(p_i)]\in\mathbb D$, $i=1, 2, ..., s$, i.e.,
\beq\label{TT-8-nouse}
[\psi_0(\Phi'(p_i))]={[}\psi_1(\Phi'(p_i)){]},\quad i=1,2,...,s.
\eneq

Set $B={\rm C}([0,1], Q^r),$ and (as before) write  $\pi_t: B\to  Q^r$ for  the point evaluation at $t\in [0,1].$
Since ${{1_{\tilde A}}}\in\mathcal P_0$, by \eqref{pert-1-2}, $[\Sigma_1(1_{\tilde A})]\in\mathbb D.$ {{Hence}}  there is
a projection $e_0\in B$ such that $\pi_0(e_0)=\psi_0^\sim (\Sigma_1(1_{\tilde{A}}))$ and $\pi_1(e_0)=\psi_1^\sim(\Sigma_1(1_{\tilde A})).$
It then follows from \eqref{pert-1-1} (applied just for $\tau$ factoring through {{$\psi_0^\sim$}}---alternatively, for $\tau$ factoring through {{$\psi_1^\sim$}}) that
\begin{equation}\label{small-trace-n-2}
\tau(e_0) < \min\{\dt_1/32,\sigma/256\},\quad \tau\in T(B).
\end{equation}

Let $j^*:  T(B)\to  T({\widetilde C_1})$ denote the continuous affine map dual to the canonical unital map $j: {\widetilde C_1}\to B.$
Let $\gamma_1:  T(B)\to T(\widetilde{A})$ be defined by $\gamma_1:={{\gamma^*}}\circ j^*$, and let $\kappa: {\bar G}_0\to K_0(B)$
be defined by $\kappa:=j_{*0}\circ \Gamma'.$
Then, by (\ref{lft-proj}) and (\ref{TT-3}),  {{for all $\tau\in  T(B),$}}
\beq\label{TT-8+}
|\tau(\kappa([p_i]) - \gamma_1(\tau)(p_i)| & = & | \tau(j_{*0}(\Gamma'([p_i])) - ({{\gamma^*}}\circ j^*)(\tau)(p_i)| \nonumber \\
 & = & |j^*(\tau)([p_i']) - {{\gamma}}(\widehat{p_i})(j^*(\tau)) |<\dt_2,\,\,\,1\le i\le k.
\eneq
%

The estimate \eqref{TT-8+} ensures that we can apply Lemma \ref{T06} with $\kappa$ and $\gamma_1$
(note that $\Gamma'([1_{\tilde A}])=[1_{{\tilde{C_1}}}]$ and hence $\kappa([1_{\tilde A}])=[1_{B}]$) 
{{to obtain}} a  ${\mathcal G}$-$\dt_1/8$-multiplicative completely positive {{linear}} map
$\Psi': A\to B$ such that
\begin{equation}\label{TT-9-nouse-1}
|\tau\circ \Psi'(a)-\gamma_1(\tau)(a)|<\min\{\dt_1/32, \sigma/256\},\quad a\in {\mathcal G},\ \tau\in  T(B).
\end{equation}

Let $\Psi^{'\sim}: {\tilde A}\to B$ {{be}} the unitization of $\Psi'.$
Amplifying $\Psi^{'\sim}$ slightly (by first identifying $Q^r$ with $Q^r\otimes Q$ and then considering
$H_0(f)(t)=f(t)\otimes(1+e_0(t))$ for $t\in[0, 1]$), we obtain a unital $\mathcal G$-$\delta_1/8$-multiplicative  completely positive {{linear}} map $\Psi: {\tilde A} \to (1\oplus e_0)\mathrm{M}_2(B)(1\oplus e_0)$
such that (by \eqref{TT-9-nouse-1} and \eqref{small-trace-n-2}), for all $a\in {\cal G}$ {{and $\tau\in T(B)$,}}
\begin{equation}\label{TT-9-use-n-1}
|\tau\circ \Psi(a)-\gamma_1(\tau)(a)| < 2\min\{\dt_1/32,\sigma/256\} =\min\{\dt_1/16, \sigma/128\}.
\end{equation}
By \eqref{14f12(e)-n0},  as $\gamma_1(\tau)\in  T(A)$ {{and $\sigma<d/2^{10}$,}}
\beq\label{12f12(e)-n10}
\tau(\Psi(f_{1/2}(e)))\ge d(2^9-1)/2^9\rforal \tau\in  T(B).
\eneq

\noindent
Note that, for any element $a \in C_1$,
\begin{equation}\label{const-trace}
\tau(\psi_0(\pi_{\mathrm{e}}(a)))=\tau(\pi_0(j(a)))\quad\mathrm{and}\quad \tau(\psi_1(\pi_{\mathrm{e}}(a))) = \tau(\pi_1(j(a))),\quad \tau\in T(Q^r).
\end{equation}
(Recall that $j: {\tilde C}_1 \to B$ is the canonical map.)
Therefore (by \eqref{const-trace}), for any $a\in\mathcal G,$
\begin{eqnarray}\label{trace-infty-1}
| \tau(\psi_0(\Phi(a))) - \gamma(\hat{a})(\tau\circ\pi_0\circ j) |&=&| \tau(\psi_0(\Phi(a))) -  \gamma(\hat{a})(\tau\circ\psi_0\circ\pi_{\mathrm{e}}) | \nonumber \\
&=& | \tau\circ\psi_0(\Phi(a))) - (\pi_{\mathrm{e}})_{\Aff}(\gamma(\hat{a}))(\tau\circ\psi_0)| \nonumber \\
& < & \min\{13\delta_1/32, \sigma/32\} \quad\quad \textrm{(by \eqref{trace-infty})}
\end{eqnarray}
{{for all $\tau\in T(Q^r).$}} The same argument shows that
\begin{equation}\label{trace-infty-1-1}
| \tau(\psi_1(\Phi(a))) - \gamma(\hat{a})(\tau\circ\pi_1\circ j) |<\min\{13\delta_1/32, \sigma/32\},\quad a\in\mathcal G,\ \tau\in T(Q^r).
\end{equation}


Then, for any $\tau\in T(Q^r)$ and any $a\in\mathcal G$, we have
\begin{eqnarray}\label{trace-mt-boundary}
&&\hspace{-0.4in} |\tau\circ \psi_0\circ \Phi'(a)-\tau\circ \pi_0\circ \Psi(a)|
= | \tau\circ\psi_0(\Phi(a)\oplus\mu_1(a)) -  \tau\circ \pi_0\circ \Psi(a) | \nonumber \\
&&<
|\tau\circ \psi_0(\Phi(a)\oplus\mu_1(a))-\gamma_1(\tau\circ\pi_0)(a)| +\min\{\dt_1/16,\sigma/128\}
 \quad\textrm{(by \eqref{TT-9-use-n-1})}\nonumber \\
&& <  |\tau\circ \psi_0(\Phi(a))-\gamma_1(\tau\circ\pi_0)(a)| + \min\{3\dt_1/32, 3\sigma/256\}
\,\,\,\quad\quad\quad\textrm{(by  {{{\eqref{pert-1-1}}}})}\nonumber \\ \nonumber
&&= |\tau\circ \psi_0(\Phi(a))-  \gamma(\hat{a})(\tau\circ\pi_0\circ j) | + \min\{3\dt_1/32, 3\sigma/256\}  \\ \nonumber
&&<  \min\{13\delta_1/32, \sigma/32\} + \min\{3\dt_1/32, 3\sigma/256\} \quad\quad\quad\quad\quad\quad\quad { {\textrm{(by \eqref{trace-infty-1})}}} \\ \label{trace-mt-boundary-n-1}
&& \leq 13\delta_1/32 + 3\delta_1/32= \delta_1/2. 
\end{eqnarray}
The same argument, using \eqref{trace-infty-1-1} instead of \eqref{trace-infty-1}, shows that
\begin{eqnarray}\label{trace-mt-boundary-n-1-1}
 |\tau\circ \psi_1\circ \Phi'(a)-\tau\circ \pi_1\circ \Psi(a)| & < & \min\{13\delta_1/32, \sigma/32\} + \min\{3\dt_1/32, 3\sigma/256\} \\
 & \leq & \delta_1/2 {{\rforal  \tau\in T(Q^r)\andeqn  a\in\mathcal G.}} \nonumber
\end{eqnarray}
(\eqref{trace-mt-boundary-n-1} and \eqref{trace-mt-boundary-n-1-1}---the $\sigma$ estimates---will be used later to verify \eqref{TT-13+1} and (\ref{TT-13+2}).)

\noindent
Noting that $\Psi$ and $\Phi'$ are $\delta_1/8$-multiplicative on $\{{{1_{\tilde A}}}, p_1, p_2, ..., p_s\}$,
we may assume ($1\le i\le s$)
$$ |\tau([\Psi(p_i)]) - \tau(\Psi(p_i)) |<\delta_1/4\quad\mathrm{and}\quad |\tau([\Phi'(p_i)]) - \tau(\Phi'(p_i)) |<\delta_1/4
\rforal \tau\in   T(Q^r).$$
Combining these inequalities with  \eqref{trace-mt-boundary-n-1} and \eqref{trace-mt-boundary-n-1-1},
we have
\beq\label{TT-11}
|\tau([\pi_0\circ \Psi(p_i)]) - \tau([\psi_0\circ \Phi'(p_i)])|<\dt_1,\quad i=1,2,...,s,\ \tau\in T(Q^r).
\eneq
Therefore (in view of \eqref{TT-11}), applying Lemma \ref{0.6} with {{$r'_i=[\pi_0\circ \Psi(p_i)] - [\psi_0\circ \Phi'(p_i)]\in \Q^r$}}, we obtain
unital ${\mathcal G}$-$\dt$-multiplicative completely positive {{linear}} maps
$\Sigma_2: {\tilde A}\to Q^{l+1}$ and $\mu_2: {\tilde A}\to Q^r$, taking $1_{\tilde A}$ into projections,
such that
\begin{equation}\label{prop-2-n-1}
[\psi_k\circ \Sigma_2(1_{{\tilde{A}}})]=[\mu_2(1_{{\tilde{A}}})],\quad k=0, 1,
\end{equation}
\begin{equation}\label{prop-2-n-2}
{[} \Sigma_2(\mathcal P){]}\subseteq  (\pi_{\mathrm{e}})_{*0}(K_0({\tilde C}_1))=\D,\quad i=1, 2, ..., s,
\end{equation}
\begin{equation}\label{prop-2-3}
 \tau(\Sigma_2(1_{{\tilde{A}}}))< \sigma/64,\quad \tau \in  T(Q^{l+1}), 
\end{equation}
and, taking {{\eqref{prop-2-n-2} into account,}}
\begin{equation}\label{prop-2-4}
{[}\psi_0\circ \Sigma_2(p_i){]}-[\mu_2(p_i)]  =  {[}\psi_1\circ \Sigma_2(p_i){]}-[\mu_2(p_i)]{{=r'_i}}
 =  [\pi_0\circ \Psi(p_i)]-[\psi_0\circ \Phi'(p_i)],
\end{equation}
where $i=1, 2, ..., s.$  It should be also noted that, since $\tau\circ \psi_k\in  T(Q^{l+1}),$ $k=0,1,$
\beq\label{14mu2-1}
\tau(\mu_2(1_{\tilde A}))<\sigma/64\rforal \tau\in  T(Q^r).
\eneq

Consider the four $\mathcal G$-$\delta$-multiplicative direct sum maps (note that $\Phi'$ and $\Psi$ are $\mathcal G$-$\delta_1/8$-multiplicative, and $\delta_1\leq 8\delta$), from ${\tilde A}$ to $\mathrm{M}_3(Q^r)$,
\beq\label{14f12(e)-21}
&&\Phi_0:=(\psi_0\circ \Phi')\oplus (\psi_0\circ \Sigma_2),\quad \Phi_1:=(\psi_1\circ \Phi')\oplus (\psi_1\circ \Sigma_2)\andeqn\\\label{14f12(e)-22}
&&\Psi_0:=(\pi_0\circ \Psi) \oplus \mu_2,\quad\quad \quad\,\,\,\,\,\,\Psi_1:=(\pi_1\circ \Psi) \oplus \mu_2.
\eneq
%
%
%
We then have that, for each $i=1, 2, ..., s$,
\begin{eqnarray*}
[\Psi_0(p_i)]-[\Phi_0(p_i)] & = & ([(\pi_0\circ \Psi)(p_i)] + [\mu_2(p_i)]) - ([(\psi_0\circ \Phi')(p_i)] + [(\psi_0\circ \Sigma_2)(p_i)])  \\
& = & ([(\pi_0\circ \Psi)(p_i)] - [(\psi_0\circ \Phi')(p_i)] ) - ([(\psi_0\circ \Sigma_2)(p_i)] - [\mu_2(p_i)] )\\
& =& 0 \quad\quad \textrm{(by \eqref{prop-2-4})}, \\
\andeqn\\
&& \hspace{-1.6in}[\Psi_1(p_i)]-[\Phi_1(p_i)] =
 ([(\pi_1\circ \Psi)(p_i)] + [\mu_2(p_i)]) - ([(\psi_1\circ \Phi')(p_i)] + [(\psi_1\circ \Sigma_2)(p_i)]) \\
& &\hspace{-1in}=  ([(\pi_1\circ \Psi)(p_i)] - [(\psi_1\circ \Phi')(p_i)] ) - ([(\psi_1\circ \Sigma_2)(p_i)] - [\mu_2(p_i)] ) \\
& &\hspace{-1in}=  ([(\pi_0\circ \Psi)(p_i)] - [(\psi_1\circ \Phi')(p_i)] ) - ([(\psi_1\circ \Sigma_2)(p_i)] - [\mu_2(p_i)] )\quad \textrm{($\pi_0$ and $\pi_1$ are homotopic)} \\
&&\hspace{-1in}{{= ([(\pi_0\circ \Psi)(p_i)] - [(\psi_0\circ \Phi')(p_i)] ) - ([(\psi_1\circ \Sigma_2)(p_i)] - [\mu_2(p_i)] )=0}}
\quad \textrm{(by \eqref{TT-8-nouse})}.
\end{eqnarray*}
Note also that, by construction,
\beq\label{14unital}
\Psi_i(1_{\tilde A})=\Phi_i(1_{\tilde A})=1\oplus\pi_i(e_0),\quad i=0, 1.
\eneq

Summarizing the calculations in the preceding paragraph, we have
\begin{equation}\label{TT-13}
[\Phi_i]|_{\mathcal P}=[\Psi_i]|_{\mathcal P},\quad i=0,1.
\end{equation}
On the other hand, for any $a\in\mathcal F\subseteq\mathcal G$ and any $\tau\in T(Q^r)$, we have
\beq\nonumber
|\tau(\Phi_0(a)) - \tau(\Psi_0(a))| & = & |\tau ((\psi_0\circ \Phi')(a)\oplus (\psi_0\circ \Sigma_2)(a)) - \tau((\pi_0\circ \Psi)(a) \oplus \mu_2(a))| \\\nonumber
 & < & |\tau ((\psi_0\circ \Phi')(a)) - \tau((\pi_0\circ \Psi)(a) )| + \sigma/32\quad\quad\textrm{(by \eqref{prop-2-3})} \\\nonumber
 & < & \min\{13\delta_1/16, \sigma/32\} + \min\{3\dt_1/16, 3\sigma/256\} +\sigma/32  \quad\quad\textrm{(by \eqref{trace-mt-boundary-n-1})}\\\label{TT-13+1}
 &\leq & 5\sigma/64.
\eneq
The same argument, using \eqref{trace-mt-boundary-n-1-1} instead of \eqref{trace-mt-boundary-n-1},  also shows that
\beq\label{TT-13+2}
|\tau(\Phi_1(a)) - \tau(\Psi_1(a))| < 5\sigma/64, \quad a\in\mathcal F,\ \tau\in T(Q^r).
\eneq

Since ${{1_{\tilde A}}} \in\mathcal P$, by \eqref{prop-2-n-2}, $[\Sigma_2({{1_{\tilde A}}})]\in\mathbb D$, and so there is  a projection $e_1\in B$ such that $\pi_0(e_1)=\psi_0(\Sigma_2({{1_{\tilde A}}}))$
and $\pi_1(e_1)=\psi_1(\Sigma_2({{1_{\tilde A}}})).$
It then follows from \eqref{prop-2-3} (applied just for $\tau$ factoring through $\psi_0$---alternatively, for $\tau$ factoring through $\psi_1$) that
\begin{equation}\label{small-2-n-proj}
\tau(e_1) < \sigma/64,\quad \tau\in T(B).
\end{equation}
Set $E_0'=1\oplus \pi_0(e_0)\oplus \pi_0(e_1),$ $E_1'=1\oplus \pi_1(e_0)\oplus \pi_1(e_1),$
and  $D_0=E_0' \mathrm{M}_2(Q^r)E_0',$  $D_1=E_1'\mathrm{M}_2(Q^r)E_1'.$
 We estimate that, using
 \eqref{12f12(e)-n10}, \eqref{14mu2-1},
 \eqref{14f12(e)-22}, \eqref{14mu2-1},
\beq\label{14f12(e)-12}
&&\tau_0(\Psi_0(f_{1/2}(e)){{)}}\ge 63d/64\,\,\,{\rm for}\,\,\, \tau_0\in T(D_0)\andeqn\\
&& \tau_1(\Psi_1(f_{1/2}(e)){{)}}\ge 63d/64\,\,\,{\rm for}\,\,\tau_1\in T(D_1).
\eneq
Then, by \eqref{TT-13+1} and \eqref{TT-13+2},
\beq\label{14f12(2)-30}
&&\hspace{-0.3in}\tau_0(\Phi_0({ {f_{1/2}(e)}}))\ge 62d/64\,\,\,
{\rm for}\,\, \tau_0\in T(D_0)
\andeqn\\
 &&\hspace{-0.3in}\tau_1(\Phi_1(f_{1/2}(e)))\ge 62d/64\,
 \,\,\,{\rm for}\,\,\tau\in T(D_1).
\eneq
 By the choice of ${\cal G}_1,$  $\dt_0$, and \eqref{14n20526-n1}, we have
 \beq\label{14f12(e)-31}
 &&\hspace{-0.3in}\tau_0(f_{1/2}(\Phi_0(e)))\ge 3d/4\,\,\,
 ,\,\,\tau_0(f_{1/2}(\Psi_0(e)))\ge 3d/4\,\,\,
{\rm for}\,\, \tau_0\in T(D_0)
\andeqn\\\label{14f12(e)-32}
 &&\hspace{-0.3in}\tau_1(f_{1/2}(\Phi_1(e))))\ge 3d/4\,
 \,\,  \tau_1(f_{1/2}(\Psi_1(e)))\ge 3d/4
 \,\,\,{\rm for}\,\,\tau\in T(D_1).
 \eneq

Pick a sufficiently small $r'\in (0, 1/4)$  {{such}} that 
\begin{equation}\label{small-pert-hor}
\| \Psi (a)((1+2r')t-r') - \Psi(a)(t) \| <\sigma/64,\quad a\in\mathcal G,\ t\in [\frac{r'}{1+2r'}, \frac{1+r'}{1+2r'}].
\end{equation}


It follows from Lemma 7.2 of \cite{eglnkk0}
(in view of (\ref{TT-13}),  \eqref{14f12(e)-31}, \eqref{14f12(e)-32},  \eqref{TT-13+1}, (\ref{TT-13+2})
and \eqref{14unital}) that there exist unitaries $u_0\in D_0$ and
 $u_1\in D_1,$
 and  unital ${\mathcal F}$-$\ep/4$-multiplicative
 completely positive {{linear}} maps
 $L_0:{\tilde A}\to {\rm C}([-r', 0], D_0)$ and
$L_1: {\tilde A}\to  {\rm C}([1, 1+r'], D_1)$, such that
\beq
&&\pi_{-r'}\circ L_0=\Phi_0,\,\,\, \pi_0\circ L_0={\rm Ad}\, u_0\circ {{\Psi_0}}, \label{small-var-trace--2} \\
&&\pi_{1+r'}\circ L_1=\Phi_1,\,\,\, \pi_{1}\circ L_1={\rm Ad}\, u_1\circ {{\Psi_1}}, \label{small-var-trace--1} \\
&&|\tau\circ \pi_t\circ L_0(a)-\tau\circ \pi_0\circ L_0(a)|<5\sigma/32,\quad t\in [-r', 0], \label{small-var-trace-0} \\
&&|\tau\circ \pi_t\circ L_1(a)-\tau\circ \pi_1\circ L_1(a)|<5\sigma/32,\quad t\in [1,1+r'], \label{small-var-trace-1}
\eneq
where $a\in {\mathcal F},$ $\tau\in  T(Q^r),$ and (as before) $\pi_t$ is the point evaluation at $t\in [-r', 1+r'].$ 

Write $E_3=1\oplus e_0\oplus e_1\in \mathrm{M}_3(\mathrm{C}([0,1], Q^r))$ and $B_1=E_3(\mathrm{M}_3(\mathrm{C}([0,1], Q^r)))E_3.$
There exists a  unitary $u\in B_1$ such
that $u(0)=u_0$ and $u(1)=u_1.$
Consider the projection $E_4\in \mathrm{M}_3(\mathrm{C}([-r',1+r'], Q^r))$ defined by    
$E_4|_{[-r,0]}=E_0',$ $E_4|_{[0,1]}=E_3$ and $E_4|_{[1,1+r]}=E_1'.$
Set
$$B_2=E_4(\mathrm{M}_3(\mathrm{C}([-r',1+r'], Q^r)))E_4.$$
Define a
unital
$\mathcal F$-$\ep/4$-multiplicative (note that $\mathcal F\subset \mathcal G$ and $\delta\leq \ep/8$) completely positive {{linear}} map $L': {\tilde A}\to B_2$ by
\begin{equation}\label{defn-new-L}
L'(a)(t) = \left\{\begin{array}{ll}
L_0(a)(t), & t\in [-r', 0), \\
{\rm Ad}\, u(t)\circ (\pi_t\circ\Psi\oplus \mu_2)(a), & t\in [0, 1] ,\\
L_1(a)(t) & t \in (1, 1+r'].
\end{array}\right.
\end{equation}

Note that for any $a\in\mathcal G$, and any $\tau\in T(Q^r)$, by \eqref{defn-new-L}, if $t\in[0, 1]$, then
\begin{eqnarray}\label{pre-realize-trace-1}
&&\hspace{-0.4in}|\tau(\pi_t(L'(a))) - \gamma_1(\pi_t^*(\tau))(a) | \nonumber \\
 & = & |\tau({\rm Ad}\, u(t)\circ ({ {\pi_t\circ}}\Psi\oplus \mu_2)(a)) - \gamma_1(\pi_t^*(\tau))(a)| \nonumber \\
& = & |\tau(\pi_t(\Psi(a))) + \tau(\mu_2(a)) - \gamma_1(\pi_t^*(\tau))(a)| \nonumber \\
& < & |(\pi_t^*(\tau))(\Psi(a))) - \gamma_1(\pi_t^*(\tau))(a)| + \sigma/64 \,\,\,\quad\quad\textrm{(by \eqref{prop-2-n-1} and \eqref{prop-2-3})} \nonumber \\
& < & \min\{\delta_1/16, \sigma/128\} + \sigma/64 \leq 3\sigma/128 \,\,\,\quad\quad\quad\textrm{(by \eqref{TT-9-use-n-1})},
\end{eqnarray}
where $\pi_t^*:  T(Q^r) \to  T(B)$ is the dual of $\pi_t: B\to Q^r$.
Furthermore, if $t \in [-r', 0]$, then for any $a\in\mathcal F$, and any $\tau\in T(Q^r)$,
\begin{eqnarray}\label{pre-realize-trace-2}
&&\hspace{-0.4in} |\tau(\pi_t(L'(a))) - \gamma_1(\pi_0^*(\tau))(a) | \nonumber \\
&=& |\tau(L_0(a)(t)) - \gamma_1(\pi_0^*(\tau))(a)| \nonumber \\
& < & |\tau(L_0(a)(0)) - \gamma_1(\pi_0^*(\tau))(a)| + 5\sigma/32\,\quad\quad\quad\quad\quad\quad\textrm{(by \eqref{small-var-trace-0})} \nonumber \\
& = & |\tau(\Psi_0(a))- \gamma_1(\pi_0^*(\tau))(a)| + 5\sigma/32 \,\,\quad\quad \quad\quad\quad \quad\quad\textrm{(by \eqref{small-var-trace--2})} \nonumber \\
& = & |\tau((\pi_0\circ \Psi)(a) \oplus \mu_2(a))- \gamma_1(\pi_0^*(\tau))(a)| + 5\sigma/32 \nonumber  \\
& < & |\tau((\pi_0\circ \Psi)(a)- \gamma_1(\pi_0^*(\tau))(a)| + \sigma/64 + 5\sigma/32 \quad\quad\textrm{(by \eqref{prop-2-n-1} and \eqref{prop-2-3})} \nonumber \\
& < & \min\{\delta_1/16, \sigma/128\} + \sigma/64 + 5\sigma/32 <23\sigma/128 \quad\quad\textrm{(by \eqref{TT-9-use-n-1})}.
\end{eqnarray}
Again, if $t \in [1, 1+r']$, then the same argument shows that for any $a\in\mathcal F$, and any $\tau\in T(Q^r)$,
\begin{equation}\label{pre-realize-trace-3}
|\tau(\pi_t(L'(a))) - \gamma_1(\pi_1^*(\tau))(a) | < 23\sigma/128.
\end{equation}


Let us modify $L'$ to a unital map from ${\tilde A}$ to $B$.
First, let us renormalize $L'$. Consider the isomorphism $\eta: {{K_0(Q^r)=\Q^r\to K_0(Q^r)=\Q^r}}$ defined by
$$
\eta(x_1,x_2,...,x_r)=({1\over{\mathrm{tr}_1(E_3)}}x_1, {1\over{\mathrm{tr}_2(E_3)}}x_2,...,{1\over{\mathrm{tr}_r(E_3)}}x_r),
$$
for all $(x_1,x_2,...,x_r)\in \Q^r,$
where (as before) $\mathrm{tr}_k$ is the tracial state supported on
the $k$th direct summand of $Q^r$.  
Then there is a (unital) isomorphism $\phi: B_2\to {\rm C}([-r', 1+r'], Q^r)$ such that $\phi_{*0}=\eta.$
Let us replace the map $L'$ {{with}} the map $\phi\circ L'$, and still denote it by $L'$. Note that it follows from \eqref{pre-realize-trace-1}, \eqref{small-trace-n-2}, and \eqref{small-2-n-proj} that for any $t\in [0, 1]$, any $a\in\mathcal F$, and any $\tau\in T(Q^r)$,
\begin{eqnarray}\label{pre-realize-trace-2-1}
&&\hspace{-0.4in} |\tau(\pi_t(L'(a))) - \gamma_1(\pi_t^*(\tau))(a) |
<  \sigma/16 + \tau(e_0) + \tau(e_1) \nonumber \\
& < &  3\sigma/128 + \min\{\delta_1/16, \sigma/64\}+ \sigma/64 \leq 7\sigma/128.
\end{eqnarray}
The same argument, using \eqref{pre-realize-trace-2} and \eqref{pre-realize-trace-3} instead of \eqref{pre-realize-trace-1}, shows that for any $a\in\mathcal F$,
\begin{eqnarray}
|\tau(\pi_t(L'(a))) - \gamma_1(\pi_0^*(\tau))(a) | & <  & 27\sigma/128,\quad t\in[-r', 0] \label{pre-realize-trace-2-2}, \\
|\tau(\pi_t(L'(a))) - \gamma_1(\pi_1^*(\tau))(a) | & < & 27\sigma/128, \quad t\in[1, 1+r'] \label{pre-realize-trace-2-3}.
\end{eqnarray}

Now, put
\begin{equation}\label{defn-new-new-L}
L''(a)(t) = L'(a)((1+2r')t-r'),\quad t\in[0, 1].
\end{equation}
This perturbation will not change the trace very much,
as for any $a\in\mathcal F$ and any $\tau\in  T(Q^r)$,  if $t\in[0, {r'}/(1+2r')]$, then
\begin{eqnarray*}
&  &\hspace{-0.4in} |\tau(L''(a)(t)) - \tau(L'(a)(t))| \\
&=& |\tau(L'(a)((1+2r')t-r')) - \tau(L'(a)(t))|  \,\,\,\quad\quad\quad\quad\quad\textrm{(by \eqref{defn-new-new-L})} \\
& = & |\tau(L_0(a)((1+2r')t-r')) - (\tau(\Psi(a)(t)) + \mu_2(a))|  \\
& < & |\tau(L_0(a)((1+2r')t-r')) - \tau(\Psi(a)(t)) | + \sigma/64 \quad\quad\textrm{(by \eqref{prop-2-n-1} {{and \eqref{prop-2-3})}}} \\
& < & |\tau(L_0(a)(0)) - \tau(\Psi(a)(t)) | + 5\sigma/32 + \sigma/64 \,\,\,\,\,\quad\quad\quad\textrm{(by \eqref{small-var-trace-0})} \\
& = &  |\tau(\Psi_0(a)(0)) - \tau(\Psi(a)(t)) | + 11\sigma/64  \,\,\,\quad\quad \quad\quad\quad\quad\textrm{(by \eqref{small-var-trace--2})} \\
& < & \sigma/64 + 11\sigma/64=3\sigma/16   \,\, \,\,\,\quad\quad\quad\quad\quad\quad\quad\quad\quad\quad\quad\textrm{(by \eqref{prop-2-n-1} {{and \eqref{prop-2-3})}}}.
\end{eqnarray*}
Furthermore, the same argument, now using \eqref{small-var-trace-1} and \eqref{small-var-trace--1}, shows that for any $a\in\mathcal F$, $\tau\in T(Q^r)$, and $t\in[(1+r')/(1+2r'), 1]$,
\begin{equation*}
|\tau(L_0(a)((1+2r')t-r')) - (\tau(\Psi(a)(t)) + \mu_2(a))| < 3\sigma/16,
\end{equation*}
and, if $t\in [{r'}/(1+2r'), (1+r')/(1+2r')]$, then
\begin{eqnarray*}
|\tau(L'(a)((1+2r')t-r')) - \tau(L'(a)(t))| & = & | \tau(\Psi (a)((1+2r')t-r') - \Psi(a)(t))| \\
& < & \sigma/64  \quad\quad\quad\quad\textrm{(by \eqref{small-pert-hor})}.
\end{eqnarray*}
Thus
\begin{equation}\label{small-pert-n-2}
|\tau(L''(a)(t)) - \tau(L'(a)(t))| < 3\sigma/16,\quad a\in\mathcal F,\ \tau\in T(Q^r),\ t\in[0, 1].
\end{equation}
Hence, {{by \eqref{small-pert-n-2}, \eqref{pre-realize-trace-2-1}, \eqref{pre-realize-trace-2-2}, and \eqref{pre-realize-trace-2-3},}}
\begin{eqnarray}\label{2nd-pert-trace}
&&\hspace{-0.6 in}|\tau(\pi_t(L''(a))) - \gamma_1(\pi_t^*(\tau))(a)|
\leq |\tau(L''(a)(t)) - \tau(L'(a)(t))| + |\tau(L'(a)(t))- \gamma_1(\pi_t^*(\tau))(a)| \nonumber \\
&&\hspace{1.34in}< 3\sigma/16 + 27\sigma/128 =51\sigma/128
\end{eqnarray}
%
Note  that  $L''$ is a unital map from ${{\tilde{A}}}$ to $B$. It is also $\mathcal F$-$\ep$-multiplicative, since $L'$ is.  
Consider the order isomorphism $\eta': {{K_0(Q^{l+1})=\Q^{l+1}\to K_0(Q^{l+1})=\Q^{l+1}}}$
defined by
$$
\eta'(y_1,y_2,...,y_l)=(a_1y_1, a_2y_2,...,a_ly_l)\,\,\, {{\rm for}}\,\,\, (y_1,y_2,...,y_l)\in {{\Q^{l+1}}},
$$
where
\begin{equation}\label{scaling-const}
a_j={1\over{{\rm tr}_j(1\oplus \Sigma_1({{1_{\tilde A}}}) \oplus \Sigma_2({{1_{\tilde A}}}))}},\quad  j=1,2,...,{{l+1}},
\end{equation}
and (as before) $\mathrm{tr_j}$ is the tracial state supported on the $j$th direct summand of ${{Q^{l+1}}}$.
There exists a unital \hm\, ${\tilde \phi}: (1\oplus \Sigma_1({{1_{\tilde A}}}) \oplus \Sigma_2({{1_{\tilde A}}}))\mathrm{M}_3({{Q^{l+1}}})(1\oplus \Sigma_1({{1_{\tilde A}}}) \oplus \Sigma_2({{1_{\tilde A}}}))
\to {{Q^{l+1}}}$ such that
\begin{equation}\nonumber
{\tilde \phi}_{*0}=\eta'.
\end{equation}
Therefore, by the constructions of $L''$, $L'$, $L_0$, and $L_1$ (\eqref{defn-new-new-L}, \eqref{defn-new-L}, \eqref{small-var-trace--2}, and \eqref{small-var-trace--1}), we may assume
that
\begin{equation}\label{mt-cond}\psi_0\circ {\tilde \phi}\circ (\Phi'\oplus \Sigma_2)=\pi_0\circ L''\quad
\textrm{and}\quad \psi_1\circ {\tilde \phi}\circ (\Phi'\oplus \Sigma_2)=\pi_1\circ L'',
\end{equation}
replacing $L''$ {{with}} ${\rm Ad}\, w\circ L''$  for a suitable unitary $w$ if necessary.

 Define $L: {{\tilde{A}\to \tilde{C}_1}}$ by
 $L(a)=(L''(a), {\tilde \phi}( \Phi'(a)\oplus \Sigma_2(a)))$, an element of ${{\tilde{C}_1}}$ by \eqref{mt-cond}. Since $L''$ and $\tilde{\phi}\circ(\Phi'\oplus\Sigma_2)$ are unital and ${\mathcal F}$-$\ep/4$-multiplicative (since $\Phi'$ and $\Sigma_2$ are $\mathcal G$-$\delta$-multiplicative, $\mathcal F\subset \mathcal G$, and $\delta\leq\ep/8$), so also is $L$.


Moreover, for any $a\in\mathcal F$, any $\tau\in T(Q^r)$, and any $t\in (0, 1)$, it follows from \eqref{2nd-pert-trace} that
\begin{equation}\label{approx-n-1}
|\tau(\pi_t(L(a))) - \gamma_1(\pi_t^*(\tau))(a)|  <  51\sigma/128.
\end{equation}
If $\tau\in T({{Q^{l+1}}})$, then, for any $a\in\mathcal F$,
\begin{eqnarray*}
&&\hspace{-0.4in} |\tau(\pi_{\mathrm{e}}(L(a))) - \gamma^*(\pi_{\mathrm{e}}^*(\tau))(a)| =
 |\tau( {\tilde \phi}( \Phi'(a)\oplus \Sigma_2(a))) - \gamma^*(\pi_{\mathrm{e}}^*(\tau))(a)| \\
& < & |\tau(\Phi'(a)\oplus \Sigma_2(a)) - \gamma^*(\pi_{\mathrm{e}}^*(\tau))(a)|+\sigma/32 \quad\quad\textrm{(by \eqref{scaling-const}, \eqref{small-trace-n-2} and \eqref{prop-2-3})}\\
& < & |\tau(\Phi'(a)) - \gamma^*(\pi_{\mathrm{e}}^*(\tau))(a)| + 3 \sigma/64  \,\quad \quad\quad\quad\quad\textrm{(by \eqref{prop-2-3})} \\
& < & |\tau(\Phi(a)) - \gamma^*(\pi_{\mathrm{e}}^*(\tau))(a)| +  \sigma/8  \,\, \quad\quad\quad\quad\quad\quad\textrm{(by \eqref{pert-1-1})} \\
& < & \sigma/32 + \sigma/8=5\sigma/32.
\quad\quad\quad\quad \quad\quad\quad\quad\quad\quad\textrm{(by \eqref{trace-infty})}.
\end{eqnarray*}
Since each  extreme trace of ${{\tilde{C}_1}}$ factors through either the evaluation map $\pi_t$ or the canonical quotient map $\pi_{\mathrm{e}}$,  by \eqref{approx-n-1},
\begin{equation}\label{approx-n-1-1}
|\tau(L(a)) - \gamma^*(\tau)(a)|  <  51\sigma/128,\quad \tau\in  T({{\tilde{C}_1}}),\ a\in\mathcal F.
\end{equation}
{{From (\ref{TT-1+Aug29}), we have
$|\gm^*(\tau_\C^{C_1})(a)|<\sigma/128$ for all $f\in{\cal F}$.  {{Combing with \eqref{approx-n-1-1}, we have}}
\beq
|\tau_\C^{C_1}(L(a))|<52\sigma/128~~\mbox{for all}~~a\in {\cal F}.
\eneq}}
{{That is $\|\pi_{\C}^{C_1}(L(a))\|<52\sigma/128$. For each $a\in {\cal F}$, {{put}}  $a'=L(a)-\ld 1_{\tilde{C}_1}$ where $\ld=\pi_{\C}^{C_1}(L(a))\in \C$.}} Choose an element
$e_{C_1}\in C_1$ with $0\le e_{C_1}\le 1$
such that
{{\beq \|e_{C_1}a'e_{C_1}-a'\|< \sigma/128 \rforal a\in {\cal F}.
\eneq}}
{{Then}}
\beq
\|e_{C_1}L(a)e_{C_1}-L(a)\|<{{\sigma/128+52\sigma/128=53\sigma/128}}
\rforal a\in {\cal F}.
\eneq
Define  ${\bar L}: A\to C_1$ by ${\bar L}(a)=e_{C_1}L(a)e_{C_1}$ for all $a\in A.$
(Note that  ${\bar L}$ is ${\cal F}$-$\ep/2$-multiplicative.)
Therefore,
for any $a\in {\mathcal F}$ and $\tau\in T(C)$, we have
\begin{eqnarray*}
 &&\hspace{-0.4in}|\tau(\imath_{1, \infty}({{\bar{L}(a)}}))-\Gamma_{\Aff}(\hat{a})(\tau)|
 \\
 &< &
 |\tau(\imath_{1, \infty}({\bar L}(a)))-\gamma^*(\imath_{1, \infty}(\tau))(a)|+|\gamma_*(\imath_{1, \infty}(\tau))(a)-\Gamma_{\Aff}(\hat{a})(\tau)| \\
 &=&|\tau(\imath_{1, \infty}({\bar L}(a)))-\gamma^*(\imath_{1, \infty}(\tau))(a)|+|(\imath_{1, \infty})_{\Aff}(\hat{a})(\tau)-\Gamma_{\Aff}(\hat{a})(\tau)| \\
 & < & 51\sigma/128 +{{53\sigma/128}}+ \sigma/128 ={{105}}\sigma/128.
  \quad\quad\quad\quad \textrm{(by \eqref{approx-n-1-1} and \eqref{TT-1})}
 \end{eqnarray*}
Recall $f_{\ep'}(e)\in {\cal F},$ combining the above inequality with \eqref{14app-0}, one has
\beq\label{14ee-1}
\tau(\imath_{1, \infty}({{\bar{L}}}(f_{\ep'}(e))))\ge (1-\sigma/128)-{{105}}\sigma/128>1-{{106}}\sigma/128\rforal \tau\in  T(C).
\eneq
Choose a strictly positive element $e_A\in A$ such that
$e_A\ge f_{\ep'}(e)$ and put $c_1=\imath_{1, \infty}(L(e_A)).$
Then, by \eqref{14ee-1},
\beq\label{14ee-2}
\tau(c_1)>1-{{106}}\sigma/128\rforal \tau\in  T(C).
\eneq
{{Let $H_1: A\to C$ be defined by $H_1= \imath_{1, \infty}\circ \bar{L}$. Then
\beq
|\tau(H_1(a))-\Gamma_{\Aff}(\hat{a})(\tau)|<105\sigma/128~~\mbox{for all}\,\, a\in {\cal F},\, \tau \in T(C_1).
\eneq}}

Since $\mathcal F$, $\ep$, and $\sigma$ are arbitrary, in this way we obtain a sequence of  completely positive {{linear}} maps $H_n: A\to C$ such that
\beq\label{14Last-1}
&&\lim_{n\to\infty}\|H_n(ab)-H_n(a)H_n(b)\|=0\rforal  a,b\in A,\andeqn\\
\label{14Last-2}
&&\lim_{n\to\infty}\sup\{|\tau\circ H_n(a)-\Gamma_{\Aff}(\hat{a})(\tau)|: \tau\in  T(C)\}
=0\rforal a\in A,
\eneq
and {{$\tau(H_n(e_A)) \to 1 $ uniformly on $\tau$ in $T(C)$.}}

{{Recall that $C_n=C_{0,n}\otimes Q,$ where $C_{0,n}\in {\cal C}_0.$
Write $Q=\overline{\cup_{n=1}^\infty M_{n!}}$ and each $1_{M_{n!}}$ is the identity of $Q.$
Choose a subsequence $\{m(n)\}\subset \{n!\}$ such that
$C=\overline{\cup_{n=1}^\infty {\bar C}_n},$ where
${\bar C}_n:=C_{0,n}\otimes M_{m(n)},$ $n=1,2,...,$ and
$\lambda_s({\bar C}_n)\nearrow 1$ (see 5.3 of \cite{eglnkk0}).
\Wlog, we may assume that, in \eqref{14Last-1} and \eqref{14Last-1},
 $H_n: A\to {\bar C}_n\to C,$ $n=1,2,....$}}

On the other hand,  $\Gamma^{-1}: (K_0(C), T(C), r_C)\cong (K_0(A), T(A),r_A)$
gives {{an affine homeomorphism $\lambda_T: T(A)\to T(C)$
such that $\Gamma_{\Aff}(\hat{a})(\lambda_T(\tau))=\tau(a)$ for all
$a\in A_{s.a.}$ and $\tau\in T(A).$}}
Since $K_1(C)=\{0\},$
{{by Corollary 7.8 of \cite{Linalmost}, there is a sequence of injective  \hm s
$h_k': C_k\to A$ such that, for any $c\in {\bar C}_m,$
\beq
\lim_{k\to\infty}\sup_{\tau\in T(A)}\{|\tau(h_k'\circ \iota_{m,k}(c))-\lambda_T(\tau)(\iota_{m, \infty}(c))|\}=0.
\eneq}}
%
It follows  that, by an appropriate choice of  a subsequence {{$\{k(n)\}$}} and {{defining}}
$h_n:={{h_{k(n)}'\circ \iota_{n, k(n)}}},$  one obtains
that
\beq\label{14Last-3}
\lim_{n\to\infty}\sup\{|\tau\circ h_n\circ H_n(a)-\tau(a)|: \tau\in  T(A)\}
=0\rforal  a\in A.
\eneq

%
%
%
%
%

{{Note that, as shown at {{the}} beginning of the proof,
for any non-zero hereditary \SCA\, $B$ of $A$ with continuous scale, $B\otimes Q\cong B.$
So all above work for any such $B.$  {{Note
that $B\cong B\otimes Q$ is tracially approximately divisible.}}
By \eqref{14Last-1}, \eqref{14Last-2},
and \eqref{14Last-3}, applying Theorem \ref{TTTAD}
to each hereditary \SCA\, $B$ of $A$ with continuous scale,
described in Theorem \ref{TTTAD},  we conclude}}
$A\otimes Q\in {\cal D}.$
%
%
%
\end{proof}
%


%
%
%
%
%


\begin{thm}\label{TLS-n}
Let $A$ be a separable simple
 \CA\, with continuous scale.
Suppose that $A\otimes U\in {\cal D}$ for some infinite dimensional UHF-algebra $U.$
Then $A\otimes B\in {\cal D}$ for any infinite dimensional UHF-algebra $B.$
\end{thm}

\begin{proof}
Suppose that $A\otimes U\in {\cal D}.$  Then $A\otimes U$ has at least one
tracial state.
Since the map $a\mapsto a\otimes 1_U$ maps  $A$ into  $A\otimes U,$
$A$ must be stably finite. Moreover, if $\tau$ is a 2-quasitrace for $A,$ then
$\tau\otimes t_U$ is a trace  since $A\otimes U\in {\cal D}$ (see Proposition  9.1 of \cite{eglnp}), where $t_U$ is the unique tracial state on $U.$
It follows that $\tau$ is a trace. In other words, $QT(A)=T(A).$
Let $B$ be a unital infinite dimensional UHF-algebra.
Choose  a strictly positive element $e_A\in A$  with $\|e_A\|=1.$
We may assume that, as $A\otimes U$ has continuous scale,
\beq
d=\inf\{\tau(e_A\otimes 1_U): \tau\in T(A\otimes U)\}>1/2.
\eneq

Fix $\ep>0,$ a finite subset ${\cal F}\subset A\otimes B$ and $a\in (A\otimes B)_+\setminus\{0\}.$
Note that $A\otimes B$ is finite {{and}} ${\cal Z}$-stable, and has {{strict}} comparison for positive elements ({see
Cor.~4.6 of \cite{Rrzstable}}).
%
%
There is a non-zero element $a_0=a_{00}\otimes b_0\in (A\otimes B)_+$
for some $a_{00}\in A_+$ and $b_0\in B_+$
such that
$a_0\lesssim a$ in $A\otimes B.$
We may also assume that $b_0=b_{0,1,1}{{\oplus}}b_{0,1,2}\oplus b_{0,2}\in B_+$
where $b_{0,1,1},$ $b_{0,1,2}$ and $b_{0,2}$ are  mutually orthogonal nonzero positive elements.
{{Put $b_{0,1}=b_{0,1,1}\oplus b_{0,1,2}.$}}

As $A \otimes B$ is simple, there is an integer $N_0\ge 1$ such that
\beq\label{27-1}
{{e_A\otimes 1_B\lesssim N_0\la a_{00}\otimes b_{0,2}\ra .}}
\eneq
We write $B=\lim_{n\to\infty}(B_n, \psi_n),$ where $B_n=M_{R(n)}$  and
$\psi_n: B_n\to B_{n+1}$ is a unital embedding.   If $n>m,$ put
$\psi_{m,n}=\psi_{n-1} {{\circ\cdots \circ}} \psi_m: B_m\to B_n.$
{{Denote by}}
$\psi_{n, \infty}: B_n\to B$ the unital embedding induced by the inductive limit.
By Proposition 2.2 and Lemma 2.3 (b) of \cite{Rr2}, to simplify notation, without loss of generality,  {{replacing
$b_0$ by}} a smaller (in Cuntz relation) element, we may assume that {{$b_{0,1,1}, b_{0,1,2}, b_{0,2}\in B_n$}} for some
large $n.$ 
Since $B$ is simple, we may assume that $R(n)>4N_0$ for all $n.$
It follows from (\ref{27-1}) that we may assume
that the range projection of $b_0$ has rank at least two { {(as a matrix)}}.

By changing {{notations,}}
without loss of generality, we may further assume that
${\cal F}\subset A\otimes B_1$ and {{$b_{0,1,1}, b_{0,1,2},
b_{0,2}\in {B_1}.$}}

Since $A_1:=A\otimes M_{R(1)}\otimes U \in  {\cal D}$ has continuous scale,
 there are ${\cal F}$-$\ep/128$-multiplicative \cpc s $\phi: A_1\to A_1$ and  $\psi: A_1\to D_0$  for some
\SCA\, $D_0\subset A_1$ with $D_0\in {\cal C}_0,$
${{D_0}} \perp \phi(A_1),$ and
\beq\label{DNtr1div-1++}
&&\|x\otimes 1_U-(\phi(x\otimes 1_U)+\psi(x\otimes 1_U))\|<\ep/128\rforal x\in {\cal F}\cup \{e_A\otimes 1\},\\\label{DNtrdiv-2+}
&&\phi(e_A\otimes 1_U)\lesssim a_{0,1}:=a_{00}\otimes {{b_{0,1,1}}}\otimes 1_U,\\\label{DNtrdiv-4+}
&&t(f_{1/4}(\psi(e_A\otimes 1)))\ge  3/4
\rforal t\in T(D_0)
\eneq
(see also Proposition 2.10 of \cite{eglnp}).
{{Replacing $\phi$ by}} the map defined by $\phi'(x)=f_\eta(\phi(e_A\otimes 1_U))\phi(x)f_\eta(\phi(e_A\otimes 1_U))$
for all $x\in A_1$ for some sufficiently small $\eta,$ we may assume
that there is $e_0\in A_1$ such that $e_0\phi(x)=\phi(x)e_0=\phi(x)$
for all $x\in A_1,$ {{$e_0\perp D_0,$}} and $e_0\lesssim   a_{00}\otimes b_{0,1,1}
\otimes 1_U.$

Let ${\cal G}\subset D_0$ be a finite subset such that, for every $x\in {\cal F},$ there exists $x'\in {\cal G}$ such
that $\|\psi(x\otimes 1_U)-x'\|<\ep/128.$  We may also assume that ${\cal G}$ contains
a strictly positive element $e_D$ of $D_0$ with $\|e_D\|=1.$

Write $U=\overline{\cup_{n=1}^{\infty} M_{r(n)}},$ where
$\lim_{n\to\infty} r(n)=\infty$ and $M_{r(n)}\subset M_{r(n+1)}$ unitally.
{{For each $n,$ there are $a_n, b_n\in A\otimes M_{R(1)}\otimes M_{r(n)},$
$0\le a_n,\, b_n\le 1$ such that}}
\beq\nonumber
&&{{a_n\perp b_n,\,\,\,
\lim_{n\to\infty}\|a_n da_n-d\|=0\rforal d\in D_0,}}\\
&&{{\lim_{n\to\infty}\|b_n-e_0\|=0\andeqn \lim_{n\to\infty} \|b_n\phi(a)b_n-\phi(a)\|=0
\rforal  a\in A_1.}}
\eneq
{{Since $e_0\lesssim a_{00}\otimes b_{0,1}
\otimes 1_U, $  by Proposition 2.2 of \cite{Rr2}, for each $n,$ there exists $k(n)$ such
that $f_{1/n}(b_{k(n)})\lesssim a_{00}\otimes b_{0,1}
\otimes 1_U.$ Therefore, \wilog, {{replacing}} $b_n$ by $f_{1/n}(b_{k(n)}),$ we may assume
(in $A_1$)
\beq\label{Add-331-n1}
b_n\lesssim a_{00}\otimes b_{0,1,1}
\otimes 1_U.
\eneq}}
Put $C_n:=\overline{a_n(A\otimes M_{R(1)}\otimes M_{r(n)})a_n}$ and
$C_n':=\overline{b_n(A\otimes M_{R(1)}\otimes M_{r(n)})b_n}.$
Note that $C_n\perp C_n'.$

Since each $D_{0}$ is weakly semiprojective, we can choose $n_0$ large enough such that
there exists a unital \hm\, $h: D_{0}  \to A\otimes M_{R(1)}\otimes M_{r(n_0)}$ ($\subset  A\otimes M_{R(1)}\otimes U$) satisfying
\beq\label{NT-2}
\|h(x')-x'\|<\ep/64\rforal x'\in {\cal G}
\andeqn\\\label{NT-2n}
 t(f_{1/4}(h(\psi(e_A\otimes 1))))\ge 1/2
 \rforal t\in T(h(D_0)).
\eneq
Consider $\Phi': A\otimes M_{R(1)}\to A_1$ defined
by $\Phi'(a)=\phi(a\otimes 1_U)$ for $a\in A\otimes M_{R(1)}.$
Let $s_0: M_{R(1)}\otimes U\to M_{R(1)}\otimes M_{r(n_0)}$ be a \cpc\,
such that ${s_0}|_{M_{R(1)}\otimes M_{r(n_0)}}=\id_{M_{R(1)}\otimes M_{r(n_0)}}.$
Define $J:=\id_A\otimes s_0: A\otimes M_{R(1)}\otimes U\to A\otimes M_{R(1)}\otimes M_{r(n_0)}$ and
define $\Phi_0: A\otimes M_{R(1)}\to   C_{n_0}'\subset A\otimes M_{R(1)}\otimes M_{r(n_0)}$
by $\Phi_0(a)=b_n(J\circ \Phi'(a)) b_n$ for all $a\in A\otimes M_{R(1)}.$
Choose {{a larger}} $n_0$ if necessary,
we may also assume that
$\Phi_0$
is {{an}} ${\cal F}$-$\ep/64$-multiplicative \cpc\,
such that
\beq
\|\Phi_0(x)-\phi(x\otimes 1_U)\|<\ep/64\rforal x\in {\cal F}.
\eneq

We also assume that (viewing $x$ as an element in $A\otimes M_{R(1)}$)
\beq\label{LS-n10}
\|x\otimes 1_{M_{r(n_0)}}-(\Phi_0(x)+h\circ \psi(x\otimes 1_U))\|<\ep/32\rforal x\in {\cal F}.
\eneq
Note {{that}} $\Phi_0(A_1)\perp h(D_0).$
%
%
{{Moreover, by \eqref{Add-331-n1}, we have,   for all $\tau\in T(A_1),$}}
\beq\label{Add331-n2}
{{d_\tau(\Phi_0(e_A))\le d_\tau(a_{00}\otimes b_{0,1,1}\otimes 1_{M_{r(n_0)}})<d_\tau(a_{00}\otimes b_{0,1}\otimes 1_{M_{r(n_0)}}).}}
\eneq
{{Recall that $\Phi_0(e_A),\, a_{00}\otimes b_{0,1}\otimes 1_{M_{r(n_0)}}\in  A\otimes M_{R(1)}\otimes M_{r(n_0)}.$
Hence \eqref{Add331-n2}  holds for  any trace with the form $t\otimes Tr,$ where $t\in T(A\otimes M_{R(1)})$ and $Tr$
is the normalized tracial state on $M_{r(n_0)}.$}}

We may assume that $\psi_{1, n_1}: B_1\to B_{n_1}$ has
multiplicities at least $N\ge 1$  such that
\beq\label{NT-4}
{2r(n_0) R(1)\over{N}}<1.
\eneq
Write
\beq\label{NT-5}
R(n_1)=NR(1)=lr(n_0)R(1)+m,
\eneq
where $l\ge 1$ and $r(n_0)R(1)>m\ge 0$ are integers. It follows that
\beq\label{NT-5+}
{m\over{R(n_1)}} <{r(n_0)R(1)\over{R(n_1)}} < {r(n_0)\over{N}}<{1\over{2R(1)}}.
\eneq
Since $R(1)|R(n_1),$ we may write $m=m^{(r)} R(1).$
Define $\rho: M_{R(1)}\to M_{m}$ by $x\to x \otimes 1_{M_{m^{(r)}}},$ if $m>0.$
If $m=0,$ then we omit $\rho,$ or view $\rho=0.$

It follows from (\ref{NT-5+}) that, if $m\not=0,$ {{by}} viewing $M_m{ {=M_m\oplus 0_{R(n_1)-m}}} \subset M_{R(n_1)}\subset B,$
\beq\label{NT-5++}
d_\tau(e_A\otimes \rho(1_{M_{R(1)}}))<{1\over{2R(1)}}\rforal \tau\in T(A\otimes B).
\eneq	
Therefore, by \eqref{27-1} and the fact that $R(1)>4N_0,$
\beq\label{NT-5+++}
\iota(e_A\otimes \rho(1_{M_{R(1)}}) )\, \lesssim \, a_{00}\otimes b_{0,2}.
\eneq
Let $\imath_{1}: M_{r(n_0)R(1)}\to M_{lr(n_0)R(1)}$ be the embedding defined
by $a\mapsto a\otimes 1_{M_{l}}.$  Let  $\imath_{2}: M_{lr(n_0)R(1)}\to M_{R(n_1)}$ be defined
by the embedding which sends rank one projections to rank one projections.
Put $\imath_{3}=\imath_{2}\circ \imath_{1}.$  Define $\imath_{4}: A\otimes M_{R(1)}\otimes M_{r(n_0)}\to A\otimes M_{R(n_1)}$
by $\imath_{4}(a\otimes b)=a\otimes \imath_{3}(b)$ for all
$a\in A$ and $b\in M_{r(n_0)R(1)}.$
Note that, for all $a\in B_1=M_{R(1)},$  we may write (modulo an inner automorphism, and
if $m=0,$ $\rho=0$)
$$
\psi_{1,n_1}(a)=
\imath_{3}\circ \imath_{0}(a)\oplus \rho(a).
$$
Define $\imath: A\otimes B_{n_1}\to A\otimes B$ to  be the map given by $a\otimes b\mapsto a\otimes \psi_{1, \infty}(b).$

Put $E_1= \imath \circ \imath_{4}(h(D_0 )) .$ Then
$E_1\in {\cal C}_0.$
Let $s: B\to B_1$  be a \cpc\, such that $s|_{B_1}=\id_{B_1}.$
Let  $j:=\id_A\otimes s: A\otimes B\to A\otimes  B_1.$
Define $\Phi_1': A\otimes B\to A\otimes B$
by $\Phi_1':=\imath\circ  \iota_4\circ {{\Phi_0}}\circ j.$
Define $\Phi_1: A\otimes B\to A\otimes B$
by $\Phi_1(a)=\Phi_1'(a)\oplus \iota((\id_A\otimes \rho)(j(b)))$
for all $a\in A$ and $b\in B.$
Note that $\Phi_1$ is a ${\cal F}$-$\ep/2$-multiplicative map
and $\Phi_1(A\otimes B)\perp E_1.$
Put $\psi':A\otimes B\to D_0$ by
$\psi'(a\otimes b)=\psi(a\otimes s(b)\otimes 1_U)$ for $a\in A$ and $b\in B.$
Define $\Psi: A\otimes B\to E_1$ by
$\Psi:=\imath\circ \imath_4\circ h\circ \psi'.$
 Then $\Psi$ {{an}} ${\cal F}$-$\ep/2$-multiplicative \cpc.
Note that $\iota(x\otimes 1_{M_{r(n_0)}})=\iota\circ \imath_4(x)$
for all $x\in {\cal F}.$ Recall also $M_{R(n_1)}=M_{R(1)}\otimes M_N.$
Hence, by \eqref{LS-n10},  for all $x\in {\cal F}$ and (viewing $x\in A\otimes M_{R(1)}$),
{{we estimate that}}
\beq
x&=&\iota(x\otimes 1_{M_N})=\iota(x\otimes 1_{M_{lr(n_0)}})\oplus \iota(x\otimes 1_{m^{(r)}})\\
&=& \iota\circ \imath_4(\imath_0(x))\oplus \iota(\id_A\otimes \rho)(x)\\\label{LS-n20}
&\approx_{\ep/32}& (\Phi_1'(x)+\Psi(x)) \oplus \iota(\id_A\otimes \rho)(x)=\Phi_1(x)+\Psi(x).
\eneq
By \eqref{NT-2n}, {{we also have}}
\beq\label{LS-n21}
 t(f_{1/4}(\Psi(e_A\otimes 1)){{)}}\ge  1/2
 \rforal t\in T(E_1).
\eneq
By \cite{Rrzstable}, $A\otimes B$ has {{strict comparison.}}
By  the lines right below \eqref{Add331-n2}
and \eqref{NT-5+++} (viewing $a_{00}\otimes b_{0,1}\otimes 1_{M_{r(n_0)}}$  as  an element in $A\otimes M_{R(1)}\otimes M_{r(n_0)}{{\subset A\otimes B}}$), {{we conclude that}}
\beq
\Phi_1(e_A\otimes 1_B)&=&\Phi_1'(e_A\otimes 1_{B} )\oplus \iota((\id_A\otimes \rho)(e_A\otimes j(1_B)))\\
&\lesssim& \iota(a_{00}\otimes b_{0,1}\otimes 1_{M_{r(n_0)}})\oplus \iota(e_A\otimes \rho(1_{B_1}))\\
&\lesssim & (a_{00}\otimes b_{0,1})\oplus (a_{00}\otimes b_{0,2})
\lesssim  a_{00}\otimes b_0\lesssim a_0\lesssim a.
\eneq
Combining the last relation with \eqref{LS-n20} and \eqref{LS-n21},  {{we obtain}}
$A\otimes B\in {\cal D}.$
%
\end{proof}

\begin{thm}\label{Treduction}
Let $A$ be {{a}} separable simple 
{{stably projectionless}}  amenable \CA\, with continuous scale 
such that
$T(A)\not=\{0\}$ and   satisfying
the UCT.
Then $A\otimes U\in {\cal D}$  for any infinite dimensional UHF-algebra $U.$
\end{thm}

\begin{proof}
Note that $U\cong U_1\otimes U_2$ for some infinite dimensional UHF-algebras $U_1$ and 
$U_2.$ Consider  $A_1=A\otimes U_1.$  Since $U_1$ is ${\cal Z}$-stable, 
so is $A_1.$ It follows from \cite{CE}  that $A_1$ has finite nuclear dimension.
Thus, to prove the theorem, we may assume that $A$ has finite nuclear dimension.
By \cite{TWW}, every tracial state of $A$ is quasidiagonal. Then, by
{{Theorem \ref{TRangeM},
\ref{4.25},
and
 \ref{Treduction1},}} $A\otimes Q\in {\cal D}.$  Thus,  by Theorem  \ref{TLS-n},
$A\otimes U\in {\cal D}$ for any infinite dimensional UHF-algebra $U.$
\end{proof}



%
%
%

%

 \providecommand{\href}[2]{#2}

ghgong@gmail.com,

hlin@uoregon.edu
      \end{document}